\documentclass{amsart}
\usepackage{mathrsfs}
\usepackage[all]{xy}

\usepackage{mathtools}
\usepackage{mathbbol}
\usepackage{wasysym}
\usepackage{amssymb}
\usepackage{MnSymbol}
\usepackage{comment}
\usepackage{enumerate}

\usepackage{ushort}

\usepackage{changepage}
\usepackage{algorithm}
\usepackage[noend]{algpseudocode}
\makeatletter
\def\BState{\State\hskip-\ALG@thistlm}
\makeatother

\usepackage{ifpdf}
\ifpdf 
  \usepackage[pdftex]{graphicx}
  \DeclareGraphicsExtensions{.pdf,.png,.jpg,.jpeg,.mps}
  \usepackage{pgf}
\else 
  \usepackage{graphicx}
  \DeclareGraphicsExtensions{.eps,.bmp}
  \usepackage{pgf}
  \usepackage{pstricks}
\fi
\usepackage{epic,bez123}
\usepackage{wrapfig}

\usepackage{soul}
\usepackage{url}

\usepackage{tikz}
\usepackage{pst-node} 
\usepackage{tikz-cd} 

\newtheorem{thm}{Theorem}[section]
\newtheorem{prop}[thm]{Proposition}
\newtheorem{cor}[thm]{Corollary}
\newtheorem{lem}[thm]{Lemma}
\newtheorem{conj}[thm]{Conjecture}

\newtheorem{assump}{}

\newtheorem{assumpB}{}

\newtheorem{conv}[thm]{Convention}

\theoremstyle{definition}
\newtheorem{defn}[thm]{Definition}

\theoremstyle{remark}
\newtheorem{quest}[thm]{Question}
\newtheorem{rem}[thm]{Remark}
\newtheorem*{example}{Example}
\newtheorem{examples}{Examples}

\newcommand{\isom}{\textrm{Isom}} 
\newcommand{\out}{\textrm{Out}} 
\newcommand{\mcg}{\textrm{Mod}(\Sigma_g)}  

\newcommand{\Gx }{\mathscr{G} (G, S)}

\newcommand{\GP }{(G, \mathcal P)}

\newcommand{\act}{\curvearrowright}

\newcommand{\Gf}{\overline{G}_{\mathrm{F}}}
\newcommand{\pGf}{\partial_{\mathrm{F}}{G}}

\newcommand{\pG}{\Lambda(Go)}
\newcommand{\dG}{\Lambda G\Join\Lambda G}

\newcommand{\cG}{\Lambda_c{(Go)}}
\newcommand{\mG}{\Lambda_m{(Go)}}

\newcommand{\dT}{\mathcal T_g \Join\mathcal T_g}

\newcommand{\p}{\mathcal P}
\newcommand{\f}{\mathscr F}
\newcommand{\s}{\mathscr S}

\newcommand{\U}{\mathrm Y}
\newcommand{\bU}{\overline{\mathrm Y}}
\newcommand{\hU}{\partial_{\mathrm H}{\mathrm Y}}
\newcommand{\pU}{\partial{\mathrm Y}}

\newcommand{\bF}{\partial_{\mathrm F}\mathrm Y} 
\newcommand{\bG}{\partial_{\mathrm G}\mathrm Y} 
\newcommand{\bB}{\partial_{\mathrm B}\mathrm Y} 
\newcommand{\bV}{\partial_{\mathrm{Vis}}\mathrm Y} 
\newcommand{\bR}{\partial_{\mathrm{R}}\mathrm Y}

\newcommand{\bTh}{\partial_{\mathrm{Th}}\mathcal T_g}
\newcommand{\bGM}{\partial_{\mathrm{GM}}\mathcal T_g}

\newcommand{\dirac}[1]{{\mbox{Dirac}}{(#1)}}

\newcommand{\e}[1]{\omega_{#1}}

\newcommand{\HD}{\textrm{HDim}}
\newcommand{\ax}{\textrm{Ax}}
 
\newcommand{\pmf}{\mathscr {PMF}}
\newcommand{\mf}{\mathscr {MF}}
\newcommand{\minf}{\mathscr {MIN}}
\newcommand{\ue}{\mathscr {UE}} 
\newcommand{\UQT}{\mathcal {Q}^1\mathcal T_g}
\newcommand{\UQM}{\mathcal {Q}^1\mathcal M_g}
\newcommand{\M}{\mathcal {M}}
\newcommand{\T}{\mathcal {T}_g}
\newcommand{\QD}{\mathcal {Q}\mathcal T_g}
\newcommand{\ext}{\mathrm {Ext}}

\newcommand{\QT}{\mathcal C(\mathscr F)}
\newcommand{\PC}{\mathcal P_K(\mathscr F)}

\newcommand{\len }{\ell}

\usepackage[bookmarks=true, pdfauthor={YANG Wenyuan}]{hyperref}

\usepackage[margin=1.1in]{geometry}
\usepackage[normalem]{ulem}

\newtheoremstyle{query}%
{}{}
{\color{red}}
{}
{\sffamily\bfseries}{:}{12pt}
{}
\theoremstyle{query}

\newcommand{\ywy}[1]{{\color{red}{#1}}}

\begin{document}

\title[Conformal dynamics at infinity]{Conformal dynamics at infinity for groups with contracting elements}

\author{Wenyuan Yang}

\email{wyang@math.pku.edu.cn}
\thanks{Partially supported by National Key R \& D Program of China (SQ2020YFA070059) and  National Natural Science Foundation of China (No. 12131009 and No. 12326601)}






\begin{abstract}
This paper develops a   theory of conformal density at  infinity  for groups with contracting elements. We start by introducing a  class of convergence boundary encompassing  many known hyperbolic-like boundaries, on which a detailed study of conical points and Myrberg points is carried out. The basic theory of conformal density   is then established on the convergence boundary, including the Sullivan shadow lemma and a Hopf--Tsuji--Sullivan dichotomy. This gives a unification of the theory of   conformal density on the Gromov and Floyd boundary for (relatively) hyperbolic groups,  the visual boundary for rank-1 CAT(0) groups, and  Thurston boundary for mapping class groups. Besides that, the conformal density on the horofunction boundary provides a new important example of our general theory. Applications include the identification of Poisson boundary of random walks,  the co-growth problem of divergent groups, measure theoretical results for CAT(0) groups and mapping class groups.


\end{abstract}




\maketitle
\setcounter{tocdepth}{1} \tableofcontents


\section{Introduction}

Suppose that a group $G$ admits a proper and  isometric action on a proper
geodesic metric space $(\U, d)$. The group $G$ is assumed to be \textit{non-elementary}: there is no finite index subgroup isomorphic to the integer group $\mathbb Z$ or the trivial group.  This paper continues the investigation of such group actions with a contracting element  from the point of  dynamics at infinity. This provides a complementary view to several studies of growth problems carried out in  \cite{YANG10,YANG11,HYANG, GYANG}.  

The contracting property captures the key feature of quasi-geodesics in Gromov hyperbolic spaces, rank-1 geodesics in CAT(0) spaces, and thick geodesics in Teichm{\"u}ller spaces, and many others. In recent years,   this notion and its variants have been proven fruitful in the setup  of general metric spaces.

Let $X$ be a closed subset of $\U$, and $\pi_X: \U\to X$ be the   shortest projection (set-valued) map. We say that $X$ is \textit{$C$-contracting} for $C\ge 0$ if     $$\|\pi_X(B)\|\le C$$ for any metric ball $B$ disjoint with $X$, where $\|\cdot \|$ denotes the diameter of a set.  An element of infinite order is called \textit{contracting}, if it acts by translation on a contracting quasi-geodesic.  The prototype of a contracting element is  a hyperbolic isometry on a hyperbolic space, but many more  examples are furnished by the following:
\begin{itemize}
\item
hyperbolic elements in the relatively hyperbolic groups, cf. \cite{GePo4}, \cite{GePo2};
\item
rank-1 elements in  proper CAT(0) spaces, cf. \cite{Ballmann}, \cite{BF2}; 
\item
pseudo-Anosov elements in Teichm{\"u}ller spaces with Teichm{\"u}ller metric, cf. \cite{Minsky}.
\end{itemize}
The first goal of this paper is aiming to develop a boundary theory for a proper action with contracting elements. In particular, our study   is motivated by and attempts to   axiomize the corresponding theory of the following  
\begin{itemize}
\item
Gromov boundary of  hyperbolic spaces;
\item
Floyd boundary and ends boundary for locally finite graphs;
\item
Visual boundary of proper CAT(0) spaces;
\item
Thurston boundary, and Gardiner--Masur boundary of Teichm{\"u}ller spaces. 
\end{itemize}
One of the basic findings in the paper is that these boundaries are in a certain sense,  quotients  or of its variants, of the horofunction boundary of general metric spaces.


Our results  comprise   the following three parts. 

In the first part, we define a class of compactification for metric spaces, called \textit{convergence boundary}, reflecting the convergence behaviour of the contracting subsets to infinity. The main novelty is that we develop the notion of conical   points and of more intrinsic Myrberg   points in line with  the corresponding notions in the theory of convergence groups.  

The second part   studies the conformal measures on the convergence compactifications via  the (partial) shadow lemma. We prove a Hopf--Tsuji--Sullivan theorem, setting up the equivalences between the divergence action, and positive/full conformal measure  supported on the Myrberg set.

The last part is devoted to applications to growth problems, random walks,  measure vanishing results in CAT(0) spaces, and Thurston boundary of mapping class groups. 

We start by introducing the so-called convergence compactification. 
\subsection{Convergence compactification}

A \textit{bordification} of a metric space $\U$ is a metrizable Hausdorff   topological space $\bU$ in which $\U$ embeds as an open and dense subset, and the \textit{boundary} is $\pU:=\bU\setminus \U$.  
We say that a sequence of subsets $X_n$ is   \textit{escaping} if $d(o, X_n)\to \infty$ for some (thus any) basepoint $o\in \U$, and  \textit{converges} to a point $\xi\in \pU$ if any sequence of points $x_n\in X_n$ converges to $\xi$.
 
We are interested in a   \textit{compactification} of $\U$:  $\bU$ and   $\pU$ are compact, though the language of bordification is convenient, as $\U$ is bordified by any subset of the boundary.
 
A fundamental notion in this work is the so-called  \textit{convergence bordification} $\bU$ so that   the following two assumptions are true: 
\begin{enumerate}
\item[\textbf{(A)}]
Any contracting geodesic ray $X$ converges to a   boundary point $\xi\in \pU$, and any sequence of    $y_n\in \U$ with escaping projections $\pi_X(y_n)$ tends to the same $\xi$.
\item[\textbf{(B)}]
Let $X_n$ be any escaping sequence of  $C$-contracting subsets. Then for any      $x\in \U$, a subsequence of their \textit{cones} defined as follows
$$
\Omega_x(X_n):=\{y\in \U: \|[x,y]\cap N_C(X_n)\|\ge 10C\}
$$
converges to a boundary point in $\pU$.
\end{enumerate} 
By definition, the  convergence   bordification persists under taking Hausdorff quotient: if  we have a quotient map of $\pU$ onto a Hausdorff metrizable space, denoted by $[\cdot]$ as follows   $$[\cdot]:\quad \pU\to  [\pU]$$  extending the identification $\U\to \U$, then $[\pU]$ is a convergence   bordification of $\U$. As one-point compactification is always a  convergence bordification, we say that  a convergence bordification   is   \textit{nontrivial} if the following assumption holds
\begin{enumerate} 
\item[\textbf{(C)}]
the set  $\mathcal C$ of \textit{non-pinched} points $\xi\in \pU$ is non-empty: if $x_n, y_n\in \U$ converge to $\xi$, then $[x_n, y_n]$ is an  escaping sequence. 
 
\end{enumerate}  
See  a schematic illustration of Assumptions (A) and (B) in Fig. \ref{fig:AssumpsAB}.

\begin{figure}[htb] 
\begin{minipage}{0.40\linewidth}

\includegraphics[width=0.98\linewidth]{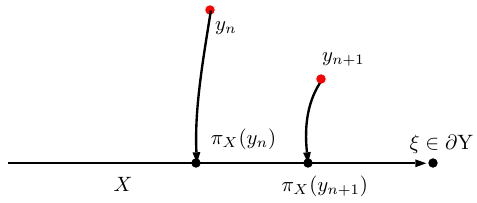} 
\end{minipage}
\medskip
\begin{minipage}{0.59\linewidth}
\includegraphics[width=0.98\linewidth]{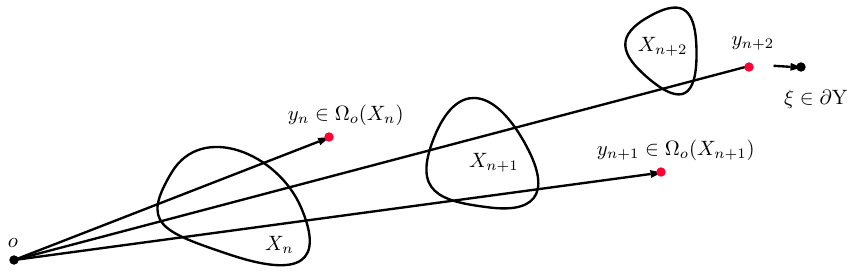} 
    
\end{minipage}
\caption{Assumption A (left) and Assumption B (right)} 
\label{fig:AssumpsAB}
\end{figure}


Let us keep in mind the  following examples, which motivate and  satisfy the above two assumptions:
\begin{examples} \label{ConvBdryEx1}
\begin{enumerate}

\item
Gromov boundary $\bG$ of a  (possibly non-proper) Gromov hyperbolic space $\U$ gives a convergence bordification, where  all boundary points are non-pinched.

\item
Bowditch boundary $\bB$ of (the Cayley graph of) a relatively hyperbolic group $\U$  gives a convergence compactification (\cite{Bow1, Ge2}). Non-pinched points  are exactly  conical points, while   certain bi-infinite geodesics exist with two half-rays tending to the same parabolic point.

\item
Floyd boundary  $\bF$ of any locally finite graph $\U$  forms a convergence compactification (\cite{Floyd, Ge2}). Any pinched point  admits a horocycle, so  non-pinched points contain at least all conical points. See Lemma \ref{FloydConvBdry}.

The same holds for the end boundary of a locally finite graph (Corollary \ref{EndConvBdry}).

\item
Visual boundary $\bV$  of a proper CAT(0) space $\U$ is a convergence boundary, where  all boundary points are non-pinched. See Lemma \ref{ContractiveVisualBdry}.
\end{enumerate}
\end{examples}

In spite of these examples, one main purpose of our investigation is to setup a unified theory for several boundaries of Teichm\"uller spaces (Thurston boundary and Gardiner--Masur boundary, etc), Roller boundary of CAT(0) cube complexes, and more generally, the horofunction boundary of any proper metric space. For boundaries of Teichm\"uller spaces, a typical phenomenon is that  a geodesic ray  may not admit a unique limit point in the boundary (\cite{Len08}), but generic ones do have (\cite{Ma82a}). Moreover, in pursuing the common feature, we found morally that all these  boundaries have much like  the horofunction boundary.   In the end, we are able to formulate   the above assumptions (A) and (B) for these boundaries, modulo   a certain ambiguity of the convergence described as follows. 

To absorb  the pathological convergence, we  endow an $\isom(\U)$-invariant partition $[\cdot]$ on the boundary. This might be artificial at the first sight, but partitioning  boundary could be found quite naturally in the study of boundary comparison. For instance, establishing the topological  comparison map  between Floyd boundary, Bowditch boundary and Martin boundary has generated many applications in the study of relatively hyperbolic groups and random walks (\cite{Ge2, GePo2, GePo3, GGPY}). Here we shall go one step further to allow  that the corresponding quotient space might be non-Hausdorff.   

Given such a partition $[\cdot]$,  the bordification $\bU=\U\cup \pU$ is called a \textit{convergence bordification} if the convergence  to a point $\xi$ in the     the assumptions (A),(B) and (C) is replaced by   the convergence  to \textbf{some} point in its equivalent class $[\xi]$.   The above Examples   \ref{ConvBdryEx1}  of convergence boundaries have  been equipped with \emph{maximal} partitions (i.e. $[\xi]$ are singletons). See \ref{AssumpA}, \ref{AssumpB}, \ref{AssumpC} in \textsection \ref{SecBoundary} for  precise formulations.

From a different point of view, we could understand the convergence in the assumptions (A), (B) and (C) after passing to the quotient  of $\pU$ by identifying points in the same $[\cdot]$-class
$$\begin{aligned}
[\cdot]: \quad  \pU &\longrightarrow [\pU]:=\pU/[\sim]\\
\xi &\longmapsto [\xi]
\end{aligned}
$$
As $[\cdot]$ may not necessarily  form a closed relation, the quotient $[\pU]$ might be non-Hausdorff or even not $T_1$ (that is, $[\cdot]$-classes are not necessarily closed subsets).  We say that the partition $[\cdot]$ is \textit{maximal} if every $[\cdot]$-class is singleton, and it is \textit{trivial} if $\pU$ is a $[\cdot]$-class. We shall refer to $[\pU]$ as the \textit{reduced} convergence boundary.

Introducing the partition is  indispensable in obtaining a hyperbolic-like boundary, as   certain natural subsets in the boundary are distinguishable in a   view of negative curvature.  The following examples illustrate some aspects of such consideration.

\begin{examples}\label{ConvBdryEx2}
Let $\T$ be the Teichm{\"u}ller space of a closed orientable surface of genus $g\ge 2$, endowed with the following two boundaries  (recalled in \textsection \ref{SecThurstonBdry} in details):
\begin{enumerate}
\item Thurston boundary $\bTh$: it is the set $\pmf$ of projective measured foliations, which is a topological sphere of dimension $6g-7$ (\cite{FLP}). Kaimanovich--Masur \cite{KaMasur} considered a partition $[\cdot]$ on $\pmf$ which is induced by  an intersection form. The restriction to the subset $\minf$ of  the minimal measured foliations is closed. Its quotient, the (Hausdorff) ending lamination space,  can be identified with the Gromov boundary of the  complex  of curves on $\Sigma_g$ (\cite{Klarreich, H06}).   Inside $\minf$, the partition is {maximal} exactly on the subset $\ue$ of   uniquely ergodic  points: their equivalent classes are singletons.

\item
Gardiner--Masur boundary $\bGM$ (\cite{GM91}): it is   homeomorphic to the horofunction boundary of $\T$ with respect to the Teichm{\"u}ller metric (\cite{LS14}). It turned out that $\pmf$ embeds as a proper subset (with a different subspace topology) in $\bGM$. The relation between $\bTh$ and $\bGM$ has been carefully investigated in \cite{Miy08, Miy13,Miy14a,Walsh19}.
\end{enumerate} 
\end{examples}

\begin{examples}\label{ConvBdryEx3}
The horofunction boundary $\hU$ was introduced  by Gromov for any metric space $\U$. Two horofunctions are said to be \textit{equivalent} if they differ  by a bounded amount. This defines a  \textit{finite difference relation} on $\hU$, and the quotient $[\hU]$ is usually called \textit{reduced} horofunction boundary. Here are three examples   we are interested in. 
\begin{enumerate}
    \item 
    If $\U$ is a hyperbolic space,  the finite difference relation can be shown  closed, for two horofunctions in a $[\cdot]$-class have a uniformly bounded difference. Thus, the quotient of the horofunction boundary is a Hausdorff space, which is homeomorphic to the Gromov boundary. A non-trivial finite difference relation arises naturally when taking a direct product of a hyperbolic graph with an interval, though the Gromov boundary is kept same.
    \item 
    Roller boundary of a locally finite CAT(0) cubical complex is homeomorphic to the horofunction boundary of $\U$ with combinatorial metric (\cite{FLM}). Every boundary point is represented by a geodesic ray and the finite difference of horofunctions is equivalent to the finite symmetric difference of boundary points viewed as ultrafilters. The quotient is homeomorphic to the combinatorial boundary studied in  \cite{Genevois}. See Lemma \ref{Roller}. 
    \item
    Martin boundary of a finitely supported random walk on a non-amenable group is homeomorphic to the horofunction boundary defined using (generally non-geodesic) Green metric.  The points  on which the   finite difference relation restricts to be a maximal partition are minimal harmonic functions. The resulting  \textit{minimal} Martin boundary is  essential in representing all positive harmonic functions. 
\end{enumerate}
\end{examples}

Our   first result  confirms that the horofunction boundary with the finite difference partition provides a  nontrivial convergence compactification.  The proof is presented  in \textsection\ref{SecHorobdry}.
\begin{thm}\label{ContractiveThm1}
The  horofunction boundary $\hU$ of a proper geodesic metric space $\U$  with contracting subsets is a convergence boundary, with finite difference relation, so that all boundary points are non-pinched. 
\end{thm}

Gardiner--Masur boundary and Roller boundary   are  the horofunction boundary of Teichm\"uller spaces and of CAT(0) cube complexes endowed with Teichm\"uller metric and combinatorial metric respectively. Hence, they provide    examples of convergence boundary with finite difference relation.

Our next result says that Thurston boundary is  also a convergence boundary for Teichm\"uller space endowed with Teichm\"uller metric. By a result of \cite{Walsh19}, Thurston boundary is actually the horofunction boundary of the Teichm\"uller space endowed with Thurston's Lipschitz metric.  The following result proved in Theorem \ref{ThurstonBdryContractive} is thus not derived as a corollary from Theorem \ref{ContractiveThm1}. 
\begin{thm}\label{ContractiveThm2}
The Thurston boundary $\bTh$   is a convergence boundary, with respect to the above Kaimanovich--Masur partition, so that the set of non-pinched points contains the set of uniquely ergodic points. 
\end{thm}

As fore-mentioned in Examples \ref{ConvBdryEx3}, the reduced horofunction boundary of a hyperbolic space is metrizable. In general,  the reduced $[\hU]$ of a metric space may be neither Hausdorff nor second countable. This leads to technical issues in the  topological and measure theoretical applications using convergence boundary. 

A key observation in fixing the issues is that, if  we consider a non-elementary group $G< \isom(\U,d)$ so that
\begin{equation}\label{GActUAssump}
G \text{ acts     properly on  a proper geodesic space } (\U, d) \text{ with a contracting element}   
\end{equation}
then a smaller subset of the boundary consisting of the so-called \textit{Myrberg   points}  would have   better properties being Hausdorff, second countable and metrizable  on the quotient space.

If $G\act \U$ is of divergent type,  we shall establish in \textsection \ref{SSConfdensity} that Myrberg   points   are generic in conformal measure, and furthermore, the push-forward measures to the quotient become unique and ergodic in \textsection \ref{SecUniqueness}. 

Before moving on the measurable theoretical applications, let us begin a topological investigation of Myrberg   points and their relatives conical   points. 

\subsection{Conical and Myrberg limit sets}

Conical   points are an important sub-class of limit points, which have been well-studied in Kleinian groups and in a more general class of convergence group actions (see Section \ref{SecRHG}). Myrberg points form a strictly smaller subset of conical points, on which the dynamics is quite powerful, but seems to be under-estimated    in literature.   The main novelty here is to formulate the notion of  Myrberg   and conical limit points for a general proper action as in (\ref{GActUAssump}), which shares desirable properties as much as in a convergence group action (see Subsection \ref{SSecConvAction}).  


\subsubsection*{Dynamics on the limit set} We   start discussing the notion of limit sets for any group $G\act \U$ as in (\ref{GActUAssump}).

A main consequence of the convergence boundary allows us to  develop a ``good'' theory of limit sets  in a sense  that the limit set is the unique and minimal invariant closed subset, up to taking $[\cdot]$-closure.   

Define the \textit{limit set} $\pG$ to be the set of accumulation points of a $G$-orbit $Go$ in $\pU$. It may depend on the choice of $o\in\U$ but    the $[\cdot]$-closure $[\pG]$ does not by Lemma \ref{[LimitSet]}. By \ref{AssumpA}, if $h$ is a contracting element, the positive and negative half rays (i.e. $n>0$ and $n<0$) of its quasi-geodesic 
$$
n\in\mathbb Z\longmapsto h^no\in \U
$$ determine two $[\cdot]$-classes of boundary points denoted by $[h^+],[h^-]$. The union $[h^\pm]=[h^-]\cup[h^+]$  are the $[\cdot]$-closure of the limit set of  $\langle h\rangle\cdot o$. We say that $h$ is  \textit{non-pinched} if its fixed points $[h^\pm]$ are non-pinched in \ref{AssumpC}. Equivalently, this is amount to asking $[h^-]\ne [h^+]$ by Lemma \ref{DisjFixpts=Non-pinched}. In a discrete group, any two non-pinched contracting elements have either the same fixed points or disjoint fixed points (Lemma \ref{DisjOrSameFixedSet}).

In what follows, we always assume that the convergence boundary $\pU$ under consideration is \textit{non-elementary} for the action $G\act \U$ in (\ref{GActUAssump}). That is, $G$ contains a non-pinched contracting element, whose fixed points are distinct $[\cdot]$-classes in $\pU$. 

\begin{lem}[NS dynamics, see Lemma \ref{SouthNorthLem}]\label{SouthNorthLemIntro}
Let $h\in \isom(\U)$ be a non-pinched contracting isometry. Then, under \ref{AssumpA} and \ref{AssumpC}, the action of $\langle h\rangle$ on $\pU\setminus [h^\pm]$ has the North--South dynamics: 

\begin{itemize}
    \item for any two open sets   $[h^+] \subseteq U$ and $[h^-]\subseteq V$ in $\pU$, there exists an integer $n>0$ such that $h^n (\pU \setminus V)\subseteq U$ and $h^{-n} (\pU \setminus U)\subseteq V$. 
\end{itemize}  
\end{lem}

Let us formulate  the main conclusion we obtained for the limit set.
\begin{lem} [see Lemma \ref{UniqueLimitSet}]
Let  $h\in G$ be a non-pinched contracting element with  minimal fixed points $[h^-]=\{h^-\},[h^+]=\{h^+\}$ on a convergence boundary $\pU$.   Let $\Lambda$  be the closure of the fixed points of all conjugates of $h$ in $G$. Then
\begin{enumerate}
    \item 
    $\Lambda$ is the   unique,  minimal, and $G$-invariant closed subset.
    \item
    $[\Lambda]=[\pG ]$ for any $o\in \U$.
    \item
    The set of the fixed point pairs of all non-pinched contracting elements in $G$ is dense in the distinct pairs of $\Lambda\times \Lambda$.
\end{enumerate} 
\end{lem}
All the boundaries  in Examples (\ref{ConvBdryEx1}, \ref{ConvBdryEx2} and \ref{ConvBdryEx3}), except the horofunction boundary of a general metric space,  contain a non-pinched contracting element  with minimal fixed points.  The  list of properties as above are mostly known in these boundaries, but seem to be new  for Roller and Gardiner--Masur boundaries, see Lemmas \ref{Roller} and  \ref{GMBdryConvergence} for concrete formulations.  

\subsubsection*{Myrberg points} We first introduce the notion of Myrberg points, which is more canonical than that of conical points. To that end, let us denote the set of distinct pairs of $[\cdot]$-classes: 
\begin{equation}\label{doubleBdryEq}
 \dG:=\{([\xi],[\eta])\in [\pG]\times [\pG]: [\xi]\ne[\eta]\}   
\end{equation}  Recall that $\mathcal C$ is the subset of non-pinched boundary points in \ref{AssumpC}.

\begin{defn}[Myrberg points]
A point $\xi \in \mathcal C\subseteq \pU$ is called \textit{Myrberg point} if for any $x\in \U$, the set of $G$-translates of the ordered pair $(x, \xi)$ is dense in $\dG$ in the following sense:
\begin{itemize}
    \item For any $[\zeta]\ne[\eta]\in [\pG]$ there exists $g_n\in G$ so that $g_nx\to [\zeta]$ and $g_n\xi\to [\eta]$.
\end{itemize}
\end{defn}
\begin{rem}
Introduced by Myrberg \cite{Myr31}, this class of limit points was proved to be full Lebesgue measure in the ideal boundary $\mathbb S^1$   for finite co-volume Fuchsian groups. This was generalized later on to higher dimension \cite{Agard, Tukia4}.
\end{rem}

In contrast with the reduced convergence boundary $[\pU]$, the Myrberg subset enjoys  quite good topological properties after passing to the quotient. 

\begin{lem}[see Lemma \ref{VisualMetriconMyrberg}]\label{MyrbergIntro} 
The quotient $[\mG]$ of the Myrberg set in a non-elementary convergence boundary $\pU$ is a Hausdorff, second countable,  metrizable topological space. 

Moreover, $\U\cup [\mG]$ forms a convergence bordification,  endowed with maximal partition, so that all boundary points are non-pinched.   
\end{lem}

If $\pU$ is the horofunction boundary we prove in Lemma \ref{MyrbergGood}, via different means, that  $[\mG]$ is a Hausdorff and second countable topological space. Lemma \ref{MyrbergIntro} is proved in Subsection \textsection \ref{SSecGromovBdryQT} by embedding $[\mG]$ into the Gromov boundary of a quasi-tree of spaces. The above properties are therefore inherited from those of the Gromov boundary.

\subsubsection*{Conical points} Our notion of  conical points is defined relative to a collection $\f$ of contracting subsets with bounded intersection. There are various (quantitatively) equivalent ways to define it.  We give a quick definition in Introduction, and the  other two elaborating ones are discussed in Section \textsection \ref{SecMyrberg}. 

Let  $F$ be a set of three (mutually) independent non-pinched contracting elements  $f_i$ ($i=1,2,3$), which form a contracting system as follows 
\begin{equation}\label{SystemFDefIntro}
\f =\{g\cdot \ax(f_i):   g\in G \}
\end{equation}
where  the axis $\ax(f_i)$    defined in (\ref{axisdefn})  depends on the choice of a basepoint $o\in \U$.

One consequence of \ref{AssumpA} and \ref{AssumpC} together is as follows:  the   shortest projection map $\pi_X:\U\to X$ to a contracting bi-infinite quasi-geodesic $X$  extends to the  subset $\mathcal C$ of non-pinched  boundary points (\textsection\ref{SSProjBdryPts}):
$$
\pi_X: \U\cup \mathcal C\setminus [\Lambda X]\longrightarrow X
$$   
which satisfies the bounded geodesic image property (for possibly larger $C$): 
$$
d_X(x,y)\ge C\quad \Longrightarrow\quad [x,y]\cap N_C(X)\ne\emptyset
$$
where $d_X(x,y):=\|\pi_X(x)\cup\pi_X(y)\|$. See Lemma \ref{BdryProjLem}.

\begin{defn}[Conical points, see Definition \ref{ConicalDef1prime}]\label{ConicalDef1}
A non-pinched point $\xi \in \mathcal C$ is called \textit{$L$-conical point} for $L>0$ relative to $\f$ if  for some $x\in \U$, there exists a sequence of $X_n\in \f$ such that $d_{X_n}(x, \xi)\ge L$. Denote by $\Lambda_L^\f(Go)$ the set of all such points.
\end{defn}

An important ingredient in our argument is the axiomized construction of a projection complex $\PC$  from $\f$,  developed by Bestvina-Bromberg-Fujiwara   \cite{BBF}.  Loosely speaking, fixing a constant $K$, $\PC$ is a graph with vertex set $\f$ so that two vertices $U\ne V\in \f$ are adjacent if and only if  they are visually $K$-small from any third party $W\in \f\setminus \{U, V\}$:
$$
d_W(U,V)\le K
$$
(see \textsection\ref{SSecProjectionComplex}  for more details). The basic fact is that for any $K\gg 0$, the graph $\PC$ is a unbounded quasi-tree (with universal hyperbolicity constant).  

Our next result, which is proved in Lemma \ref{EmbedGromov}, describes the above-defined conical points via the familiar Gromov boundary of the projection complex.
\begin{lem}\label{EmbedGromovThm}
For any   $L\gg 0$, there exists $K>0$ such that the Gromov boundary of the projection complex $\PC$ admits a topological embedding into the set of $L$-conical points relative to $\f$.
\end{lem}

In \cite{BBF}, a quasi-tree of spaces  denoted by $\QT$  is built from the projection complex $\PC$ in the following way. Keeping the edge incidence,  $\QT$ is obtained from $\PC$ by replacing each vertex with the corresponding subspace in $\U$. The technical advantage is that the fixed points of $f\in F$ are incorporated into the Gromov boundary of $\QT$ (but not in $\PC)$.  Consequently, we are able to embed  the conical points $\Lambda_r^F(Go)$ into the Gromov boundary $\partial \QT$ in Lemma \ref{EmbedConical}. In particular,  Myrberg points are thus embedded in the Gromov boundary as a subset of a separable metrizable space. This proves Lemma \ref{MyrbergIntro}.

We close the discussion with a connection between  Myrberg and conical points. Namely,  Corollary \ref{MyrbergConical}  sets up an equivalent characterization of Myrberg points in terms of conical points.
\begin{prop}\label{MyrbergConicalThm}
The    Myrberg  set $\mG$ is the (countable) intersection of all $(L, \f)$-conical sets $\Lambda_L^\f(Go)$, over all possible $\f$   given by (\ref{SystemFDefIntro})   and all large integers $L \in \mathbb N$. 
\end{prop}

This result admits several applications, which shall be delivered in the forth-coming work. At the moment, let us mention the following the topological one (with measurable one in next subsection) concerning  Bowditch boundary of relatively hyperbolic groups. 
\begin{thm}[see Theorem \ref{HomeoMybergRHG}]\label{HomeoMyrbergIntro}
Let $\GP$ be a relatively hyperbolic pair and $\partial_{\mathrm{B}}\GP$  be the associated Bowditch boundary. Then the Myrberg set for the action on $\partial_{\mathrm{B}}\GP$ is homeomorphic to the Myrberg set for the action the Floyd boundary. 
\end{thm}
We refer the reader to  \textsection\ref{SSecRHG} for a brief introduction to relatively hyperbolic groups.

A finitely generated group may admit different peripheral structure $\mathcal P$ to be relatively hyperbolic (see \cite{YANG3} for their relations). Hence, this result says that the Myrberg set is invariant up to changing the peripheral structure on $G$. This could be compared with Floyd mapping theorem (\cite{Floyd, Ge2}). A similar conclusion holds for Cannon--Thurston maps; see  Theorem \ref{HomeoMybergCT} for precise statement.


\subsection{Conformal density on the convergence boundary}\label{SSConfdensity}

A family of conformal measures on the limit set of Fuchsian groups was firstly constructed by Patterson \cite{Patt} and studied extensively by Sullivan \cite{Sul} in higher dimensional Kleinian groups. Patterson's construction is very robust and carries over verbatim  to the general metric spaces with horofunction boundary  (see \cite{BuMo}).      The second part of our study is to give a vast generalization of these works in groups with contracting elements, endowed with a  convergence compatification.  Let us first introduce the setup for the Patterson's construction.

Fix a basepoint $o \in \U$. Consider the  growth function of the ball of radius $R>0$:
$$N(o, R):=\{v\in Go: d(o, v)\le n\}.$$ The \textit{critical exponent} $\e \Gamma$ for a subset $\Gamma \subseteq G$:
\begin{equation}\label{criticalexpo}
\omega_\Gamma = \limsup\limits_{R \to \infty} \frac{\log \sharp(N(o, R)\cap \Gamma o)}{R},
\end{equation}
  is independent of the choice of $o \in \U$, and intimately related  to the Poincar\'e series 
\begin{equation}\label{PoincareEQ}
s\ge 0, x,y\in\U, \quad \p_\Gamma(s,x, y) = \sum\limits_{g \in \Gamma} e^{-sd(x, gy)}
\end{equation}
as $\p_{\Gamma}(s,x,y)$ diverges for $s<\e \Gamma$ and converges for $s>\e \Gamma$. Thus, the   action   $G\act \U$ as in (\ref{GActUAssump}) is called of \textit{divergent type} (resp.
\textit{convergent type})   if $\p_{G}(s,x,y)$ is divergent (resp. convergent) at
$s=\e G$.  

Consider a further assumption, \ref{AssumpE}, on a convergence compactification $\bU=\U\cup\pU$, so that for non-pinched points $\xi\in \mathcal C$, as $z\to \xi,$ the Buseman cocycle 
$$
\forall x,y, z\in \U:\; B_z(x, y) := d(x,z)-d(y,z)
$$
extends coarsely to a well-defined Buseman quasi-cocycle $B_\xi(x,y)$. See \ref{AssumpE} for precise formulation and Remark \ref{ExamplesConformalDensity} for further clarification.

Denote by $\mathcal M_+(\bU)$  the set of positive Radon measures on $\bU$. If $G\act\U$ is of divergent type,  the \textit{Patterson--Sullivan measures} $\{\mu_x\}_{x \in \U}$ are accumulation points in $\mathcal M_+(\bU)$ of the following family of measures $\{\mu_x^{s}\}_{x \in \U}$ supported on $Go$ 
$$
\mu_{x}^{s} = \frac{1}{\p_G(s, o, o)} \sum\limits_{g \in G} e^{-sd(x, go)} \cdot \dirac{go},
$$
as $s >\e G$ tends to $\e G$. Similar construction works for groups of convergent type  via  a Patterson's trick \cite{Patt}  on the Poincar\'e series. 
The $G$-equivariance of such measures $\mu_x$ on basepoints $x$  and their conformality on switching   basepoints  are formulated in a general notion of (quasi)-conformal density (see Definition \ref{ConformalDensityDefn}).

In \cite{Sul}, Sullivan proved a shadow lemma describing the local nature of Patterson's measure on the limit set of Kleinian groups.  We extend it to our setting. Define the usual shadow $\Pi_x(y, r)$ as the topological closure in $\pU$ of the cone as follows:
$$\Omega_x(y, r)=\{z\in \U: [x,z]\cap B(y, r)\ne\emptyset \}$$

Analogous to the partial cone and shadow in \cite{YANG7}, we are actually working with a similar version of shadows via contracting segments. With a set $F$ of three independent contracting elements as in (\ref{SystemFDefIntro}) and $r>0$,  a segment $\alpha$ in $\U$ contains an \textit{$(r, F)$-barrier} at $go$ if  for some $f\in F$, we have $go,\; gfo\in N_r([x,z]).$
For $x\in \U, y\in Go$ and $r>0$, the \textit{$(r, F)$-cone} is defined as follows
$$\Omega_{x}^F(y, r):=\{z\in \U: y \text{ is an } (r, F) \text{-barrier for some geodesic } [x, z]\}
 $$
whose  topological closure  in $\pU$  gives the \textit{$(r, F)$-shadow} $\Pi_{x}^F(y, r)$. See \textsection\ref{SecBoundary} for details. The following version of Shadow Lemma is proved in Lemma \ref{ShadowLem}.

\begin{lem}
Let $G\act \U$ be as in (\ref{GActUAssump}). 
Let $\{\mu_x\}_{x\in\U}$ be a $\omega$-dimensional $G$-quasi-equivariant, quasi-conformal density on $\pU$ for some $\omega>0$, charging positive measure to the non-pinched points $\mathcal C$. Then there exists $r_0 > 0$ such that  
$$
\begin{array}{rl}
e^{-\omega \cdot d(o, go)}\; \prec \; \mu_o(\Pi_o^F(go, r)\cap \mathcal C)\; \le \; \mu_o(\Pi_o(go, r)\cap \mathcal C) \; \prec_r\;  e^{-\omega \cdot  d(o, go)}\\
\end{array}
$$
for any $g\in G$ and $r \ge  r_0$, where   $\prec_r$ means the inequality holds up to a multiplicative constant depending on $r$ (or a universal constant if no $r$ is displayed). 
\end{lem}

\begin{rem}
By Theorem \ref{ContractiveThm1}, we have $\mathcal C=\hU$: the complete version of shadow lemma holds in the horofunction boundary. This is the main new instance, see \textsection\ref{Connections} for a short history of it.    
\end{rem}
 
In order to state the next result, we also introduce the   usual version of \textit{conical points} defined via the usual shadow without involving $\f$:  
\begin{equation}\label{ConicalEQ}
\Lambda_{c}(Go) :=\bigcup_{r\ge 0}\bigcup_{x\in Go} \left(\limsup_{y\in Go} \Pi_x(y, r)\right).
\end{equation}
By definition, $L$-conical points and Myrberg points are conical. Note that every  point in the visual boundary is conical in this sense for any geometric action on a CAT(0) space. This is partly the reason why we restrict our attention to the smaller subset of $L$-conical points and Myrberg points. 

The main result of this investigation is relating the divergence of Poincar\'e series to the positive conformal measure charged on conical points. This type of results is usually refereed to as Hopf--Tsuji--Sullivan dichotomy by researchers. Note here that the item on the ergodicity of product measures  is missing.

\begin{thm}\label{HTSThm} 
Let $G\act \U$ be as in (\ref{GActUAssump}). Let $\{\mu_x\}_{x\in \U}$ be a $\omega$-dimensional $G$-quasi-equivariant, quasi-conformal density on   a non-elementary convergence boundary $\pU$ for some $\omega>0$. Assume that $\mu_o(\mathcal C)>0$.  Then the following statements are equivalent: Let  $\f$ be as in (\ref{SystemFDefIntro}) and any $L\gg 0$.
\begin{enumerate}
\item \label{itm:1}
The Poincar\'e series $\p_G(s,o,o)$ diverges at $s=\omega$;
\item \label{itm:2}
The Myrberg set has  the positive   $\mu_x$-measure;
\item \label{itm:3}
The Myrberg  set has the full $\mu_x$-measure in $\mathcal C$; 
\item \label{itm:4}
The set of conical points (or $(L, \f)$-conical points)   has positive   $\mu_x$-measure;
\item \label{itm:5}
The set of conical points (or $(L, \f)$-conical points) has full   $\mu_x$-measure in $\mathcal C$.
 
\end{enumerate} 
If one of the above statements is true, then $\omega=\e G$. 
\end{thm}
\begin{rem}
The theorem is proved via the routine: (1) $\Longleftrightarrow$ (4) $\Longrightarrow$ (5) $\Longrightarrow$ (3)  $\Longrightarrow$ (2)  $\Longrightarrow$ (4). We comment on the following directions while the others   are  trivial or classic.   
\begin{itemize}
    \item 
     (1) $\Longrightarrow$ (4): this is the most difficult direction, whose proof uses the projection complex to build a family of the \textit{visual sphere} decomposition of group elements. With light source at the basepoint $o\in \U$, each visual sphere is shadowed by the previous one and blocks the following one. Such a structure allows to re-arrange the series $\p_G(s, o, o)$ as a geometric-like series. See \textsection\ref{SecHTSTProof} for more details. 
     \item
     (5) $\Longrightarrow$ (3): This follows from Proposition \ref{MyrbergConicalThm}, as $\mG$ is the countable intersection of $\mu_o$-full $(L, \f)$-conical points, over all possible $\f$ in (\ref{SystemFDefIntro}).
\end{itemize}

\end{rem}

Under any statement in Theorem \ref{HTSThm}, the conformal density turns out to be unique when passed to the quotient of Myrberg  set (which is Hausdorff and second countable by Lemma \ref{MyrbergIntro}). This is proved in Theorem \ref{Unique}.
\begin{thm} \label{UniqueConfThm}
Let $G\act \U$ be as in (\ref{GActUAssump})   of divergent type, compactified by a non-elementary convergence boundary $\hU$. Then the $\e G$-dimensional $G$-quasi-equivariant, quasi-conformal density on the quotient $[\mG]$ of $\mG\subseteq \pU$   is unique up to a bounded multiplicative constant, and ergodic.
\end{thm}
By Rohklin's theory,  $[\mG]$ endowed with conformal measures are Lebesgue  spaces (\cite[Lemma 15.4]{C16Book}), i.e.  measurably isomorphic to    the interval $[0,1]$ with Lebesgue measure.

\subsection{More applications}
We now present various applications to the growth problems, and measure theoretical results in CAT(0) groups and mapping class groups.

\subsubsection*{Cogrowth tightness}
We start with the following co-growth   results. This extends the work of \cite{MYJ20} about  divergent group actions on hyperbolic spaces, and answers positively \cite[Questions 4.1 and 4.2]{AC20}. The following is proved in Theorem \ref{SecCogrowthThm}.

\begin{thm}\label{CogrowthThm}
Let $G\act \U$ be as in (\ref{GActUAssump}).  Suppose that  the group $G$ is of divergent type. Then for any infinite normal subgroup $H$ of $G$, we have $$\e H> \frac{ \e G}{2}$$
Moreover, if $H$   is of divergent type, then $\e H=\e G$.
\end{thm}
\begin{rem}
In \cite{AC20}, Arzhantseva-Cashen proved the same result under the assumption that $G\act \U$ has purely exponential growth.  This already includes many interesting examples, but the divergent action is weaker than having purely exponential growth.
\end{rem}


\subsubsection*{Myrberg set as Poisson boundary of random walks}
We now give an application of previous results to random walks on groups with contracting elements. A probability measure $\mu$ on $G$ is called \emph{irreducible} if its support generates the group $G$ as a semi-group.   A \emph{$\mu$-random walk} on $G$ is the stochastic process $(w_n)$ defined as 
$$w_n := g_1 g_2 \dots g_n$$
where $(g_n)$ is a sequence of i.i.d. random variables valued in $G$ with law $\mu$. The \emph{entropy}  is defined by $$H(\mu) := - \sum_{g \in \textrm{supp}(\mu)} \mu(g) \cdot |\log \mu(g)|.$$ 

By the work of Maher--Tiozzo \cite{MT}, the following statement holds for the random walk for acylindrical actions on hyperbolic spaces.  The previous two results transfer the results to the horofunction boundary. More details on the  proof is given in \textsection\ref{ProofPoissonBdry}.

\begin{thm}\label{PoissonBdry}
Let $G\act \U$ as in (\ref{GActUAssump}). Consider an irreducible  $\mu$-random walk driven by  a probability measure $\mu$ on $G$ with  finite logarithmic moment on $\U$: $$\sum_{g\in G}\mu(g)\cdot |\log d(o, go)|<\infty.$$ Then almost every trajectory of $\mu$-random walk converges to the $[\cdot]$-class of a Myrberg  point in $\hU$.

Moreover, if  $H(\mu)<\infty$, then  the reduced horofunction boundary $[\hU]$  is the Poisson boundary,  with harmonic measure supported on the quotient $[\mG]$ of Myrberg  set.
\end{thm} 
\begin{rem}
If $\U$ is a CAT(0) space with rank-1 elements, this completes the theorem of Karlsson--Margulis  \cite{KM99} by showing  the positive drift.  For   the {proper} essential action on CAT(0) cube complex, our result recovers the main results \cite[Thm 1.2]{FLM} \cite{F18}  that almost every trajectory  tends to a point in the Roller boundary. See Lemma \ref{Roller} for the relevant facts on Roller boundary. 

Using the recent work \cite{CFFT}, the finite logarithmic moment condition could be removed in the above statement, still under finite entropy assumption. See the remark \ref{NoMomentCondtion}. 
\end{rem} 

\subsubsection*{Conformal measure on the contracting boundary}
Contracting boundary for CAT(0) spaces was introduced by Charney--Sultan \cite{ChaSul} as quasi-isometric invariant, and has attracted active research interests in recent years. It is  observed that the contracting boundary is measurably negligible in harmonic measures. We derive the same  result in conformal measures from a  more general Theorem  \ref{nullity}. Here let us state it in CAT(0) spaces.

The  underlying set of the contracting boundary  consists of the endpoints of contracting geodesic rays in the visual boundary.

\begin{thm}\label{CATnullity}
Let $G\act \U$  in (\ref{GActUAssump}) be co-compact, where $\U$ is a proper \textrm{CAT}(0) space. Let $\{\mu_x\}_{x\in \U}$ be the $\e G$-dimensional conformal density on the visual boundary $\bV$. Then  the underlying set of contracting boundary is $\mu_o$-null if and only if $G$ is not a hyperbolic group. 
\end{thm}

\subsubsection*{Hopf--Tsuji--Sullivan Dichotomy for mapping class groups and cubical groups}
At last,  for subgroups of mapping class groups, we can complete the Hopf--Tsuji--Sullivan  Dichotomy \ref{SecTeichFlow} with ergodicity of product measures.

\begin{thm}\label{HTSModThm}
Consider any non-elementary subgroup $G$ of $\mcg$ ($g\ge 2$) with pseudo-Anosov elements. Let $\{\mu_x\}_{x\in \T}$ be a $\omega$-dimensional $G$-equivariant conformal density on $\bTh=\pmf$. Then the following are equivalent:
\begin{enumerate}
\item
The Poincar\'e series $\p_G(s,o,o)$  of $G$ diverges at $\omega$;
\item
The Myrberg  set $\mG$ has either full or positive $\mu_x$-measure;

\item
The conical  set $\cG$ has either full or positive $\mu_x$-measure;

\item
The diagonal action on $\pmf\times \pmf$ is ergodic with respect to the product measure  $\mu_o\times \mu_o$. 
\end{enumerate} 
\end{thm}

\begin{rem}
If $G=\mcg$   the uniqueness of $\e G$-dimensional conformal density was proved \cite{H06} \cite{LM08}, and the ergodicity was known long before by work of Masur \cite{Ma82a} and Veech \cite{V82}. The above statement, however,  seems new for    non-elementary proper subgroups.
\end{rem}

At last, we note the following corollary to the non-elementary  action on a CAT(0) cube complex for further reference. Such an action is called \textit{essential} if no half-space contains an orbit in its fixed finite neighborhood (See \cite{CapSag}). In this setup, a Myrbert limit point is minimal, so we do not need take the quotient to obtain the uniqueness of PS measures.
\begin{thm}  \label{UniqueConfCubeThm}
Let $G\act \U$ in (\ref{GActUAssump}) be   an essential action of divergent type on  a CAT(0) cube complex compactified by the Roller  boundary $\bR$. Then the $\e G$-dimensional $G$-quasi-equivariant, quasi-conformal density on $\bR$   is unique up to a bounded multiplicative constant, and ergodic.
\end{thm}

\subsection{Historical remarks and further questions} \label{Connections}
In a series of works \cite{YANG10,YANG11,HYANG, GYANG} without using      ergodic theory,  the elementary and  geometric   arguments  are employed  to establish  coarse counting results for groups   with contracting elements acting on    general metric spaces. 

Analogous to Anosov shadowing / closing property, a geometric tool  called extension lemma was used (recalled here in Lemma \ref{extend3}): roughly speaking,  any two geodesics $\alpha, \beta$ can be connected by a short arc $c$ to  form    a quasi-geodesic $\alpha\cdot c\cdot \beta$. The short arc is provided by a contracting segment,  whose existence    facilities the extensive use of Gromov-hyperbolic geometry. In particular, if $\alpha=\beta$, then a repeated application of extension lemma concatenates copies of $\alpha \cdot c$ to get a  periodic quasi-geodesic. In other words, this gives   a closed geodesic nearby, with length coarsely equal $\ell(\alpha)$ on the quotient $\U/G$.   On account of our previous works,   the complementary view from   the  ergodic theory is anticipated, and actually forms the primary goal of the present paper.
  
The conformal density on the limit set of Fuchsian groups was famously constructed by Patterson \cite{Patt}, and further developed by Suillvan \cite{Sul} in Kleinian groups with a number of applications. Particularly, their work sets the right track for the further generalization of (quasi)-conformal density on boundaries of hyperbolic spaces \cite{Coor}, CAT(-1) spaces \cite{BuMo}\cite{Roblin}, and CAT(0) spaces \cite{R17,L18}.  The Sullivan Shadow Lemma     \cite{Sul} provides a very useful tool in applications, and thus is most desirable in any ``good" theory of conformal density. Although the Sullivan's proof carries  over to Gromov hyperbolic spaces,   its generalization in CAT(0) manifolds and spaces follows  a different argument given by Knieper  (\cite{Kneiper1},\cite[Remark on p.781]{Kneiper2}) and  Link \cite{L18}. In Teichm\"uller space, the Shadow lemma is obtained in \cite{TYANG} with applications to fundamental inequality of random walks (see \cite[Lemma 5.1]{Gekht2} for a special case). 
 
Attempting to formulate a unified framework for the  (quasi-)conformal density in the above examples motivates the investigation of this paper. 
 
In \cite{YANG7}, the quasi-conformal density on Floyd   and Bowditch boundary of  relatively hyperbolic groups is used to study growth problems in word metrics. Like  Teichm\"uller space, the Cayley graph of a relatively hyperbolic group is generally neither Gromov hyperbolic  nor CAT(0) spaces. Moreover, the lack of Buseman cocyles at all boundary points forces us to define Buseman (quasi-)cocyles only at conical points. To handle  non-conical points, we  prove   the constructed Patterson--Sullivan measures are fully supported on conical points. This     strategy is axiomized in the present paper as the  \ref{AssumpE}. As another instance, the conformal density obtained  on the Thurston boundary \cite{ABEM}  follows a similar route   detailed in \textsection\ref{SecMCGDensity}.

Concerning  the convergence boundary, it would be   interesting to compare with Karlsson's theory of \textit{stars at infinity} \cite{Ka2}. Defined using half-spaces, the stars describe some incidental relation on boundary,   but do not form a partition in general. A nice feature of his theory allows to formulate  an   analogue of convergence group property. Some partial convergence property  also exists on  our conical points but is not included as we do not see an application.  
  
 
\subsubsection*{\textbf{Further questions}}
We discuss some  questions we found interesting about the convergence boundary. The notion of contracting subsets in this paper is usually called \textit{strongly contracting} in literature: there are many co-existing notions of contracting subsets (see \cite{ACGH}).  We are working  with proper {geodesic} metric spaces for simplicity, but the framework proposed here could be adapted to non-geodesic and non-symmetric metric spaces with strongly contracting elements. Here are two concrete examples in potential applications:
\begin{enumerate}
    \item 
    Martin boundary of  groups with nontrivial Floyd boundary, that is the horofunction boundary of Green metric (see \cite{GGPY}).
    \item
    Horofunction boundary of outer space for $\out(\mathbb F_n)$ with Lipschitz (asymmetric) metric. 
\end{enumerate}
On the other hand, a  challenging question would be the following.
\begin{quest}
Does there exist  a ``good" theory of the convergence boundary  for groups with weakly contracting elements / Morse elements, and the corresponding conformal density at infinity?  
\end{quest}

Boundary comparison is an interesting research topic, and  has been well-studied for convergence group actions, see the work of Floyd and Gerasimov \cite{Floyd, Ge2}. Recently, a surjective continuous map from Martin boundary to Floyd boundary was discovered in \cite{GGPY}, where the Martin boundary can be seen as the horofunction boundary of the Green metric.  As    the dynamical notions of conical points and Myrberg points  witness, the convergence boundary shows many similarities with  convergence group actions. In view of these results, we are interested in understanding.
\begin{quest}
Which group action on a convergence compactification $G \act \bU$  admit a  quotient with  a nontrivial convergence group action?     
\end{quest}

All known examples with this property  have   non-trivial Floyd boundary. It is easy to check that the horofunction boundary for word metric surjects to the Floyd boundary (cf. Lemma \ref{HoroFloydMap}).  It is still open that whether such groups are necessarily relatively hyperbolic. 

The last question concerns about the Myrberg  set. We  prove in Lemma \ref{HomeoMybergRHG} that Myrberg  set in Floyd boundary persists in quotients: there exists a homeomorphism from the Myrberg  set in the Floyd boundary onto the  Myrberg  set in any of its quotient with a non-trivial convergence group action.  This result seems  providing a positive evidence to the following.

\begin{conj}
Suppose a group $G$ acts geometrically on two proper CAT(0) spaces $X$ and $Y$. Then there exists a $G$-equivariant homeomorphism  between the corresponding Myrberg  sets. 
\end{conj}

As the two CAT(0) actions share the same set of rank-1 elements, it  still has some chance  to have a positive answer. Compare with the famous examples of Croke--Kleiner \cite{CK00} and the        quasi-isometric invariant  boundaries recently studied in  \cite{ChaSul, Co17, QR22}.
    
At the completion of writing of this paper, the author received a preprint of Gekhtman--Qing--Rafi \cite{GQR22} which contains similar results as Theorems \ref{PoissonBdry} and \ref{CATnullity} for sublinearly Morse directions.   On this regard, it would be interesting to compare sublinearly Morse boundary with Myrberg  set. 

Recently, R. Coulon independently developed a Patterson--Sullivan theory on the horofunction boundary     in \cite{C22}, in a large overlap with ours, including  Shadow lemma \ref{ShadowLem},   Theorems \ref{HTSThm}, \ref{UniqueConfThm}, and   \ref{CogrowthThm}. However, our proof of Theorem \ref{HTSThm}  via the projection complex machinery are very different from his approach, and is actually a novelty of our study.  The measurable quotient of horofunction boundary in Theorem \ref{UniqueConfThm}   is a Lebesgue space, which is   not known in \cite{C22}.

\subsubsection*{\textbf{Organization of the paper}} 
The paper is  organized into three parts. The first part from Sections \textsection \ref{SecPrelim} -- \textsection\ref{SecHorobdry} presents the basic theory of convergence boundary.  After the preparatory section \textsection\ref{SecPrelim}, we introduce the notion of convergence boundary in \textsection\ref{SecBoundary} and derive basic facts for later use. The study of conical points and Myrberg points is carried out in  \textsection \ref{SecMyrberg}, with   Proposition \ref{MyrbergConicalThm} and Theorem \ref{EmbedGromovThm} proved among others. Theorem \ref{ContractiveThm1} about convergence of horofunction boundary is shown in \textsection\ref{SecHorobdry}.

The second part from Sections \textsection\ref{SecDensity} -- \textsection\ref{SecUniqueness} develops  the theory of  conformal density on the convergence boundary. Shadow Lemma \ref{ShadowLem} and Shadow Principle \ref{ShadowPrinciple} are established in \textsection\ref{SecDensity}, and the Hopf--Tsuji--Sullivan Theorem \ref{HTSThm} in \textsection\ref{SecHTST}. Section \textsection\ref{SecUniqueness} contains two different but related results:  the unique  conformal density Theorem \ref{UniqueConfThm} for divergent groups, and the reduced Myrberg  set as Poisson boundary  stated in Theorem \ref{PoissonBdry}.    

The final Sections \ref{SecCogrowth}, \ref{SecRHG}, \ref{SecCAT0} and \ref{SecMCG}  collect various applications of this study; in particular, the last three Sections are devoted accordingly to Case Studies of Relatively hyperbolic groups, CAT(0) groups and Mapping class groups. In \textsection\ref{SecCogrowth}, we establish a co-growth tightness for normal subgroups in a general divergent action, answering questions in \cite{AC20}. Among others, Theorem \ref{HomeoMyrbergIntro} about homeomorphic Myrberg sets is proved in \textsection\ref{SecCAT0} for relatively hyperbolic groups, and Theorem \ref{CATnullity} on nullity of contracting boundary  is derived as a corollary of a general theorem \ref{nullity} for CAT(0) groups in \textsection\ref{SecCAT0},  and Hopf--Tsuji-Sullivan Dichotomy  \ref{HTSModThm}  for  mapping class groups is proved in  \textsection\ref{SecMCG}.

\subsection*{Acknowledgments} 
This work   started in December 2015, where the author was visiting  Universit\'e de Lille 1 under a 3-month CNRS research fellowship. He wishes to thank Professor Leonid Potyagailo for the invitation and the hospitality of Math department. During the visit, the proof of shadow lemma was obtained after discussions with Ilya Gekhtman   and  was explained to   Fanny Kassel. The author thanks  Ilya   for    his  interests since then, and R\'emi Coulon for related conversations in August 2018 and May 2019.  Thanks  also to Hideki Miyachi for sharing his prints, and Anthony Genevois, Weixu Su and  Giulio Tiozzo for  many helpful discussions. Part of the writing was completed during the visit in 2022 to the Institute for Advanced Study in Mathematics in Zhejiang University. 
    

\section{Preliminary}\label{SecPrelim}

\subsection{Notation and Convention}
Let $(\U, d)$ be a proper geodesic metric space. The \textit{shortest projection} of  a point $y \in \U$ to a closed subset $X \subseteq \U$ is given by $$\pi_X(y)=\{x\in X: d(y, x)
= d(y, X)\},$$ and of a subset
$A \subseteq \U$,  $\pi_X(A) = \bigcup_{a \in A} \pi_X(a).$ Denote the diameter of a set $A$: $$\|A\|=\sup\{d(x,y): x,y\in A\}$$ If $F$ is a finite set of isometries and $o\in \U$ is a base point, we also denote 
$$\begin{aligned}
\|Fo\|=\max\{d(o,fo): f\in F\}\\
\|Fo\|_{\min}=\min\{d(o,fo): f\in F\}
\end{aligned}$$


Let $\alpha: [s,t]\to\U$ be an arc-length parameterized   path  from the initial point $\alpha^-=\alpha(s)$ to the terminal point $\alpha^+=\alpha(t)$. Denote by $[x,y]_\alpha$ the parametrized
subpath of $\alpha$ going from $x$ to $y$. Precisely, if $x=\alpha(s_0)$ and $y=\alpha(t_0)$, $[x,y]_\alpha$ denotes the restriction of $\alpha$ on $[s_0,t_0]$.  Without the sub-index, 
 $[x, y]$ denotes a choice of a geodesic between $x, y\in \U$. 

Let $\alpha :\mathbb R\to\U$ be a bi-infinite path. The restriction of $\alpha$ to $[a, +\infty)$ for $a\in\mathbb R$  is referred to  as a \textit{positive ray}, and its complement a \textit{negative ray}. By abuse of language, we often denote them by   $\alpha^+$ and $\alpha^-$. In particular,  if $\U$ is compactified with boundary $\pU$ and two half ways $\alpha^\pm$ converge in the compactification,  $\alpha^+$ and $\alpha^-$ acquire the usual sense as the end points of $\alpha$ in the compact space $\U\cup\pU$.

A path $\alpha$ is called a \textit{$c$-quasi-geodesic} for $c\ge 1$ if for   any  subpath $\beta$,
$$\len(\beta)\le c \cdot d(\beta^-, \beta^+)+c$$
 where   $\ell(\beta)$ denotes the length of $\beta$.
 
Denote by $\alpha\cdot \beta$ (or simply $\alpha\beta$) the concatenation of two paths $\alpha, \beta$  provided that $\alpha^+ =
\beta^-$.

Let $f, g$ be real-valued functions. Then $f \prec_{c_i} g$ means that
there is a constant $C >0$ depending on parameters $c_i$ such that
$f < Cg$. The symbol $\succ_{c_i}  $ is defined similarly, and  $\asymp_{c_i}$ means both $\prec_{c_i}  $ and $\succ_{c_i}  $  are true. The constant $c_i$ will be omitted if it is a universal constant.

The group action on $\U$ is always assumed  by isometry and the identity in a group $G$ shall be denoted by $1$.

\subsection{Contracting subsets}\label{SContractingSubset}
\begin{defn}[Contracting subset]\label{ContrDefn}
For given $C\ge 1$, a subset $U$ in $\U$ is called $C$-\textit{contracting}   if for any geodesic $\gamma$ with $d(\gamma,
U) \ge C$, we have
$$\|\pi_{U} (\gamma)\|  \le C.$$
 A
collection of $C$-contracting subsets is referred to
as a $C$-\textit{contracting system}. 
\end{defn}

Contracting property has several equivalent characterizations. When speaking about $C$-contracting property, the constant $C$ shall be assumed to satisfy  the following three statements.

\begin{lem}\label{BigThree}
Let $U$ be a contracting subset. Then there exists $C>0$ such that 
\begin{enumerate}
\item
If $d(\gamma,
U) \ge C$ for a geodesic  $\gamma$, we have
$\|\pi_U (\gamma)\|  \le C.$
\item
If $\|\pi_U (\gamma)\|  \ge  C$ then $d(\pi_U(\gamma^-),\gamma)\le C, \;d(\pi_U(\gamma^+),\gamma)\le C$.
\item
For a metric ball $B$ disjoint with $U$, we have $\|\pi_U(B)\|\le C$.
\end{enumerate}
\end{lem}

\begin{figure}
    \centering

\tikzset{every picture/.style={line width=0.75pt}} 

\begin{tikzpicture}[x=0.75pt,y=0.75pt,yscale=-1,xscale=1]

\draw    (61.5,190) .. controls (120.5,154) and (184.5,157) .. (256.5,190) ;
\draw   (124.5,108.25) .. controls (124.5,91.27) and (138.27,77.5) .. (155.25,77.5) .. controls (172.23,77.5) and (186,91.27) .. (186,108.25) .. controls (186,125.23) and (172.23,139) .. (155.25,139) .. controls (138.27,139) and (124.5,125.23) .. (124.5,108.25) -- cycle ;
\draw  [dash pattern={on 4.5pt off 4.5pt}]  (186,108.25) .. controls (166.5,126.78) and (166.49,135.56) .. (166.5,161.01) ;
\draw [shift={(166.5,163)}, rotate = 270] [color={rgb, 255:red, 0; green, 0; blue, 0 }  ][line width=0.75]    (10.93,-3.29) .. controls (6.95,-1.4) and (3.31,-0.3) .. (0,0) .. controls (3.31,0.3) and (6.95,1.4) .. (10.93,3.29)   ;
\draw [shift={(186,108.25)}, rotate = 136.47] [color={rgb, 255:red, 0; green, 0; blue, 0 }  ][fill={rgb, 255:red, 0; green, 0; blue, 0 }  ][line width=0.75]      (0, 0) circle [x radius= 3.35, y radius= 3.35]   ;
\draw  [dash pattern={on 4.5pt off 4.5pt}]  (124.5,108.25) .. controls (139.77,125.73) and (147.27,120.17) .. (144.63,161.1) ;
\draw [shift={(144.5,163)}, rotate = 273.99] [color={rgb, 255:red, 0; green, 0; blue, 0 }  ][line width=0.75]    (10.93,-3.29) .. controls (6.95,-1.4) and (3.31,-0.3) .. (0,0) .. controls (3.31,0.3) and (6.95,1.4) .. (10.93,3.29)   ;
\draw [shift={(124.5,108.25)}, rotate = 48.87] [color={rgb, 255:red, 0; green, 0; blue, 0 }  ][fill={rgb, 255:red, 0; green, 0; blue, 0 }  ][line width=0.75]      (0, 0) circle [x radius= 3.35, y radius= 3.35]   ;
\draw    (316,184) .. controls (356,154) and (454,157) .. (498.5,182) ;
\draw    (341,102) .. controls (346.5,132) and (348.5,148) .. (375.5,152) .. controls (402.5,156) and (431.5,154) .. (450.5,152) .. controls (469.5,150) and (476.5,125) .. (479.5,101) ;
\draw [shift={(479.5,101)}, rotate = 277.13] [color={rgb, 255:red, 0; green, 0; blue, 0 }  ][fill={rgb, 255:red, 0; green, 0; blue, 0 }  ][line width=0.75]      (0, 0) circle [x radius= 3.35, y radius= 3.35]   ;
\draw [shift={(341,102)}, rotate = 79.61] [color={rgb, 255:red, 0; green, 0; blue, 0 }  ][fill={rgb, 255:red, 0; green, 0; blue, 0 }  ][line width=0.75]      (0, 0) circle [x radius= 3.35, y radius= 3.35]   ;
\draw  [line width=0.75]  (144.5,167) .. controls (144.5,170.29) and (146.15,171.94) .. (149.44,171.94) -- (149.44,171.94) .. controls (154.15,171.94) and (156.5,173.59) .. (156.5,176.88) .. controls (156.5,173.59) and (158.85,171.94) .. (163.56,171.94)(161.44,171.94) -- (163.56,171.94) .. controls (166.85,171.94) and (168.5,170.29) .. (168.5,167) ;
\draw    (352.5,168) -- (355.26,144.99) ;
\draw [shift={(355.5,143)}, rotate = 96.84] [color={rgb, 255:red, 0; green, 0; blue, 0 }  ][line width=0.75]    (10.93,-3.29) .. controls (6.95,-1.4) and (3.31,-0.3) .. (0,0) .. controls (3.31,0.3) and (6.95,1.4) .. (10.93,3.29)   ;
\draw [shift={(352.5,168)}, rotate = 276.84] [color={rgb, 255:red, 0; green, 0; blue, 0 }  ][fill={rgb, 255:red, 0; green, 0; blue, 0 }  ][line width=0.75]      (0, 0) circle [x radius= 3.35, y radius= 3.35]   ;
\draw    (465.5,170) -- (463.67,148.99) ;
\draw [shift={(463.5,147)}, rotate = 85.03] [color={rgb, 255:red, 0; green, 0; blue, 0 }  ][line width=0.75]    (10.93,-3.29) .. controls (6.95,-1.4) and (3.31,-0.3) .. (0,0) .. controls (3.31,0.3) and (6.95,1.4) .. (10.93,3.29)   ;
\draw [shift={(465.5,170)}, rotate = 265.03] [color={rgb, 255:red, 0; green, 0; blue, 0 }  ][fill={rgb, 255:red, 0; green, 0; blue, 0 }  ][line width=0.75]      (0, 0) circle [x radius= 3.35, y radius= 3.35]   ;

\draw (139,177.4) node [anchor=north west][inner sep=0.75pt]    {$\leq C$};
\draw (60,163.4) node [anchor=north west][inner sep=0.75pt]    {$U $};
\draw (150,97.4) node [anchor=north west][inner sep=0.75pt]    {$B$};
\draw (405,172.4) node [anchor=north west][inner sep=0.75pt]    {$U $};
\draw (471,149.4) node [anchor=north west][inner sep=0.75pt]    {$\leq C$};
\draw (323,81.4) node [anchor=north west][inner sep=0.75pt]    {$x$};
\draw (488,86.4) node [anchor=north west][inner sep=0.75pt]    {$y$};
\draw (316,147.4) node [anchor=north west][inner sep=0.75pt]    {$C\geq $};
\draw (333,174.4) node [anchor=north west][inner sep=0.75pt]    {$\pi _{U }( x)$};
\draw (446,174.4) node [anchor=north west][inner sep=0.75pt]    {$\pi _{U }( y)$};

\end{tikzpicture}
    \caption{Contracting property in Lemma \ref{BigThree}}
    \label{fig:contracting}
\end{figure}



We continue to list a few more consequences used through out this paper.

A subset $U\subseteq\U$ is called \textit{$\sigma$-Morse} for a function $\sigma: \mathbb R^+ \to \mathbb R^+$ if  given $c \ge
1$,  any $c$-quasi-geodesic with endpoints in $U$ lies in $N_{\sigma(c)}(U)$. 

\begin{lem}\label{BigFive}
Let $U\subseteq \U$ be a  $C$-contracting subset for $C>0$. 

\begin{enumerate}
\item

There exists $\sigma=\sigma(C)$ such that   $U$ is $\sigma$-Morse. 
\item
For any $r>0$, there exists $\hat C=\hat C(C, r)$ such that a subset $V\subseteq \U$ within $r$-Hausdorff distance to $U$ is $\hat C$-contracting.

\item
If  a geodesic $\gamma$  intersects $N_r(U)$ only at the endpoint $\gamma^-$ for given $r\ge C$, then $$ \pi_U(\gamma^+)\subseteq B(\gamma^-, r+ C).$$
\item
For any  geodesic $\gamma$, we have $$\big |d_U(\gamma^-,\gamma^+)- \|\gamma\cap N_C(U)\|\big|\leq 4C.$$

\item
There exists $\hat C=\hat C(C)$ such that   $d_U(y, z)\le d(y, z)+\hat C$ for any $y,z\in \U$.

\end{enumerate}
\end{lem}
\begin{proof}
All the statements are well-known or straightforward. We leave them as exercises to the interested reader.    
\end{proof}


In this paper, we are interested in a  contracting system $\f$ with  \textit{$\tau$-bounded intersection} property for a function $\tau: \mathbb R_{\ge 0}\to \mathbb R_{\ge 0}$ so that the
following holds
$$\forall U, V\in \f: \quad \|{N_r (U) \cap N_r (V)}\| \le \tau(r)$$
for any $r \geq 0$. This  is, in fact, equivalent to a \textit{bounded projection
property} of $\f$:  there exists a constant $B>0$ such that the
following holds
$$\|\pi_{U}(V)\| \le B$$
for any $U\ne V \in \f$.  See \cite[Lemma 2.3]{YANG7} for a proof.

We now state two elementary lemmas used later on, which  precises the same idea: the projection  of a geodesic to a contracting subset equals the intersection up to a bounded amount. 
This could be thought of as a  general version of the assertion (4) in Lemma \ref{BigFive}. 

Consider   a $C$-contracting subset $U\subseteq \U$ and  $r\ge 100C$. Let $\alpha$ be a geodesic so that $\| N_r(U)\cap \alpha\|>3r$. Recall that the entry and exit points, $x,y\in \alpha\cap N_r(U)$, of $\alpha$ in $N_r(U)$  satisfy $d(x,y)=\|N_r(U)\cap \alpha\|$. 

  
\begin{lem}\label{Transform}
The   entry and exit points of $\alpha$ in $N_C(U)$ are $(3r/2)$-close to the entry and exit points  of $\alpha$ in $N_r(U)$ respectively. In particular, 
\begin{align}
\label{EQchangeNbhd} \| N_C(U)\cap \alpha\|\ge \|N_r(U)\cap \alpha \| -3r\\
\label{EQProj=Intersec} |d_U(\alpha^-, \alpha^+)- \|N_r(U)\cap \alpha\||\le 4r
\end{align}
\end{lem} 
\begin{proof}
By hypothesis, we have $d(x,y)> 3r$. First notice that $N_C(U)\cap \alpha\ne\emptyset$. Otherwise, if $N_C(U)\cap \alpha=\emptyset$, the contracting property gives  $$d(x,y)\le d(x,U)+d(y,U)+\|\pi_U(\alpha)\|\le 2r+C<3r.$$ This is a contradiction, confirming the above observation. 

Let us   thus choose  $u, v\in [x,y]\cap N_C(U)$ such that $d(u, v)=\|N_C(U)\cap \alpha \|$.  Using again contracting property of $U$, we obtain $$d(x,u)\le d(x,U)+d(u,U)+\|\pi_U([x,u])\|\le r+2C\le 3r/2.$$ Similarly,  $d(v, y)\le r+2C\le 3r/2$. Thus, $d(u,v)\ge d(x,y)-3r$, so the (\ref{EQchangeNbhd}) follows. 

To show the inequality (\ref{EQProj=Intersec}), note that $\pi_U(\alpha^-)$ (resp. $\pi_U(\alpha^+)$) is $2C$-close to $u$ (resp. $v$). Hence, (\ref{EQProj=Intersec}) follows from (\ref{EQchangeNbhd}). The proof is complete.
\end{proof}

\begin{lem}\label{TwoGeodesicsEnterULem}
Let $\beta$ be a geodesic such that 
$d_U(\alpha^-,\beta^-)\le 10C$ and $d_U(\alpha^+,\beta^+)\le 10C$. Then the entry and exit points of  $\beta$ in $N_r(U)$ are $(4r)$-close to $x$ and $y$ respectively. 
\end{lem}
\begin{proof}
We first prove that $\beta\cap N_C(U)\ne\emptyset$. Indeed, assume to the contrary that $\beta\cap N_C(U)=\emptyset$, so $\|\pi_U(\beta)\|\le C$ by the contracting property.

Since $[x, \alpha^-]\cap N_r(U)=\emptyset$ and $[ \alpha^-, y]\cap N_r(U)=\emptyset$, the contracting property again shows $d_U(x, \alpha^-) \le C$ and $d_U(\alpha^-, y)\le C$.  We compute
$$
\begin{array}{rl}
d(x, y) \le& d(x, U)+d_U(x, \alpha^-)  + d_U(\alpha^-,\beta^-) \\
&+\|\pi_U(\beta)\|+ d_U(\alpha^+,\beta^+) +d_U(\alpha^-, y) + d(y, U)\\
\le & 14C+2r <3r.
\end{array}
$$ 
This contradicts the assumption that $d(x,y)>3r$.  Thus, $\beta\cap N_C(U)\ne\emptyset$ is proved, so we can   choose  $u, v\in \beta\cap N_C(U)$ such that $d(u, v)=\|N_C(U)\cap \beta \|$.  Using the contracting property of $U$, we obtain 
$$\begin{array}{rl}
d(x,u)\le & d(x,U)+d_U(x,\alpha^-)+d_U(\alpha^-,\beta^-)+d_U(\beta^-,u)+ d(u,U)\\
\le &2 r+12C\le (5/2)r.
\end{array}
$$ Similarly, we have $d(v, y)\le 2r+12C\le (5/2)r$.

By Lemma \ref{Transform}, the   enter and exit points of $\beta$ in $N_r(U)$ are $(3r/2)$-close to $u$ and $v$ respectively, so are $(4r)$-close to the enter and exit points $x,y$ of  $\alpha$ respectively. The proof of the lemma is proved. 
\end{proof}

\subsection{Contracting elements}
A group  $H$ is
called \textit{contracting} if for some (hence any) $o \in \U$, the
subset $Ho$ is contracting in $\U$. 


An  isometry $h \in \isom(\U)$ is called  
\textit{contracting} if it is of infinite order and the subgroup $\langle h \rangle$ is contracting. The set of contracting isometries is preserved under conjugacy. We say that $h$ has the \textit{quasi-isometric image} property if the orbital map
\begin{equation}\label{QIEmbed}
n\in \mathbb Z\to h^no \in \mathrm Y
\end{equation}
is a quasi-isometric embedding. We then define the axis of $h$ as the following concatenating path:
\begin{equation}\label{axisdefn0}
\ax(h)=\cup_{n\in \mathbb Z} h^n[o,ho].
\end{equation}

Let $G<\isom(\U)$ be a discrete group, that is, which acts properly on $\U$. 

\begin{lem}\cite[Lemma 2.11]{YANG10}\label{elementarygroup}
Let $h\in G$ be a contracting element  with quasi-isometric image. Then $h$ is contained in a maximal elementary subgroup defined as follows 
$$
E(h)=\{g\in G: \exists n\in \mathbb N_{> 0}, (\;gh^ng^{-1}=h^n)\; \lor\;  (gh^ng^{-1}=h^{-n})\}.
$$
\end{lem}
 
Keeping in mind the basepoint $o\in\U$, the \textit{axis} of $h$  is defined as the following quasi-geodesic 
\begin{equation}\label{axisdefn}
\ax(h)=\{f o: f\in E(h)\}.
\end{equation} Notice that $\ax(h)=\ax(k)$ and $E(h)=E(k)$    for any contracting element   $k\in E(h)$.
\begin{rem} [Definition of axis]
For most results in the paper,   contracting elements   are contained in a discrete group, for which we make use of the axis in (\ref{axisdefn}). As $\langle h\rangle$ is of finite index, $E(h)\cdot o$ has a finite Hausdorff distance to $\langle h\rangle\cdot o$. Hence, there exists no essential difference in these two definitions.        
\end{rem}

An element $g\in G$ \textit{preserves the orientation} of a bi-infinite quasi-geodesic $\gamma$ if $\alpha$ and $g\alpha$ has finite Hausdorff distance for any half ray $\alpha$ of  $\gamma$. Let $E^+(h)$ be the subgroup of $E(h)$ with possibly index 2 whose elements  preserve  the orientation of their axis. Then we have 
$$
E^+(h)=\{g\in G: \exists n\in \mathbb N_{> 0}, \;gh^ng^{-1}=h^n\}.
$$
and $E^+(h)$ contains all contracting elements in $E(h)$. In the sequel, unless explicitly mentioned, we always assume  contracting elements to satisfy quasi-isometrically embedded image property (\ref{QIEmbed}).

Two  contracting elements $h_1, h_2\in G$  are called \textit{independent} if the collection $\{g\ax(h_i): g\in G;\ i=1, 2\}$ is a contracting system with bounded intersection. Note that two conjugate contracting elements with disjoint fixed points are not independent in our sense. This is slightly different from the use of  independence   by other researchers (cf. \cite{MT}).  



\begin{defn}\label{barriers}
Fix $r>0$ and a set $F$ in $G$.  A geodesic $\gamma$ contains an \textit{$(r, f)$-barrier} for $f\in F$   if there exists    an element $g \in G$ so that 
\begin{equation}\label{barrierEQ}
\max\{d(g\cdot o, \gamma), \; d(g\cdot fo, \gamma)\}\le r.
\end{equation}
By abuse of language,    the point $go$ or  the axis $g\ax(f)$ is called {$(r, F)$-barrier} of type $f$ on $\gamma$.
\end{defn}

Recall that $\|Fo\|$ denotes the diameter of $Fo$, and   $\|Fo\|_{\min}:=\min\{d(o,fo): f\in F\}$ the minimal length of elements in $F$. As $F$ is usually chosen as a finite set (of three elements), their distinction does not really matter and could be ignored on a first reading.  

\begin{lem}\label{BarrierFellowLem}
For any $r>0$ there exists $\hat r=\hat r(r) >0$ with the following property. 
 
Let $f$ be a contracting element with $d(o,fo)> 3r$. Suppose that a geodesic $\alpha$ contains an $(r, f)$-barrier $go$. Set $X=g\ax(f)$. Let $\beta$ be a geodesic such that $d_X(\alpha^-,\beta^-)\le 10C$ and $d_X(\alpha^+,\beta^+)\le 10C$.  Then $go$ is an $(\hat r, f)$-barrier for $\beta$.
\end{lem}

\begin{proof}
Let $x, y\in \alpha\cap N_r(X)$ be the corresponding entry and exit points. The   $(r, f)$-barrier $go$   for $\alpha$  implies that $d(go, [x,y]_\alpha)\le r$ and $d(gfo, [x,y]_\alpha)\le r$. By Lemma \ref{TwoGeodesicsEnterULem}, the entry and exit points $u, v$ of  $\beta$ in $N_r(X)$ are $(4r)$-close to those $x, y$ respectively. 

As $X$ is a contracting quasi-geodesic, any geodesic subsegment (say $[u, v]$) in its $r$-neighborhood is also uniformly contracting by \cite[Prop. 2.2]{YANG11}. By    Morse property, there exists $\tau=\tau(r)>0$ such that   $[x,y]$  is contained in a $\tau$-neighborhood of $[u,v]$. Setting  $\hat r=4r+\tau$ then completes the proof. 
\end{proof}


Let $h$ be a contracting element  with a $C$-contracting axis  $\gamma:=\ax(h)$  in (\ref{axisdefn}).  Here we understand  $\gamma$  as a quasi-geodesic path for simplicity.  Let          $f\in E^+(h)$ be an orientation-preserving element which must be contracting. By the Morse property, there exist constants   $r, L$ depending only on $C$ such that    any geodesic segment in $N_r(\gamma)$ of length $L+ d(o, fo)$ contains $[o,fo]$.  We record this fact as the following.
 \begin{lem}\label{AnyBarrierLem}
There exist  $r=r(C), L=L(C)>0$ with the following property.

Let    $x, y\in \gamma\cap N_r(\alpha)$ be the entry and exit points of $\gamma$ in a geodesic segment $\alpha$.  If $z\in [x,y]_\gamma$ is a point such that $\min\{d(z,x), d(z,y)\}\ge L+ d(o, fo)$, then $z$  is   an $(r, F)$-barrier for   $\alpha$, where $F:=\{f,f^{-1}\}$
\end{lem}

Lastly,   the following result will be used later on.
\begin{lem}\label{InfiniteNormalLem}\cite[Lemma 4.6]{YANG6}
Let $G\act \U$ be as in (\ref{GActUAssump}) and  $\Gamma<G$ be an infinite normal subgroup. Then $\Gamma$ must be non-elementary and contains infinitely many independent contracting elements.
\end{lem}

\subsection{Admissible paths}
Let $\f$ be a $C$-contracting system in the space $\U$; that is, a collection of $C$-contracting subsets for a common $C>0$. A class of admissible paths  introduced in \cite{YANG6} is of great importance in our study,  which are  quasi-geodesics with   fellow travel property relative to   such a system $\f$. In most cases, this paper considers the $C$-contracting system $\f$ as in (\ref{SystemFDefIntro}).

\begin{figure}
    \centering

\tikzset{every picture/.style={line width=0.75pt}} 

\begin{tikzpicture}[x=0.75pt,y=0.75pt,yscale=-1,xscale=1]

\draw   (80.5,98.5) .. controls (80.5,83.31) and (99.53,71) .. (123,71) .. controls (146.47,71) and (165.5,83.31) .. (165.5,98.5) .. controls (165.5,113.69) and (146.47,126) .. (123,126) .. controls (99.53,126) and (80.5,113.69) .. (80.5,98.5) -- cycle ;
\draw   (221,100.5) .. controls (221,85.31) and (236.33,73) .. (255.25,73) .. controls (274.17,73) and (289.5,85.31) .. (289.5,100.5) .. controls (289.5,115.69) and (274.17,128) .. (255.25,128) .. controls (236.33,128) and (221,115.69) .. (221,100.5) -- cycle ;
\draw   (461,109) .. controls (461,95.19) and (482.15,84) .. (508.25,84) .. controls (534.35,84) and (555.5,95.19) .. (555.5,109) .. controls (555.5,122.81) and (534.35,134) .. (508.25,134) .. controls (482.15,134) and (461,122.81) .. (461,109) -- cycle ;
\draw   (361,100.5) .. controls (361,85.59) and (375.33,73.5) .. (393,73.5) .. controls (410.67,73.5) and (425,85.59) .. (425,100.5) .. controls (425,115.41) and (410.67,127.5) .. (393,127.5) .. controls (375.33,127.5) and (361,115.41) .. (361,100.5) -- cycle ;
\draw  [dash pattern={on 4.5pt off 4.5pt}]  (165.5,98.5) -- (219,100.43) ;
\draw [shift={(221,100.5)}, rotate = 182.06] [color={rgb, 255:red, 0; green, 0; blue, 0 }  ][line width=0.75]    (10.93,-3.29) .. controls (6.95,-1.4) and (3.31,-0.3) .. (0,0) .. controls (3.31,0.3) and (6.95,1.4) .. (10.93,3.29)   ;
\draw  [dash pattern={on 4.5pt off 4.5pt}]  (425,100.5) -- (459.02,105.7) ;
\draw [shift={(461,106)}, rotate = 188.69] [color={rgb, 255:red, 0; green, 0; blue, 0 }  ][line width=0.75]    (10.93,-3.29) .. controls (6.95,-1.4) and (3.31,-0.3) .. (0,0) .. controls (3.31,0.3) and (6.95,1.4) .. (10.93,3.29)   ;
\draw    (289.5,100.5) -- (361,100.5) ;
\draw [shift={(361,100.5)}, rotate = 0] [color={rgb, 255:red, 0; green, 0; blue, 0 }  ][fill={rgb, 255:red, 0; green, 0; blue, 0 }  ][line width=0.75]      (0, 0) circle [x radius= 3.35, y radius= 3.35]   ;
\draw [shift={(289.5,100.5)}, rotate = 0] [color={rgb, 255:red, 0; green, 0; blue, 0 }  ][fill={rgb, 255:red, 0; green, 0; blue, 0 }  ][line width=0.75]      (0, 0) circle [x radius= 3.35, y radius= 3.35]   ;
\draw    (299.5,140) .. controls (305.29,130.35) and (319,118.84) .. (296.13,107.26) ;
\draw [shift={(293.5,106)}, rotate = 24.36] [fill={rgb, 255:red, 0; green, 0; blue, 0 }  ][line width=0.08]  [draw opacity=0] (8.93,-4.29) -- (0,0) -- (8.93,4.29) -- cycle    ;
\draw    (346.5,140) .. controls (338.7,134.15) and (322.34,121.65) .. (353.04,106.19) ;
\draw [shift={(355.5,105)}, rotate = 154.8] [fill={rgb, 255:red, 0; green, 0; blue, 0 }  ][line width=0.08]  [draw opacity=0] (8.93,-4.29) -- (0,0) -- (8.93,4.29) -- cycle    ;
\draw    (92.5,119) -- (165.5,98.5) ;
\draw [shift={(165.5,98.5)}, rotate = 344.31] [color={rgb, 255:red, 0; green, 0; blue, 0 }  ][fill={rgb, 255:red, 0; green, 0; blue, 0 }  ][line width=0.75]      (0, 0) circle [x radius= 3.35, y radius= 3.35]   ;
\draw [shift={(92.5,119)}, rotate = 344.31] [color={rgb, 255:red, 0; green, 0; blue, 0 }  ][fill={rgb, 255:red, 0; green, 0; blue, 0 }  ][line width=0.75]      (0, 0) circle [x radius= 3.35, y radius= 3.35]   ;
\draw    (221,100.5) -- (289.5,100.5) ;
\draw [shift={(221,100.5)}, rotate = 0] [color={rgb, 255:red, 0; green, 0; blue, 0 }  ][fill={rgb, 255:red, 0; green, 0; blue, 0 }  ][line width=0.75]      (0, 0) circle [x radius= 3.35, y radius= 3.35]   ;
\draw    (361,100.5) -- (425,100.5) ;
\draw [shift={(425,100.5)}, rotate = 0] [color={rgb, 255:red, 0; green, 0; blue, 0 }  ][fill={rgb, 255:red, 0; green, 0; blue, 0 }  ][line width=0.75]      (0, 0) circle [x radius= 3.35, y radius= 3.35]   ;
\draw    (461,106) -- (550.5,122) ;
\draw [shift={(550.5,122)}, rotate = 10.14] [color={rgb, 255:red, 0; green, 0; blue, 0 }  ][fill={rgb, 255:red, 0; green, 0; blue, 0 }  ][line width=0.75]      (0, 0) circle [x radius= 3.35, y radius= 3.35]   ;
\draw [shift={(461,106)}, rotate = 10.14] [color={rgb, 255:red, 0; green, 0; blue, 0 }  ][fill={rgb, 255:red, 0; green, 0; blue, 0 }  ][line width=0.75]      (0, 0) circle [x radius= 3.35, y radius= 3.35]   ;

\draw (108,48.4) node [anchor=north west][inner sep=0.75pt]    {$X_{0}$};
\draw (248,43.4) node [anchor=north west][inner sep=0.75pt]    {$X_{i}$};
\draw (379,49.4) node [anchor=north west][inner sep=0.75pt]    {$X_{i+1}$};
\draw (501,57.4) node [anchor=north west][inner sep=0.75pt]    {$X_{n}$};
\draw (248,144.4) node [anchor=north west][inner sep=0.75pt]    {$\| \pi _{X_{i}}( q_{i}) \| ,$};
\draw (322,144.4) node [anchor=north west][inner sep=0.75pt]    {$\| \pi _{X_{i+1}}( q_{i}) \| \leq B$};
\draw (250,79.4) node [anchor=north west][inner sep=0.75pt]    {$p_{i}$};
\draw (116,87.4) node [anchor=north west][inner sep=0.75pt]    {$p_{0}$};
\draw (380,80.4) node [anchor=north west][inner sep=0.75pt]    {$p_{i+1}$};
\draw (319,76.4) node [anchor=north west][inner sep=0.75pt]    {$q_{i}$};
\draw (501,92.4) node [anchor=north west][inner sep=0.75pt]    {$p_{n}$};
\draw (239,103.4) node [anchor=north west][inner sep=0.75pt]    {$\geq L$};
\draw (374,101.4) node [anchor=north west][inner sep=0.75pt]    {$\geq L$};

\end{tikzpicture}
    \caption{Admissible path}
    \label{fig:admissiblepath}
\end{figure}

\begin{defn}[Admissible Path]\label{AdmDef}
Given $L, B \ge 0$, a path $\gamma$  is called \textit{$(L,
B)$-admissible} relative to $\f$, if  the path $\gamma = p_0 q_1
p_1 \ldots q_n p_n$ is  a piece-wise geodesic path  satisfying     the     \textit{Long Local} and \textit{Bounded Projection} properties:
\begin{enumerate}

\item[\textbf{(LL)}]
Each  $p_i$ for $0\le i\le n$ has the two endpoints in $X_i\in \f$ and  length bigger than  $L$ unless $i=0$ or $i=n$,

\item[\textbf{(BP)}]
For each $X_i$,  we have $X_i\ne X_{i+1}$ and 
$$
\|\pi_{X_i}(q_{i+1})\|\le B, \quad \|\pi_{X_i}(q_i)\|\le B
$$
whenever the previous and next paths $q_{i}, q_{i+1}$ are defined. See Fig. \ref{fig:admissiblepath} for an illustration.


\end{enumerate}
The indexed collection of $X_i \in \f$  as above, denoted by $\f(\gamma)$, will be referred to as the \textit{saturation} of $\gamma$. 

In the case that $p_n$ (or $q_n$ if $p_n$ is trivial)  is a geodesic ray, we say that $\gamma$ is an admissible ray.
\end{defn}

As the  properties \textbf{(LL)} and (\textbf{BP}) are  local conditions, we can produce admissible paths via the following operations: 
 
\begin{itemize}
    \item
    {\textbf{Subpath}} A subpath of an $(L, B)$-admissible path is $(L, B)$-admissible.  
    \item 
    {\textbf{Concatenation}} Let  $\alpha, \beta$ be two $(L, B)$-admissible paths so that  the last contracting subset associated to the geodesic   $p$ of $\alpha$ is the same as the first contracting subset  to the geodesic    $q$ of $\beta$. If  $d(p^-, q^+)>L$, the path $\gamma$  by concatenating paths $$\gamma:=[\alpha^-, p^-]_{\alpha} \cdot  [p^-, q^+] \cdot [q^+, \beta^+]_{\beta}$$ 
 is $(L, B)$-admissible.  
 
\end{itemize}


A sequence of points $x_i$ in a path $p$   is called \textit{linearly ordered} if $x_{i+1}\in [x_i, p^+]_p$ for each $i$. 

\begin{defn}[Fellow travel]\label{Fellow}
Let   $\gamma = p_0 q_1 p_1 \cdots q_n p_n$ be an $(L, B)-$admissible
path. We say $\gamma$ has \textit{$r$-fellow travel} property for some $r>0$   if for any geodesic  
$\alpha$  with the same endpoints as $\gamma$,   there exists a sequence of linearly ordered points $z_i,
w_i$ ($0 \le i \le n$) on $\alpha$ such that  
$$d(z_i, p_{i}^-) \le r,\quad d(w_i, p_{i}^+) \le r.$$
In particular, $\|N_r(X_i)\cap \alpha\|\ge L$ for each $X_i\in \f(\gamma)$. 
\end{defn}
\begin{rem}
Note that $\|N_r(X_i)\cap \alpha\|\ge L$ is basically amount to saying that  $X_i$ is an $(r, F)$-barrier for $\alpha$. Thus, the $R$-fellow travel property is closely related to the existence of $(r, F)$-barriers. These cumbersome terminologies may have their own convenience in practice.
\end{rem}
 
The following result  says that   a local long admissible path enjoys the fellow travel property.
 
\begin{prop}\label{admisProp}\cite{YANG6}
For any $B>0$, there exist $L,  r>0$ depending only on $B,C$ such that  any $(L, B)$-admissible path   has $r$-fellow travel property.
\end{prop}

The constant $B$   in an $(L, B)$-admissible path  can be made uniform by the following operation.
\begin{itemize}
\item {\textbf{Truncation}}
 By the contracting property, a geodesic issuing from a $C$-contracting subset  diverges in an orthogonal way  from its $C$-neighborhood. At the expense of decreasing $L$, we can thus  truncate these  diverging parts to obtain the constant $B$ depending only on $\f$. 
\end{itemize}

We explain  this procedure in a more general form. 

\begin{figure}
    \centering

\tikzset{every picture/.style={line width=0.75pt}} 

\begin{tikzpicture}[x=0.75pt,y=0.75pt,yscale=-1,xscale=1]

\draw   (256.97,105.4) .. controls (256.97,89.13) and (279.59,75.94) .. (307.48,75.94) .. controls (335.37,75.94) and (357.99,89.13) .. (357.99,105.4) .. controls (357.99,121.68) and (335.37,134.87) .. (307.48,134.87) .. controls (279.59,134.87) and (256.97,121.68) .. (256.97,105.4) -- cycle ;
\draw  [dash pattern={on 4.5pt off 4.5pt}] (241.46,105.4) .. controls (241.46,85.47) and (271.02,69.31) .. (307.48,69.31) .. controls (343.94,69.31) and (373.5,85.47) .. (373.5,105.4) .. controls (373.5,125.34) and (343.94,141.5) .. (307.48,141.5) .. controls (271.02,141.5) and (241.46,125.34) .. (241.46,105.4) -- cycle ;
\draw  [dash pattern={on 0.84pt off 2.51pt}]  (243.99,92.15) -- (361.59,84.78) ;
\draw [shift={(361.59,84.78)}, rotate = 356.42] [color={rgb, 255:red, 0; green, 0; blue, 0 }  ][fill={rgb, 255:red, 0; green, 0; blue, 0 }  ][line width=0.75]      (0, 0) circle [x radius= 3.35, y radius= 3.35]   ;
\draw [shift={(243.99,92.15)}, rotate = 356.42] [color={rgb, 255:red, 0; green, 0; blue, 0 }  ][fill={rgb, 255:red, 0; green, 0; blue, 0 }  ][line width=0.75]      (0, 0) circle [x radius= 3.35, y radius= 3.35]   ;
\draw    (256.97,103.93) -- (357.27,96.56) ;
\draw [shift={(357.27,96.56)}, rotate = 355.8] [color={rgb, 255:red, 0; green, 0; blue, 0 }  ][fill={rgb, 255:red, 0; green, 0; blue, 0 }  ][line width=0.75]      (0, 0) circle [x radius= 3.35, y radius= 3.35]   ;
\draw [shift={(256.97,103.93)}, rotate = 355.8] [color={rgb, 255:red, 0; green, 0; blue, 0 }  ][fill={rgb, 255:red, 0; green, 0; blue, 0 }  ][line width=0.75]      (0, 0) circle [x radius= 3.35, y radius= 3.35]   ;
\draw    (177.5,86) -- (254.05,103.27) ;
\draw [shift={(256.97,103.93)}, rotate = 192.71] [fill={rgb, 255:red, 0; green, 0; blue, 0 }  ][line width=0.08]  [draw opacity=0] (8.93,-4.29) -- (0,0) -- (8.93,4.29) -- cycle    ;
\draw    (357.27,96.56) -- (433.53,106.61) ;
\draw [shift={(436.5,107)}, rotate = 187.5] [fill={rgb, 255:red, 0; green, 0; blue, 0 }  ][line width=0.08]  [draw opacity=0] (8.93,-4.29) -- (0,0) -- (8.93,4.29) -- cycle    ;
\draw    (228.5,153) .. controls (268.1,123.3) and (247.42,174.95) .. (285.82,146.88) ;
\draw [shift={(287,146)}, rotate = 143.13] [color={rgb, 255:red, 0; green, 0; blue, 0 }  ][line width=0.75]    (10.93,-3.29) .. controls (6.95,-1.4) and (3.31,-0.3) .. (0,0) .. controls (3.31,0.3) and (6.95,1.4) .. (10.93,3.29)   ;

\draw (232.54,65.61) node [anchor=north west][inner sep=0.75pt]    {$z_{i} '$};
\draw (334.54,98.61) node [anchor=north west][inner sep=0.75pt]    {$w_{i}$};
\draw (266.54,102.61) node [anchor=north west][inner sep=0.75pt]    {$z_{i}$};
\draw (363.54,64.61) node [anchor=north west][inner sep=0.75pt]    {$w_{i} '$};
\draw (172,145.4) node [anchor=north west][inner sep=0.75pt]    {$N_{C}( X_{i})$};
\draw (147,74.4) node [anchor=north west][inner sep=0.75pt]    {$\hat{q}_{i-1}$};
\draw (446,103.4) node [anchor=north west][inner sep=0.75pt]    {$\hat{q}_{i+1}$};
\draw (303,104.4) node [anchor=north west][inner sep=0.75pt]    {$\hat{q}_{i}$};

\end{tikzpicture}
    \caption{Truncating a geodesic around $N_C(X_i)$.}
    \label{fig:truncation}
\end{figure}

\begin{lem}\label{Truncation}
Assume that $\|Fo\|_{\min}>3r$ for some $r>0$. There exist $L=L(\|Fo\|_{\min},r), B=B(\f)>0$ with the following property.

Let $\alpha$  be a geodesic    with a set of   distinct  $(r, F)$-barriers $X_i\in \f$. Then there exists  an $(L,  B)$-admissible  path $\hat \alpha$   with the same endpoints as $\alpha$ with saturation $\{X_i\}$.
\end{lem}
 The    path $\hat \alpha$ is called \textit{truncation} of $\alpha$. If $\alpha$ is a  fellow travel geodesic for the admissible path $\gamma$ in Proposition \ref{admisProp}, then   $\alpha$ has $(r, F)$-barriers  $X\in \f(\gamma)$. The truncation   $\hat \alpha$ will   also be referred as \textit{truncation} of $\gamma$ and has a  uniform constant $B$. 
\begin{proof}
By definition,   $\|\alpha\cap N_r(X_i)\|\ge \|Fo\|_{\min}$. Consider the entry and exit points $z_i',w_i'\in \alpha$ in $N_C(X_i)$   by Lemma \ref{Transform}  such that $$d(z_i',w_i')=\|\alpha\cap N_C(X_i)\|\ge \|Fo\|_{\min}-3r.$$ Let $z_i, w_i\in X_i$ such that $d(z_i,z_i')=d(w_i,w_i')=C$, where $z_0:=\alpha^-$ and $w_n:=\alpha^+$. See Fig. \ref{fig:truncation}. Since $\f$ has $\tau$-bounded intersection and then bounded projection, one deduces     $$\|\pi_{X_i}([w_i,z_{i+1}])\|\le  B$$ for some   $ B$ depending only on $\tau$ and $C$. Note that  $d(z_i,w_i)\ge L:= \|Fo\|_{\min}-3r-2C$. Denoting $\hat p_i=[z_i,w_i]$ and $\hat q_{i+1}=[w_i, z_{i+1}]$, the path $$\hat \alpha := \hat p_0 \hat q_1 \hat p_1 \ldots \hat q_n \hat p_n$$   is an $(  L,   B)$-admissible  path   with the same endpoints as $\alpha$.
\end{proof}

\subsection{Extension lemma}\label{extensionlemSec}

Admissible paths have demonstrated strong fellow travel properties. We now introduce a basic tool   to build admissible paths. 

\begin{lem}[Extension Lemma]\label{extend3}
There exist   $L, r, B>0$ depending only on $C$ with the following property.  

Let $h_1,h_2,h_3\in G$ be three independent contracting elements. Choose any element $f_i\in \langle h_i\rangle$ for each $1\le i\le 3$  to form the set $F$ satisfying $\|Fo\|_{\min}\ge L$. Let $g\in G$ and $\alpha$ be a geodesic starting at $\alpha^-=o$.
\begin{enumerate}
\item
There exists an element $f \in F$ such that   the path  $$\gamma:=[o, go]\cdot(g[o, fo])\cdot(gf\alpha)$$ is an $(L, B)$-admissible path relative to $\f$ (defined in (\ref{SystemFDefIntro})). 
\item
The point  $go$  is an $(r, f)$-barrier for any geodesic  $[\gamma^-,\gamma^+]$.	
\end{enumerate}
\end{lem}

\begin{rem}
\begin{enumerate}
\item
In \cite{YANG10}, the geodesic   $\alpha=[o, ho]$ is chosen to be ending at $Go$ for some $h\in G$. However, the proof given there works for a   geodesic starting from $o$ possibly ending at any point in $\U$. 
\item
Since admissible paths are local conditions, we can connect via $F$  any number of elements  $g\in G$ to satisfy (1) and (2). We refer the reader to \cite{YANG10} for precise formulation.
\end{enumerate}
\end{rem}

\subsection{Projection complex}\label{SSecProjectionComplex}
In this subsection, we  briefly recall  the work of Bestvina-Bromberg-Fujiwara  \cite{BBF} on constructing  a quasi-tree of spaces.

\begin{defn}[Projection axioms]
\label{defn:projaxioms}
Let $\f$ be a collection of   metric spaces equipped with (set-valued) projection maps  $$\{\pi_{U}:  \f\setminus \{U\}\to U\}_{U\in \f}.$$  Denote $d_{U}(V,W) := \| \pi_{U}(V) \cup \pi_{U}(W)\|$ for $V\ne U\ne W \in \f$.  The pair $(\f, \{\pi_{U}\}_{U\in \f})$ satisfies {\it projection axioms} for a {\it   constant} $\kappa \ge 0$ if
 \begin{enumerate}
     \item
     \label{axiom1}  $\|\pi_{U}(V)\| \le \kappa$ when $U \neq V$.
     \item
     \label{axiom2} if $U,V,W$ are distinct and $d_{V}(U, W) > \kappa$ then $d_{U}(V,W) \le \kappa$.
     \item
     \label{axiom3} the set $\{ U \in \f \,:\, d_{U}(V, W) > \kappa \}$ is finite for $V \neq W$.
 \end{enumerate}
 \end{defn}

By definition,  the following triangle inequality holds
\begin{equation}\label{axiom4}
d_Y(V,W)\le d_Y(V,U)+d_Y(U,W).
\end{equation}

It is well-known that the projection axioms hold for a contracting system  with bounded intersection (cf. \cite[Appendix]{YANG6}).
\begin{lem}\label{bp}
Let $h$ be any contracting element in a proper action of $G\act\U$. Then the collection  $\f=\{g\ax(h): g\in G\}$ with shortest   projection maps $\pi_U(V)$  satisfies the   projection   axioms with   constant $\kappa=\kappa(\f)>0$.
\end{lem}

In \cite[Def. 3.1]{BBF}, a modified version of $d_U$  is introduced so that it is symmetric and agrees with the original $d_U$ up to an additive amount $2\kappa$. So, the axioms (1)-(3) are still true for $3\kappa$, and the triangle inequality in (\ref{axiom4}) holds up to a uniform error. In what follows, we   actually need to work with this modified $d_U$ to define the complex projection, but for sake of simplicity, we stick on the above definition of $d_U$.

We consider the interval-like set for $K>0$ and $V, W\in \f$ as follows $$\f_K(V,W):=\{U\in\f: d_U(V,W)>\kappa\}.$$ Denote $\f_K[V,W]:=\f_K(V,W)\bigcup\{V,W\}$.
It possesses a total order described in the next lemma.

\begin{lem}\label{OrderLem}\cite [Theorem 3.3.G]{BBF}
There exist  constants $D=D(\kappa), K=K(\kappa)>0$ for the above $\f$  such that  the set $\f_K[V,W]$ with order ``$<$'' is totally ordered  with least element $V$ and great element $W$, such that given $U_0$, $U_1$, $U_2\in \f_K[V,W]$, if $U_0<U_1<U_2$, then
    $$d_{U_1}(V,W)-D\leq d_{U_1}(U_0,U_2)\leq d_{U_1}(V,W),$$  and $$d_{U_0}(U_1,U_2)\leq D \quad and \quad d_{U_2}(U_0,U_1)\leq D.$$

\end{lem}


We now give the definition of a projection complex.

\begin{defn}
The \textit{projection complex} $\p_K(\f)$ for $K$ satisfying Lemma \ref{OrderLem}  is  a graph with the vertex set consisting of the elements in $\f$. Two  vertices $u$ and $v$ are connected if $\f_K(U,V)=\emptyset$. We equip $\PC$ with a length metric $d_{\p}$ induced by assigning unit length to each edge.
\end{defn}
The projection complex $\p_K(\f)$ is connected since by  \cite[Proposition 3.7]{BBF}   the interval set $\f_K[U, V]$  gives    a connected path between $U$ and $V$ in $\PC$ called \textit{standard path}: the consecutive elements directed by the total order  are  adjacent in $\p_K(\f)$.  By \cite[Corollary 3.7]{BBFS}, standard paths are $(2,1)$-quasi-geodesics. The structural  result about the projection complex is the following.
\begin{thm}\label{projcplxThm}\cite{BBF}
For $K\gg 0$ as in Lemma \ref{OrderLem}, the projection complex  $\PC$ is a quasi-tree, on which $G$ acts co-boundedly.
\end{thm}
\begin{proof}
It is proved there that the standard path  from $U$ to $V$ lies in the $2$-neighborhood of any path with the same endpoints. This implies the conclusion via (a variant of) Manning's bottleneck criterion of quasi-trees in \cite{Man05}.     
\end{proof}
A triangle formed by standard paths in the projection complex looks almost like a  tripod.
\begin{lem}\label{Tripod} \cite[Lemma 3.6]{BBFS}
For every $U,V,W\in \f$, the path $\f_K[U,V]$ is contained in  $\f_K[U,W]\cup \f_K[W,V]$ except for at most two consecutive vertices. Precisely, the (possibly empty) initial and terminal subpaths of $\f_K[U,V]$ are contained in $\f_K[U,W]\cup \f_K[W,V]$, with an exception of at most two vertices in the middle.  
\end{lem}

For any two points $x\in X$ and $z\in Z$, we often need to lift a standard path to $\U$. It is a piece-wise geodesic path (admissible path) as concatenation of the normal paths between two consecutive vertices and geodesics contained in vertices.
This is explained by the following lemma proved in \cite[Lemma 4.5]{HLYANG}; we include the proof for completeness.  
\begin{lem}\label{LiftPathLem}
For any $K>0$, there exist a constant $L=L(K,\kappa)\ge 0$ with $L\to \infty$ as $K\to\infty$, and a  constant $B=B(\kappa)>0$ with the following property.

For any two points $u\in U, v\in V$ there exists an $(L, B)$-admissible path $\gamma$ in $\U$ from $u$ to $v$ with saturation $\f_K[U, V]$. 
\end{lem}
\begin{proof}
List $\f_K[U, V]=\{S_0=U, S_1, S_2, \cdots, S_k, S_{k+1}=V\}$   by the total order in Lemma \ref{OrderLem}. The admissible path is constructed by connecting projections between $S_i$ and $S_{i
+1}$. Namely, choose a sequence of points $x_i\in \pi_{S_i}(S_{i+1}), y_i\in \pi_{S_{i+1}}(S_{i})$ for $0\le i\le k$. We connect  consecutively the points in $\{u, x_0, y_0, \cdots, x_k, y_k, v\}$  to get a piecewise geodesic path $\gamma$. See Fig. \ref{fig:liftstandard}. By bounded projection $\|\pi_{S_{i}}(S_{i+1})\|\le B$, we have $$d_{S_i}(x_{i-1}, y_{i-1}),\; d_{S_i}(x_{i}, y_{i}) \le B.$$
By Lemma \ref{OrderLem},  we have $$d_{S_i}(S_{i-1}, S_{i+1})\ge d_{S_i}(U, V)-D\ge K-D,$$ for a constant $D=D(\kappa)$. Thus, $d(y_{i-1}, x_i)\ge  d_{S_i}(S_{i-1}, S_{i+1})-2B\ge K-D-2B:=L(K,\kappa)$. So, $\gamma$ is an $(L, B)$-admissible path relative to $\f_K[U, V]$. The lemma is then proved.
\end{proof}
\begin{figure}
    \centering

\tikzset{every picture/.style={line width=0.75pt}} 

\begin{tikzpicture}[x=0.75pt,y=0.75pt,yscale=-1,xscale=1]

\draw   (60.5,78.5) .. controls (60.5,63.31) and (79.53,51) .. (103,51) .. controls (126.47,51) and (145.5,63.31) .. (145.5,78.5) .. controls (145.5,93.69) and (126.47,106) .. (103,106) .. controls (79.53,106) and (60.5,93.69) .. (60.5,78.5) -- cycle ;
\draw   (201,80.5) .. controls (201,65.31) and (216.33,53) .. (235.25,53) .. controls (254.17,53) and (269.5,65.31) .. (269.5,80.5) .. controls (269.5,95.69) and (254.17,108) .. (235.25,108) .. controls (216.33,108) and (201,95.69) .. (201,80.5) -- cycle ;
\draw   (441,89) .. controls (441,75.19) and (462.15,64) .. (488.25,64) .. controls (514.35,64) and (535.5,75.19) .. (535.5,89) .. controls (535.5,102.81) and (514.35,114) .. (488.25,114) .. controls (462.15,114) and (441,102.81) .. (441,89) -- cycle ;
\draw   (341,80.5) .. controls (341,65.59) and (355.33,53.5) .. (373,53.5) .. controls (390.67,53.5) and (405,65.59) .. (405,80.5) .. controls (405,95.41) and (390.67,107.5) .. (373,107.5) .. controls (355.33,107.5) and (341,95.41) .. (341,80.5) -- cycle ;
\draw  [dash pattern={on 4.5pt off 4.5pt}]  (145.5,78.5) -- (199,80.43) ;
\draw [shift={(201,80.5)}, rotate = 182.06] [color={rgb, 255:red, 0; green, 0; blue, 0 }  ][line width=0.75]    (10.93,-3.29) .. controls (6.95,-1.4) and (3.31,-0.3) .. (0,0) .. controls (3.31,0.3) and (6.95,1.4) .. (10.93,3.29)   ;
\draw  [dash pattern={on 4.5pt off 4.5pt}]  (405,80.5) -- (439.02,85.7) ;
\draw [shift={(441,86)}, rotate = 188.69] [color={rgb, 255:red, 0; green, 0; blue, 0 }  ][line width=0.75]    (10.93,-3.29) .. controls (6.95,-1.4) and (3.31,-0.3) .. (0,0) .. controls (3.31,0.3) and (6.95,1.4) .. (10.93,3.29)   ;
\draw    (269.5,80.5) -- (341,80.5) ;
\draw [shift={(341,80.5)}, rotate = 0] [color={rgb, 255:red, 0; green, 0; blue, 0 }  ][fill={rgb, 255:red, 0; green, 0; blue, 0 }  ][line width=0.75]      (0, 0) circle [x radius= 3.35, y radius= 3.35]   ;
\draw [shift={(269.5,80.5)}, rotate = 0] [color={rgb, 255:red, 0; green, 0; blue, 0 }  ][fill={rgb, 255:red, 0; green, 0; blue, 0 }  ][line width=0.75]      (0, 0) circle [x radius= 3.35, y radius= 3.35]   ;
\draw    (279.5,120) .. controls (285.38,110.2) and (299.42,98.48) .. (275.04,86.72) ;
\draw [shift={(273.5,86)}, rotate = 24.36] [color={rgb, 255:red, 0; green, 0; blue, 0 }  ][line width=0.75]    (10.93,-3.29) .. controls (6.95,-1.4) and (3.31,-0.3) .. (0,0) .. controls (3.31,0.3) and (6.95,1.4) .. (10.93,3.29)   ;
\draw    (326.5,120) .. controls (318.62,114.09) and (302.01,101.39) .. (334,85.72) ;
\draw [shift={(335.5,85)}, rotate = 154.8] [color={rgb, 255:red, 0; green, 0; blue, 0 }  ][line width=0.75]    (10.93,-3.29) .. controls (6.95,-1.4) and (3.31,-0.3) .. (0,0) .. controls (3.31,0.3) and (6.95,1.4) .. (10.93,3.29)   ;

\draw (76,68.4) node [anchor=north west][inner sep=0.75pt]    {$S_{0} =U$};
\draw (228,69.4) node [anchor=north west][inner sep=0.75pt]    {$S_{i}$};
\draw (359,69.4) node [anchor=north west][inner sep=0.75pt]    {$S_{i+1}$};
\draw (455,81.4) node [anchor=north west][inner sep=0.75pt]    {$S_{k+1} =V$};
\draw (235,124.4) node [anchor=north west][inner sep=0.75pt]    {$\pi _{S_{i}}( S_{i+1})$};
\draw (315,125.4) node [anchor=north west][inner sep=0.75pt]    {$\pi _{S_{i+1}}( S_{i})$};

\end{tikzpicture}
    \caption{Schematic illustration of lifting standard paths}
    \label{fig:liftstandard}
\end{figure}
In practice, we always assume that $ K$ is large enough so that $L$ satisfies Proposition \ref{admisProp} and the path $\gamma$ is a $c$-quasi-geodesic for some $c>1$. 

Before moving on, let us put into perspective  what we have proved so far about projection complex.

Consider the union $\widetilde X:= \bigcup \{X\in \f\}$ of $\f$ with the restricted metric $d$ from $\U$. Then  the natural projection map collapsing $X\in \f$ to the corresponding vertex $\bar x\in \PC$ defined as follows
\begin{equation}\label{ShadowMap}\begin{aligned}
\Phi_1: (\widetilde X, d)   & \longrightarrow (\p_K(\f), d_{\mathcal P})\\    
 X& \longmapsto \bar x
\end{aligned}
\end{equation}
is a coarsely Lipschitz map: there exists a constant $c>1$ so that for any $x, y\in \widetilde X$,
$$
d_{\mathcal P}(\bar x, \bar y)\le c d(x,y)+c
$$
Indeed, this follows from the fact that the lifted path $\gamma$ in Lemma \ref{LiftPathLem} is a uniform quasi-geodesic. Moreover, by Proposition \ref{admisProp}, as $\gamma$ fellow travels any geodesic $\alpha$ with the same  endpoints as $\gamma$, we could deduce that the map $\Phi$ sends the geodesic $\alpha$  to a unparametrized quasi-geodesic (standard path) in $\PC$: the gaps between two consecutive axis in the saturation $\f(\gamma)$ are sent to the corresponding   edges. A map between two metric spaces with the property is usually called {\it shadowing map} in literature.

We now recall the following  result to be used later on. 
\begin{lem}\label{ForcingLem}\cite[Lemma 3.18]{BBF}
For any $K>0$ there exists $\hat K>0$ such that $\f_{\hat K}(X,Z)$ is contained in the geodesic  from $X$ to $Z$ in $\PC$.
\end{lem}

 
\subsubsection*{Quasi-tree of spaces.} Fix   a positive number $L$ so that  $1/2K\leq L\leq 2K$. We now define a blowup version, $\QT$, of the projection complex $\PC$  by remembering the geometry of each $U\in \f$.  Precisely,

We replace each $U\in \f$, a vertex in $\PC$, by the corresponding subspace $U\subset\U$, and  keep the adjacency relation in $\PC$: if $U$ and $V$ are adjacent in $\PC$ (i.e.  $d_{\p}(U,V)=1$), then we attach an edge of length $L$    from every point $u\in \pi_UV$ to $v\in \pi_VU$.  This choice of $L$ by \cite [Lemma 4.2]{BBF}  ensures that $U\subseteq\U$ are geodesically embedded in $\QT$ (for simplicity, we mention $K$ rather than $L\in [K/2,2K]$).

Since $h$ is contracting,  the infinite cyclic subgroup $\langle h\rangle$ is of finite index in $E(h)$ by Lemma \ref{elementarygroup}, so  $\ax(h)=E(h)o$ is quasi-isometric to a line $\mathbb{R}$. Thus, $\f$ consists of  uniform quasi-lines. By \cite [Theorem B]{BBF}, we have the following.

\begin{thm}\cite{BBF}\label{quasitreeThm}
The quasi-tree of spaces  $\QT$ is a quasi-tree of infinite diameter, with each $U\in \f$ totally geodesically embedded into $\QT$.  Moreover, the shortest projection from $U$ to $V$ in $\QT$ agrees with the projection $\pi_UV$ up to a uniform finite Hausdorff distance.
\end{thm}

From the construction, there exists a natural map which collapses each totally  geodesically embedded subspace $U\subseteq \QT$  as a point and identifies edges of length $L$ to  edges of unit length in $\PC$.  In a reverse direction, the connected path in $\PC$ given by $\f_K[X, Y]$ lifts to a standard path in  $\QT$ described as follows.   The notion of standard paths plays a key role to establish Theorem \ref{quasitreeThm}.

\begin{defn}\label{stdpathDef}
A path $\gamma$ from $u\in U$ to $v\in V$ in the space $\QT$ is called \textit{a $K$-standard path} if it passes through the set of vertices in $\f_K[U,V]$ in a natural order given in Lemma \ref{OrderLem} and, within each vertex, the path is a geodesic.
\end{defn}
Hence, standard paths in $\QT$ are also uniform quasi-geodesics. 

Summarizing the above discussion, we also have the coarsely Lipschitz map:
\begin{equation}\label{ShadowMap2}\begin{aligned}
\Phi_2: (\QT, d)   & \longrightarrow (\p_K(\f), d)\\    
 X& \longmapsto \bar x
\end{aligned}
\end{equation}
which sends a standard path in $\QT$ to the corresponding standard path in $\PC$. 

As each $U\in \f$ is isometrically embedded into $\QT$, the identification 
\begin{equation}\label{ShadowMap3}\begin{aligned}
\Phi: (\widetilde X, d)   &  \longrightarrow(\QT, d)\\    
 x& \longmapsto  x
\end{aligned}
\end{equation}
is a  coarsely Lipschitz map, where we recall $\widetilde X= \bigcup \{X\in \f\}$.

We shall continue the investigation on the Gromov boundary of projection complex in \textsection \ref{SSecGromovBdryPC}.

\subsection{Horofunction boundary}
We recall the definition of horofunction boundary and setup the notations for further reference.

Fix a basepoint $o\in \U$. For  each $y \in  \U$, we define a Lipschitz map $b_y^o:  \U\to \mathbb R$     by $$\forall x\in \U:\quad b^o_y(x)=d(x, y)-d(o,y).$$ This   family of $1$-Lipschitz functions sits in the set of continuous functions on $ \U$ vanishing at $o$.  Endowed  with the compact-open topology, the  Arzela-Ascoli Lemma implies that the closure  of $\{b^o_y: y\in  \U\}$  gives a compactification of $ \U$.  The complement, denoted by $\hU$, of $ \U$ is called  the \textit{horofunction boundary}. 

A \textit{Buseman cocycle} $B_\xi:  \U\times \U \to \mathbb R$ at $\xi\in \hU$ (independent of $o$) is given by $$\forall x_1, x_2\in  \U: \quad B_\xi(x_1, x_2)=b_\xi^o(x_1)-b_\xi^o(x_2).$$
 

The topological type of horofunction boundary is independent of  the choice of   basepoints, since if $d(x, y_n)-d(o, y_n)$ converges as $n\to \infty$, then so does $d(x, y_n)-d(o', y_n)$ for any $o'\in \U$. Moreover, the corresponding horofunctions differ by an additive amount: $$b^o_\xi(\cdot)-b_\xi^{o'}(\cdot)=b_\xi^o(o'),$$ so we will omit the upper index $o$. Every isometry $\phi$ of $\U$ induces a homeomorphism on $\hU$:  
$$
\forall y\in \U:\quad\phi(\xi)(y):=b_\xi(\phi^{-1}(y))-b_\xi(\phi^{-1}(o)).
$$
According to the context, both $\xi$ and $b_\xi$ are used to denote  the boundary points.

\subsubsection*{Finite difference relation}
Two horofunctions $b_\xi, b_\eta$ have   \textit{$K$-finite difference} for $K\ge 0$ if the $L^\infty$-norm of their difference is $K$-bounded: $$\|b_\xi-b_\eta\|_\infty\le K.$$ 
The   \textit{locus} of     $b_\xi$ consists of  horofunctions $b_\eta$ so that $b_\xi, b_\eta$ have   $K$-finite difference for some $K$ depending on $\eta$.  The loci   $[b_\xi]$  of    horofunctions $b_\xi$ form a \textit{finite difference equivalence relation} $[\cdot]$ on $\hU$. The \textit{locus} $[\Lambda]$ of a subset $\Lambda\subseteq \hU$ is the union of loci of all points in $\Lambda$. We say that $\Lambda$ is \textit{saturated} if $[\Lambda]=\Lambda$.

The following fact follows directly by definition.
\begin{lem}\label{bddLocusLem}
Let $x_n\in \U\to \xi\in \hU$ and  $y_n\in \U\to\eta\in \hU$ as $n\to \infty$. If  $\sup_{n\ge 1}d(x_n, y_n)<\infty$, then  $[\xi]=[\eta]$. 
\end{lem}

We say that the      partition $[\cdot]$ restricted on a saturated subset $\Lambda$ is called \textit{$K$-finite} for $K\ge 0$ if any $[\cdot]$-class in $\Lambda$ consists of horofunctions with pairwise $K$-finite difference. 

At last, we mention the following topological criterion   to obtain second countable, Hausdorff quotient spaces.

\begin{lem}\label{GoodPartition}
Let $\Lambda\subseteq\hU$ be a saturated subset so that the restricted relation $[\cdot]$ on $\Lambda$ is $K$-finite for some $K\ge 0$.  Then
the quotient map 
$$
[\cdot]:\quad \xi\in \Lambda\longmapsto [\xi]\in [\Lambda]
$$
is a closed map from $\Lambda$ to $[\Lambda]$ with compact fibers. Moreover,     $[\Lambda]$ has a metrizable and  second countable quotient topology.
\end{lem}
\begin{proof}
A closed continuous surjective map $f: M\to N$ between two topological spaces with compact fibers is called \textit{perfect} in \cite[Ex \textsection 31.7]{M00}. A prefect map preserves the Hausdorff, regular, and second countable properties of $\Lambda$, so by Urysohn's metrization theorem, $[\Lambda]$ is metrizable. Hence, we only need to prove that the quotient map $[\cdot]$ is a closed map.

Note first that the restricted $K$-finite difference relation $[\cdot]$ on $\Lambda$ is closed. It suffices to recall that   $b_{\xi_n}\to b_{\xi}$ is amount to  the uniform convergence of $b_{\xi_n}(\cdot)\to b_{\xi}(\cdot)$ on any compact subset of $\U$. Thus, $|b_{\xi_n}-b_{\eta_n}|_{\infty}\le K$ implies $|b_{\xi}-b_{\eta}|_{\infty}\le K$. Moreover, the $[\cdot]$-class of a point in $\Lambda$ is a closed subset, so is compact in $\hU$. Thus, the map $[\cdot]$ on $\Lambda$ has compact fibers.

We now show   the quotient map $[\cdot]$ is closed:   the locus $[Z]$ (the preimage of the map) of a closed subset $Z\subseteq \Lambda$ is closed in $\Lambda$. Indeed, let $w_n\in [Z]$ tend to a limit point $\xi\in\Lambda$. Take a sequence     $z_n\in Z$ with $w_n\in [z_n]$, and for $\pU$ is compact, let $\eta\in \pU$ be any accumulation point of $z_n$. As the   relation $[\cdot]$ on $\Lambda$ is closed,   we have $[\xi]=[\eta]$ and then $\eta\in \Lambda$ for $\Lambda$ is saturated. As $Z$ is a closed subset, we have $\eta\in Z$, so $\xi\in[Z]$. This implies that $[Z]$ is a closed subset in $\Lambda$. Hence,  the quotient map is closed, and the lemma is proved. 
\end{proof}

\section{Convergence compactification}\label{SecBoundary}
We develop a general theory of a convergence compactification.  As in (\ref{GActUAssump}), let $(\U, d)$ be a proper metric space admitting an isometric action of a non-elementary countable group $G$ with a contracting element. Consider a metrizable compactification $\bU:=\pU\cup \U$, so  $\U$ is open and dense in $\bU$. We also assume that the action of $\textrm{Isom}(\U)$ extends by homeomorphism to  $\pU$. 

\subsection{Partition on a boundary}

We   equip $\pU$    with an  $\isom(\U)$-invariant  partition $[\cdot]$:   $[\xi]=[\eta]$ implies $[g\xi]=[g\eta]$ for any $g\in \isom(\U)$.   The  \textit{locus} $[Z]$  of a subset $Z\subseteq\pU$ is the union of all $[\cdot]$-classes of $\xi\in Z$.  We say that $\xi$ is \textit{minimal} if $[\xi]=\{\xi\}$.

Partitions     on the boundary might look   unnatural.   It is, however, amount to  descending the action to the quotient  $[\pU]$ of $\pU$ induced by the relation $[\cdot]$ via the map:
$$
\xi\in \pU\quad \longmapsto \quad [\xi]\in [\pU]
$$
so the open subsets of $[\pU]$ are  precisely the images of  saturated open sets $U$ in $\pU$ with $U=[U]$.
In general, $[\pU]$ may not be Hausdorff or $T_1$.

Say that $[\cdot]$ is a \textit{closed} partition if  $x_n\to \xi\in\pU$ and $y_n\to\eta\in\pU$ are two sequences with $[x_n]=[ y_n]$, then $[\xi]=[\eta]$. Equivalently, the relation $\{(\xi,\eta): [\xi]=[\eta]\}$ is a closed subset in $\pU\times \pU$, so the quotient space $[\pU]$ is Hausdorff. 

\begin{lem}\label{closedpartition}
Assume that  the partition $[\cdot]$ is closed.  For any given $Z\subseteq\pU$, we have  $\overline{[Z]}\subseteq [\overline Z]$. In particular, the quotient map is a closed map: if $Z$ is a closed subset, then $[Z]$ is closed.
\end{lem}
 
\begin{proof}
Let $w_n\in [Z]$ tend to $\xi\in\pU$. Then  there exists $z_n\in Z$ with $z_n\in [w_n]$, and for $\pU$ is compact, let $\eta\in \bar Z$ be any accumulation point of $z_n$. As $[\cdot]$ is closed,   we have $[\xi]=[\eta]$. Thus, we proved $\xi\in [\overline Z]$ and the proof is complete.
\end{proof}

Recall that   a point $\xi$ is  an \textit{accumulation point} of a sequence  $\{x_n\in \U:n\ge 1\}$, if there exists a subsequence denoted by $y_n$ so that  $y_n$ converges to $\xi$. We say that $x_n$ \textit{tends} (resp. \textit{accumulates}) to $[\xi]$ if the limit point (resp. any accumulate point) is contained in $[\xi]$. This implies that $[x_n]$ tends or accumulates to $[\xi]$ in the quotient space $[\pU]$. So, an infinite ray $\gamma$ \textit{terminates} at a point in $[\xi] \in \pU$ if any sequence of points on $\gamma$ accumulates in $[\xi]$.

Let $X$ be a subset in $\U$. The \textit{limit set} of $X$, denoted by $\Lambda X$, is the topological closure of $X$ in $\pU$.

\begin{conv}
We write  ``$[Z]\subseteq \pU$" to emphasize $[Z]$ as a subset in $\pU$, and ``$[Z]\subseteq [\pU]$" as a subset in $[\pU]$. In the latter case, the convergence to $[\xi]$  makes the usual sense in the quotient topology.
\end{conv}

\subsection{Assumptions A and B}

The first two assumptions  relate the  contracting property of subsets in the space to the boundary. Recall that a sequence of subsets $X_n$ is \textit{escaping} if $d(o,X_n)\to\infty$ for some basepoint $o\in \U$. 
\begin{assump}\label{AssumpA} 
Any contracting geodesic ray $\gamma$ accumulates into a closed subset $[\xi]$ for some  $\xi\in \pU$. Moreover, any sequence of    $y_n\in \U$ with escaping projections $\pi_\gamma(y_n)$ tends to $[\xi]$.
\end{assump} 
Note that, by assumption, $[\xi]$ is requested to be a closed subset.
 
\begin{lem}\label{RayLimitsLem}
Under \ref{AssumpA}, any contracting quasi-geodesic ray $\gamma$ and   its finite neighborhood  accumulates in only one   $[\cdot]$-class, denoted by $[\gamma^+]\in [\pU]$. 
\end{lem} 

\begin{proof}
By Morse property in Lemma \ref{BigFive}, a contracting quasi-geodesic ray is contained in a finite neighborhood of a contracting geodesic  ray.  
 \end{proof}

Let $\gamma: \mathbb R\to \U$ be a contracting   quasi-geodesic, with   the   positive and negative half rays denoted by $\gamma^+, \gamma^-$ respectively. Following   \cite{GePo2},  $\gamma$ is called  \textit{horocycle}    if $\gamma^+, \gamma^-$ accumulate into the same $[\cdot]$-class: $[\gamma^+]=[\gamma^-]$. This is opposite to a class of non-pinched points, which shall be consider in \ref{AssumpC}.   If the points $[\xi]$ in \ref{AssumpA} are non-pinched,  we show that the only possible way for $x_n$ tending to $[\xi]$    is that their projection to $\gamma$ is escaping. See Lemma \ref{ConverseAssumpALem} for a precise statement. 

\begin{defn}
Let $h\in \isom(\U)$ be a contracting isometry. The \textit{attracting fixed set} $[h^+]$ and \textit{repelling fixed set} $[h^-]$ in $\pU$ are the locus of   all accumulation points of $\{h^{n}o: n>0\}$ and $\{h^{-n}o: n>0\}$ respectively. 

Abusing language, we also refer $[h^-], [h^+]$ as the fixed points of $h$, though they are usually non-singleton. Write      $[h^\pm]:=[h^+]\cup[h^-]$.
\end{defn}
Note that $h$ is not assumed to be in a discrete group, so the axis is defined as in (\ref{axisdefn0}) to be the path $\ax(h)=\cup_{n\in\mathbb Z} h^n[o,ho]$.
 
By Lemma  \ref{RayLimitsLem}, the sets  $[h^-], [h^+]$ do not depend on $o\in \U$. If $\gamma$ is a quasi-geodesic axis,   $[h^-], [h^+]$ are also represented by the negative/positive half rays $\gamma^+,\gamma^-$  (in the obvious parametrization).

Our second assumption deals with an escaping sequence of contracting subsets, which roughly says that their cones sub-converge to a boundary point. Recall that give $r\ge 0$, the \textit{$r$-cone} of a subset $X\subseteq \U$ relative to $o\in \U$ is defined as follows:
$$
\Omega_o(X,r):=\{y\in \U: d_X(o,y)\ge r\}
$$
\begin{assump}\label{AssumpB} 
Let $\{X_n\subseteq \U :n\ge 1\}$ be an escaping sequence of $C$-contracting  quasi-geodesics for some $C>0$. For given $o\in \U$, consider the union $Z_n= X_n\cup \Omega_o(X_n,10C)$.  Then there exist a  subsequence of $Z_n$ (still denoted by $Z_n$)    and a point $\xi\in \pU$ such that $Z_n$  accumulates    into $[\xi]$: 
\begin{itemize}
    \item Any convergent sequence of points $z_n\in Z_n$ has the limit  point in $[\xi]$. 
\end{itemize} 


\end{assump}  
\begin{rem}
In the above assumption, we could replace $Z_n=\Omega_o(X_n,10C)$ with the following simpler defined set $$Z_n=\{y\in \U: [o,y]\cap Z_n\ne \emptyset \}$$ for  hyperbolic and CAT(0) spaces by Lemma \ref{UnifDiffLemCAT0}, and for Teichm\"uller space by Lemma \ref{StrongerAssumpB}, .     
\end{rem}

The following weaker assumption suffices in most results below, which asserts the sub-convergence of $X_n$ instead of the union with  $10C$-cones.  
\begin{assumpB}\label{AssumpB'}
Let $\{X_n\subseteq \U :n\ge 1\}$ be an escaping sequence of $C$-contracting   quasi-geodesics for some $C>0$. Then there exists   $\xi\in \pU$  so that a  subsequence of $X_n$ (still denoted by $X_n$)  accumulates    into $[\xi]$: 
\begin{itemize}
    \item Any convergent sequence of points $x_n\in X_n$ has the limit  point in $[\xi]$. 
\end{itemize}
\end{assumpB}

The next two results under \ref{AssumpB} are particularly useful in defining conical points later on. The first result says that if the escaping sequence of contracting subsets intersect a geodesic ray in a large diameter, then it converges to a boundary point without passing to subsequence.  
\begin{lem}\label{FreqContrLimits}
Assume that  $X_n$ is an escaping sequence of $C$-contracting  quasi-geodesics with bounded projection by a constant $B$. Let $\gamma$ be a geodesic ray in $\U$ so that $\|X_n\cap \gamma\|\ge 14C+B$ for each $n\ge 1$. Then under \ref{AssumpB}, $\gamma$ and the sequence $X_n$ accumulate in one $[\cdot]$-class, denoted by $[\gamma^+]$.
\end{lem}
\begin{proof}
According to \ref{AssumpB}, assume that the cones of a subsequence $X_{n_i}$ converges to a $[\cdot]$-class, denoted by $[\gamma^+]$. We first prove that any unbounded sequence $z_n$ of  points on a geodesic ray $\gamma$ also  converges to $[\gamma^+]$. Indeed, fixing $X_{n_i}$, we obtain  from the assumption that $z_n \in \Omega_o(X_{n_i},10C)$  for any $n\gg 0$ and $o:=\gamma_-$. Thus, \ref{AssumpB} implies that $z_n\to [\gamma^+]$.  

It now remains to prove that any subsequence of $X_n$  converges to $[\gamma^+]$. Fix  $X_{n_i}$ and a point  $x\in X_n\cap \gamma$   for any given $n\gg n_i$. By Lemma \ref{Transform}, we have $d_{X_{n_i}}(o,x)\ge \|N_{C}(X_{n_i})\cap \gamma\|-4C \ge   10C+B$. As $X_n$ projects to a set on $X_{n_i}$ with  diameter at most $B$, we see that $d_{X_{n_i}}(o,y)\ge 10C$ for any $y\in X_n$, so $X_n\subseteq \Omega_o(X_{n_i},10C)$. By the above assumption, the cones of a subsequence $X_{n_i}$ converges  to  $[\gamma^+]$, hence $X_n$ does so. The proof is complete. 
\end{proof}
Via Proposition \ref{admisProp}, we can have the convergence of  an admissible path and the sequence of contracting subsets in its saturation.  
\begin{lem}\label{AdmisContrLimits}
Let $\gamma$ be an $(L,B)$-admissible ray in $\U$ where the saturation $\f(\gamma)$ is an escaping sequence of $C$-contracting  quasi-geodesics with $B$-bounded projection. Then under \ref{AssumpB}, for $L\gg C+B$, the ray  $\gamma$ and any infinite sequence $X_n\in \f(\gamma)$ accumulate in one $[\cdot]$-class, denoted by $[\gamma^+]$.
\end{lem}
\begin{proof}
By Proposition \ref{admisProp}, if $L\gg C+B$, any geodesic $\alpha$ from $o:=\gamma^-$ to a point $y$ on $\gamma$ $r$-fellow travel the subpath of $\gamma$ between $o$ and $y$. Let $d(o,y)\to\infty$, we obtain a limiting geodesic ray $\alpha$ so that   $\|X_n\cap \alpha\|\ge L-2r$. If $L$ is sufficiently large relative to $B+10C$, then the proof is completed by Lemma \ref{FreqContrLimits}.  
\end{proof}

In what follows, we apply \ref{AssumpB'} to study the limit set of a discrete group action.

Consider a proper action $G\act \U$ as in (\ref{GActUAssump}). The {limit set} $\pG$  of a $G$-orbit $Go$  may   depend on the basepoint $o\in \U$, though $[h^\pm]$ is the locus of the limit set of $\langle h\rangle o$, which is independent of $o$ by \ref{AssumpA}.

\begin{lem}\label{[LimitSet]}
Assume that $G$ contains a contracting element. Under  \ref{AssumpB'}, for any two points $o, o'\in \U$, we have $[\Lambda (G o)]=[\Lambda (Go')]$.
\end{lem}
\begin{proof}
Fix a  contracting element $h$ with $C$-contracting axis $\ax(h)=E(h)\cdot o$.  
Let $g_no \to \xi\in \pG $. After passage to a subsequence, it suffices to prove  that $g_no'\to\eta\in [\xi]$. Indeed, consider the sequence $g_no\in X_n:=g_n\ax(h)$.  If $X:=X_n$ is eventually the same, then we are done by Lemma \ref{RayLimitsLem}. Otherwise,  assume that $X_n$ is escaping. For  $r:=d(o,o')$,     $N_r(X_n)$ is an escaping sequence of  $(C+2r)$-contracting subsets containing $g_no'$ and $X_n$. By \ref{AssumpB'},    any convergent subsequence of $g_no'$ has the limit point in $[\xi]$.
\end{proof}
We now derive that the normal subgroup shares the  limit set with the ambient group, which crucially uses \ref{AssumpB} instead of \ref{AssumpB'}.
\begin{lem} \label{NormalLimitSetLem}
Let     $H<G$ be an infinite normal subgroup. Then under \ref{AssumpA} and \ref{AssumpB}, $[\Lambda (H o)]=[\Lambda (G o)]$.
\end{lem} 
\begin{proof}
By Lemma \ref{InfiniteNormalLem}, we can choose a set of three  independent contracting elements  $F\subseteq H$.
Let $g_n\in G$ such that $g_no\to \xi\in \pG$. For each pair $(g_n, g_n^{-1})$, after passage to subsequences, there exists $f\in F$  such that $h_n:=g_nfg_n^{-1}\in H$ satisfies       Lemma \ref{extend3}, so {$[o, h_no]\cap N_C(g_n\ax(f))\ne\emptyset$. According to whether $X_n=g_n\ax(f)$ is escaping or not, we apply \ref{AssumpB} or \ref{AssumpA} to  see that  $h_n o$ accumulates into $[\xi]$. We thus obtain $\pG \subseteq [\Lambda(Ho)]$}. The converse inclusion is obvious, so the lemma is proved. 
\end{proof}
\begin{rem}(alternative proof) 
If   $H$ contains a    contracting element $h$ with \textit{minimal} fixed point (i.e.: $[h^+], [h^-]$ being singletons), we provide  an alternative argument under \ref{AssumpB'}. 

Indeed, consider a sequence of elements $h_n\in \langle h\rangle$ so that $h_no\to \xi\in \Lambda (Ho)$. On one   hand,  $[\pG ]=[\overline{G\xi}]\subseteq [\pG ]$ by Lemma \ref{FixPtsDense} below. We now prove that $G\xi\subseteq [\Lambda (Ho)]$.  For any given $g\in G$,  as the conjugates $gh_ng^{-1}$ are contained in $H$ and $d(gh_ng^{-1}o, gh_no)=d(o, go)$,  \ref{AssumpB'} implies that $gh_ng^{-1}o$ and $gh_no$ converge to the same $[\cdot]$-class $[g\xi]$ which are singleton by assumption, so we obtain $g\xi\in \Lambda (Ho)$.  Thus $G\xi\subseteq [\Lambda (Ho)]$  is proved. 
\end{rem}

The following lemma says that for any contracting element $h\in G$,
the set of all $G$-translates of $[h^\pm]$  is dense in the limit  set $[\pG ]$ in the quotient $[\pU]$.  
\begin{lem}\label{FixPtsDense}
Fix a contracting element $h\in G$ and a basepoint $o\in \U$. Take any point $\xi\in \Lambda(E(h) o)$. Under \ref{AssumpA} and \ref{AssumpB'},   we have $\overline {G\xi}\subseteq \pG \subseteq [\overline {G\xi}]$. In particular, $[\pG ] =  [\overline {G\xi}]$. 
\end{lem}
\begin{rem}
If  the partition $[\cdot]$ is closed, we can  obtain $[\overline{G\xi}]=[\overline{G\eta}]$ from $[\xi]=[\eta]$. In this case, the above conclusion works for any $\xi\in [h^\pm]$.
\end{rem}

\begin{proof}
Let  $\eta \in \pG $, and  $x_n:=g_n o\to \eta$ a sequence of points. Set $\gamma_n:=g_n \ax(h)$. Up to taking a subsequence of $\gamma_n$, we consider   two cases as follows.

\textbf{Case 1}. Assume  that $d(o, \gamma_n)\to \infty$ as $n\to \infty$.  
By \ref{AssumpB'}, as  $x_n\in \gamma_n\to \eta$,   any convergent sequence of $y_n \in \gamma_n$   tends to  $\eta'\in [\eta]$. By assumption, for each $n\ge 1$, the point $g_n\xi\in g_n[h^\pm]$ is an accumulation point of the set $\gamma_n$. The   compactification $\U\cup\pU$ is assumed to be metrizable, so let us fix a compatible metric $\delta$.  {Choose $y_n\in \gamma_n$ so that $\delta(y_n, g_n\xi)\le 1/n$. Thus, $g_n\xi\to\eta'\in [\eta]$.}

\textbf{Case 2}. $\gamma:=\gamma_n$ for all $n\ge 1$, as  the $G$-translated axis  of $h$ forms a discrete sequence  by the proper action. Note that $\gamma$ is a translated  axis of $h$ on which $E(h)$ acts cocompactly. As a non-elementary  $G$      contains infinitely many distinct conjugates of $h$, we can thus choose a sequence of their axes $\alpha_n$ (i.e. copies of $\gamma$) so that $\pi_{\gamma}(\alpha_n)$ is uniformly close to $x_n$ independent of $n$. Pick up any sequence $y_n\in \alpha_n$. Since $d(o, x_n)\to \infty$ we have $y_n\to[\eta]$ by \ref{AssumpA}. Considering the escaping sequence $\alpha_n$, we are therefore reduced to the setup of the Case 1, so  the proof of  the lemma is finished. 
\end{proof}

We now prove the double density of fixed points of  contracting elements.  Recall that two contracting elements  $h, k$ (in a discrete group $G$) are independent if  their $G$-translated axes have bounded projection / intersection.
\begin{lem}\label{contractingcriteron2}
Let $h, k\in G$ be two independent contracting elements. Then under \ref{AssumpA}, for any large $n\gg 0$, the element $g_n:=h^nk^n$ is a   contracting element. Moreover, there exist $\xi_n\in [g_n^+]$ and $\eta_n\in [g_n^-]$ such that $\xi_n, \eta_n$ tend to $[h^+], [k^-]$ respectively.
\end{lem}
\begin{rem}
We explain some subtlety about existence of $\xi_n\in [g_n^+]$ and $\eta_n\in [g_n^-]$. In the proof, it suffices to  choose $\xi_n$ and $\eta_n$ to be accumulation points of $\langle g_n\rangle o$. However,   the $[\cdot]$-locus of $[g_n^+]$ may contain points which are not accumulation points of $\langle g_n\rangle o$.     
\end{rem}
\begin{proof}
By assumption,  $\ax(h)$ and $\ax(k)$ have $B$-bounded projection for some $B>0$. Consider the admissible path relative to $\mathbb X=\{g\ax(k): g\in G\}$. For large enough $n>0$,  the element $g_n:=h^nk^n$ is contracting: indeed, its orbital points  form a  bi-infinite $(L, B)$-admissible path $\gamma_n$ as follows:
$$\gamma_n =\bigcup_{i\in \mathbb Z} g_n^i\big([o, h^no]h^n[o, k^no]\big)$$
where $L:=d(o, k^no)$.  By \cite[Proposition 2.9]{YANG10}, $\gamma_n$ is a uniformly contracting quasi-geodesic independent of $n$, so by definition, $g_n$ is a contracting element. The positive (resp. negative) ray $\gamma_n^+$  (resp. $\gamma_n^-$) denotes the half ray with positive (resp. negative) indices $i$ in $\gamma_n$. 

Let   $\xi_n\in [g_n^+]$ be an accumulation point of $\gamma_n^+$. We now only prove   that $\xi_n$ tends to $[h^+]$. The  case for repelling points $\eta_n \to [k^-]$ is symmetric.

If  denote $\alpha:=\bigcup_{i\in \mathbb Z}h^i[o, ho]$ and $\beta:=\bigcup_{i\in \mathbb Z}k^i[o, ko]$,   the positive ray of $\alpha$ accumulates into $[h^+]$, and the negative ray of $\beta$ into $[k^-]$. Since $\gamma_n$ has the overlap $[o, h^no]$ with $\alpha$,   any     point $y_n$ on the positive ray $\gamma_n^+$ that is sufficiently far from $o$  projects to a fixed neighborhood of $h^no$, whose radius depends only on the contracting constant.   Applying \ref{AssumpA} with the contracting quasi-geodesic $\alpha$, we obtain   $y_n \to [h^+]$ for $x_n\to[h^+]$. Since $y_n$ can be chosen arbitrarily close to $\xi_n$, we conclude that the attracting fixed point $\xi_n$  tends to $[h^+]$. 
\end{proof}
It is convenient to understand the next result   in the quotient space  $[\pU]$.
\begin{lem}\label{DoubleDense}
Let $\dG$  be the set of distinct $[\cdot]$-pairs in $[\pG]$ with quotient topology defined in (\ref{doubleBdryEq}). Then under \ref{AssumpA} and \ref{AssumpB'}, the set of (ordered) fixed point pairs $([h^+],[h^-])$ of all   contracting elements $h\in G$ with $[h^+]\ne [h^-]$ is dense in $\dG$.
\end{lem} 
\begin{proof}
Let $\xi,\eta\in \pG $ so that $[\xi]\ne[\eta]$. Consider  any  $[\cdot]$-saturated open neighborhoods $[\xi]\subseteq U=[U]$ and $[\eta]\subseteq V=[V]$. By Lemma \ref{FixPtsDense}, there exist two contracting elements $h, k$  so that $[h^+]\subseteq U, [k^-]\subseteq V$.  By Lemma \ref{contractingcriteron2}, there exist contracting elements $g_n=h^nk^n$ for $n\gg 0$ such that $[g_n^+]\subseteq  U$ and $[g_n^-]\subseteq V$. This implies that $([g_n^+], [g_n^-])$ converges to $([\xi], [\eta])$ in the quotient topology.
\end{proof}
\subsection{Assumption C} The third assumption demands the boundary to contain enough interesting points. This is a non-trivial condition to get a useful theory, as the one-point compactification satisfies the assumptions (A) and (B).
\begin{assump}\label{AssumpC} 
Assume that the following set   $\mathcal C$ of \emph{non-pinched} points  $\xi\in \pU$ is non-empty. If $x_n, y_n\in \U$ are two sequence of points converging to $[\xi]$, then $[x_n,y_n]$ is an escaping sequence of geodesic segments.  
\end{assump}

By definition,   $\mathcal C$ is $\textrm{Isom}(\U)$-invariant and $\mathcal C=[\mathcal C]$ is saturated. Hence, non-pinched points are actually a property of $[\cdot]$-classes.

Under \ref{AssumpC}, we first  give a converse statement to  \ref{AssumpA}, and two useful corollaries of it. 
\begin{lem}\label{ConverseAssumpALem}
Assume that a $C$-contracting geodesic ray $\gamma$ accumulates into  $[\xi]$ by \ref{AssumpA}, where $\xi\in\mathcal C$ is some non-pinched point. If $x_n\in \U\to [\xi]$, then $\pi_\gamma(x_n)$ forms an unbounded sequence.     
\end{lem}
\begin{proof}
Indeed, the contracting property of $\gamma$ implies that $$d(\pi_\gamma(x_n),[x_n,y])\le 2C$$ for any $y\in \gamma$. Choose $y_n\in \gamma$ so that $y_n\to [\xi]$. If $\pi_\gamma(x_n)$ stays bounded, then we found a non-escaping sequence of $[x_n,y_n]$ so that $x_n\to[\xi]$ and $y_n\to[\xi]$. This contradicts \ref{AssumpC}.
\end{proof}

As a corollary, we obtain. 
\begin{lem}\label{DisjointContr}
Given $\xi\in \mathcal C$, let  $\alpha$ be a contracting geodesic ray ending at $[\xi]$. If   $\beta$ is a geodesic ray also ending at $[\xi]$, then $\beta$ is contained in     a finite neighborhood of $\alpha$.  
\end{lem}
\begin{proof}
Take a sequence $y_n\in \beta$ tending to $[\xi]$, where $y_0=\beta^-$. By Lemma \ref{ConverseAssumpALem} for $\alpha$,  $\pi_\alpha(y_n)$ is escaping, then applying \ref{AssumpA}, any sequence of $x_n\in \pi_\alpha(y_n)$ tends to $[\xi]$. By $C$-contracting property of $\alpha$, the segment $[\beta_-,y_n]_\beta$ intersects the  two $C$-balls at $x_0$ and  $x_n$ respectively. As $d(o,x_n)\to\infty$,  we see that   $\beta$ intersects $N_C(\alpha) $ in a unbounded  set $\{x_n\}$,  so the Morse property of $\alpha$ implies that $\beta$  is contained in a finite neighborhood of $\alpha$. 
\end{proof}
Another corollary is that non-pinched boundary points at which  a contracting geodesic ray terminates are visual. This is a baby version of the more complicated Lemma \ref{VisualPointsLem}.
\begin{lem}\label{ContractingBdryVisual}
Let $\alpha,\beta$ be two $C$-contracting geodesic rays ending at $[\xi]\ne [\eta]\subseteq \mathcal C$ respectively. Then there exists a bi-infinite geodesic  $\gamma$ such that two half ways eventually lies in a $\tilde C$-neighborhood of $\alpha$ and $\beta$, where $\tilde C$ depends only on $C$.  
\end{lem}
\begin{proof}
Let $y_n\in \beta$ tend to $[\eta]$. By Lemma \ref{ConverseAssumpALem}, as $[\xi]\ne [\eta]$, the projection $\pi_\alpha(y_n)$ stays  in a bounded set $A\subseteq \alpha$. Similarly, for any $x_n\in \alpha\to [\xi]$,  $\pi_\alpha(y_n)$ lies in a bounded set $B\subseteq \beta$. Thus,  $[x_n,y_n]$ intersects $N_C(A)$ and $N_C(B)$ by the contracting property. By Ascoli-Arzela Lemma, we could extract a limiting bi-infinite geodesic $\gamma$ from the sequence of $[x_n,y_n]$, so that the two half rays of $\gamma$ lies in a finite neighborhood of $\alpha$ and $\beta$.  
\end{proof}

The set $\mathcal C$  of non-pinched points allows us to introduce  non-pinched contracting elements, which are crucial in the following discussion.  

A contracting isometry $h\in \isom(\U)$ is called  \textit{non-pinched} if   their fixed points $[h^+],[h^-]$    are contained in $\mathcal C$; otherwise,  it is \textit{pinched}. 

Note that $h^no\to [h^+]$ and $h^{-n}o\to [h^-]$ as $n\to\infty$. Thus, it follows that $[h^+]\ne[h^-]$ for  a non-pinched   $h$. It turns out that distinct fixed points property characterizes the non-pinched contracting elements.   

\begin{figure}
    \centering

\tikzset{every picture/.style={line width=0.75pt}} 

\begin{tikzpicture}[x=0.75pt,y=0.75pt,yscale=-1,xscale=1]

\draw    (187.12,54.55) .. controls (143.27,169.89) and (207.18,222.11) .. (302.6,238.51) ;
\draw [shift={(305.5,239)}, rotate = 189.27] [fill={rgb, 255:red, 0; green, 0; blue, 0 }  ][line width=0.08]  [draw opacity=0] (10.72,-5.15) -- (0,0) -- (10.72,5.15) -- (7.12,0) -- cycle    ;
\draw [shift={(188.5,51)}, rotate = 111.55] [fill={rgb, 255:red, 0; green, 0; blue, 0 }  ][line width=0.08]  [draw opacity=0] (10.72,-5.15) -- (0,0) -- (10.72,5.15) -- (7.12,0) -- cycle    ;
\draw    (215.5,72) .. controls (167.5,123) and (183.5,197) .. (259.5,103) ;
\draw [shift={(259.5,103)}, rotate = 308.96] [color={rgb, 255:red, 0; green, 0; blue, 0 }  ][fill={rgb, 255:red, 0; green, 0; blue, 0 }  ][line width=0.75]      (0, 0) circle [x radius= 3.35, y radius= 3.35]   ;
\draw [shift={(215.5,72)}, rotate = 133.26] [color={rgb, 255:red, 0; green, 0; blue, 0 }  ][fill={rgb, 255:red, 0; green, 0; blue, 0 }  ][line width=0.75]      (0, 0) circle [x radius= 3.35, y radius= 3.35]   ;
\draw    (259.5,103) -- (177.4,130.37) ;
\draw [shift={(175.5,131)}, rotate = 341.57] [color={rgb, 255:red, 0; green, 0; blue, 0 }  ][line width=0.75]    (10.93,-3.29) .. controls (6.95,-1.4) and (3.31,-0.3) .. (0,0) .. controls (3.31,0.3) and (6.95,1.4) .. (10.93,3.29)   ;
\draw  [line width=5.25] [line join = round][line cap = round] (173.3,131.22) .. controls (173.3,131.22) and (173.3,131.22) .. (173.3,131.22) ;
\draw  [line width=5.25] [line join = round][line cap = round] (238.3,219.22) .. controls (238.3,219.22) and (238.3,219.22) .. (238.3,219.22) ;
\draw    (269.5,133) .. controls (221.5,184) and (237.5,258) .. (313.5,164) ;
\draw [shift={(313.5,164)}, rotate = 308.96] [color={rgb, 255:red, 0; green, 0; blue, 0 }  ][fill={rgb, 255:red, 0; green, 0; blue, 0 }  ][line width=0.75]      (0, 0) circle [x radius= 3.35, y radius= 3.35]   ;
\draw [shift={(269.5,133)}, rotate = 133.26] [color={rgb, 255:red, 0; green, 0; blue, 0 }  ][fill={rgb, 255:red, 0; green, 0; blue, 0 }  ][line width=0.75]      (0, 0) circle [x radius= 3.35, y radius= 3.35]   ;
\draw  [dash pattern={on 4.5pt off 4.5pt}]  (336.5,137) .. controls (345.32,109.56) and (309,65.79) .. (280.25,55.58) ;
\draw [shift={(278.5,55)}, rotate = 17.24] [color={rgb, 255:red, 0; green, 0; blue, 0 }  ][line width=0.75]    (10.93,-3.29) .. controls (6.95,-1.4) and (3.31,-0.3) .. (0,0) .. controls (3.31,0.3) and (6.95,1.4) .. (10.93,3.29)   ;
\draw  [dash pattern={on 4.5pt off 4.5pt}]  (288.5,116) .. controls (312.38,89.26) and (287,71.64) .. (275.09,63.14) ;
\draw [shift={(273.5,62)}, rotate = 36.03] [color={rgb, 255:red, 0; green, 0; blue, 0 }  ][line width=0.75]    (10.93,-3.29) .. controls (6.95,-1.4) and (3.31,-0.3) .. (0,0) .. controls (3.31,0.3) and (6.95,1.4) .. (10.93,3.29)   ;

\draw (149,33.4) node [anchor=north west][inner sep=0.75pt]    {$\left[ h^{+}\right] =h^{-m}\left[ h^{+}\right]$};
\draw (310,221.4) node [anchor=north west][inner sep=0.75pt]    {$\left[ h^{-}\right]$};
\draw (124,181.4) node [anchor=north west][inner sep=0.75pt]    {$A\mathrm{x}( h)$};
\draw (195,59.4) node [anchor=north west][inner sep=0.75pt]    {$x_{n}$};
\draw (254,79.4) node [anchor=north west][inner sep=0.75pt]    {$y_{n}$};
\draw (67,123.4) node [anchor=north west][inner sep=0.75pt]    {$\pi_{\ax(h)}(y_n)=w_{n}$};
\draw (211,222.4) node [anchor=north west][inner sep=0.75pt]    {$h^{-m} w_{n}$};
\draw (278,119.4) node [anchor=north west][inner sep=0.75pt]    {$h^{-m} x_{n}$};
\draw (317,143.4) node [anchor=north west][inner sep=0.75pt]    {$h^{-m} y_{n}$};

\end{tikzpicture}
    \caption{Disjoint fixed points are non-pinched by Lemma \ref{DisjFixpts=Non-pinched}.}
    \label{fig:horocycle}
\end{figure}
\begin{lem}\label{DisjFixpts=Non-pinched}
Let $h\in \isom(\U)$ be a contracting isometry  so that $[h^-]\cap [h^+]=\emptyset$. Then $h$ is non-pinched.     
\end{lem} 
\begin{proof}
We shall show that if $[h^+]$ is pinched, then $[h^-]=[h^+]$. By definition, as $[h^+]$ is pinched, there exist $x_n, y_n\in \U$ so that $x_n, y_n\to [h^+]$ and $\sup_{n\ge 1} d(o,[x_n, y_n])\le M<\infty$ for some $M$. The geodesic $\alpha_n:=[x_n,y_n]$ has length tending to $\infty$.  The same property holds for the pair $(h^mx_n, h^my_n)$ for any $m\ne 0$, since $h$ fixes $[h^+]$.

Here is the   key observation. Note that   $\alpha_n$ intersects a $M$-neighborhood of $\ax(h)$. So if we project the two endpoints $x_n, y_n$ to $u_n, v_n$ on $\ax(h)$ respectively,  it follows by Lemma \ref{BigThree} that  $d(u_n, \alpha_n),d(v_n, \alpha_n)\le C$ for any $n\gg 0$. We claim that there exists $D>0$ so that   $d(o,w_n)\le D$ for some $w_n\in \{u_n, v_n\}$ and each $n\ge 1$. Indeed, if not, we have $d(o, \{u_n, v_n\})\to\infty$ as $n\to\infty$. Recalling that $\sup_{n\ge 1} d(o,\alpha_n)\le M$ for some $M<\infty$, this implies that the geodesic segment $\alpha_n$ stays within a fixed finite neighborhood (depending on $C,M$) of  the subpaths $[u_n, o]_{\ax(h)}$ and $[o,v_n]_{\ax(h)}$ of the quasi-geodesic $\ax(h)$. As $n\to\infty$ and $\ell(\alpha_n)\to\infty$, this contradicts that $\alpha_n$ is a geodesic segment. Thus, $d(o,w_n)\le D$ for some $w_n\in \{u_n, v_n\}$. See Fig. \ref{fig:horocycle}.

For any fixed large $m>0$,  we apply $h^{-m}$ to the sequence $(x_n,y_n)$: $h^{-m}x_n, h^{-m}y_n\to [h^+]$ as $n\to\infty$, and  $d(h^{-m}w_n, h^{-m}o)\le D$. Letting $m\to \infty$, we can choose a sequence of  points $z_m$ in $\{h^{-m}x_n, h^{-m}y_n: n\ge 1\}$,   so that $z_m$ tends to $[h^+]$, and $z_m$ projects into the $D$-neighborhood of $h^{-m}o$.  \ref{AssumpA} shows that  $z_m$ tends to $[h^-]$ as $m\to\infty$.  Hence, $[h^+]=[h^-]$ follows. This concludes the proof of lemma.
\end{proof}
 
Unless explicitly stated, the statements from now on until the end of this section are under \ref{AssumpA} and \ref{AssumpC}, but without \ref{AssumpB}.  

As expected, we deduce below that two non-pinched contracting isometries have either the same  or disjoint fixed points.  
Recall that if $h$ is not assumed to be in a discrete group,  the axis is defined as in (\ref{axisdefn0}) to be the path $\ax(h)=\cup_{n\in\mathbb Z} h^n[o,ho]$.

\begin{lem}\label{DisjFixedSet}
Let $h, k\in \isom(\U)$ be two  contracting isometries  so that $\ax(h)$ and $\ax(k)$ have bounded intersection.   If $h$ is non-pinched, then  $[h^\pm]\cap [k^\pm]=\emptyset$.
\end{lem}

\begin{proof}
Up to taking inverses, assume by contradiction that $\xi\in [h^+]\cap[k^+]$. 

Take a sequence of points $x_n\in \ax(h)\to \xi$. By $C$-contracting property, we have $[o,x_n]$ intersects the $C$-ball  centered  at $\pi_{\ax(k)}(x_n)$. Let us take  $u_n \in [o, x_n]$ so that $d(u_n, \ax(k)) \le C$. If    $\{u_n\}$ is unbounded,  the infinite intersection of $N_C(\ax(h))$ and $N_C(\ax(k))$  then  follows from  their   Morse property, contradicting the hypothesis. Thus,    $\{u_n\}$ is a bounded set. 

Let $y_n\in \ax(h)$ tend  to $\xi$. The contracting property of $\ax(k)$ implies that $[x_n, y_n]$ intersects the $C$-neighborhood of $\pi_{\ax(k)}(x_n)$, that is contained in a fixed compact ball by boundedness of $\{u_n\}$. This contradicts   \ref{AssumpC} for $[\xi]\subseteq \mathcal C$. Thus the lemma is proved.
\end{proof} 

\begin{cor}\label{DisjOrSameFixedSet}
Let $h, k$ be two non-pinched contracting elements in a discrete group $G$.  Then  either $[h^\pm]\cap [k^\pm]=\emptyset$ or $[h^\pm]= [k^\pm]$.
\end{cor}
\begin{proof}
Assume to the contrary that  $[h^+]= [k^+]$ and $[h^-]\ne [k^-]$. By Lemma \ref{DisjFixedSet}, $\ax(h)$ has infinite intersection with $\ax(k)$. We shall get a contradiction from the proper action. We present the proof for the convenience of the reader. 

As $[h^-]\ne [k^-]$,     the Morse property of $\ax(h)$ and $\ax(k)$ shows that    $\{h^no: n\ge 1\}$ and $\{k^no: n\ge 1\}$ are   $R$-close for some finite  $R>0$. In particular, there are two unbounded sequences of integers $n_i, m_i$ so that  $d(h^{n_i}o, k^{m_i}o)\le R$. The proper action   implies $\{h^{-n_i}k^{m_i}\}$ is a finite set. Passing to a subsequence, we may assume $\{h^{-n_i}k^{m_i}=\{h^{-n_j}k^{m_j}$ for any $i\ne j$, and thus $h^{n_i-n_j}=k^{m_j-m_i}$ where $n_i\ne n_j$ and  $m_i\ne m_j$.  As a consequence, we obtain $\ax(h)=\ax(k)$, contradicting the assumption $[h^-]\ne [k^-]$.  
\end{proof}

Note, however,  that two independent pinched elements may share the same fixed point. 
\begin{example}
Let us consider the free group $\mathbb F_3=\mathbb Z\star \mathbb F_2$ of rank 3, which is hyperbolic relative to the factor $\mathbb F_2=\langle g, h\rangle$. The Bowditch boundary is then obtained from the visual boundary ($\simeq$ Cantor set) of $\mathbb F_3$ by collapsing equivariantly  the limit set of $\mathbb F_2$ to one point.  It is easy to verify that Bowditch boundary is a convergence boundary, where the two pinched contracting elements $g, h$ are independent, but share the same unique fixed point.     
\end{example}


\subsection{Shortest projection   for boundary points}\label{SSProjBdryPts} 
We introduce two ways of defining the shortest projection points of a boundary points to a contracting subset.  Before describing the definitions, let us repeat the convention that given a quasi-geodesic $\gamma$, we denote by $\gamma^+$ either a positive half ray of $\gamma$ or the endpoint in $\pU$ determined by \ref{AssumpA}. In the later use, we often write $[\gamma^+]$ for the $[\cdot]$-class, since any unbounded sequence on the half ray $\gamma^+$ is  assumed to accumulate into the same $[\cdot]$-class. 

In what follows, let $\gamma$ be a $C$-contracting bi-infinite quasi-geodesic for $C>0$.  We shall define a  shortest projection map
\begin{align*}
  \pi_\gamma:   \pU\setminus [\gamma^\pm]&\longrightarrow \gamma\\
     \xi&\longmapsto \pi_\gamma(\xi)
\end{align*}
with the following properties:
\begin{enumerate}
    \item
     $\pi_\gamma(\xi)$ is a bounded set on $\gamma$ (possibly of diameter depending on $\xi$).
    \item 
    the map extends the shortest projection map $\pi_\gamma: \U\to {\gamma}$ in the following sense: if $y_n\in\U \to \xi$, then $\pi_\gamma(y_n)\subseteq \pi_\gamma(\xi)$ for all $n\gg 0$.
    \item 
    the map is coarsely equivariant: $\pi_{g\gamma}(g\xi)=g\pi_\gamma(\xi)$ for any $g\in \isom(\U)$.
\end{enumerate}

We first define the map for non-pinched points $\xi$ in $\mathcal C $ as in \ref{AssumpC}. In this case, the diameter of $\pi_\gamma(\xi)$ is bounded by a  constant depending on the contracting constant of $\gamma$ but independent of $\xi$.  

\begin{lem}\label{BdryProjLem}
For any $\xi \in \mathcal C\setminus [\gamma^\pm]$,  there exists a bounded set denoted by $\pi_\gamma(\xi)$ of diameter at most $2C$ with the following property.

Let $y_n\in \U\to [\xi]$. Then $\pi_\gamma(y_n)$ is contained in $\pi_\gamma(\xi)$ for all $n\gg 0$.
\end{lem}

\begin{proof}
We claim that  for some    $N>0$, the following set  has diameter at most $C$: $$U:= \{\pi_\gamma(y_n): n\ge N\} $$ Indeed, if not, there exists two subsequences $\{x_n\}, \{z_n\}$  of points from $\{y_n\}$   such that  $d_\gamma(x_n, z_n)>C$ for each $n\ge 1$.  By Lemma \ref{BigThree},        $[x_n,z_n]$ intersects $N_C(\pi_X(x_n))$.  By assumption, $[\xi]$ is disjoint with $[\gamma^{\pm}]$,    so  all projection points $\{\pi_\gamma(y_n): n\ge 1\}$ are contained in a fixed  ball $Z$ of finite radius  by  \ref{AssumpA}.  Consequently,    $[x_n,z_n]$ intersects $N_C(Z)$, but $x_n,z_n\to[\xi]\subseteq\mathcal C$, so we get a contradiction with    \ref{AssumpC}. Setting $\pi_\gamma(\xi):=N_C(U)$ completes the proof.
\end{proof}

The following fact   obtained in the above proof will be  used later on.
\begin{cor}\label{BdryProjCor}
Let $x_n, y_n\in \U$ be two sequences of points  tending to $[\xi] \subseteq \mathcal C\setminus [\gamma^\pm]$. Then $d_\gamma(x_n, y_n)\le C$  for all $n\gg 0$.  
\end{cor}

We now  define, without \ref{AssumpC}, the shortest projection $\pi_\gamma(\xi)$ for any boundary point $\xi\in\pU\setminus [\gamma^\pm]$.  

\begin{lem}\label{BdryProjLem2}
Assume that   $K\subseteq \pU\setminus [\gamma^+]$ is a closed subset which may contain $[\gamma^-]$. Then there is a negative half ray $\gamma^-$ (i.e.: a restriction to some $(-\infty,c)$ of $\gamma: (-\infty,\infty)\to \U$) depending on $K$  such that for any $\xi\in K\setminus [\gamma^-]$ and any $y_n\in \U\to \xi$, we have $\pi_\gamma(y_n)\subseteq \gamma^-$.

In particular, if $K\cap [\gamma^\pm]=\emptyset$, then there exists a bounded  subset $Z$ of $\gamma$ such that for any   $y_n\in \U\to \xi\in K$, $\pi_\gamma(y_n)$ is contained in $Z$.
\end{lem}
\begin{proof}
We proceed by way of contradiction. Let us assume that for some $\xi\in K$ and a sequence $y_n\in \U\to \xi$, so that $\{  \pi_\gamma(y_n): n\ge 1\}$    is not contained in any negative half way of $\gamma$. By \ref{AssumpA}, (a subsequence of) $\pi_\gamma(y_n)$ is escaping and tends to $[\gamma^+]$. Let us fix a  metric $\delta$ compatible with the topology $\U\cup \pU$. We thus have $\delta(y_n, [\gamma^+])\to 0$ on one hand,  and $\delta(y_n, \xi)\to 0$ on the other hand. This contradicts the assumption $\delta(K, [\gamma^+])>0$. 
\end{proof}
As a consequence, we can define a projection $\pi_\gamma(\xi)$  to $\gamma$ for any boundary point $\xi\in \pU\setminus [\gamma^\pm]$ as  the bounded set $Z$ for a sufficiently small compact neighborhood $K$ of $\xi$. 

\begin{lem}\label{PreimageProjLem}
For any bounded subset $Z\subseteq \gamma$, the preimage of $Z$ under $\pi_\gamma: \U\cup \pU\setminus [\gamma^\pm]\to \gamma$ has the closure disjoint with $[\gamma^\pm]$.    
\end{lem}
\begin{proof}
Let $K$ be the topological closure of $\pi_\gamma^{-1}(Z)$ in $\pU\cup \U$. If $K\cap [\gamma^\pm]\ne\emptyset$, let $\xi_n\in \pi_\gamma^{-1}(Z)\to [\gamma^+]$ for definiteness. By definition of $\pi_\gamma$, choose $x_n\in \U$ in a small neighborhood of  $\xi_n$ such that $x_n\to [\gamma^+]$, and $\pi_\gamma(x_n) \cap Z\ne\emptyset$. Choose $y_n\in\gamma$ so that $y_n\to [\gamma^+]$.  By the $C$-contracting property of $\gamma$, each $[x_n, y_n]$ intersects the $C$-neighborhood of $\pi_\gamma(x_n)$, and thus of a compact subset as $\pi_\gamma(x_n) \cap Z\ne\emptyset$.
This contradicts \ref{AssumpC} with the assumption $[\gamma^\pm]\subseteq\mathcal C$.    
\end{proof}

\begin{lem}\label{LocusProj} 
Let  $ \xi\in \mathcal C$.  
Then for any $\eta\in [\xi]$, we have $d_{\gamma}(\xi, \eta)\le 6C$.   
\end{lem}
\begin{proof}  
Assume by contradiction that   $d_{\gamma}(\xi, \eta)>  6C$.  If    $x_m\in\U\to \xi$ and $z_m\in \U\to \eta$,  the triangle inequality for $d_{\gamma}$ with Lemma \ref{BdryProjLem} implies  $$\forall m\gg 1: \quad d_{\gamma}(x_m, z_m)\ge d_{\gamma}(\xi, \eta)-4C\ge  2C$$ so by Lemma \ref{BigThree},   $[x_m, z_m]$  intersects the $C$-ball around $\pi_{\gamma}(\xi)$. As $\pi_{\gamma}(\xi)$ is a fixed subset of diameter at most $2C$,  the non-escaping sequence of $[x_m, z_m]$       contradicts the assumption of  $\xi\in \mathcal C$. So the lemma is proved.
\end{proof}

\subsection{North--South dynamics of non-pinched contracting elements}

Let $h\in \isom(\U)$ be a  non-pinched contracting isometry, where the fixed points $[h^\pm]$ are non-pinched. This is equivalent to have   $[h^-]\ne[h^+]$  by Lemma \ref{DisjFixpts=Non-pinched}. We shall prove that the action of $\langle h\rangle$ on the complement to $[h^\pm]$ in $\pU$ has so-called North--South dynamics. 

Such type of dynamics can be found in many contexts, but surprisingly, our proof is quite general and only uses  the above \ref{AssumpA} and \ref{AssumpC}.  This property shall be often used later on.


\begin{lem}[NS dynamics]\label{SouthNorthLem}
The action of $\langle h\rangle$ on $\pU\setminus [h^\pm]$ has the North--South dynamics: 

for any two open sets   $[h^+] \subseteq U$ and $[h^-]\subseteq V$ in $\pU$, there exists an integer $n>0$ such that $h^n (\pU \setminus V)\subseteq U$ and $h^{-n} (\pU \setminus U)\subseteq V$.   
\end{lem}


\begin{proof}
Since $h$ is a contracting element, the path $\gamma:=\bigcup_{i\in\mathbb Z} h^i[o,ho]$ is a contracting quasi-geodesic. 
Let $K$ be  the  set of points $x\in \U$ such that $[o, ho]\cap \pi_\gamma(x)\ne\emptyset$, so the following holds $$\bigcup_{g\in G} g  K = \U.$$ 
Denote by $\widetilde K$ the topological closure of $K$ in $\U\cup \pU$, and    $\Lambda K:=\pU\cap \widetilde K$. 

Let $\pi_\gamma:\pU\setminus [h^\pm]\to \gamma$ be the shortest projection map given by Lemma \ref{BdryProjLem2}.  By definition, $[h^\pm]=[\gamma^\pm]$.  By Lemma \ref{PreimageProjLem} we have    $\widetilde K\cap [h^\pm]=\emptyset$.

Observe that 
\begin{equation}\label{CoveringEQ}
\pU\cup \U= \Big(\bigcup_{g\in G} g  \widetilde K \Big)\bigcup {[h^\pm]}.
\end{equation}
Indeed, if $\xi\in \pU\setminus [h^\pm]$, any sequence of points $x_n\in \U\cup\pU \to \xi$  has bounded projection to $\gamma$ by \ref{AssumpA}. Then $\{x_n\}$ are   contained in  finitely many copies of $K$, so the   inclusion $\subseteq$ follows. The other direction follows by definition $\widetilde K$ and $\cup_{g\in G} gK=\U$. 

By a similar argument, Lemma \ref{BdryProjLem2} also implies that $\pi_\gamma(\widetilde K)$ lies in a finite  neighborhood of $[o,ho]$. Consequently, the projection map $\pi_\gamma$ coarsely commutes two actions of $\langle h\rangle$ on $\U\cup \pU\setminus [h^\pm]$  and on $\gamma$: for any $h\in G$ and $\eta\in \pU\setminus  [h^\pm]$,  
\begin{equation}\label{ActionCommutesEQ}
\pi_\gamma( h\eta)  \simeq h\pi_\gamma(\eta)
\end{equation}
where $\simeq$ means   one of the two sets being contained into   a uniform  neighborhood of the other.

As $\langle h\rangle$ acts by translation on $\gamma$, we see from (\ref{ActionCommutesEQ}) that $\langle h\rangle$ acts properly on $\U\cup\pU\setminus [h^\pm]$: for any compact subset $L\subseteq \pU\setminus [h^\pm]$, the set $\{g\in G: gL\cap L\ne\emptyset\}$ is finite. Moreover,  $\{h^iL: i\in \mathbb Z\}$ is a locally finite family of closed subsets, so the union of any subfamily  is a closed subset in $\U\cup \pU\setminus [h^\pm]$. 

If $W\subseteq \pU\cup U$ is a compact neighborhood of $\widetilde K$ disjoint with $[h^\pm]$, then $$U_N:=\pU\setminus \left(\bigcup_{i<N}  h^iW\cup [h^-]\right)$$ for $N\ge 1$  forms a neighborhood basis for $[h^+]$. Indeed, let $[h^+]\subseteq U$  be any open subset. By Lemma \ref{BdryProjLem2} applied to $U^c$, there exists    $N>0$ such that $U^c\subseteq \bigcup_{i< N}h^i W\cup [h^-]$.    That is to say, $U_N\subseteq U$.


Similarly, $V_M:=\pU\setminus (\bigcup_{i> -M} h^iW\cup[h^+])$ for $M>0$ is a neighborhood basis for $[h^-]$. Taking large enough $n>N+M$,  we obtain $h^{n} (\pU \setminus V_M)\subseteq U_N$, so the North--South dynamics follows. 
\end{proof}

Let $\gamma=\ax(h)$ be the axis of the non-pinched contracting element. The following result is implicit in the proof. 
\begin{cor}\label{OpenNBHD}
For any $L\ge 0$ there exist open neighborhoods $[h^-]\subseteq V$ and $[h^+]\subseteq U$ in $\U\cup\pU$ with the following property. For any $x\in V\setminus [h^-]$ and $y\in U\setminus [h^+]$ we have 
$$
d_\gamma(x, y)\ge L.
$$
Conversely, if $d_\gamma(x_n, y_n)\to \infty$ for $x_n, y_n\in \U\cup \pU\setminus [h^\pm]$, we have $(x_n,y_n)$ converges to $([h^-],[h^+])$.
\end{cor}
\begin{proof}
Keep the notions as in the proof of Lemma \ref{SouthNorthLem}. As $\widetilde K\subseteq W$ is a compact set disjoint with $[h^\pm]$,   $\pi_\gamma(\widetilde K)\subseteq \pi_\gamma(W)$ is bounded by Lemma \ref{BdryProjLem2}.  By (\ref{ActionCommutesEQ}), we can choose $U=U_N$ and $V=V_M$ for  large $M, N \gg L$ so that $d_\gamma(x,y)>L$ for $x\in U, y\in V$.
 The converse direction also follows by the North--South dynamics.
\end{proof}

We remark that the following result could be used to give another proof of Shadow Lemma \ref{ShadowLem} (as given in \cite{TYANG}). It is recorded here only for further applications. 
\begin{lem}\label{IntShadow}
Let $F$ be a set of non-pinched contracting elements given by Lemma \ref{extend3}.
For any  $f\in F$ and $r\gg 0$, there exist     open neighborhoods $U\subseteq V$ of the attracting fixed point $[f^+]$  so that for any $g\in G$,  we have $$gU\subseteq \Pi_o^F(go,r) \subseteq gV.$$

\end{lem}
\begin{proof}
Set  $\gamma:=\ax(f)$. Choose a shortest projection point $w\in \pi_{g\gamma}(o)$ from $o$ to $g\gamma$, so $d(o, go)\ge d(o, w)$ for $go\in g\gamma$. 
Let $U$ be an open neighborhood of $[f^+]$ by Corollary \ref{OpenNBHD} such that 
$$
\forall z\in U\cap \U:\quad d_{\gamma}(g^{-1}o, z)=d_{g\gamma}(o, gz)\gg d(o, fo)
$$  
This is to say, the  distance from  $go\in g\gamma$  to the exit point of $[o,gz]$ at $N_C(g\gamma)$ is sufficiently larger than $d(o,fo)$.   For the constant $r$ given by Lemma \ref{AnyBarrierLem},  $go$ is an $(r, f)$-barrier for  $[o,gz]$.  Thus, the open set $gU\cap \pU$ is contained in $\Pi_o^F(go,r)$.

The definition of $z\in \Pi_o^F(go, r)$ implies $\|N_r(g\gamma)\cap [o, z]\|\ge \|Fo\|_{\min}$, so by Lemma \ref{Transform}, we have $d_{g\gamma}(o, z)\ge  \|Fo\|_{\min}-4r$.  For large enough $\|F\|\gg 0$, Corollary \ref{OpenNBHD} gives an open neighborhood $V$ of $[f^+]$ such that $z\in gV$. That is, we proved that $\Pi_o^F(go, r)\subseteq gV$.
\end{proof}

By Lemma \ref{[LimitSet]} (under \ref{AssumpB'}), the limit set $\pG $ is well-defined up to taking the $[\cdot]$-locus. One desirable property for  limit sets is their topological minimality and uniqueness.  We  describe  a typical situation where this happens. Recall that a point in $\pU$ is called \textit{minimal} if its $[\cdot]$-locus contains only one point. 
\begin{lem}\label{UniqueLimitSet} 
Fix a non-pinched contracting element $h\in G$  with  minimal fixed points. Let $\Lambda=\overline{\{gh^\pm: g\in G\}}$  be the closure of the fixed points of all conjugates. Then under \ref{AssumpA}, \ref{AssumpB'} and \ref{AssumpC},
\begin{enumerate}
    \item 
    $\Lambda$ is the   unique,  minimal, $G$-invariant closed subset.
    \item
    $[\Lambda]=[\pG]$ for any $o\in \U$.
    \item
    The fixed point pairs $([g^+],[g^-])$ of all  non-pinched contracting elements $g$  in $G$ are dense in the distinct pairs of $\Lambda\times \Lambda$ in the following sense:
    \begin{itemize}
        \item For any $\xi\ne \eta \in \Lambda$, there exists $\xi_n\in [f_n^+], \eta_n\in [f_n^-]$ for some contracting element $f_n\in G$ so that $\xi_n\to \xi$  and $\eta_n\to \eta$.
    \end{itemize}
    
\end{enumerate} 
\end{lem}

\begin{proof}
We only need to prove the assertion (1).  The assertion $[\Lambda]=[\pG ]$ follows by  Lemma \ref{FixPtsDense}. Assuming   assertion (1), the double density follows the same proof of Lemma \ref{DoubleDense} (without taking  the $[\cdot]$-closure of fixed points of $h$).

Note that the conjugates of  $h$ are also non-pinched with minimal fixed points. A  non-elementary group $G$ thus contains infinitely many conjugates of $h$ with pairwise bounded intersection, so the set $Z$ of their fixed points   are   infinite by Lemma \ref{DisjFixedSet}. Let $A$ be any $G$-invariant closed subset in $\pU$. First of all, $A$ contains at least three points. Otherwise, if $G$ fixes a point or a pair of points, this contradicts   the   North--South dynamics of conjugates of  $h$, as the set $Z$ of their fixed points is infinite.   As  $[h^+]=\{h^+\}$ is minimal, we can choose $\xi\in A\setminus  \{h^+\}$.  By Lemma \ref{SouthNorthLem},    $h^n\xi\in A$ converges to  $h^+$. This implies   $h^+\in \Lambda$, since $A$ is  closed and $G$-invariant.  Thus, $A$ contains $Z$ and $\Lambda\subseteq A$ is proved. 
\end{proof}

\begin{cor}\label{doubledenseCor}
Assume that {all contracting elements in $G$ are non-pinched with minimal fixed points}.  Let $\Lambda$ be as in Lemma \ref{UniqueLimitSet}. Then the set of the fixed point pairs of all  contracting element  in $G$ is dense in the distinct pairs of $\Lambda\times \Lambda$.    
\end{cor}

\begin{rem} 
The above two results   are a  generalization of the known results on CAT(0) groups with rank-1 elements (\cite[Theorem III.3.4]{Ballmann}\cite{H09a}) and dynamically irreducible subgroups  of mapping class groups (\cite{McPapa}). The applications to Roller   and Gardiner--Masur boundary  seem new, and are recorded in Lemmas \ref{Roller} and  \ref{GMBdryConvergence}.  The assumption in Corollary \ref{doubledenseCor} is satisfied for all these examples.
\end{rem}
\section{Conical and Myrberg points}\label{SecMyrberg}

In this section, we introduce the notion of conical   and Myrberg  limit points of a proper action with contracting elements  on the convergence boundary. This  shall be central in the study of  the conformal density performed in Section \textsection\ref{SecDensity}. 

We begin with some non-elementary assumption on the group action under consideration, throughout the remainder of this paper.

Let $G$ be a non-elementary group  acting properly on a proper geodesic space $\U$, as in (\ref{GActUAssump}), with a contracting element. Assume further that $\U$   is equipped with a \textit{non-elementary} convergence boundary $\pU$:  $G$ contains a non-pinched contracting element with fixed points being non-pinched in a sense in \ref{AssumpC}. Equivalently, this is plain to say that the repelling and attracting fixed points of some contracting element are distinct $[\cdot]$-classes in $\pU$ (Lemma \ref{DisjFixpts=Non-pinched}). 

To define conical limit points, we need to choose three independent non-pinched contracting elements $F=\{f_1,f_2, f_3\}$ and form the $C$-contracting system  for some $C\ge 0$ as in (\ref{SystemFDefIntro})
$$
\f=\{g\ax(f): f\in F, g\in G\}    
$$

\begin{conv}\label{ConvF}
Fix a constant $B$  satisfying Lemma \ref{Truncation}. The constant $L$ in Definition \ref{AdmDef} of an admissible path is chosen sufficiently large with the following properties:  
\begin{enumerate}
    \item 
    Any $(L,B)$-admissible path has $r$-fellow travel relative to $\f$, where the constant $r$ satisfies Proposition \ref{admisProp} and Lemma  \ref{extend3}. 
    \item 
    If $\alpha$ contains a $(r,f)$-barrier with $d(o,fo)>L$, then $\beta$ with endpoints at most $10C$-apart contains $(\hat r, f)$-barrier.  The constant $\hat r$ depends only on $r$  by Lemma \ref{BarrierFellowLem}. 
    \item 
    An intersection $\alpha\cap N_r(X)$ of diameter at least $L$ for $X\in \f$ produces an $(r,f)$-barrier for any geodesic $\alpha$. See Lemma \ref{AnyBarrierLem}.
\end{enumerate}
In particular, all the constants depend only on the contracting constant  $C$ of $\f$. 
\end{conv}

\subsection{Conical points using shadows}

From different perspectives, we are going to present three  definitions of conical limit points, which are  quantitatively equivalent. We first give the main definition of conical points formulated using shadows.

Fix a basepoint $o\in \U$, so everything we talk about below depends on $o$.
First of all, define the usual cone and shadow:  
$$\Omega_{x}(y, r)\; :=\; \{z\in \U: \exists [x,z]\cap B(y,r)\ne\emptyset\}$$
and $\Pi_{x}(y, r) \subseteq \pU$ be the topological closure  in $\pU$ of $\Omega_{x}(y, r)$.

We are actually working with a technically involving but very useful       notion of partial cones, which are defined  relative to $\f$.

\begin{defn}[Partial cone and shadow]\label{ShadowDef}
For $x\in \U, y\in Go$, the \textit{$(r, F)$-cone} $\Omega_{x}^F(y, r)$ is the set of elements $z\in \U$ such that $y$ is a $(r, F)$-barrier for some geodesic $[x, z]$.  

The \textit{$(r, F)$-shadow} $\Pi_{x}^F(y, r) \subseteq \pU$ is the topological closure  in $\pU$ of the cone $\Omega_{x}^F(y, r)$.
\end{defn}

We give  the first   definition  of a conical point. Recall that $\mathcal C\subseteq \pU$ is the  set of non-pinched boundary points in \ref{AssumpC}. 

\begin{defn}\label{ConicalDef2}
A point $\xi \in \mathcal C$ is called \textit{$(r, F)$-conical point}   if for some $x\in Go$, the point $\xi$ lies in infinitely many $(r, F)$-shadows $\Pi_x^{F}(y_n, r)$ for $y_n\in Go$.  We denote by  $\Lambda_{r}^F(Go) $ the set of $(r, F)$-conical points.

\end{defn}

\begin{rem}
\begin{enumerate}
    \item 
    If a geodesic ray $\gamma$ starting at $x\in Go$  contains infinitely many $(r, F)$-barriers $y_n\in Go$, then by definition, any accumulation points of $\gamma$ are   $(r, F)$-conical. 
    \item 
    According to our convention, $F$ consists of non-pinched contracting elements $f$. The fixed points $[f^\pm]$ of $f$ are the $[\cdot]$-locus of the accumulation points of the axes $\ax(f)$. It is   obvious that the accumulation points of the axes $\ax(f)$ are contained in infinitely many shadows $\Pi_x^{F}(y_n, r)$ over $y_n\in \ax(f)$, so are $(r,\{f,f^{-1}\})$-conical points for $r\gg 0$. Thus,  fixed points $[f^+]$   are $(r, F)$-conical points.
\end{enumerate}

\end{rem}

By definition,  $\Lambda_{r}^F(Go) $ is a $G$-invariant set, and   depends on the choice of $o\in \U$.  It is the limit superior  of  the family of partial shadows:
\begin{equation}\label{rFConicalEQ}
\Lambda_{r}^F(Go) =\bigcup_{x\in Go} \Big(\limsup_{y\in Go} \Pi_x^F(y, r)\Big).
\end{equation}
Recall that $\Lambda_{c}(Go)$ is the usual conical limit set given in (\ref{ConicalEQ}),  defined similarly as above but using the usual shadows $\Pi_x(y, r)$. Thus, $\Lambda_{r}^F(Go)\subseteq \Lambda_{c}(Go)$. 

Let us restate the Definition \ref{ConicalDef1} of conical points given in Introduction.
\begin{defn}\label{ConicalDef1prime}
A non-pinched point $\xi \in \mathcal C$ is called \textit{$(L,\f)$-conical point} for $L>0$ if  for some $x\in \U$, there exists a sequence of $X_n\in \f$ such that $d_{X_n}(x, \xi)\ge L$.  
\end{defn}

We clarify that these two definitions introduced so far are equivalent in the following quantitative sense.
\begin{lem}\label{QuantativeEquiv12}
Let $r, F$ be as above. There exist  $L_1, L_2>0$ depending on $r, F$ such that \begin{enumerate}
    \item 
     An $(L_1,\f)$-conical point is $(r, F)$-conical;
    \item 
     An $(r, F)$-conical point is $(L_2,\f)$-conical.
\end{enumerate} 
\end{lem}
\begin{proof}

Let $\xi$ be an $L_1$-conical point in Definition \ref{ConicalDef1prime} so there exists $X_n\in \f$ such that for each $n\ge 1$, we have
$$
d_{X_n}(x,\xi)\ge L_1
$$
Consider a convergent sequence of points $z_m\in \U\to\xi$.  By Lemma \ref{BdryProjLem} we have $d_{X_n}(x, z_m)\ge L_1-2C$  for any fixed $n\ge 1$ and for any $m\gg 0$. Choosing $L_1\gg \|Fo\|$   by Lemma \ref{AnyBarrierLem}, we have $[x,z_m]$ contains an $(r, f)$-barrier at $y_n\in \pi_{X_n}(x)$ for some $f\in F$. As $z_m\to\xi$, we have $\xi$ lies in the closure of $ \Omega_x^F(y_n, r)$, that is, the shadow $ \Pi_x^F(y_n, r)$. Thus, $\xi$ is $(r, F)$-conical.

Conversely,  if $\xi\in  \Pi_x^F(g_no, r)$, choose a sequence of points  $z_m\in \Omega_x^F(g_no, r)$ tending  to $\xi$ (as $m\to\infty$). There exists an $(r, F)$-barrier $y_n\in X_n=g_n\ax(f_n)$ for some $f_n\in F$ so that $$\|[x,z_m]\cap N_r(X_n)\|\ge \|Fo\|_{\min}$$ Lemma \ref{Transform} implies $d_{X_n}(x,z_m)\ge  L_2:=\|Fo\|_{\min}-4r$. Thus, $\xi$ is an $L_2$-conical point. The proof is complete.
\end{proof}

A point $\xi\in \pU$ is called \textit{$[\cdot]$-visual} (or simply \textit{visual}): for any $x\in\U$, there exists a geodesic ray starting at $x$ ending at $[\xi]$. 
We first prove that conical points are visual. 

We remark that the next result uses crucially \ref{AssumpB}, which underlies the proof of Lemma \ref{FreqContrLimits}. 

\begin{lem}\label{CharConicalLem}

Let $\xi \in \Lambda_{r}^F(Go)$ and $\hat r$ given by Lemma \ref{BarrierFellowLem}. For any $x\in \U$,  there exists a geodesic ray $\gamma$ from $x$, which accumulates  into $[\xi]$. Moreover, 
\begin{enumerate}
    \item 
    $\gamma$ contains infinitely many distinct $(\hat r, F)$-barriers $y_n\in Go$, and $y_n\to [\xi]$.
    \item 
    There exists  an $(L,B,\f)$-admissible ray starting from $x$ and ending at $[\xi]$, where $L, B$ depending on $F,r$ are given by Lemma \ref{Truncation}.
\end{enumerate}


\end{lem}
 
\begin{proof}
By definition of $\xi \in \Lambda_{r}^F(Go)$,  for each fixed $n\ge 1$, there exists a sequence of  points $z_{n,m}\in\U$ in the cone $\Omega_x^{F}(g_n o, r)$ for some $x_0\in Go$ and $g_n\in G$ such that   $z_{n,m}$ tends to $\xi$ when $m\to\infty$.  A geodesic segment $\alpha_m'=[x_0, z_{n,m}]$ contains the $(r, f_n)$-barrier $y_n=g_no$ for some $f_n\in F$:  that is, $$d(g_no, \alpha_m'),\; d(g_nf_no,\alpha_m')\le r.$$ 
Given $x\in \U$, consider a geodesic segment $\alpha_m=[x,z_{n,m}]$. We claim that  $g_no$ is $(\hat,f_n)$-barrier for $\alpha_m$, where   $\hat r$ is given by Lemma \ref{BarrierFellowLem}.  Indeed, we first note that $\alpha_m$ and $\alpha_m'$ both intersect $N_r(g_n\ax(f_n))$.  if $d(x, g_n\ax(f_n))\ge d(x_0,x)$, we have $d_{g_n\ax(f_n)}(x,x_0)\le C$. 

By assumption, $\U$ is a proper metric space. Up to passing to a subsequence, the sequence of $\alpha_m$  converges pointwise to  a limiting geodesic ray denoted by $\gamma_n$,   whence  $g_no$ is an  $(\hat r, f_n)$-barrier for $\gamma_n$. Passing to a subsequence, a limiting geodesic ray $\gamma$ of $\gamma_n$   still possesses  infinitely many distinct $(\hat r, F)$-barriers  $g_no$.   

 
In particular, $\|\gamma\cap N_{\hat r}(g_n\ax(f))\|\ge \|Fo\|_{\mathrm{min}}$. If $\|Fo\|_{\mathrm{min}}$ is sufficiently large, we obtain the convergence of $g_no\to [\xi]$   from Lemma \ref{FreqContrLimits}.

The general case for $x\in \U$ could further derived from  Lemma \ref{BarrierFellowLem}:  we leave the proof to the interested reader.

By Lemma \ref{Truncation}, the truncation  of $\gamma$ relative to barriers is an $(L, B)$-admissible ray of infinite type, for some $L=L(\|Fo\|_{\min})$. The assertion (2) thus follows by Lemma \ref{AdmisContrLimits}.
\end{proof}




As an immediate corollary, we can fix the light source $x$ in (\ref{rFConicalEQ}) for convenience.

\begin{cor}\label{FixLightSource}
For any $r>0$ there exists $\hat r>0$ such that for any $x\in Go$, we have 
$$
\Lambda_{r}^F(Go) \subseteq  \limsup_{y\in Go} \Pi_x^F(y, \hat r)
$$
\end{cor}


\subsection{Conical points using admissible paths}
The second assertion in the ``moreover" statement of Lemma \ref{CharConicalLem} motivates the third  formulation of conical points. It is based on a special class of $(L,B)$-admissible paths. Namely,  an $(L, B)$-admissible ray $\gamma$ is said of \textit{infinite type} if one of the following holds:
\begin{enumerate}
    \item 
    either its saturation $\f(\gamma)$ contains infinitely many elements, or 
    \item 
    $\gamma$ is eventually contained in a finite neighborhood of a contracting subset $X \in \f(\gamma)$. In particular, $\f(\gamma)$ is finite.
\end{enumerate} 
In the first case, $\f(\gamma)$ is infinite, so it forms an escaping sequence of $C$-contracting subsets (as the action $G\act \U$ is proper).
In the second case, if an admissible path $\gamma$ has  a finite collection $\f(\gamma)$ of associated contracting subsets,  then we can write $\gamma=p_0q_1p_1\cdots q_np_n$, where $p_n$ is a non-trivial geodesic ray in a finite neighborhood of $X$.  We refer the reader to  Definition \ref{AdmDef} of admissible paths.


Let the constant $B$ satisfy the truncation of a geodesic in Lemma \ref{Truncation}.
\begin{defn}\label{ConicalDef3}
A non-pinched point $\xi \in \mathcal C$ is called an \textit{$(L, B, \f)$-conical point}    if there exists an   $(L, B)$-admissible  ray of infinite type, which  accumulates into $[\xi]$. 

By definition,  the set of $(L, B, \f)$-conical points is $G$-invariant.
\end{defn}

We continue to clarify the quantitative equivalence between different formulations of conical points. 
\begin{lem}\label{QuantativeEquiv23}
Let $r>4C, F\subseteq G$ be as above. There exist $L_1, L_2> 0$ such that 
\begin{enumerate}
    \item 
    the $[\cdot]$-locus of  an $(L_1, B, \f)$-conical point $\xi\in \pU$  contains  an $(r, F)$-conical point,
    \item 
    the $[\cdot]$-locus of  an $(r, F)$-conical point contains $(L_2,B, \f)$-conical point $\xi\in \pU$.
\end{enumerate} 
\end{lem} 
\begin{proof}
The   statement (2) already follows from the second assertion of Lemma \ref{CharConicalLem}: if  $\xi$ is an $(r, F)$-conical point, then there exists an $(L_2,B,\f)$-admissible ray $\gamma$,  which accumulates  in $[\xi]$.  

Let  $L_1=\|Fo\|+L_0+6r>0$, where $L_0$ is given by Lemma \ref{AnyBarrierLem}. Let us now assume that $\xi$ is  an $(L_1, B,  \f)$-conical point, and prove the statement (1). 

Let  $\gamma$ be an $(L_1, B)$-admissible ray of infinite type ending in $[\xi]$. Take a sequence of geodesics $\alpha_n:=[o,x_n]$ for an unbounded sequence of points $x_n$ on $\gamma$. Any subpath of $\gamma$ is $(L_1, B)$-admissible, so has the $r$-fellow travel property by Proposition \ref{admisProp}: $$\|\alpha_n\cap N_r(X_i)\|>L_1-2r$$ for a sequence of  $X_i\in \f(\gamma)$. An Cantor argument using Ascoli-Arzela lemma    obtains  a limiting    geodesic ray $\alpha$  from $o$,   satisfying $\|\alpha\cap N_r(X_i)\|>L_1-2r$ as well. 

From Lemma \ref{FreqContrLimits}, we know that  $\alpha$ accumulates into $[\xi]$. By Lemma \ref{Transform}, we have 
$$
d_{X_i}(\alpha^-,\alpha^+) \ge \|\alpha\cap N_r(X_i)\|-4r \ge L_1-6r\ge \|Fo\|+L_0.
$$
By Lemma \ref{AnyBarrierLem}, $\alpha$ contains infinitely many  $(r, F)$-barriers, so by definition \ref{ConicalDef2}, any accumulation point of $\alpha$ in the boundary   is an $(r, F)$-conical point. 
\end{proof}


Furthermore, conical points are visual from each other.
 
\begin{lem}\label{VisualPointsLem}
For any          $[\xi]\ne[\eta]$ in  $\Lambda_{r}^F(Go)$, there exists  a bi-infinite geodesic so that the two half rays, both with infinitely many $(\hat r, F)$-barriers, accumulate into $[\xi], [\eta]$ respectively.  
\end{lem}
\begin{proof}
By Lemma \ref{QuantativeEquiv23}, let $\alpha,\beta$ be two $(L,B)$-admissible rays of infinite type ending at $[\xi], [\eta]$ respectively.    Let $\f(\alpha)=\{X_n\}$ and $\f(\beta)=\{Y_n\}$ be the corresponding saturation of $\alpha$ and $\beta$. 

According to the definition of infinite type, if $\f(\alpha), \f(\beta)$ are both finite, then  $\alpha, \beta$ are  eventually contained in a finite neighborhood of  contracting quasi-geodesics, the last one in $\f(\alpha)$ and $\f(\beta)$ respectively. In this case, the proof is completed by Lemma \ref{ContractingBdryVisual}.

Let us thus assume that one of $\f(\alpha)$ and $\f(\beta)$, say $\f(\alpha)$, is infinite. By Lemma \ref{DisjointContr}, we have $\xi\notin [\Lambda X_m]$ for each   fixed $X_m$.
Now, let $x_n\in \alpha\cap N_r(X_n)$ and $y_n\in \beta\cap N_r(Y_n)$ be two unbounded sequences.  Consider the sequence of   geodesics  $\gamma_n:=[x_n, y_n]$. We claim that $d(o, \gamma_n)$ are uniformly bounded for all $n\gg 0$. See Fig. \ref{fig:visualconicalpts}.

\begin{figure}
    \centering

\tikzset{every picture/.style={line width=0.75pt}} 

\begin{tikzpicture}[x=0.75pt,y=0.75pt,yscale=-1,xscale=1]

\draw   (413.81,39.93) .. controls (409.58,30.63) and (415.43,18.88) .. (426.86,13.68) .. controls (438.3,8.48) and (451,11.8) .. (455.23,21.1) .. controls (459.46,30.4) and (453.62,42.16) .. (442.18,47.36) .. controls (430.74,52.56) and (418.04,49.23) .. (413.81,39.93) -- cycle ;
\draw   (338.81,69.02) .. controls (334.58,59.71) and (340.43,47.96) .. (351.86,42.76) .. controls (363.3,37.56) and (376,40.88) .. (380.23,50.18) .. controls (384.46,59.48) and (378.62,71.24) .. (367.18,76.44) .. controls (355.74,81.64) and (343.04,78.32) .. (338.81,69.02) -- cycle ;
\draw  [dash pattern={on 4.5pt off 4.5pt}]  (294,108) -- (336.96,72.28) ;
\draw [shift={(338.5,71)}, rotate = 140.26] [color={rgb, 255:red, 0; green, 0; blue, 0 }  ][line width=0.75]    (10.93,-3.29) .. controls (6.95,-1.4) and (3.31,-0.3) .. (0,0) .. controls (3.31,0.3) and (6.95,1.4) .. (10.93,3.29)   ;
\draw  [dash pattern={on 4.5pt off 4.5pt}]  (380.23,50.18) -- (413.33,38.84) ;
\draw [shift={(415.22,38.19)}, rotate = 161.09] [color={rgb, 255:red, 0; green, 0; blue, 0 }  ][line width=0.75]    (10.93,-3.29) .. controls (6.95,-1.4) and (3.31,-0.3) .. (0,0) .. controls (3.31,0.3) and (6.95,1.4) .. (10.93,3.29)   ;
\draw    (455.23,21.1) -- (473.36,13.85) ;
\draw [shift={(475.22,13.11)}, rotate = 158.21] [color={rgb, 255:red, 0; green, 0; blue, 0 }  ][line width=0.75]    (10.93,-3.29) .. controls (6.95,-1.4) and (3.31,-0.3) .. (0,0) .. controls (3.31,0.3) and (6.95,1.4) .. (10.93,3.29)   ;
\draw   (403.84,199.24) .. controls (409.09,190.48) and (422.09,188.6) .. (432.87,195.06) .. controls (443.65,201.51) and (448.13,213.85) .. (442.88,222.62) .. controls (437.63,231.38) and (424.64,233.26) .. (413.86,226.8) .. controls (403.08,220.35) and (398.6,208.01) .. (403.84,199.24) -- cycle ;
\draw   (334.1,155.14) .. controls (339.35,146.37) and (352.35,144.5) .. (363.13,150.95) .. controls (373.91,157.4) and (378.39,169.74) .. (373.14,178.51) .. controls (367.89,187.28) and (354.9,189.15) .. (344.12,182.7) .. controls (333.34,176.24) and (328.85,163.9) .. (334.1,155.14) -- cycle ;
\draw  [dash pattern={on 4.5pt off 4.5pt}]  (294,108) -- (335.14,152.53) ;
\draw [shift={(336.5,154)}, rotate = 227.26] [color={rgb, 255:red, 0; green, 0; blue, 0 }  ][line width=0.75]    (10.93,-3.29) .. controls (6.95,-1.4) and (3.31,-0.3) .. (0,0) .. controls (3.31,0.3) and (6.95,1.4) .. (10.93,3.29)   ;
\draw  [dash pattern={on 4.5pt off 4.5pt}]  (368.5,183) -- (404.24,198.61) ;
\draw [shift={(406.07,199.41)}, rotate = 203.59] [color={rgb, 255:red, 0; green, 0; blue, 0 }  ][line width=0.75]    (10.93,-3.29) .. controls (6.95,-1.4) and (3.31,-0.3) .. (0,0) .. controls (3.31,0.3) and (6.95,1.4) .. (10.93,3.29)   ;
\draw    (442.88,222.62) -- (459.15,233.41) ;
\draw [shift={(460.81,234.52)}, rotate = 213.57] [color={rgb, 255:red, 0; green, 0; blue, 0 }  ][line width=0.75]    (10.93,-3.29) .. controls (6.95,-1.4) and (3.31,-0.3) .. (0,0) .. controls (3.31,0.3) and (6.95,1.4) .. (10.93,3.29)   ;
\draw    (413.81,39.93) .. controls (371.5,101) and (370.5,131) .. (403.84,199.24) ;
\draw    (294,108) -- (378.5,111.91) ;
\draw [shift={(380.5,112)}, rotate = 182.65] [color={rgb, 255:red, 0; green, 0; blue, 0 }  ][line width=0.75]    (10.93,-3.29) .. controls (6.95,-1.4) and (3.31,-0.3) .. (0,0) .. controls (3.31,0.3) and (6.95,1.4) .. (10.93,3.29)   ;
\draw [shift={(294,108)}, rotate = 2.65] [color={rgb, 255:red, 0; green, 0; blue, 0 }  ][fill={rgb, 255:red, 0; green, 0; blue, 0 }  ][line width=0.75]      (0, 0) circle [x radius= 3.35, y radius= 3.35]   ;

\draw (417.07,26.91) node [anchor=north west][inner sep=0.75pt]  [rotate=-335.55]  {$X_{n}$};
\draw (342.07,56.91) node [anchor=north west][inner sep=0.75pt]  [rotate=-335.55]  {$X_{m}$};
\draw (416.41,194.52) node [anchor=north west][inner sep=0.75pt]  [rotate=-30.91]  {$Y_{n}$};
\draw (349.36,152.99) node [anchor=north west][inner sep=0.75pt]  [rotate=-30.91]  {$Y_{m}$};
\draw (477.22,16.51) node [anchor=north west][inner sep=0.75pt]    {$[ \xi ]$};
\draw (456.22,208.4) node [anchor=north west][inner sep=0.75pt]    {$[ \eta ]$};
\draw (413,48.4) node [anchor=north west][inner sep=0.75pt]    {$x_{n}$};
\draw (410,165.4) node [anchor=north west][inner sep=0.75pt]    {$y_{n}$};
\draw (274,98.4) node [anchor=north west][inner sep=0.75pt]    {$o$};
\draw (392,102.4) node [anchor=north west][inner sep=0.75pt]    {$[ x_{n} ,y_{n}]\rightarrow \gamma $};

\end{tikzpicture}   
    \caption{Conical points are visual.}
    \label{fig:visualconicalpts}
\end{figure}
Indeed, suppose to the contrary that the sequence of $[x_n,y_n]$ is escaping. For $\xi\notin [\Lambda X_m]$,     Corollary \ref{BdryProjCor} shows $$\forall n\gg m: \quad d_{X_m}(x_n, y_n)\le C$$  By Proposition \ref{admisProp}, $\alpha$ has $r$-fellow travel property for $X_m$ and then by Lemma \ref{Transform}, we have $$d_{X_m}(o, x_n)\ge \|[o,x_n]\cap N_r(X_m)\| -4r\ge L- 4r$$ which yields by the triangle inequality $d_{X_m}$:   $$d_{X_m}(o, y_n)>L-C-4r.$$  By  Lemma \ref{Transform}, if $L> C-7r$, we have $\beta$ intersects $N_r(X_m)$.  As $X_n$ accumulates into $[\xi]$ by Lemma \ref{FreqContrLimits}, we obtain that $\beta$ does so as well. This leads to a contradiction with  $\alpha\ne \beta$. The claim follows.

As $\U$ is a locally compact complete metric space,  a subsequence of  $\gamma_n$ converges locally uniformly to a   bi-infinite geodesic $\gamma$, which intersects $N_r(X_m)$ and $N_r(Y_m)$ for all $m\gg 0$. Moreover, $X_m$ and $Y_m$ produce an infinite sequence of $(\hat r, F)$-barriers on the two half rays $\gamma^-,\gamma^+$ of $\gamma$, where $\hat r$ is given by Lemma \ref{BarrierFellowLem}. According to Lemma \ref{FreqContrLimits},  $\gamma^-,\gamma^+$ accumulate into $[\xi]\ne[\eta]$ respectively. 

In summary,  we proved a bi-infinite geodesic $\gamma$ with desired properties, connecting   $[\xi]$ and $[\eta]$. The proof of lemma is complete.
\end{proof}

We next show that our notion of conical points satisfies the dynamical property of conical points in the theory of convergence groups.

\begin{lem}\label{ConvOnConicalLem}
Let $\xi \in \Lambda_{r}^F(Go)$ be a  conical  point. Then  there exists a 
sequence of elements $g_n \in G$ and a pair of     $[a]\ne [b]
\in [\pU]$ such that the following holds
$$g_n(\xi, x) \to ([a], [b]),$$
for any point $x\in \U$. 
\end{lem}
\begin{rem}
This conclusion could hold for certain boundary points $x\in \pU$, for instance if $\zeta$ is a conical point. We restrict to the points in  $\U$ to simplify the proof.
\end{rem}

\begin{proof}
By the proof of Lemma \ref{CharConicalLem}, there exists a   geodesic $\gamma$ from $x$ to $\xi$, so that there exists a sequence   $X_n\in \f$ such that $\|\gamma\cap N_r(X_n)\|\ge \|Fo\|_{\min}$.   Passing to subsequence, we assume that $X_n=g_n^{-1}\ax(f)$ for the same $f\in F$ and $g_n\in G$. 

If $X:=X_n$  are eventually constant,  then $\gamma$ is eventually contained in a finite neighborhood of $X$, so $[\xi]$ is the end point of a positive ray of $X$ (with appropriate arc-length parametrization). Thus, the diameter of  $g_n \gamma\cap N_r(\ax(f))$ tends to $\infty$, that is: $\pi_{\ax(f)}(g_nx,g_n\xi)\to \infty$. According to our convention, $F$ consists of non-pinched contracting elements. Let $[a]\ne [b]$ be the attracting and repelling fixed points of non-pinched $f$. The conclusion in this case now follows from   Lemma \ref{OpenNBHD}. 

Passing to a subsequence, we  next  assume  that $X_n$ are distinct. 
	Let $\alpha_n$ be the geodesic segment of $g_n\gamma$ before entering $N_r(\ax(f))$, and $\beta_n$  the remaining   half ray.  By Lemma \ref{BarrierFellowLem}, two subsequences of $\alpha_n$ and $\beta_n$ converge respectively to  two geodesic rays $\alpha, \beta$ with infinitely many $(\hat r, F)$-barriers. By Lemma \ref{FreqContrLimits} using crucially \ref{AssumpB}, $\alpha, \beta$ accumulate to $[a], [b]$ respectively, which depend only on the sequence $\{X_n\}$.  The proof  is complete.
\end{proof}


\subsection{Myrberg   points}
We are in a position to formulate a more intrinsic subclass of conical  points, called Myrberg  points, in a sense that  Myrberg  points do not refer to a contracting system as conical points do. It has several equivalent characterizations (cf. \cite{Tukia4}); we give the one in line with  the convergence property of conical points in Lemma \ref{ConvOnConicalLem}. 

Recall that a point $\xi\in \pU$ is called visual if there exists a geodesic ray starting from any point in $\U$ and accumulating into $[\xi]$. In Lemma \ref{VisualPointsLem}, we proved that conical points are visual.
\begin{defn}
A visual point $\xi \in \mathcal C$ is called  \textit{Myrberg   point} if for any  pair of    $[a]\ne [b]
\subseteq
[\pG ]$ there exists a sequence of elements $g_n \in G$  such that the following  
$$g_n(\xi, x) \longrightarrow ([a], [b])$$
holds for any $x \in \U$. Denote by $\mG$ the set of Myrberg   points.
\end{defn}
 
First of all, let us note that the fixed points of a contracting element are conical, but never    Myrberg points.
\begin{lem}\label{ContrMyrberg}
Assume that a non-elementary $G$ contains a non-pinched contracting element. Then the   locus of fixed points of any  contracting element  contain no Myrberg point.
\end{lem}
\begin{proof}
By Lemma \ref{DisjFixedSet}, a non-elementary $G$ must contain at least  two independent  non-pinched contracting elements $k_1, k_2$. Any  contracting element  $h$ must be    independent with one, say $k$, of  $\{k_1,k_2\}$. Hence, by independence, any $G$-translated axis of $h$ has uniformly bounded projection   to the axis of $k$. If $[h^+]$ contains a Myrberg point $\xi$, we have $g_nx\to [k^-]$ and $g_n\xi\to [k^+]$. By Corollary \ref{OpenNBHD},  it follows that $d_{\ax(k)}(g_nx, g_n\xi)\to\infty$. This is a contradiction. Thus, $[h^+]$ contains no Myrberg points. The proof is complete.
\end{proof}

\begin{figure}
    \centering

\tikzset{every picture/.style={line width=0.75pt}} 

\begin{tikzpicture}[x=0.75pt,y=0.75pt,yscale=-1,xscale=1]

\draw    (81.5,195) .. controls (121.3,165.15) and (300.69,115.5) .. (413.8,122.88) ;
\draw [shift={(415.5,123)}, rotate = 184.05] [color={rgb, 255:red, 0; green, 0; blue, 0 }  ][line width=0.75]    (10.93,-3.29) .. controls (6.95,-1.4) and (3.31,-0.3) .. (0,0) .. controls (3.31,0.3) and (6.95,1.4) .. (10.93,3.29)   ;
\draw    (173.5,102) .. controls (196.5,156) and (295.5,129) .. (291.5,70) ;
\draw    (187.5,122) -- (193.82,149.08) ;
\draw [shift={(194.5,152)}, rotate = 256.87] [fill={rgb, 255:red, 0; green, 0; blue, 0 }  ][line width=0.08]  [draw opacity=0] (8.93,-4.29) -- (0,0) -- (8.93,4.29) -- cycle    ;
\draw [shift={(187.5,122)}, rotate = 76.87] [color={rgb, 255:red, 0; green, 0; blue, 0 }  ][fill={rgb, 255:red, 0; green, 0; blue, 0 }  ][line width=0.75]      (0, 0) circle [x radius= 3.35, y radius= 3.35]   ;
\draw    (283,104) -- (288.82,129.08) ;
\draw [shift={(289.5,132)}, rotate = 256.93] [fill={rgb, 255:red, 0; green, 0; blue, 0 }  ][line width=0.08]  [draw opacity=0] (8.93,-4.29) -- (0,0) -- (8.93,4.29) -- cycle    ;
\draw [shift={(283,104)}, rotate = 76.93] [color={rgb, 255:red, 0; green, 0; blue, 0 }  ][fill={rgb, 255:red, 0; green, 0; blue, 0 }  ][line width=0.75]      (0, 0) circle [x radius= 3.35, y radius= 3.35]   ;

\draw (69,201.4) node [anchor=north west][inner sep=0.75pt]    {$o$};
\draw (198,80.4) node [anchor=north west][inner sep=0.75pt]    {$g_{n} A\mathrm{x}( h)$};
\draw (155,113.4) node [anchor=north west][inner sep=0.75pt]    {$g_{n} o$};
\draw (290,91.4) node [anchor=north west][inner sep=0.75pt]    {$g_{n} h^{k} o$};
\draw (422,131.4) node [anchor=north west][inner sep=0.75pt]    {$\xi $};
\draw (290,112.4) node [anchor=north west][inner sep=0.75pt]    {$\leq r$};
\draw (193,131.4) node [anchor=north west][inner sep=0.75pt]    {$\leq r$};

\end{tikzpicture}
    \caption{Myrberg Points: set $f=h^k\in E(h)$ in Lemma \ref{CharMyrberg}}
    \label{fig:mybergpoints}
\end{figure}
 
The following result explains the canonical feature of a Myrberg point. Let $A$ be a set of non-pinched contracting elements  so that its fixed points pair $([h^+], [h^-])$ of elements $h\in A$ is dense in the distinct pairs $\dG$ of limit points defined (\ref{doubleBdryEq}). 
\begin{lem}\label{CharMyrberg}
A  point $\xi\in \mathcal C$ is a Myrberg   point if and only if the following holds.

Let  $f\in E(h)$ be a non-pinched contracting element for some  $h\in A$. Any geodesic ray $\gamma$ ending at $[\xi]$ contains infinitely many distinct $(r, f)$-barriers. Here $r$ depends only on the contracting constant of $\ax(h)$ (but not on the choice of $f$).
\end{lem}

\begin{proof}
``$\Longrightarrow$":  Let $\gamma$ be a geodesic ray starting at $x\in\U$ and ending at  a Myrberg   point $[\xi]$. Fix    a non-pinched contracting element $h$. Applying the definition of Myrberg points to  $[a]:=[h^-], [b]:=[h^+]$, we have $g_n(\xi, x) \to ([a], [b])$ for a sequence of $g_n\in G$. By Lemma \ref{BdryProjLem}  
we have a  well-defined projection map   $\pi_{X}:  \U\cup \mathcal C\setminus [\Lambda X]\to X$ where $X=\ax(h)$.   
 
Set $X_n:=g_n^{-1}X$. As $g_n(\xi, x) \to ([a], [b])$,  by  Corollary \ref{OpenNBHD}, there exist $n_0>0$    such that    for any $n\ge n_0$, $$
d_{X_n}(x,\xi)\gg d(o, fo)
$$   By Lemma \ref{AnyBarrierLem},     $\gamma$ contains an  $(r, f)$-barrier $g_n^{-1}X$. So the direction ``$\Longrightarrow$" follows.


``$\Longleftarrow$": As the set of      the pairs $([h^-], [h^+])$ for   $h\in A$ is dense in the distinct $[\cdot]$-pairs in $[\pG ]$, it suffices to verify the definition of Myrberg  point for each pair $([h^-],  [h^+])$.  By assumption, $\gamma=[x,\xi]$ contains   a sequence of $(r,f_n)$-barrier $g_n^{-1}\ax(h)$ for some $g_n\in G$ and $f_n\in E(h)$  with $d(o,f_no)\to \infty$. Thus, we obtain $$\|g_n\gamma\cap N_r(X)\|\ge d(o,f_no)$$
which tends to $\infty$ and thus yields $g_n(\xi, x) \to ([h^-], [h^+])$ by   Corollary \ref{OpenNBHD}. As $h$ is arbitrary,   Lemma \ref{DoubleDense} shows that $\xi$ is a Myrberg point. The proof is complete. 
\end{proof}

The following corollary summarizes the relation between conical points and Myrberg points.
\begin{cor}\label{MyrbergConical}
We have
$$
\mG= \bigcap \Lambda_r^F(Go)
$$
where the intersection is taken over   the sets  $F\subseteq G$ of three mutually independent non-pinched contracting elements in $G$.  
\end{cor}

\subsection{Gromov boundary of projection complex}\label{SSecGromovBdryPC}
Recall that $\f$ defined in (\ref{SystemFDefIntro}) is the collection of $G$-translated axis of three independent non-pinched elements $F$ in $G$.  We recall some facts in \textsection
\ref{SSecProjectionComplex} on the projection complex and quasi-tree of spaces built from $\f$, which are used in this and next subsections.

Let   $\PC$  be the projection complex built from $\f$, where $K\gg 0$ satisfies  Lemma \ref{OrderLem}. We denote by   $\partial \PC$  the Gromov boundary.  

Set $\widetilde X:= \bigcup \{X\in \f\}$, which is the union of all quasi-lines in $\f$ as a subset in $\U$. Let us recall the shadow map $\widetilde X\to \PC$ in  (\ref{ShadowMap}), which is  a coarsely Lipschitz map,  collapses all points in $X\in \f$ to a vertex $\bar x$: 
\begin{equation}
\begin{tikzcd}\label{LipProjectionMap}
\widetilde X \arrow[d,"(\ref{ShadowMap3})"'] \arrow[dr,"(\ref{ShadowMap})"] & \\
\QT \arrow[r,"(\ref{ShadowMap2})"'] & \PC
\end{tikzcd}
\end{equation}
By Lemma \ref{LiftPathLem}, a standard path in $\PC$ is lifted via (\ref{ShadowMap}) to an admissible path in $\U$, which is sent  under this map to an unparametrized quasi-geodesic.

The goal of this subsection is examining how the map  (\ref{ShadowMap}) induces a boundary map from the Gromov boundary $\partial \PC$ to the conical limit set.

Let us start with the definition of Gromov boundary. 

Fix a basepoint $\bar o\in \PC$. We say that a (unbounded) sequence   $\{\bar  u_n\}$ is \textit{admissible} if $\langle \bar  u_n , \bar u_m\rangle_{\bar o}\to\infty$ as $n, m\to\infty$, and  two  sequences $\{\bar  u_n\},   \{\bar v_n\}$ are \textit{equivalent} if 
$\langle   \bar u_n ,  \bar v_n\rangle_{\bar o}\to\infty$. By definition, the Gromov boundary   $\partial \PC$ consists of equivalent classes of admissible sequences     $\{\bar  u_n\}$.

\begin{lem}\label{standardray}
For any $\xi\in \partial \PC$, there exists a  quasi-geodesic ray $\gamma$ in $\PC$ starting at $\bar o$ and ending at $\xi$, so that any subpath of $\gamma$ is a standard path (cf. \textsection
\ref{SSecProjectionComplex}). 
\end{lem}
The path $\gamma$ in the statement shall be refereed  to as a standard ray.
\begin{proof}
Let $\bar u_n\to \xi\in \partial \PC$, and   $\alpha_n$ be the standard path from $\bar o$ to $\bar u_n$, which is a uniform quasi-geodesic. The triangle of standard paths has the tripod-like   property in Lemma \ref{Tripod}. As  $\langle  \bar u_n ,  \bar u_m\rangle_{\bar o}\to\infty$, one can extract a subsequence of $\alpha_n$ such that it converges to a quasi-geodesic ray $\alpha$ in a sense that restricting on any bounded subsets in $\PC$, $\alpha_n$ coincides with $\alpha$. Thus, any subpath of $\alpha$   are standard paths, and $\alpha$ converges to the boundary point $\xi$. The tripod-like property actually implies  that $\alpha$ is  unique, provided that $\bar o$ is fixed.
\end{proof}
The first main result is  that the Gromov boundary $\partial \PC$ embeds into the set of $(r,F)$-conical points.  We emphasize this result crucially uses  \ref{AssumpB} instead of \ref{AssumpB'}. 

\begin{lem}\label{EmbedGromov}
For any $r, F$ satisfying Convention \ref{ConvF}, there exists $K>0$ such that the following holds.
 
There is a continuous embedding of  the Gromov boundary $\partial \PC$ of $\PC$   into    $[\Lambda_r^F(Go)]$ of $(r, F)$-conical points in the quotient space $[\pU]$.
 
\end{lem}

\begin{proof}
Given $\xi\in \partial \PC$, let $\alpha$ be a  standard ray ending at $\xi$ given by Lemma \ref{standardray}. By Lemma \ref{LiftPathLem}, there exist $L=L(K), B=B(\f)$ so that we can lift $\alpha$ to get an $(L, B)$-admissible  quasi-geodesic ray $\tilde \alpha$. According to definition \ref{ConicalDef3},  $\tilde \alpha$ accumulates into the $[\tilde \alpha^+]$-class of an $(L,B,\f)$-conical point in $\pU$. Moreover, $L\to \infty$  as $K\to\infty$.
\begin{figure}
    \centering

\tikzset{every picture/.style={line width=0.75pt}} 

\begin{tikzpicture}[x=0.55pt,y=0.55pt,yscale=-1,xscale=1]

\draw    (30.5,105) -- (129.5,77) ;
\draw [shift={(129.5,77)}, rotate = 344.21] [color={rgb, 255:red, 0; green, 0; blue, 0 }  ][fill={rgb, 255:red, 0; green, 0; blue, 0 }  ][line width=0.75]      (0, 0) circle [x radius= 3.35, y radius= 3.35]   ;
\draw [shift={(30.5,105)}, rotate = 344.21] [color={rgb, 255:red, 0; green, 0; blue, 0 }  ][fill={rgb, 255:red, 0; green, 0; blue, 0 }  ][line width=0.75]      (0, 0) circle [x radius= 3.35, y radius= 3.35]   ;
\draw    (129.5,77) -- (198.5,49) -- (282.69,9.84) ;
\draw [shift={(284.5,9)}, rotate = 155.06] [color={rgb, 255:red, 0; green, 0; blue, 0 }  ][line width=0.75]    (10.93,-3.29) .. controls (6.95,-1.4) and (3.31,-0.3) .. (0,0) .. controls (3.31,0.3) and (6.95,1.4) .. (10.93,3.29)   ;
\draw [shift={(164,63)}, rotate = 337.91] [color={rgb, 255:red, 0; green, 0; blue, 0 }  ][fill={rgb, 255:red, 0; green, 0; blue, 0 }  ][line width=0.75]      (0, 0) circle [x radius= 3.35, y radius= 3.35]   ;
\draw [shift={(241.5,29)}, rotate = 335.06] [color={rgb, 255:red, 0; green, 0; blue, 0 }  ][fill={rgb, 255:red, 0; green, 0; blue, 0 }  ][line width=0.75]      (0, 0) circle [x radius= 3.35, y radius= 3.35]   ;
\draw    (129.5,77) -- (192.5,98) ;
\draw [shift={(192.5,98)}, rotate = 18.43] [color={rgb, 255:red, 0; green, 0; blue, 0 }  ][fill={rgb, 255:red, 0; green, 0; blue, 0 }  ][line width=0.75]      (0, 0) circle [x radius= 3.35, y radius= 3.35]   ;
\draw    (192.5,98) -- (277.6,125.39) ;
\draw [shift={(279.5,126)}, rotate = 197.84] [color={rgb, 255:red, 0; green, 0; blue, 0 }  ][line width=0.75]    (10.93,-3.29) .. controls (6.95,-1.4) and (3.31,-0.3) .. (0,0) .. controls (3.31,0.3) and (6.95,1.4) .. (10.93,3.29)   ;
\draw   (367,73.5) .. controls (367,63.28) and (377.19,55) .. (389.75,55) .. controls (402.31,55) and (412.5,63.28) .. (412.5,73.5) .. controls (412.5,83.72) and (402.31,92) .. (389.75,92) .. controls (377.19,92) and (367,83.72) .. (367,73.5) -- cycle ;
\draw   (434,74.5) .. controls (434,64.28) and (444.19,56) .. (456.75,56) .. controls (469.31,56) and (479.5,64.28) .. (479.5,74.5) .. controls (479.5,84.72) and (469.31,93) .. (456.75,93) .. controls (444.19,93) and (434,84.72) .. (434,74.5) -- cycle ;
\draw   (569,77.5) .. controls (569,67.28) and (579.19,59) .. (591.75,59) .. controls (604.31,59) and (614.5,67.28) .. (614.5,77.5) .. controls (614.5,87.72) and (604.31,96) .. (591.75,96) .. controls (579.19,96) and (569,87.72) .. (569,77.5) -- cycle ;
\draw   (504,75.5) .. controls (504,65.28) and (514.19,57) .. (526.75,57) .. controls (539.31,57) and (549.5,65.28) .. (549.5,75.5) .. controls (549.5,85.72) and (539.31,94) .. (526.75,94) .. controls (514.19,94) and (504,85.72) .. (504,75.5) -- cycle ;
\draw    (412.5,73.5) -- (432,74.41) ;
\draw [shift={(434,74.5)}, rotate = 182.66] [color={rgb, 255:red, 0; green, 0; blue, 0 }  ][line width=0.75]    (10.93,-3.29) .. controls (6.95,-1.4) and (3.31,-0.3) .. (0,0) .. controls (3.31,0.3) and (6.95,1.4) .. (10.93,3.29)   ;
\draw    (479.5,74.5) -- (501,75.41) ;
\draw [shift={(503,75.5)}, rotate = 182.44] [color={rgb, 255:red, 0; green, 0; blue, 0 }  ][line width=0.75]    (10.93,-3.29) .. controls (6.95,-1.4) and (3.31,-0.3) .. (0,0) .. controls (3.31,0.3) and (6.95,1.4) .. (10.93,3.29)   ;
\draw    (549.5,75.5) -- (569,76.41) ;
\draw [shift={(571,76.5)}, rotate = 182.66] [color={rgb, 255:red, 0; green, 0; blue, 0 }  ][line width=0.75]    (10.93,-3.29) .. controls (6.95,-1.4) and (3.31,-0.3) .. (0,0) .. controls (3.31,0.3) and (6.95,1.4) .. (10.93,3.29)   ;
\draw    (345.5,72.5) -- (365,73.41) ;
\draw [shift={(367,73.5)}, rotate = 182.66] [color={rgb, 255:red, 0; green, 0; blue, 0 }  ][line width=0.75]    (10.93,-3.29) .. controls (6.95,-1.4) and (3.31,-0.3) .. (0,0) .. controls (3.31,0.3) and (6.95,1.4) .. (10.93,3.29)   ;
\draw    (614.5,77.5) -- (634,78.41) ;
\draw [shift={(636,78.5)}, rotate = 182.66] [color={rgb, 255:red, 0; green, 0; blue, 0 }  ][line width=0.75]    (10.93,-3.29) .. controls (6.95,-1.4) and (3.31,-0.3) .. (0,0) .. controls (3.31,0.3) and (6.95,1.4) .. (10.93,3.29)   ;

\draw (24,79.4) node [anchor=north west][inner sep=0.75pt]    {$\overline{o}$};
\draw (117,86.4) node [anchor=north west][inner sep=0.75pt]    {$\overline{u}_{0}$};
\draw (250,28.4) node [anchor=north west][inner sep=0.75pt]    {$\xi =\alpha ^{+}$};
\draw (233,128.4) node [anchor=north west][inner sep=0.75pt]    {$\eta =\beta ^{+}$};
\draw (185,112.4) node [anchor=north west][inner sep=0.75pt]    {$\overline{v}$};
\draw (224,10.4) node [anchor=north west][inner sep=0.75pt]    {$\overline{u}$};
\draw (151,38.4) node [anchor=north west][inner sep=0.75pt]    {$\overline{w}$};
\draw (379,64.4) node [anchor=north west][inner sep=0.75pt]    {$U_{0}$};
\draw (448,66.4) node [anchor=north west][inner sep=0.75pt]    {$W$};
\draw (520,67.4) node [anchor=north west][inner sep=0.75pt]    {$V$};
\draw (581,68.4) node [anchor=north west][inner sep=0.75pt]    {$W$};

\end{tikzpicture}
    \caption{Injectivity of $\Phi$: standard rays $\alpha,\beta$ in $\p_K(\f)$ (left); admissible path $\gamma$ in $\U$ (right) }
    \label{fig:injectivity}
\end{figure}

\textbf{Defining the map $\Phi$.}  For given $F$, we can choose $K$ and thus $L$ large enough, so that an $(L,B,\f)$-conical point is an $(r,F)$-conical point by Lemma \ref{QuantativeEquiv23}. Therefore, $\tilde \alpha$ determines a unique $(r, F)$-conical point $[\gamma^+]\in [\pU]$.  Setting  $\Phi(\xi):=[\tilde \alpha^+]$ defines the desired map $$\Phi: \partial \PC\to [\Lambda_r^F(Go)]$$ 

\textbf{The map $\Phi$ is injective.} Indeed, if $\Phi(\xi)=\Phi(\eta)$,  assume to the contrary that the corresponding standard rays $\alpha, \beta$ ending at $\xi\ne \eta$ coincide until a  vertex $\bar x_0$.  

By the tripod triangle property in Lemma \ref{Tripod}, any $\bar u\in \alpha$ and $\bar w, \bar v\in \beta$ far from the departure  $\bar x_0$ is contained on a standard path, and if they appear  in this order, then   Lemma \ref{OrderLem} implies $d_W(U, V)>K$, where $U,V,W$ are associated vertex spaces. See Fig. \ref{fig:injectivity} for a schematic illustration. 

Let $\tilde \alpha$ (resp. $\tilde \beta$) be the lifted path of $\alpha$ (resp. $\beta$) described in Lemma \ref{LiftPathLem}. 
From the construction, each  vertex space $U$ associated to    a vertex $\bar u$ on $\alpha$ (resp. $\beta$) is contained in the saturation of the  $(L, B)$-admissible path $\tilde \alpha$ (resp. $\tilde \beta$). 

As $\tilde \alpha,\tilde \beta$  terminate at the same conical point $\Phi(\xi)=\Phi(\eta)$, there exists a geodesic ray $\gamma$ by Lemma \ref{CharConicalLem} which  fellow travels both $\tilde \alpha$ and $\tilde \beta$ by  Proposition \ref{admisProp}. In particular, if $W, U, V$ appear in this order on $\gamma$, then $d_W(U, V)\le  B$. As  $\bar u\in \alpha, \bar v\in \beta$  can be chosen arbitrarily far from $\bar w$, such an order exists. Thus, for $K\gg B$, we obtain a contradiction with the above inequality $d_W(U, V)>K$, whence the injectivity of $\Phi$ follows.

\textbf{The continuity of $\Phi$}. Let $\xi_n\to\xi$ in the Gromov boundary. Our goal is to prove that  $\Phi(\xi_n)$ converges to $[\Phi(\xi)]$. 

By Lemma \ref{standardray},   there exist a standard ray $\alpha_n$ ending at $\xi_n$, and its lift $(L, B)$-admissible path $\tilde \alpha_n$ in $\U$ ending at $\Phi(\xi_n)$.  Moreover, $\alpha_n$ converges in the above sense to a standard ray $\alpha$ ending at $\xi$, and the lift $\tilde \alpha$ is an $(L, B)$-admissible path, ending at $[\Phi(\xi)]$. 

If $\bar x_n$ is the last vertex  of the overlap $\alpha_n\cap \alpha$, the corresponding vertex space  $X_n\in \f$  appears in  the saturation of the $(L, B)$-admissible path $\tilde \alpha$.  As $\bar x_n\to \xi$, we have  $d(\bar o,\bar x_n)\to\infty$. The map $X_n\to \bar x_n$ is Lipschitz, so we obtain $d(o,X_n)\to\infty$. By Lemma \ref{FreqContrLimits}, the escaping sequence of $X_n$ accumulates into $[\Phi(\xi)]$. 

Let $v_n\in \U$ be the exit point of $\tilde \alpha_n$ at $N_r(X_n)$, where  $r\ge 0$ is given by Proposition \ref{admisProp}. Look at the sub-ray $[v_n, \Phi(\xi_n)]_{\tilde \alpha_n}$ of $\tilde \alpha_n$ starting from $v_n$. For any  $z_n\in [v_n, \Phi(\xi_n)]_{\tilde \alpha_n}$ sufficiently far from $v_n$,     we have $\pi_{X_n}(z_n)$ is $(r+C)$-close to $v_n$ by Lemma \ref{BigFive}.  So the triangle inequality for $d_{X_n}$ gives
$$
d_{X_n}(o, z_n) \ge d_{X_n}(o, v_n) -r-C
$$
As $d_{X_n}(o, v_n)\ge \|\gamma_n\cap N_r(X_n)\|-4r\ge L-4r$ by Lemma \ref{Transform}, we have 
$$
d_{X_n}(o, z_n) \ge      L-5r-C>10C 
$$
implying $z_n\in \Omega_o(N_C(X_n), 10C)$, where $L>5r+11C$ is assumed. As $X_n$ converges to $[\Phi(\xi)]$, \ref{AssumpB} implies that $z_n$ converges to $\Phi(\xi)$.  

{Fix a metric $\delta$ on the metrizable boundary $\pU$, and choose $z_n$ so that $\delta(z_n, \Phi(\xi_n))\le 1/n$. Thus, $\Phi(\xi_n)$ converges to $[\Phi(\xi)]$.}  

\begin{figure}
    \centering

\tikzset{every picture/.style={line width=0.75pt}} 

\begin{tikzpicture}[x=0.75pt,y=0.75pt,yscale=-1,xscale=1]

\draw    (72.5,131) -- (147.5,110) ;
\draw [shift={(147.5,110)}, rotate = 344.36] [color={rgb, 255:red, 0; green, 0; blue, 0 }  ][fill={rgb, 255:red, 0; green, 0; blue, 0 }  ][line width=0.75]      (0, 0) circle [x radius= 3.35, y radius= 3.35]   ;
\draw [shift={(72.5,131)}, rotate = 344.36] [color={rgb, 255:red, 0; green, 0; blue, 0 }  ][fill={rgb, 255:red, 0; green, 0; blue, 0 }  ][line width=0.75]      (0, 0) circle [x radius= 3.35, y radius= 3.35]   ;
\draw   (351.35,113.39) .. controls (351.35,102.13) and (369.52,93) .. (391.93,93) .. controls (414.33,93) and (432.5,102.13) .. (432.5,113.39) .. controls (432.5,124.65) and (414.33,133.78) .. (391.93,133.78) .. controls (369.52,133.78) and (351.35,124.65) .. (351.35,113.39) -- cycle ;
\draw    (421.5,98) -- (545.57,64.52) ;
\draw [shift={(547.5,64)}, rotate = 164.9] [color={rgb, 255:red, 0; green, 0; blue, 0 }  ][line width=0.75]    (10.93,-3.29) .. controls (6.95,-1.4) and (3.31,-0.3) .. (0,0) .. controls (3.31,0.3) and (6.95,1.4) .. (10.93,3.29)   ;
\draw [shift={(421.5,98)}, rotate = 344.9] [color={rgb, 255:red, 0; green, 0; blue, 0 }  ][fill={rgb, 255:red, 0; green, 0; blue, 0 }  ][line width=0.75]      (0, 0) circle [x radius= 3.35, y radius= 3.35]   ;
\draw    (320.5,112.29) -- (349.35,113.32) ;
\draw [shift={(351.35,113.39)}, rotate = 182.05] [color={rgb, 255:red, 0; green, 0; blue, 0 }  ][line width=0.75]    (10.93,-3.29) .. controls (6.95,-1.4) and (3.31,-0.3) .. (0,0) .. controls (3.31,0.3) and (6.95,1.4) .. (10.93,3.29)   ;
\draw [shift={(320.5,112.29)}, rotate = 2.05] [color={rgb, 255:red, 0; green, 0; blue, 0 }  ][fill={rgb, 255:red, 0; green, 0; blue, 0 }  ][line width=0.75]      (0, 0) circle [x radius= 3.35, y radius= 3.35]   ;
\draw    (147.5,110) -- (279.6,154.36) ;
\draw [shift={(281.5,155)}, rotate = 198.56] [color={rgb, 255:red, 0; green, 0; blue, 0 }  ][line width=0.75]    (10.93,-3.29) .. controls (6.95,-1.4) and (3.31,-0.3) .. (0,0) .. controls (3.31,0.3) and (6.95,1.4) .. (10.93,3.29)   ;
\draw    (147.5,110) -- (273.57,75.53) ;
\draw [shift={(275.5,75)}, rotate = 164.71] [color={rgb, 255:red, 0; green, 0; blue, 0 }  ][line width=0.75]    (10.93,-3.29) .. controls (6.95,-1.4) and (3.31,-0.3) .. (0,0) .. controls (3.31,0.3) and (6.95,1.4) .. (10.93,3.29)   ;
\draw    (477.5,136) -- (555.53,150.63) ;
\draw [shift={(557.5,151)}, rotate = 190.62] [color={rgb, 255:red, 0; green, 0; blue, 0 }  ][line width=0.75]    (10.93,-3.29) .. controls (6.95,-1.4) and (3.31,-0.3) .. (0,0) .. controls (3.31,0.3) and (6.95,1.4) .. (10.93,3.29)   ;
\draw [shift={(477.5,136)}, rotate = 10.62] [color={rgb, 255:red, 0; green, 0; blue, 0 }  ][fill={rgb, 255:red, 0; green, 0; blue, 0 }  ][line width=0.75]      (0, 0) circle [x radius= 3.35, y radius= 3.35]   ;
\draw    (426.5,125) -- (477.5,136) ;
\draw [shift={(426.5,125)}, rotate = 12.17] [color={rgb, 255:red, 0; green, 0; blue, 0 }  ][fill={rgb, 255:red, 0; green, 0; blue, 0 }  ][line width=0.75]      (0, 0) circle [x radius= 3.35, y radius= 3.35]   ;

\draw (64,111.4) node [anchor=north west][inner sep=0.75pt]    {$\overline{o}$};
\draw (135,119.4) node [anchor=north west][inner sep=0.75pt]    {$\overline{x}_{n}$};
\draw (217,54.4) node [anchor=north west][inner sep=0.75pt]    {$\xi =\alpha ^{+}$};
\draw (194,151.4) node [anchor=north west][inner sep=0.75pt]    {$\xi_{n} =\alpha _{n}^{+}$};
\draw (363.36,104.26) node [anchor=north west][inner sep=0.75pt]    {$ \begin{array}{l}
N_{r}(X_{n})\\
\end{array}$};
\draw (506,156.4) node [anchor=north west][inner sep=0.75pt]    {$\Phi ( \xi _{n})$};
\draw (470,51.4) node [anchor=north west][inner sep=0.75pt]    {$\tilde{\alpha }$};
\draw (430,143.4) node [anchor=north west][inner sep=0.75pt]    {$\ z_{n} \in \ \widetilde{\alpha _{n}}$};
\draw (518,76.4) node [anchor=north west][inner sep=0.75pt]    {$\Phi ( \xi )$};
\draw (311,118.4) node [anchor=north west][inner sep=0.75pt]    {$o$};
\draw (406,74.4) node [anchor=north west][inner sep=0.75pt]    {$v_{n}$};

\end{tikzpicture}
    \caption{Continuity of $\Phi$: Standard rays in $\p_K(\f)$ (left); Admissible paths in $\U$ (right)}
    \label{fig:continuity}
\end{figure}
\textbf{The continuity of $\Phi^{-1}$}. Conversely, let $[p_n]\to [p]$ in the convergence boundary $[\pU]$, where $[p_n],[p]\in [\Lambda_r^F(Go)]$.    Let  $\alpha_n,\alpha$ be  the corresponding  standard rays ending at $\xi_n:=\Phi^{-1}([p_n])$ and $ \xi :=\Phi^{-1}([p])$ by Lemma \ref{standardray}. We assume that  $\alpha_n,\alpha$  depart at the vertex $\bar x_n$. To prove $\xi_n\to \xi$, it suffices to show that   $d(\bar o, \bar x_n)$  is unbounded. Let us assume  to the contrary that $d(\bar o, \bar x_n)$ is bounded for infinitely many $\alpha_n$. 

Let us fix any $\bar u_n\in \alpha_n$ and $\bar w, \bar v_n\in \alpha$ far from $\bar x_n$. By Lemma \ref{Tripod}, $\bar u_n, \bar w,\bar v_n$   are necessarily contained in a standard path (say, in this order), so $d_W(U_n,V_n)>K$. 
Recall that the saturation of $\tilde \alpha_n$ (resp. $\tilde \alpha$) contains the vertex spaces associated to the vertices on $\alpha_n$ (resp. $\alpha$). 

By definition of $\Phi$, the admissible paths $\tilde \alpha_n$ and $\tilde \alpha$  terminate at $[p_n], [p]$ respectively, so by Lemma \ref{CharConicalLem}, $U_n$ (resp.  $V_n$) accumulates into $[p_n]$ (resp. $[p]$). However, as $d_W(U_n,V_n)>K$, the contracting property implies that any geodesic from $u_n\in U_n$ to $v_n\in V_n$ intersects a fixed ball. Since $u_n\to [p], v_n\to [p]$, we obtain a contradiction with $p\in \mathcal C$. Hence, the continuity of $\Phi^{-1}$ follows, and the proof of the lemma is complete.
\end{proof}

\subsection{Gromov boundary of quasi-tree of spaces}\label{SSecGromovBdryQT}

Starting from a family $\f$ of quasi-lines as in (\ref{SystemFDefIntro}), the construction of the projection complex $\p_K(\f)$ records no information from those lines in $\f$. The  quasi-tree of spaces $\QT$ is a blown-up version of $\p_K(\f)$, so that it     contains isometrically embedded   copies of    quasi-lines in $\f$. {It is exercise to verify that the Gromov boundary $\partial \QT$ as a set is surjective to the union of $\partial \PC$ and the vertex set $\f$:
$$
\partial \QT \longrightarrow (\partial \PC) \cup \f
$$
though this map is not necessarily continuous with respect to the obvious topology on the target. Note that the non-injectivity only comes  from   collapsing    the two endpoints  of each quasi-line  $X \in \f$ to be one corresponding vertex $X$ in $\PC$. }

\begin{cor}\label{SurjectGromov}
Under the assumption of Lemma \ref{EmbedGromov}, there is a continuous embedding $\Phi$ from the Gromov boundary $\partial \QT$ into $[\Lambda_r^F(Go)]$.
\end{cor}
\begin{proof}
By the construction of $\QT$, there are two types of sequences $ u_n\in \QT$ tending to Gromov boundary. Either the corresponding vertices $\bar u_n$ are unbounded in the original projection complex $\PC$ or $u_n$ are eventually contained in one  vertex space, i.e.: an axis $U$ in $\f$. The former case has been dealt in Lemma \ref{EmbedGromov}. In the latter case,  $U$ is a   axis of non-pinched contracting element  in $\U$ with distinct endpoints $[U^-]\ne [U^+]$.  Let $\Phi(\xi)$   be  the corresponding $[U^-]$ or $[U^+]$ to which  $u_n$ converges.  

The continuity follows a similar argument as in Lemma \ref{EmbedGromov}, and is left to the interested reader.
\end{proof}

The map $\partial \QT\to [\Lambda_r^F(Go)]$  in Corollary \ref{SurjectGromov}  admits the following inverse in a qualitative way.  Compare  the  order of quantifiers in the two statements. 

\begin{lem}\label{EmbedConical}
For any $K\gg 0$ in Lemma \ref{OrderLem}, if  $\|Fo\|_{\min}\gg K$, then the set   $[\Lambda_r^F(Go)]$ of $(r, F)$-conical points admits a continuous embedding into the Gromov boundary $\partial\QT$.  
\end{lem}
\begin{proof}
Let $[\xi]$ be the $[\cdot]$-class of an $(r, F)$-conical point. By Lemma \ref{QuantativeEquiv23}, there exists an $(L,B)$-admissible ray $\gamma$ ending at $[\xi]$, where $L$ is comparable with $\|Fo\|_{\min}$ and $B$ depends on $\f$. Let $\f(\gamma)$ be the collection of contracting subsets, which is a subset of $\f$, associated to the admissible path $\gamma$ (cf. Definition \ref{AdmDef}). 

\textbf{Defining the desired map} $\Psi: [\Lambda_r^F(Go)]\to \partial \QT$. The discussion is divided into the following two cases.

If $\gamma$ is eventually contained in the $r$-neighborhood of a quasi-line $U\in \f$, then the corresponding half-ray of $U$ defines a Gromov boundary point of $\QT$. Here, $U$ is   isometrically embedded convex into $\QT$.  

We are thus left   consider the case where the saturation $\f(\gamma)$  is infinite. That is, $\gamma$ passes through the $r$-neighborhood of  infinitely many distinct axes in $\f(\gamma)$.

If $\|Fo\|_{\min}\gg K$, then $L$ is large enough so that for any $U,W,V\in  \f(\gamma)$  in this order, we have $d_W(U,V)\ge K$ by Proposition \ref{admisProp}. In different words, any subset of consecutive elements in $\f(\gamma)$ forms a standard path in $\PC$.

Applying   a limiting argument as in Lemma \ref{standardray} to the sequence of standard paths between two vertices in $\f(\gamma)$, we obtain a standard ray $\alpha$ in $\PC$ which contains   $\f(\gamma)$  as a (possibly proper) subset of vertices. 

Now, the standard ray $\alpha$ in $\PC$ gives naturally    a  standard ray $\beta$ in $\QT$ by \cite[Def. 4.3]{BBF}: $\beta$ is obtained from $\alpha$ by replacing  vertices $\bar u$ of $\alpha$ in $\PC$ with appropriate geodesic segments in the associated axes $U$ (cf. Theorem \ref{quasitreeThm}). As $\alpha$ is a quasi-geodesic by \cite[Cor. 3.7]{BBFS},    so  $\beta$ is a quasi-geodesic in $\QT$. Thus,   we see that $\beta$   determines a  boundary point $p$ in $\partial \QT$. We thus obtain  the map $\Psi: [\Lambda_r^F(Go)]\to \partial \QT$ by defining $\Psi([\xi])=p$. 

\textbf{The map $\Psi$ is injective.}  Indeed, assume to the contrary that $\Psi([\xi])=\Psi([\eta])=p$ for two distinct $[\xi]\ne [\eta]\in [\Lambda_r^F(Go)]$.  

Let $\gamma, \gamma'$ be two $(L,B)$-admissible rays ending at $[\xi]$ and $[\eta]$ respectively. If $\gamma$ or $\gamma'$ are eventually contained in the finite neighborhood of $X\in \f(\gamma)$ or of $X'\in \f(\gamma')$, then it is easy to arrive a contradiction, as $X\ne X'$ have disjoint limit sets in $\QT$. Hence, let us focus on the generic case that  $\f(\gamma)$ and $\f(\gamma')$ are both infinite.      

Let  $\alpha$  be the standard ray ending at $p$. 
According to the above construction of $\Psi$, as $\Psi([\xi])=\Psi([\eta])=p$, we have that $\f(\gamma)$ and $\f(\gamma')$ are  contained in the vertex set of  $\alpha$.     On the one hand, the coarsely Lipschitz map $\widetilde X \to \PC$ in (\ref{LipProjectionMap}) sends $\f(\gamma)\cup \f(\gamma')$ into the set of vertices on $\alpha$. On the other hand, by Lemma \ref{VisualPointsLem}, there exists a bi-infinite geodesic between $[\xi]$ and $[\eta]$, which intersects each quasi-line in $\f(\gamma)\cup \f(\gamma')$.  Hence,   we would obtain a contradiction, and  the injectivity of $\Phi$ follows. 

The continuity of $\Psi$ (resp. $\Psi^{-1}$) follows by repeating the  same arguments as in Lemma \ref{EmbedGromov} for the continuity of $\Phi^{-1}$ (resp. $\Phi$). Hence,  $\Psi$ is a topological embedding.    The proof is complete.     
\end{proof} 

As a corollary, we have
\begin{cor}\label{VisualMetriconMyrberg}
The quotient $[\mG]$ for $\mG\subseteq \pU$ is a Hausdorff, second countable, metrizable topological space.     
\end{cor}
If $\pU$ is the horofunction boundary we prove in Lemma \ref{MyrbergGood}, via different means, that  $[\mG]$ is a Hausdorff and second countable topological space. 
\begin{proof} 
Via the   embedding in Lemma \ref{EmbedConical}, $[\mG]$ admits an topological embedding into the Gromov boundary $\partial \QT$. Moreover, as fixed points of a contracting element cannot be Myrberg point (Lemma \ref{ContrMyrberg}), the image of $[\mG]$ is disjoint with the endpoints of any quasi-line in $\f$.

Recall that a subset of a separable metric space is separable, and a metric space is separable if and only if it is second countable. It suffices to prove that the Gromov boundary $\partial \QT$  is a complete separable space with respect to the visual metric. This actually holds for hyperbolic graphs with countably many vertices.  Indeed, as $\QT$ has countably many vertices, the shadows $\Pi_o(v,r)$ at vertices $v\in \QT$ form a countable base for the Gromov boundary. To see the completeness, we make use of the standard rays in $\QT$ for simplicity; in general, one could appeal to the construction of quasi-geodesic rays (\cite{BriHae}). 

Every boundary point $\xi$ is represented by a standard  ray $\alpha$ by Lemma \ref{standardray}. If $\xi_n$ is a Cauchy sequence, then the almost tripod property of standard paths implies that the associated ray $\alpha_n$ converges to a standard ray. This shows the completeness of $\partial \QT$. 
\end{proof}

\section{Horofunction boundary as convergence boundary}\label{SecHorobdry}
In this section, we verify that the horofunction boundary with finite difference relation  is a convergence boundary, i.e. satisfying \ref{AssumpA},\ref{AssumpB} and \ref{AssumpC}.



Let us first recall the definition of limit sets with respect to the horofunction boundary. The \textit{limit set} $\pG $ of a group $G$ acting isometrically on $\U$ is defined to be the topological closure of a $G$-orbit $Go$ in $\hU$. We have $[\pG ]=[\Lambda(Go')]$ by Lemma \ref{bddLocusLem} (note that $G$ is not necessarily discrete). We sometimes denote $[\Lambda(G)]=[\pG]$ for simplicity.

\subsection{Verifying Assumptions}\label{SSecAssump}
This subsection is devoted to the proof of Theorem \ref{ContractiveThm1}.
We first  verify \ref{AssumpA} for the horofunction boundary.

\begin{lem}\label{ZLocusLem}
Let $X\subseteq \U$ be a $C$-contracting subset for $C\ge 0$. If $y_n \in \U\to \eta\in \hU$ and $x_n\to\xi\in\hU$ for a sequence $x_n\in \pi_X(y_n)$, then  $$\|b_\xi -b_\eta\|_\infty \le 4C$$ 
\end{lem}
 
\begin{proof}
The proof boils down to show that the difference $$\lim\limits_{n\to\infty} [b_{x_n}(z)-b_{y_n}(z)]=\lim\limits_{n\to\infty} [d(z,{x_n})-d(o,x_n) +  d(o, y_n)-d(z, y_n)]$$ is uniformly bounded over any $z\in \U$.   As $x_n\in \pi_X(y_n)$  exits every compact set, the $C$-contracting property of $X$ thus implies  for all $n\gg 0$:  $$ d(x_n, [y_n, o])\le C, \quad d(x_n, [y_n, z])\le C.$$
Combining  these estimates, we get   for $n\gg 0$,
$$ 
\begin{array}{lll}
|d(z,{x_n})-d(o,x_n) +  d(o, y_n)-d(z, y_n) | \le 4C
\end{array}
$$
completing the proof.
\end{proof}
The following corollary is implicit in the proof.
\begin{cor}\label{ContractUnifDiff}
Let $\gamma$ be  a $C$-contracting geodesic ray for some $C\ge 0$. For any $ \xi,\eta\in[\gamma^+]$, we have $$\|b_\xi -b_\eta\|_\infty \le 4C$$
\end{cor}
\begin{proof}
Let $y_n\in \U\to \zeta\in [\gamma^+]$. By  Lemma \ref{ZLocusLem}, it suffices to prove  that $\pi_\gamma(y_n)$ is an escaping sequence. Indeed, if not, $\{\pi_\gamma(y_n):n\ge 1\}$ is contained in an $R$-ball  at $\gamma^-$ for $R>0$. By $C$-contracting property,   as $z\in \gamma$ tends to infinity, the geodesic $[y_n,z]$ intersects $B(\gamma^-, R+C)$. Estimating the value of $b_{y_n}(\cdot)$ at $z\in\gamma$ shows   $b_{y_n}(z)\to +\infty$. However, the horofunction $b_\gamma$ associated to $\gamma$ takes negative values as  $z\in \gamma$ tends to infinity. This contradiction  completes the proof. 
\end{proof}

We now associate the fixed points to a contracting element. Let $h$ be a contracting element on $\U$, so that  $\langle h\rangle o$ is a contracting quasi-geodesic by definition.
The repelling and attracting fixed points $[h^-], [h^+]$  in $\hU$  do  not depend on $o\in \U$. Indeed,  by Corollary \ref{ContractUnifDiff}, any two horofunctions in $[h^+]$ (or $[h^-]$)   have a finite  difference depending on the contracting constant of $\langle h\rangle o$.    

We now prove \ref{AssumpB} and \ref{AssumpB'} for the horofunction boundary.   
\begin{lem}\label{UnifDiffLem}
Let $X_n\subseteq \U$ be an escaping sequence of $C$-contracting subsets for some $C\ge 0$. Fix a     basepoint $o\in \U$. Let $x_n\in \pi_{X_n}(o)$  be a sequence of points such that $x_n\to\xi\in \hU$.  
Consider another sequence of points   $y_n\in \U\to \eta\in \hU$   so that either $y_n\in X_n$ or $\|N_C(X_n)\cap [o, y_n]\|>6C$ for all $n\ge 1$. Then   $\|b_\xi -b_\eta\|_\infty \le 6C$.

\end{lem}
\begin{proof}
Given  $z\in \U$, we are going to give a uniform bound on $|b_\xi(z) -b_\eta(z)|$. To this end,  choose a projection point $z_n\in \pi_{X_n}(z)$. 
As  $d(o, X_n)\to \infty$, the ball centered at $o$ of radius  $d(o,z)$ misses $X_n$, so    the $C$-contracting property by Lemma \ref{BigThree}  implies    $$d(z_n, x_n)\le d_{X_n}(z,o)\le C$$
See Fig. \ref{fig:assumpBHorofunction}. If $y_n\in X_n$, the contracting property gives $d(z_n,[z,y_n])\le 2C$ and then $d(x_n,[z,y_n])\le 3C$. If $\|N_C(X_n)\cap [o, y_n]\|>6C$, Lemma \ref{BigFive}  implies   $d_{X_n}(o,y_n)>2C$, yielding  $$d_{X_n}(z,y_n)\ge d_{X_n}(o,y_n) - d_{X_n}(z,o)  \ge C$$  
by triangle inequality of $d_{X_n}$.
\begin{figure}
    \centering

\tikzset{every picture/.style={line width=0.75pt}} 

\begin{tikzpicture}[x=0.75pt,y=0.75pt,yscale=-1,xscale=1]

\draw    (55.5,110) .. controls (110.5,109) and (269.5,129) .. (312.5,167) ;
\draw [shift={(312.5,167)}, rotate = 41.47] [color={rgb, 255:red, 0; green, 0; blue, 0 }  ][fill={rgb, 255:red, 0; green, 0; blue, 0 }  ][line width=0.75]      (0, 0) circle [x radius= 3.35, y radius= 3.35]   ;
\draw [shift={(55.5,110)}, rotate = 358.96] [color={rgb, 255:red, 0; green, 0; blue, 0 }  ][fill={rgb, 255:red, 0; green, 0; blue, 0 }  ][line width=0.75]      (0, 0) circle [x radius= 3.35, y radius= 3.35]   ;
\draw   (168,92) .. controls (188.5,82) and (256,107) .. (280.5,135) .. controls (305,163) and (188,182) .. (168,152) .. controls (148,122) and (147.5,102) .. (168,92) -- cycle ;
\draw    (56,159) .. controls (90.5,129) and (162.5,109) .. (312.5,167) ;
\draw [shift={(56,159)}, rotate = 318.99] [color={rgb, 255:red, 0; green, 0; blue, 0 }  ][fill={rgb, 255:red, 0; green, 0; blue, 0 }  ][line width=0.75]      (0, 0) circle [x radius= 3.35, y radius= 3.35]   ;
\draw  [line width=6] [line join = round][line cap = round] (153.8,117.22) .. controls (153.8,117.22) and (153.8,117.22) .. (153.8,117.22) ;
\draw  [line width=6] [line join = round][line cap = round] (155.8,130.22) .. controls (155.8,130.22) and (155.8,130.22) .. (155.8,130.22) ;
\draw  [line width=6] [line join = round][line cap = round] (275.8,154.22) .. controls (275.8,154.22) and (275.8,154.22) .. (275.8,154.22) ;
\draw  [line width=0.75]  (152.5,162) .. controls (151.72,166.6) and (153.63,169.29) .. (158.23,170.07) -- (201.76,177.45) .. controls (208.33,178.56) and (211.23,181.42) .. (210.45,186.02) .. controls (211.23,181.42) and (214.91,179.68) .. (221.48,180.79)(218.53,180.29) -- (262.43,187.73) .. controls (267.03,188.51) and (269.72,186.6) .. (270.5,182) ;
\draw  [dash pattern={on 4.5pt off 4.5pt}] (8.75,111) .. controls (8.75,84.49) and (30.24,63) .. (56.75,63) .. controls (83.26,63) and (104.75,84.49) .. (104.75,111) .. controls (104.75,137.51) and (83.26,159) .. (56.75,159) .. controls (30.24,159) and (8.75,137.51) .. (8.75,111) -- cycle ;

\draw (50,83.4) node [anchor=north west][inner sep=0.75pt]    {$o$};
\draw (65,157.4) node [anchor=north west][inner sep=0.75pt]    {$z$};
\draw (323,151.4) node [anchor=north west][inner sep=0.75pt]    {$y_{n}$};
\draw (201,103.4) node [anchor=north west][inner sep=0.75pt]    {$X_{n}$};
\draw (131,92.4) node [anchor=north west][inner sep=0.75pt]    {$x_{n}$};
\draw (139,130.4) node [anchor=north west][inner sep=0.75pt]    {$z_{n}$};
\draw (188,188.4) node [anchor=north west][inner sep=0.75pt]    {$\geq 6C$};
\draw (119,115.4) node [anchor=north west][inner sep=0.75pt]    {$C\geq $};

\end{tikzpicture}
    \caption{Schematic configuration in verifying Assumption B}
    \label{fig:assumpBHorofunction}
\end{figure}
Using again the contracting property in Lemma \ref{BigFive},    $$d(z_n, [z,y_n])\le 2C$$ which implies also $d(x_n,[z,y_n])\le 3C$. In both cases,   a straightforward computation  gives that $$|b_{x_n}(z) - b_{y_n}(z)|\le |d(z,x_n)-d(z,y_n)|\le 2d(x_n,[z,y_n]) \le 6C.$$ Letting $x_n\to \xi$ and $y_n\to \eta$ completes the proof.
\end{proof}

If $\U$ is  CAT(0) or hyperbolic,  the assumption ``$\|N_C(X_n)\cap [o, y_n]\|>6C$" could be relaxed  as follows.
\begin{lem}\label{UnifDiffLemCAT0}
Under the assumption of Lemma \ref{UnifDiffLem}, assume that     $y_n\in \U\to \eta\in \hU$   so that $X_n\cap [o, y_n]\ne\emptyset$ for each $n\ge 1$. If  $\U$ is  CAT(0) or hyperbolic, then   $\|b_\xi -b_\eta\|_\infty \le 6C$.
\end{lem}
\begin{proof}
We resume the notation in the proof of Lemma \ref{UnifDiffLem}.

As $x_n\in X_n$ are the shortest projection points of $o$,  the $C$-contracting property of $X_n$ shows that $d(x_n,[o,y_n])\le 2C$. Consider the triangle   $\Delta(o,z,y_n)$. Note that $d(o,z)$ is fixed and $x_n$ is $2C$-close to the side $[z,y_n]$. If $\U$ is CAT(0) or $\delta$-hyperbolic,  it follows that $d(x_n,[z,y_n])<2C+\delta$: in the former case, we use the Euclidean comparison triangle and in the latter, the $\delta$-thin triangle property.

The same proof as in Lemma \ref{UnifDiffLem} then proves the conclusion with   $|b_{x_n}(z) - b_{y_n}(z)|\le 2C+\delta$ over $z\in \U$.
\end{proof}

We prove now that all boundary points in the  horofunction boundary $\hU$ are non-pinched; that is, $\mathcal C=\hU$ in \ref{AssumpC}. As a corollary, every contracting elements are non-pinched in the horofunction boundary.  
\begin{lem}\label{NoPinchedLineLem}
For any $\xi\in \hU$, there exist no sequences $x_n, y_n\in \U$ such that $x_n,y_n\to[\xi]$ and $[x_n,y_n]$ intersects  a fixed ball for any $n\ge 1$.
\end{lem}
\begin{proof}
For concreteness, assume that $x_n\to\xi, y_n\to\eta$ such that $[\xi]=[\eta]$. The finite difference relation implies that there exists $K\ge 10$ such that $|B_\xi(z,w)-B_\eta(z,w)|\le K$ for any $z,w\in \U$.
Arguing by contradiction, as $\U$ is proper, assume that  $[x_n,y_n]$ converges locally uniformly to a bi-infinite geodesic $\alpha$. Choose $z,w\in \alpha$ with distance at least $2K$ such that $z, w\in N_1([x_n, y_n])$. For concreteness, we may assume that  $z$ is closer to $x_n$ than $w$. By direct computation, $$|B_{x_n}(z,w)+d(z,w)|\le 2, \quad |B_{y_n}(z,w)-d(z,w)|\le 2.$$ For $x_n, y_n\to \xi$, we have $B_{x_n}(z,w), B_{y_n}(z,w)\to B_\xi(z,w)$ as $n\to\infty$. With $|B_\xi(z,w)-B_\eta(z,w)|\le K$, this contradicts the assumption of $d(z,w)>2K$.
\end{proof}

\subsection{Horofunctions associated to conical points}
We are not presumably working with the horofunction boundary in this subsection. Still, the goal is to  define a horofunction for conical points in  any nontrivial convergence boundary  $\bU=\U\cup\pU$.

\begin{lem}\label{ClosedFinDiff}
Let $\xi\in \Lambda_r^F(Go)$. Consider two sequences $x_n\in \U\to \zeta\in [\xi]$,   $z_n\in \U\to \eta\in [\xi]$ in $\pU$, and simultaneously, $x_n\to b_\zeta$, $z_n\to b_\eta$ in the horofunction boundary $\hU$. Then 
$$
\|b_\zeta -b_\eta \|_\infty \le 20C.
$$ 
Moreover, the locus of an $(r, F)$-conical point $\xi$ consists of     $(\hat r, F)$-conical points. 
\end{lem}
As a consequence, we can associate to  any $\xi\in \Lambda_r^F(Go)$ a   \textit{Busemann quasi-cocycle} as follows
$$B_{\xi}(x,y)=\limsup_{y\to \xi} \big(d(x,z_n)-d(y,z_n)\big)$$
where the convergence  $z\to\xi$ takes place in $\bU=\U\cup\pU$. 

\begin{proof}
By Lemma \ref{CharConicalLem}, let $\gamma$ be a geodesic ray with infinitely many distinct $(\hat r, F)$-barriers $y_n:=g_no \in X_n\in \f$ and ending at $[\xi]$.  For given $X_n$ and $m\gg n$,  we have
\begin{equation}\label{EntryExitEq}
 d_{X_n}(o, y_m)\ge \|N_r(X_n)\cap \gamma\|-4r\ge \|Fo\|_{\min}-4r> 10C,
\end{equation} 

If $X:=X_n$ are eventually the same, we see that  $\gamma$ is (eventually) contracting (after removing a finite initial path). In this case,  the conclusion follows by Corollary \ref{ContractUnifDiff}.

We now assume that $\{X_n\}$ are escaping. Take any sequence   $z_m\to \eta\in [\gamma^+]$, where $z_m\in \U$ may  be not on $\gamma$. As $y_n\to [\xi]=[\eta]$, Lemma \ref{LocusProj} implies that  $d_{X_n}(y_m, z_m)\le 6C $ for all $m\gg 0$.   By triangle inequality with (\ref{EntryExitEq}),  we obtain         $\pi_{X_n}(z_m, o) \ge 4C$. So $[o, z_m]\cap N_C(X_n)\ne\emptyset$ by Lemma \ref{BigThree}. In other words,   $z_m\in \Omega_o(N_C(X_n))$ for   every  $m\gg n$. The same argument applied  for $x_n\to\zeta$ shows that $x_m\in \Omega_o(N_C(X_n))$. By Lemma \ref{UnifDiffLem}, we  thus see $\|b_\zeta -b_\eta \|_\infty \le 20C$.

It remains to see that $\eta$ is an $(\hat r, F)$-conical point, where  $\hat r$ is given by Lemma \ref{BarrierFellowLem}.  As  $d_{X_n}(y_m, z_m)\le 6C$, $[o,z_m]$ contains an $(\hat r, f)$-barrier $y_n=g_n o$ by Lemma \ref{BarrierFellowLem}. Thus,  $z_m\in \Omega_o^F(g_no,\hat r)$, so  $\eta$ is an $(\hat r, F)$-conical point.
\end{proof}

We close this section with the following property of Myrberg   set in the horofunction boundary.

\begin{lem}\label{MyrbergGood}
The Myrberg  set $\mG\subseteq \pU$ is saturated in any nontrivial convergence boundary $\pU$, and the finite difference relation $[\cdot]$ of the horofunction boundary restricted on $\mG\subseteq \hU$ is $K$-finite for some uniform constant $K>0$.  
\end{lem}
 
\begin{proof}
As Myrberg points are conical, the finite difference relation on $\mG$ is $K$-finite for some uniform $K$ by Lemma \ref{ClosedFinDiff}. 
By Lemma  \ref{CharMyrberg},   a point $\xi\in \mG$ is Myrberg if and only if for any non-pinched contracting element $h$ and for any $L\gg 10C$, there exists $g\in G$ such that  
$$
d_X(o, \xi)\ge L 
$$
where $X=g\ax(h)$ is $C$-contracting for some $C>0$.  It remains   to show that $\eta\in [\xi]$ is Myrberg. By Lemma \ref{LocusProj},  we have $d_X(\xi, \eta)\le 6C$, so by the triangle inequality for $d_X$, $$d_X(o, \eta)\ge L -6C.$$
As $h$ and $L$ is arbitrary, it is proved that $\eta$ is Myrberg as well. 
\end{proof}

As a corollary of Lemma \ref{GoodPartition} and Lemma \ref{VisualMetriconMyrberg}, we have
\begin{cor}\label{MyrberginHorobdry}
The quotient $[\mG]$ for $\mG\subseteq \hU$ is a Hausdorff, second countable, metrizable topological space.     
\end{cor}

\section{Conformal density on  the convergence boundary}\label{SecDensity}

This section   develops a general theory of quasi-conformal density on a convergence compactification $\bU=\U\cup \pU$, under an additional \ref{AssumpE} introduced below.  
\subsection{Quasiconformal density} 
Buseman (quasi-)cocyles are indispensable to formulate (quasi-)conformal density on $\pU$. 
As shown in Lemma \ref{ClosedFinDiff}, conical points  in any convergence boundary  can be associated    with  Buseman quasi-cocycles. Let  $\mathcal C$ be the $G$-invariant subset of non-pinched points in \ref{AssumpC}.  Our last assumption requires that Buseman quasi-cocycles extends to the set $\mathcal C$.

Given $z \in \U$, let $B_z(x, y): = d(x,
z) -d(y,z)$ for $x, y \in \U$. 
\begin{assump}\label{AssumpE}
There exists a family of Buseman quasi-cocycles based at points in $\mathcal C$
$$\big\{B_\xi: \quad \U \times \U \to \mathbb R\big\}_{\xi \in
\mathcal C}$$ so that  for any $x, y \in \U$, we have 
$$\limsup_{z\to\xi} |B_\xi(x, y)-B_z(x, y)| \le \epsilon, $$
where $\epsilon\ge 0$ is a universal constant not depending on $x,y$ and $\xi$. If $\epsilon=0$, we say that Buseman cocyles extend \textit{continuously} to $\mathcal C$.

\end{assump}
 
Let $ \mathcal M^+(\bU)$ be the set of finite positive Borel measures on $\bU:=\pU\cup\U$, on which  $G$ acts  by push-forward: $$g_\star\mu(A)=\mu(g^{-1}A)$$ 
for any Borel set $A$.

\begin{defn}\label{ConformalDensityDefn}

Let $\omega \in [0, \infty[$. Under \ref{AssumpE}, a  map $\mu: \U \to \mathcal M^+(\bU)$
$$x \longmapsto \mu_x$$ is a
\textit{$\omega$-dimensional $G$-quasi-equivariant quasi-conformal density}  if for any $g,
h \in G$ and any $x, y\in \U$, we have
\begin{equation}\label{almostInv}
\forall \xi\in \pU: \quad
\frac{dg_\star\mu_{x}}{d\mu_{x}}(\xi) \in [\frac{1}{\lambda}, \lambda],
\end{equation}
\begin{equation}\label{confDeriv}
\forall \mu_y \textrm{ a.e. } \xi\in \mathcal C: \quad
\frac{1}{\lambda} e^{-\omega B_\xi (x, y)}  \le  \frac{d\mu_{x}}{d\mu_{y}}(\xi) \le \lambda e^{-\omega B_\xi (x, y)}
\end{equation}
for a  universal constant $\lambda\ge 1$.  We normalize $\mu_o$ to be a probability measure: its mass $\|\mu_o\|=\mu_o(\bU)=1$.
\end{defn}

 If $\lambda=1$   for (\ref{almostInv}), the map $\mu: \U \to \mathcal M^+(\bU)$ is $G$-equivariant. If both $\lambda=1$, we call $\mu$ a conformal density. Note that $\mu_{gx} = g_{\star}\mu_x$; equivalently, $\mu_{gx}(gA)=\mu_x(A)$.
\begin{rem}\label{ExamplesConformalDensity}
We clarify by examples the advantage of  the conformality   (\ref{confDeriv})  restricted   on a smaller subset $\mathcal C$ of $\pU$. 
\begin{enumerate}
\item
First of all, it seems implausible  to define Buseman cocyles at pinched points, e.g.  if the pinched point is the two endpoints of a bi-infinite geodesic (i.e.: a  horocycle). For example, the  horocycles  exist in the Cayley graph of relatively hyperbolic groups, so that in \cite{YANG7}, we can only define Buseman (quasi-)cocyles at conical points in Floyd boundary. 
\item
Secondly, if the action $G\act \U$ is of divergent type,   one may prove via other means that the quasi-conformal measures are  fully supported on the set   $\mathcal C$ (in view of  Theorem \ref{HTSThm}).  

As finitely generated groups are of divergent type for the action on the Cayley graph, we    then   prove  in \cite{YANG7} the atomless of the   Patterson--Sullivan measures on Bowditch boundary via an argument in \cite{DOP}. As there are only countably  many non-conical points,  the conformality is recovered  on almost every boundary points.   

As another example, the conformal density constructed from Thurston measures in \cite{ABEM} (not using Patterson's construction) is supported on uniquely ergodic points in Thurston boundary (this follows by Masur-Veech's result \cite{Ma82a, V82} on Teichm\"uller geodesic flow). It turned out that the horofunction could be defined on those points (this is nontrivial as  Thurston boundary is not the horofunction boundary of Teichm\"uller metric), on which the conformality suffices in application.  
\end{enumerate}
\end{rem}
 
For simplicity, we  write $\mu_g=\mu_{go}$ if  the basepoint $o\in \U$ does not matter in context. In particular, $\mu_1$ denotes the measure $\mu_o$ where $1$ is the identity of $G$. 

\subsubsection*{Patterson--Sullivan measures}
Let $(\U, d)$ be a proper geodesic space. Choose a basepoint $o \in \U$. 
The Poincar\'e series for the   action of $G \act \U$
$$
\p_G(s, x,y) =\sum\limits_{g \in G} e^{-sd(x, gy)}, \; s \ge 0
$$
diverges at $0\le s< \e G$ and converges at $s>\e G$. The   action of $G$ on $\U$ is called \textit{divergent} if $\p_{G}(s,x,y)$ diverges at the critical exponent $\e G$. Otherwise, $G$ is called \textit{convergent}.

We start to construct a family of measures $\{\mu_x^{s,y}\}_{x \in \U}$ supported on $Gy$ for any given $s >\e G$ and $x, y\in \U$. Assume that $\p_G(s, x,y)$ is divergent at $s=\e G$. Set
\begin{equation}\label{PattersonEQ}
\mu_{x}^{s, y} = \frac{1}{\p_G(s, o, y)} \sum\limits_{g \in G} e^{-sd(x, gy)} \cdot \dirac{gy},
\end{equation}
where $s >\e G$. Note that $\mu^{s, y}_o$ is a probability
measure supported on $Gy$. 
If $\p_G(s, x,y)$ is convergent at $s=\e G$, the Poincar\'e series in (\ref{PattersonEQ}) needs to be replaced by a modified series as in \cite{Patt}.

Fix $y\in \U$.
Choose $s_i \to \e G$ such that $\mu_x^{s_i, y}$ are convergent in
$\mathcal M(\Lambda (Gy))$.  The \textit{Patterson--Sullivan measures} $\mu_x^y = \lim \mu_v^{s_i,y}$ are the limit
measures, which may depend on the sequence $s_i$ and $y$. Note that $\mu_o^y(\Lambda (Gy)) = 1$. In what follows, we usually write $\mu_x=\mu_x^o$ for $x\in \U$.

\begin{lem} \label{confdensity}
Let $G$ act properly on a proper geodesic space $(\U,d)$ with a contracting element, and compactified by a nontrivial convergence boundary $\pU$. Then under \ref{AssumpE}, for any fixed $y\in \U$, the family of Patterson--Sullivan measures $\{\mu_x^y\}_{x\in \U}$  on $\pU$ is a $\e G$-dimensional $G$-equivariant quasi-conformal density supported on the limit set $\Lambda (Gy)$. 

Moreover, if $\pU$ is the horofunction boundary, then it is a conformal density.  
\end{lem}
\begin{proof}
As the orbit $Gy$ is discrete in $\U$, the measure $\mu_x^y$ is supported on the closure of $Gy$.
The proof for quasi-conformal density follows the same argument in \cite{Coor}. If $\pU$ is the horofunction boundary, the Buseman cocycle extends continuously to $\hU$ so the constant $\epsilon=0$ in  \ref{AssumpE}, and we obtain a  conformal density: $\lambda=1$. See \cite{BuMo} for more details.
\end{proof}

In the sequel, we write PS-measures as shorthand for
Patterson--Sullivan measures.

\subsection{Shadow Lemma}

We   assume   the set $\mathcal C$ in \ref{AssumpC} is measurably significant: $\mu_1(\mathcal C)>0$. The goal of this subsection is to prove the shadow lemma.
 
The partial shadows $\Pi_o^F(go, r)$ and cones $\Omega_o^F(go, r)$   given in Definition \ref{ShadowDef} depend on the choice of a contracting system $\f$ as in (\ref{SystemFDefIntro}). Without index $F$,   $\Pi_o(go, r)$   denotes the     usual shadow.

\begin{lem}[Shadow Lemma]\label{ShadowLem}

Let $\{\mu_g\}_{g\in G}$ be a $\omega$-dimensional $G$-quasi-equivariant quasi-conformal density on $\pU$ for some $\omega>0$. Assume that $\mu_1(\mathcal C)>0$. Then there exists $r_0 > 0$ such that  
$$
\begin{array}{rl}
e^{-\omega \cdot d(o, go)} \; \prec_{\lambda} \; \mu_1(\Pi_o^F(go, r)\cap \mathcal C) \; \le\; \mu_1(\Pi_o(go, r)\cap \mathcal C)  \;\prec_{\lambda, \epsilon, r} \; e^{-\omega \cdot  d(o, go)}\\
\end{array}
$$
for any $g\in G$ and $r \ge  r_0$, where $ \epsilon,\lambda$ are given respectively in \ref{AssumpE} and Definition \ref{ConformalDensityDefn}.
\end{lem}

\begin{proof}
We start proving the lower bound which is the key part of the proof.

\textbf{1. Lower Bound.} Fix a Borel subset $U$  in $\mathcal C$ such that $\mu_1(U)>0$ (e.g.:  let $U=\mathcal C$). For any given $\xi \in U$,  there exists a sequence of points $z_n\in \U$ such that $z_n \to \xi$. By Lemma \ref{extend3}, there exist $f_n\in F$ and $r>0$ such that  $g$ is an $(r, f_n)$-barrier for any geodesic $[o, gfz_n]$. 
Since $F$ consists of three elements,  we assume that $f:=f_n$ are all equal, up to taking a subsequence of $z_n$. This implies  that $gfz_n\in \Omega_o^F(go, r)$ tends to $gf\xi$, so by definition, 
\begin{equation}\label{TransferEQ}
gf\xi\in \Pi_o^F(go, r)\cap \mathcal C.
\end{equation}
We refer to Fig. \ref{fig:shadowlem} for an illustration.



Consequently, we  decompose the set $U$ into a disjoint  union of three sets $U_1, U_2, U_3$: for each $U_i$, there exists $f_i \in F$ such that $gf_iU_i \subseteq  \Pi_o^F(go,r)$.   
\begin{figure}
    \centering

\tikzset{every picture/.style={line width=0.75pt}} 

\begin{tikzpicture}[x=0.75pt,y=0.75pt,yscale=-1,xscale=1]

\draw    (122.5,64) .. controls (279.71,110.77) and (465.63,87.24) .. (554.17,93.9) ;
\draw [shift={(555.5,94)}, rotate = 184.55] [color={rgb, 255:red, 0; green, 0; blue, 0 }  ][line width=0.75]    (10.93,-3.29) .. controls (6.95,-1.4) and (3.31,-0.3) .. (0,0) .. controls (3.31,0.3) and (6.95,1.4) .. (10.93,3.29)   ;
\draw [shift={(122.5,64)}, rotate = 16.57] [color={rgb, 255:red, 0; green, 0; blue, 0 }  ][fill={rgb, 255:red, 0; green, 0; blue, 0 }  ][line width=0.75]      (0, 0) circle [x radius= 3.35, y radius= 3.35]   ;
\draw   (240.6,98.45) .. controls (243.68,78.59) and (272.32,66.54) .. (304.57,71.54) .. controls (336.83,76.54) and (360.48,96.69) .. (357.4,116.55) .. controls (354.32,136.41) and (325.68,148.46) .. (293.43,143.46) .. controls (261.17,138.46) and (237.52,118.31) .. (240.6,98.45) -- cycle ;
\draw    (268.5,113) -- (329.5,114) ;
\draw [shift={(329.5,114)}, rotate = 0.94] [color={rgb, 255:red, 0; green, 0; blue, 0 }  ][fill={rgb, 255:red, 0; green, 0; blue, 0 }  ][line width=0.75]      (0, 0) circle [x radius= 3.35, y radius= 3.35]   ;
\draw [shift={(268.5,113)}, rotate = 0.94] [color={rgb, 255:red, 0; green, 0; blue, 0 }  ][fill={rgb, 255:red, 0; green, 0; blue, 0 }  ][line width=0.75]      (0, 0) circle [x radius= 3.35, y radius= 3.35]   ;
\draw    (268.5,113) -- (266.76,92.99) ;
\draw [shift={(266.5,90)}, rotate = 85.03] [fill={rgb, 255:red, 0; green, 0; blue, 0 }  ][line width=0.08]  [draw opacity=0] (8.93,-4.29) -- (0,0) -- (8.93,4.29) -- cycle    ;
\draw    (330.5,113) -- (330.5,95) ;
\draw [shift={(330.5,92)}, rotate = 90] [fill={rgb, 255:red, 0; green, 0; blue, 0 }  ][line width=0.08]  [draw opacity=0] (8.93,-4.29) -- (0,0) -- (8.93,4.29) -- cycle    ;
\draw    (122.5,64) -- (249.09,190.59) ;
\draw [shift={(250.5,192)}, rotate = 225] [color={rgb, 255:red, 0; green, 0; blue, 0 }  ][line width=0.75]    (10.93,-3.29) .. controls (6.95,-1.4) and (3.31,-0.3) .. (0,0) .. controls (3.31,0.3) and (6.95,1.4) .. (10.93,3.29)   ;
\draw    (330.5,113) -- (510.5,104.1) ;
\draw [shift={(512.5,104)}, rotate = 177.17] [color={rgb, 255:red, 0; green, 0; blue, 0 }  ][line width=0.75]    (10.93,-3.29) .. controls (6.95,-1.4) and (3.31,-0.3) .. (0,0) .. controls (3.31,0.3) and (6.95,1.4) .. (10.93,3.29)   ;
\draw  [line width=5.25] [line join = round][line cap = round] (511.8,103.22) .. controls (511.8,103.22) and (511.8,103.22) .. (511.8,103.22) ;
\draw  [line width=5.25] [line join = round][line cap = round] (555.5,93) .. controls (555.5,93) and (555.5,93) .. (555.5,93) ;
\draw  [line width=5.25] [line join = round][line cap = round] (251.8,193.22) .. controls (251.8,193.22) and (251.8,193.22) .. (251.8,193.22) ;

\draw (261.28,115.91) node [anchor=north west][inner sep=0.75pt]  [rotate=-1.93]  {$go$};
\draw (312.29,117.42) node [anchor=north west][inner sep=0.75pt]  [rotate=-1.93]  {$gf o$};
\draw (134,42.4) node [anchor=north west][inner sep=0.75pt]    {$o$};
\draw (426,62.4) node [anchor=north west][inner sep=0.75pt]    {$\Pi _{o}^{F}( g o,r) \ni \ gf\xi $};
\draw (270.5,98.4) node [anchor=north west][inner sep=0.75pt]    {$\leq r$};
\draw (329.5,99.4) node [anchor=north west][inner sep=0.75pt]    {$\leq r$};
\draw (260.5,181.4) node [anchor=north west][inner sep=0.75pt]    {$z_{n}\rightarrow \xi \in U$};
\draw (468.5,115.4) node [anchor=north west][inner sep=0.75pt]    {$gfz_{n}$};
\draw (292,44.4) node [anchor=north west][inner sep=0.75pt]    {$g\ax(f)$};

\end{tikzpicture}
    \caption{Mass transport of $U$ into shadows}
    \label{fig:shadowlem}
\end{figure}
Since $B_\xi(\cdot, \cdot)$ is Lipschitz, we obtain a constant $\theta=\theta(\|Fo\|)>0$ from  (\ref{confDeriv}) such that for each $1\le i\le 3$, we have $$\mu_{1} (f_iU_i)\ge \theta \cdot \mu_{f_i} (f_iU_i)\ge  \theta\cdot  \lambda^{-1} \cdot \mu_{1} (U_i)$$ 
where the universal constant $\lambda$  comes from the almost $G$-equivariance  (\ref{almostInv}). Thus, $$\displaystyle\sum_{1\le i\le 3} \mu_{1} (f_iU_i) \ge \theta \lambda^{-1} \cdot\mu_1(U)/3$$   so $f_iU_i \subseteq  g^{-1}\Pi_{o}^F(go,r)$ as above shows $$\mu_{1}(g^{-1}\Pi_o^F(go, r_0)\cap \mathcal C)\ge {\theta}\lambda^{-1}\cdot  \mu_1(U)/3.$$ 
Using again (\ref{almostInv}) and (\ref{confDeriv}),  we have
$$
\begin{array}{rl}
\mu_1(\Pi_o^F(go,r_0)\cap \mathcal C)&\ge \lambda^{-1}\cdot\mu_{g^{-1}}(g^{-1}\Pi_o^F(go,r_0)\cap \mathcal C) \\
\\
&\ge \lambda^{-2} \cdot e^{-\omega\cdot d(o, go)} \cdot \mu_1( g^{-1}\Pi_o^F(go,r_0))
\end{array}
$$
where the last line uses $ B_\xi (g^{-1}o, o)\le d(o, go)$.
Setting  $M:=   \theta\cdot \lambda^{-3}\cdot \mu_1(U)/3$, we get the lower bound:
$$
\mu_1(\Pi_o^F(go,r_0)\cap \mathcal C) \ge M \cdot e^{-\omega\cdot d(o, go)}.
$$
We remark here that by the same proof, the subset $\mathcal C$ could be replaced with any $G$-invariant subset of positive $\mu_1$-measure. We restrict to the   $\mathcal C$ satisfying \ref{AssumpE} only for the upper bound. 

\textbf{2. Upper Bound.} Fix $r\ge r_0$.
Given $\xi \in g^{-1} \Pi_o(go,r)\cap \mathcal C$, there is a sequence of   $z_n\in \U$ tending to $\xi$ such that $\gamma_n \cap B(o, r)
\neq \emptyset$ for $\gamma_n:=[g^{-1}o, z_n]$.  By \ref{AssumpE}, we obtain  $$|B_\xi(g^{-1}o, o)
- d(g^{-1}o, o))| \le 2r+\epsilon.$$
As     $\mu_1( g^{-1}\Pi_o(go,r)) \le 
\|\mu_1\| =1$,  the upper bound is given as follows:
$$
\begin{array}{rl}
\mu_1(\Pi_o(go,r)\cap \mathcal C) & \le \lambda \mu_{g^{-1}}(g^{-1}\Pi_o(go,r)\cap \mathcal C) \\
\\
&\le   \lambda e^{2r+\epsilon}  \displaystyle \int_{g^{-1}\Pi_o(go,r)}e^{-\omega B_\xi(g^{-1}o, o)} d \mu_1(\xi)  \\
\\
& \le  \lambda e^{2r+\epsilon}  \cdot e^{-\omega d(go, o)}
\end{array}
$$
The proof of lemma is complete.
\end{proof}
\begin{rem}[Alternative proof]
We sketch another proof  for the lower bound which was carried out in Teichm\"uller space \cite{TYANG}.  In the proof of Lemma \ref{SouthNorthLem}, as $\langle f\rangle\cdot \partial K = \pU\setminus [f^\pm]$, the set $\partial K\cap \mathcal C$ has positive $\mu_1$-measure. By Lemma \ref{IntShadow}, there exists   an open set $U$ of $[f^+]$ of positive $\mu_1$-measure so that  $U\subseteq \Pi_{g^{-1}o}^F(o, r)$, and then the remaining proof proceeds as above.
\end{rem}

\subsubsection*{Standing Assumption}
From now on, we assume further that the constant $r>0$  in Convention \ref{ConvF} satisfies the Shadow Lemma \ref{ShadowLem}.

As all boundary points in the horofunction boundary  are non-pinched by Lemma \ref{NoPinchedLineLem},  we obtain a full version of  shadow lemma.
\begin{lem}\label{HoroShadowLem}
Let $\{\mu_g\}_{g\in G}$ be a $\omega$-dimensional $G$-equivariant conformal density for some $\omega>0$ on the horofunction boundary $\hU$. Then there exists $r_0 > 0$ such that  
$$
\begin{array}{rl}
   \exp(-\omega \cdot  d(o, go))  \prec   \mu_1(\Pi_o^F(go,r))  \le \mu_1(\Pi_o(go,r))   \prec_r     \exp(-\omega \cdot  d(o, go)) 
\end{array}
$$
for any $go\in Go$ and $r \ge  r_0$.
\end{lem}

\begin{cor}\label{NoAtom}
Under the asusmption as Lemma \ref{HoroShadowLem},  $\mu_1$ has no atoms at conical points in $\hU$.
\end{cor}
\begin{proof}
Recall that  any conical point $\xi\in \Lambda_r^F(Go)$ is contained in the $(r, F)$-shadows of an unbounded sequence of orbital points in $Go$. By Lemma \ref{IntShadow}, the $(r, F)$-shadows are sandwiched between open sets, so by increasing $r$, we can assume that $\xi$  is contained in a sequence of open sets whose $\mu_1$-measure tends to $0$. This shows that $\xi$ is not atom.   
\end{proof}

By the existence of $\e G$-dimensional conformal density in Lemma \ref{confdensity},  we obtain a upper bound on the growth of balls first proved in \cite{YANG10}, by other methods without using PS-measures.
\begin{prop}\label{ballgrowth}
Let $\{\mu_g\}_{g\in G}$ be a $\omega$-dimensional  $G$-quasi-equivariant, quasi-conformal density on  $\pU$ for $\omega>0$. Assume that $\mu_1$ charges positive measure on $\mathcal C$. Then  for any
$n\ge 0$, $$| N(o, n) | \prec \exp(\omega n).$$ In particular,  $\omega \ge \e G$ and $| N(o, n) | \prec \exp(\e G n).$
\end{prop}

To prove this result, we need the following technical result, which shall be also used in later on. Recall that $[\Pi_{o}^F(v, r)]$ denotes the union of $[\cdot]$-classes of $\xi\in \Pi_{o}^F(v, r)$.
\begin{lem}\label{IntersectShadow}
There exist $\hat r,R>0$ with the following property. Assume that $\|Fo\|_{\min}\gg 0$.
Let $v,w\in Go$ such that the intersection $$\xi\in [\Pi_{o}^F(v, r)]\cap [\Pi_{o}^F(w,r)]\cap \mathcal C$$ is non-empty. If $d(o, v)\le d(o, w)$, then $d(v, [o, w])\le R$ and $\Pi_{o}^F(w, r)\subseteq \Pi_{o}^F(v, \hat r)$.
\end{lem}
\begin{proof}
Write $v=go, w=g'o$.  By definition, there exist $x_m\in \Omega_{o}^F(go, r), y_m\in \Omega_{o}^F(g'o, r)$ such that $x_m$ and $y_m$ converge to $[\xi]$. Let $X=g\ax(f)$ and $Y=g'\ax(f')$ be the corresponding   $(r, F)$-barriers   satisfying 
\begin{equation}\label{BarriergfoEq}
go, gfo\in N_r([o, x_m]),\quad g'o,g'f'o\in N_r([o, y_m])
\end{equation}   

\textbf{Case 1.} Assume that $X=Y$.  For  $go, g'o\in X=Y$, $[o, go]$ and $[o,g'o]$  intersect the   $C$-ball at $\pi_X(o)$ by Lemma \ref{BigThree}. As $d(o, go)\le d(o, g'o)$ is assumed,   the  Morse  quasi-geodesic $X$ implies the   existence of a constant $R=R(C)>0$     independent of $go, g'o$ such that $d(go, [o, g'o])\le R$. 

\begin{figure}
    \centering

\tikzset{every picture/.style={line width=0.75pt}} 

\begin{tikzpicture}[x=0.75pt,y=0.75pt,yscale=-1,xscale=1]

\draw    (45.5,173) -- (130.5,98) ;
\draw [shift={(130.5,98)}, rotate = 318.58] [color={rgb, 255:red, 0; green, 0; blue, 0 }  ][fill={rgb, 255:red, 0; green, 0; blue, 0 }  ][line width=0.75]      (0, 0) circle [x radius= 3.35, y radius= 3.35]   ;
\draw [shift={(45.5,173)}, rotate = 318.58] [color={rgb, 255:red, 0; green, 0; blue, 0 }  ][fill={rgb, 255:red, 0; green, 0; blue, 0 }  ][line width=0.75]      (0, 0) circle [x radius= 3.35, y radius= 3.35]   ;
\draw    (130.5,98) -- (188.5,75) ;
\draw [shift={(188.5,75)}, rotate = 338.37] [color={rgb, 255:red, 0; green, 0; blue, 0 }  ][fill={rgb, 255:red, 0; green, 0; blue, 0 }  ][line width=0.75]      (0, 0) circle [x radius= 3.35, y radius= 3.35]   ;
\draw    (45.5,173) -- (228.5,181) ;
\draw [shift={(228.5,181)}, rotate = 2.5] [color={rgb, 255:red, 0; green, 0; blue, 0 }  ][fill={rgb, 255:red, 0; green, 0; blue, 0 }  ][line width=0.75]      (0, 0) circle [x radius= 3.35, y radius= 3.35]   ;
\draw    (45.5,173) .. controls (167.27,76.97) and (294.92,71.11) .. (351.8,75.85) ;
\draw [shift={(353.5,76)}, rotate = 185.1] [color={rgb, 255:red, 0; green, 0; blue, 0 }  ][line width=0.75]    (10.93,-3.29) .. controls (6.95,-1.4) and (3.31,-0.3) .. (0,0) .. controls (3.31,0.3) and (6.95,1.4) .. (10.93,3.29)   ;
\draw    (228.5,181) -- (309.5,164) ;
\draw [shift={(309.5,164)}, rotate = 348.15] [color={rgb, 255:red, 0; green, 0; blue, 0 }  ][fill={rgb, 255:red, 0; green, 0; blue, 0 }  ][line width=0.75]      (0, 0) circle [x radius= 3.35, y radius= 3.35]   ;
\draw    (45.5,173) .. controls (200.72,167.03) and (316.34,158.09) .. (369.7,112.69) ;
\draw [shift={(370.5,112)}, rotate = 139.04] [color={rgb, 255:red, 0; green, 0; blue, 0 }  ][line width=0.75]    (10.93,-3.29) .. controls (6.95,-1.4) and (3.31,-0.3) .. (0,0) .. controls (3.31,0.3) and (6.95,1.4) .. (10.93,3.29)   ;
\draw  [line width=5.25] [line join = round][line cap = round] (391.3,90.22) .. controls (391.3,90.22) and (391.3,90.22) .. (391.3,90.22) ;
\draw  [line width=5.25] [line join = round][line cap = round] (371.3,112.22) .. controls (371.3,112.22) and (371.3,112.22) .. (371.3,112.22) ;
\draw  [line width=5.25] [line join = round][line cap = round] (353.3,76.22) .. controls (353.3,76.22) and (353.3,76.22) .. (353.3,76.22) ;
\draw   (211.81,187.4) .. controls (206.74,167.95) and (228.24,145.51) .. (259.83,137.28) .. controls (291.41,129.06) and (321.12,138.15) .. (326.19,157.6) .. controls (331.26,177.05) and (309.76,199.49) .. (278.17,207.72) .. controls (246.59,215.94) and (216.88,206.85) .. (211.81,187.4) -- cycle ;
\draw   (109.81,100.4) .. controls (104.74,80.95) and (126.24,58.51) .. (157.83,50.28) .. controls (189.41,42.06) and (219.12,51.15) .. (224.19,70.6) .. controls (229.26,90.05) and (207.76,112.49) .. (176.17,120.72) .. controls (144.59,128.94) and (114.88,119.85) .. (109.81,100.4) -- cycle ;
\draw    (309.5,164) -- (303.49,146.83) ;
\draw [shift={(302.5,144)}, rotate = 70.71] [fill={rgb, 255:red, 0; green, 0; blue, 0 }  ][line width=0.08]  [draw opacity=0] (8.93,-4.29) -- (0,0) -- (8.93,4.29) -- cycle    ;
\draw    (228.5,181) -- (225.95,163.97) ;
\draw [shift={(225.5,161)}, rotate = 81.47] [fill={rgb, 255:red, 0; green, 0; blue, 0 }  ][line width=0.08]  [draw opacity=0] (8.93,-4.29) -- (0,0) -- (8.93,4.29) -- cycle    ;
\draw    (188.5,75) -- (194.55,93.15) ;
\draw [shift={(195.5,96)}, rotate = 251.57] [fill={rgb, 255:red, 0; green, 0; blue, 0 }  ][line width=0.08]  [draw opacity=0] (8.93,-4.29) -- (0,0) -- (8.93,4.29) -- cycle    ;
\draw    (130.5,98) -- (137.34,114.24) ;
\draw [shift={(138.5,117)}, rotate = 247.17] [fill={rgb, 255:red, 0; green, 0; blue, 0 }  ][line width=0.08]  [draw opacity=0] (8.93,-4.29) -- (0,0) -- (8.93,4.29) -- cycle    ;

\draw (35,178.4) node [anchor=north west][inner sep=0.75pt]    {$o$};
\draw (116,83.4) node [anchor=north west][inner sep=0.75pt]    {$go$};
\draw (221.5,182.4) node [anchor=north west][inner sep=0.75pt]    {$g'o$};
\draw (309,47.4) node [anchor=north west][inner sep=0.75pt]    {$x_{m} \in \Omega _{o}^{F}( go,r)$};
\draw (401,79.4) node [anchor=north west][inner sep=0.75pt]    {$\xi \in \mathcal C$};
\draw (353.5,125.4) node [anchor=north west][inner sep=0.75pt]    {$y_{m} \in \Omega _{o}^{F}( g'o,r)$};
\draw (183,55.4) node [anchor=north west][inner sep=0.75pt]    {$gfo$};
\draw (295,168.4) node [anchor=north west][inner sep=0.75pt]    {$g'f'o$};
\draw (88,70.4) node [anchor=north west][inner sep=0.75pt]    {$X$};
\draw (189,190.4) node [anchor=north west][inner sep=0.75pt]    {$Y$};

\end{tikzpicture}
    \caption{Case (2) in the proof of Lemma \ref{IntersectShadow}}
    \label{fig:boundedmulticity}
\end{figure} 
\textbf{Case 2.} Assume now that $X\ne Y\in \f$. As $\f$ has bounded intersection,  we have $[\Lambda X]\cap [\Lambda Y]=\emptyset$. For concreteness, assume that $\xi\notin [\Lambda X]$. By Corollary \ref{BdryProjCor}, we have $d_X(x_m, y_m)\le C$ for all $m\gg 0$.  
By triangle inequality  for $d_X$, we have $d_X(o, y_m) \ge d_X(o, x_m) -C$, so Lemma \ref{BigFive} shows  
$$
\|N_C(X)\cap [o, y_m]\|\ge  d_X(o, x_m) -5C \ge L-4r-5C>0.
$$
By Lemma \ref{BigThree}, the entry points of $[o, x_m]$ and of $[o, y_m]$ into $N_C(X)$ have a distance at most $2C$. As   $\|N_C(X)\cap N_C(Y)\|\le B$ for some $B=B(C)$, we see that $[o, y_m]$  exits $N_C(X)$ within at most $B$ distance to the entry point in $N_C(Y)$. By (\ref{BarriergfoEq}), we have $d(go, [o,x_m]), d(g'o, [o, y_m])\le r$. Thus, there exists    a constant $R=R(C)>0$ such that $d(go, [o,g'o])\le R$.

Apply   Lemma \ref{BarrierFellowLem} to the geodesic $\alpha:=[o,x_m]$ and $\beta:=[o, y_m]$, where $\alpha$ contains $(r, f)$-barrier $go\in X $. Recalling  $d_X(x_m, y_m)\le C$ as above, we have  $go$ is an $(\hat r, f)$-barrier for $[o, y_m]$ with $m\gg 0$. This implies by definition that   $\Pi_{o}^F(w, r)\subseteq \Pi_{o}^F(v, \hat r)$. Hence, the proof is complete.
\end{proof}

Consider the annulus of radius $n$ centered at $o\in \U$ with width $\Delta\ge 1$: 
\begin{equation}\label{AnnulusEQ}
A(o, n, \Delta)=\{v\in Go: |d(o, v)-n|\le\Delta\}
\end{equation}

\begin{lem}\label{OverlapAnnulus}
For given $\Delta>0$, there exists $N=N(\Delta)$ such that  for every $n\ge 1$,  any $\xi\in \mathcal C$ is contained in at most $N$   shadows $\Pi_{o}^F(v, r)$ where $v\in A(o,n,\Delta)$.
\end{lem}
\begin{proof}
 Indeed, let $v,w\in A(o,n, \Delta)$ such that $\xi\in \Pi_{o}^F(v, r)\cap \Pi_{o}^F(w,r)$.   Lemma \ref{IntersectShadow} gives  an    constant $R>0$ so that     $d(v, [o,  w])\le R$.  Taking on account of the relation $$ 
d(o,v),\; d(o,w)\in [n-\Delta, n+\Delta],
$$ 
there exists $D=D(\Delta, R)$ such that $d(v, w)\le D$. Therefore it suffices to set $N=|N(o, D)|$ to complete the proof.
\end{proof}

All the ingredients  are prepared for the proof of  Proposition \ref{ballgrowth}.
 
\begin{proof}[Proof of Proposition \ref{ballgrowth}]

By Lemma \ref{OverlapAnnulus},  that every   point $\xi\in \mathcal C$ are contained in a uniformly finite number of shadows.  Hence,
$$
\sum_{v\in A(o, n,\Delta)} \mu_1(\Pi_{o}^F(v, r)\cap \mathcal C)\prec \mu_1(\mathcal C)
$$ 
which by the shadow lemma gives $\sharp A(o,n,\Delta)\le e^{n\omega}$, and thus the desired    estimates on $|N(o, n)|$.
\end{proof}

\subsection{Shadow Principle}

Let $N(G)=\{g\in \isom(\U): gGg^{-1}=G\}$   be the normalizer of $G$ in $\isom(\U)$.  Following Roblin, the next result  is   called Shadow Principle, which  was first proved in CAT(-1) spaces \cite{Roblin2}.  The shadow lemma is a special case where $x, y\in G o$ and $\|\mu_x\|=\|\mu_y\|$.

\begin{lem}[Shadow Principle]\label{ShadowPrinciple}
Let $\{\mu_x\}_{x\in \U}$ be a $\omega$-dimensional $G$-quasi-equivariant quasi-conformal density supported on $\pU$ with $\mu_x(\mathcal C)>0$ for some $\omega>0$. Let $Z=\Gamma o$ for $\Gamma=N(G)$. Then there exists $r_0 > 0$ such that  
$$
\begin{array}{rl}
\mu_y(\mathcal C) e^{-\omega  d(x, y)} \; \prec_\lambda \;   \mu_x(\Pi_{x}^F(y,r)\cap \mathcal C)\;  
 \le  \mu_x(\Pi_{x}(y,r)\cap \mathcal C)\; \prec_{\lambda,\epsilon, r} \; \mu_y(\mathcal C) e^{-\omega  d(x, y)}
\end{array}
$$
for any $x,y\in Z$ and $r \ge  r_0$,  where $ \epsilon,\lambda$ are given respectively in \ref{AssumpE} and Definition \ref{ConformalDensityDefn}.
\end{lem}

\begin{proof}
Write explicitly $x=g_1o$ and $y=g_2o$ for $g_1,g_2\in \Gamma$.  Set $U=\mathcal C$. Repeating the proof for the lower bound   of  Lemma \ref{ShadowLem} with $g:=g_1^{-1}g_2$, we decompose the set $U$ into a disjoint  union of three sets $U_1, U_2, U_3$: for each $U_i$, there exists $f_i \in F$ such that 
\begin{equation}\label{MoveSetEQ}
g_2f_iU_i \subseteq  \Pi_{g_1o}^F(g_2o,r)\cap \mathcal C.
\end{equation}

By quasi-conformality (\ref{confDeriv})  for $d(o,f_io)\le L,$ there exists    $\theta=\theta(L)>0$ such that for each $1\le i\le 3$, 
$$
\mu_{g_2o} (g_2f_iU_i)\ge \theta \cdot \mu_{g_2f_io} (g_2f_iU_i).
$$
By applying the $G$-quasi-equivariance (\ref{almostInv}) to the above right-hand  with the element $g_2 f_i^{-1}g_2^{-1}\in G$, we  obtain 
\begin{equation}\label{GEquivEQ}
\mu_{g_2o} (g_2f_iU_i) \ge \lambda^{-1}\theta\cdot  \mu_{g_2o} (g_2U_i).
\end{equation} ($g_2f\in \Gamma\setminus G$ cannot be cancelled  directly for  $\mu_x$ is not known $\Gamma$-equivariant). Summing (\ref{GEquivEQ}) over $1\le i\le 3$, 
$$\displaystyle\sum_{1\le i\le 3} \mu_{g_2o} (g_2f_iU_i) \ge \lambda^{-1}\theta\cdot\mu_{g_2o}(g_2U).$$ As $U=\mathcal C$ is   $\Gamma$-invariant and so $g_2U=U$, we   obtain from (\ref{MoveSetEQ}): $$\mu_{g_2o}(\Pi_{r}^F(g_1o,g_2o))\ge \lambda^{-1}{\theta} \cdot \mu_{g_2o}(g_2U)/3\ge\lambda^{-1} \theta \cdot\mu_{g_2o}(U)/3.$$ 
By quasi-conformality (\ref{confDeriv}) again, 
$$\begin{array}{rl}
\mu_{g_1o}(\Pi_{g_1o}^F(g_2o, r)\cap \mathcal C) &\ge \lambda^{-2} \cdot e^{-\omega\cdot d(g_1o, g_2o)} \cdot \mu_{g_2o}( \Pi_{g_1o}^F(g_2o,r)) \\
\\
&\ge M \cdot \mu_{g_2o}(U) \cdot e^{-\omega\cdot d(g_1o, g_2o)}.
\end{array}
$$
where  $M:=  \lambda^{-3}\cdot \theta/3$.   The upper bound  proceeds exactly as in Lemma \ref{ShadowLem} so we omit it. 
\end{proof}

\section{Hopf--Tsuji--Sullivan Dichotomy}\label{SecHTST}
This section is devoted to the proof of  Theorem \ref{HTSThm}, the so-called Hopf--Tsuji--Sullivan dichotomy, for the conformal density   for groups with contracting elements. 

Let $G\act \U$ as in (\ref{GActUAssump}), equipped with a non-elementary convergence boundary $\pU$.
Let $\{\mu_x\}_{x\in \U}$  be any $\omega$-dimensional $G$-quasi-equivariant quasi-conformal density charging positive measure on the set $\mathcal C$ of non-pinched points, where $\mathcal C$ satisfies \ref{AssumpC} and \ref{AssumpE}. By Lemma \ref{ballgrowth}, it is necessary that $\omega\ge \e G$.

The proof  is based on the   construction   of a family of visual spheres described in the sequel.  A similar strategy was used by Tukia \cite{Tukia4} for real hyperbolic spaces whose arguments are hard to be implemented in general metric spaces. We are eventually turned     to  the technique of projection complex. 

\subsection{Outline of the proof} Before giving the details, let us outline the key feature we want for visual spheres, and how it is used in the equivalence between divergence  of Poincar\'e series and positivity of conical points, the core of Theorem \ref{HTSThm}.

The general principle to prove the equivalence is well-known in the fields, which relies on a certain variation of the Borel-Cantelli lemma: if 
$$
\sum_{i=1}^\infty \mu_o(A_n)<\infty$$ then setting $\Lambda=\cap_{n\ge 1}\cup_{i\ge n} A_i$, we have $\mu_o(\Lambda)=0$. 

In our specific setup, $A_n$ is usually the union of shadows $\Pi_o(v,r)$ over a sphere $v\in S_n$ (which is actually the visual sphere to be precised in next subsection), then $\Lambda$ coincides with the set of conical points, which is by definition contained in infinitely many shadows. Recall that $$\mathrm{e}^{-\e G d(o,v)}\asymp \mu_o(\Pi_o(v,r))$$ from the Shadow Lemma \ref{ShadowLem}. As a consequence, this implies the easy direction:  if the conical point has positive measure, the Poincar\'e series $\sum_{g\in G} \mathrm{e}^{-\e G d(o,go)}$ is divergent.

The converse direction is the  difficult part,  boiling down to  secure a converse to  Borel-Cantelli lemma. The basic idea which is standard in the community is to establish a certain independence between the events $A_n$ such as 
\begin{align}\label{IndependenceAn}
\mu_o(A_{n+m}) \le c\mu_o(A_n)\mu_o(A_m)    
\end{align} for a constant $c$ independent of $n$. Therefore, it becomes crucial how to choose the ``sphere" $S_n$ with this desired property, which usually hinges on the hyperbolicity. 

In our setup, let us  describe the scenario of constructing  $S_n$ as follows. Standing on   a base point $o$,  the first visual sphere $S_1$ consists of the set of vertices in $Go$ visible from $o$. In general, $S_1$ would be an infinite set, as $Go$ is typically not a subset in any sense of  quasi-convexity. This  is actually the advantage of the construction, rather than a drawback, to make the above inequality (\ref{IndependenceAn}) plausible. 

Inductively, the next visual sphere $S_n$ would be the set of vertices, which are visually blocked by the previous one $S_{n-1}$. 

This construction could be worked out with explicit hyperbolicity, which was  done  by Tukia \cite{Tukia4} for real hyperbolic spaces. However, if the construction is similarly performed in a general metric space, one might run into cumbersome technicalities due to the  lack of hyperbolicity of the ambient space.   

As fore-mentioned, we shall implement the construction of visual spheres  using the genuine   sphere $\bar S_n$ in the projection complex in its combinatorial metric. To be precisely, the visual sphere, to be denoted by $T_n$ below, are the set of orbits in $Go$ carried by the sphere $\bar S_n$. With hyperbolicity of $\p_K(\f)$, this allows us to use our version of Shadow Lemma to  prove (\ref{IndependenceAn}) in Lemma \ref{KeyInequalitiesLem} in the end. 

Let us now     precise the construction of the visual spheres, which is the main novelty in this proof.

\subsection{Family of visual spheres}


Let  $F\subseteq G$  be a set of three independent non-pinched contracting elements and $\f=\{g\ax(f): g\in G, f\in F\}$ be the family of their $G$-translated  $C$-contracting axes. 

By the construction, the quasi-tree of spaces $\QT$ contains totally geodesically embedded copies of every $U\in \f$, and the collection $U\in \f$ is the vertex set of the projection complex $\PC$. We are therefore making the following.

\begin{conv}\label{ConvNotations}
In the sequel,      $U\in \f$   denotes either the subset $U$ in $\U$ or  $U$ in  $\QT$, so the points   $u\in U$ are understood according to the context. The corresponding vertex   in $\PC$ is  denoted by $\bar u$.   
\end{conv}




We fix a preferred axis  $O=\ax(f)\in  \f$, and the corresponding base vertex $\bar o$ in $\PC$. The basepoint $o\in \U$ is contained in the axis $O=\ax(f)$.



\subsubsection*{Visual sphere construction}
For an integer $n\ge 0$, consider  the   sphere  of radius $n$ centered at the preferred base vertex $\bar o$ in the projection complex: $$\bar S_n:=\{\bar v\in\PC: d(\bar o, \bar v)=n\}$$  The \textit{$n$-th  visual sphere}  of orbital points in $Go$ is defined as  $$T_n:=\bigcup_{\bar v\in \bar S_n} V$$
where $V$ is of the form $g\ax(f)=gE(f)o$.
That is,  $T_n$  is the union of all $V\in\f$ corresponding to $\bar v\in \bar S_n$.  If  $go\in U$,  we   say   that $U$ 
(resp. $\bar u$)   are  an associated  (non-unique) vertex space 
(resp. vertex).


We now introduce some auxiliary sets. 
For  any $V\in \f$ and $L,\Delta>0$, denote  
\begin{equation}\label{TnLEQ}
\begin{array}{ll}
V_L&:=\quad \{v\in V: d_V(o,v)> L\}\\
T_n'(L)&:=\quad \bigcup_{\bar v\in \bar S_n} V_L
\end{array}
\end{equation}
(recall that, according to Convention \ref{ConvNotations}, $V$ is the axis in $\U$ corresponding to the vertex $\bar v$.)
By definition, $T_n'(L)$ forms an \textit{$L$-net} in $T_n$: any point in $T_n$ has a distance less than $L$ to $T_n'$.
\begin{figure}
    \centering

\tikzset{every picture/.style={line width=0.75pt}} 

\begin{tikzpicture}[x=0.75pt,y=0.75pt,yscale=-1,xscale=1]

\draw   (163,152) .. controls (163,146.48) and (167.48,142) .. (173,142) .. controls (178.52,142) and (183,146.48) .. (183,152) .. controls (183,157.52) and (178.52,162) .. (173,162) .. controls (167.48,162) and (163,157.52) .. (163,152) -- cycle ;
\draw   (101,92) .. controls (101,86.48) and (105.48,82) .. (111,82) .. controls (116.52,82) and (121,86.48) .. (121,92) .. controls (121,97.52) and (116.52,102) .. (111,102) .. controls (105.48,102) and (101,97.52) .. (101,92) -- cycle ;
\draw   (85,190) .. controls (85,184.48) and (89.48,180) .. (95,180) .. controls (100.52,180) and (105,184.48) .. (105,190) .. controls (105,195.52) and (100.52,200) .. (95,200) .. controls (89.48,200) and (85,195.52) .. (85,190) -- cycle ;
\draw   (195,81) .. controls (195,75.48) and (199.48,71) .. (205,71) .. controls (210.52,71) and (215,75.48) .. (215,81) .. controls (215,86.52) and (210.52,91) .. (205,91) .. controls (199.48,91) and (195,86.52) .. (195,81) -- cycle ;
\draw   (230,205) .. controls (230,199.48) and (234.48,195) .. (240,195) .. controls (245.52,195) and (250,199.48) .. (250,205) .. controls (250,210.52) and (245.52,215) .. (240,215) .. controls (234.48,215) and (230,210.52) .. (230,205) -- cycle ;
\draw   (248,142) .. controls (248,136.48) and (252.48,132) .. (258,132) .. controls (263.52,132) and (268,136.48) .. (268,142) .. controls (268,147.52) and (263.52,152) .. (258,152) .. controls (252.48,152) and (248,147.52) .. (248,142) -- cycle ;
\draw   (157,229) .. controls (157,223.48) and (161.48,219) .. (167,219) .. controls (172.52,219) and (177,223.48) .. (177,229) .. controls (177,234.52) and (172.52,239) .. (167,239) .. controls (161.48,239) and (157,234.52) .. (157,229) -- cycle ;
\draw    (173,162) -- (167.21,217.01) ;
\draw [shift={(167,219)}, rotate = 276.01] [color={rgb, 255:red, 0; green, 0; blue, 0 }  ][line width=0.75]    (10.93,-3.29) .. controls (6.95,-1.4) and (3.31,-0.3) .. (0,0) .. controls (3.31,0.3) and (6.95,1.4) .. (10.93,3.29)   ;
\draw    (181.5,158) -- (229.97,198.71) ;
\draw [shift={(231.5,200)}, rotate = 220.03] [color={rgb, 255:red, 0; green, 0; blue, 0 }  ][line width=0.75]    (10.93,-3.29) .. controls (6.95,-1.4) and (3.31,-0.3) .. (0,0) .. controls (3.31,0.3) and (6.95,1.4) .. (10.93,3.29)   ;
\draw    (163.5,157) -- (105.32,183.18) ;
\draw [shift={(103.5,184)}, rotate = 335.77] [color={rgb, 255:red, 0; green, 0; blue, 0 }  ][line width=0.75]    (10.93,-3.29) .. controls (6.95,-1.4) and (3.31,-0.3) .. (0,0) .. controls (3.31,0.3) and (6.95,1.4) .. (10.93,3.29)   ;
\draw    (183,152) -- (246.02,142.3) ;
\draw [shift={(248,142)}, rotate = 171.25] [color={rgb, 255:red, 0; green, 0; blue, 0 }  ][line width=0.75]    (10.93,-3.29) .. controls (6.95,-1.4) and (3.31,-0.3) .. (0,0) .. controls (3.31,0.3) and (6.95,1.4) .. (10.93,3.29)   ;
\draw    (166.5,145) -- (118,102.32) ;
\draw [shift={(116.5,101)}, rotate = 41.35] [color={rgb, 255:red, 0; green, 0; blue, 0 }  ][line width=0.75]    (10.93,-3.29) .. controls (6.95,-1.4) and (3.31,-0.3) .. (0,0) .. controls (3.31,0.3) and (6.95,1.4) .. (10.93,3.29)   ;
\draw    (178.5,143) -- (197.79,91.87) ;
\draw [shift={(198.5,90)}, rotate = 110.67] [color={rgb, 255:red, 0; green, 0; blue, 0 }  ][line width=0.75]    (10.93,-3.29) .. controls (6.95,-1.4) and (3.31,-0.3) .. (0,0) .. controls (3.31,0.3) and (6.95,1.4) .. (10.93,3.29)   ;
\draw    (393.5,203) -- (443.5,162) ;
\draw [shift={(443.5,162)}, rotate = 320.65] [color={rgb, 255:red, 0; green, 0; blue, 0 }  ][fill={rgb, 255:red, 0; green, 0; blue, 0 }  ][line width=0.75]      (0, 0) circle [x radius= 3.35, y radius= 3.35]   ;
\draw [shift={(393.5,203)}, rotate = 320.65] [color={rgb, 255:red, 0; green, 0; blue, 0 }  ][fill={rgb, 255:red, 0; green, 0; blue, 0 }  ][line width=0.75]      (0, 0) circle [x radius= 3.35, y radius= 3.35]   ;
\draw    (506.5,72) .. controls (468.5,100) and (426.5,143) .. (446.5,167) .. controls (466.5,191) and (525.5,145) .. (554.5,119) ;
\draw  [draw opacity=0][dash pattern={on 4.5pt off 4.5pt}] (427.25,136.65) .. controls (430.15,134.87) and (433.4,133.52) .. (436.93,132.73) .. controls (453.56,128.99) and (469.98,139.07) .. (473.62,155.24) .. controls (476.18,166.67) and (471.66,178.1) .. (462.9,185.12) -- (443.5,162) -- cycle ; \draw  [dash pattern={on 4.5pt off 4.5pt}] (427.25,136.65) .. controls (430.15,134.87) and (433.4,133.52) .. (436.93,132.73) .. controls (453.56,128.99) and (469.98,139.07) .. (473.62,155.24) .. controls (476.18,166.67) and (471.66,178.1) .. (462.9,185.12) ;  
\draw    (443.5,162) -- (428.33,138.34) ;
\draw [shift={(427.25,136.65)}, rotate = 57.34] [color={rgb, 255:red, 0; green, 0; blue, 0 }  ][line width=0.75]    (10.93,-3.29) .. controls (6.95,-1.4) and (3.31,-0.3) .. (0,0) .. controls (3.31,0.3) and (6.95,1.4) .. (10.93,3.29)   ;
\draw  [line width=0.75]  (424.5,139) .. controls (421.34,141.06) and (420.79,143.67) .. (422.85,146.82) -- (422.85,146.82) .. controls (425.79,151.33) and (425.68,154.62) .. (422.53,156.68) .. controls (425.68,154.62) and (428.73,155.84) .. (431.68,160.35)(430.35,158.32) -- (431.68,160.35) .. controls (433.73,163.51) and (436.34,164.06) .. (439.5,162) ;
\draw    (443.5,162) -- (501.95,114.27) ;
\draw [shift={(503.5,113)}, rotate = 140.76] [color={rgb, 255:red, 0; green, 0; blue, 0 }  ][line width=0.75]    (10.93,-3.29) .. controls (6.95,-1.4) and (3.31,-0.3) .. (0,0) .. controls (3.31,0.3) and (6.95,1.4) .. (10.93,3.29)   ;
\draw  [line width=0.75]  (444.5,165) .. controls (447.51,168.57) and (450.8,168.84) .. (454.36,165.83) -- (473.01,150.08) .. controls (478.1,145.78) and (482.16,145.41) .. (485.17,148.97) .. controls (482.16,145.41) and (483.2,141.48) .. (488.29,137.17)(486,139.11) -- (501.67,125.87) .. controls (505.24,122.86) and (505.51,119.57) .. (502.5,116) ;
\draw  [line width=0.75]  (248.5,134) .. controls (247.75,129.39) and (245.07,127.47) .. (240.46,128.23) -- (223.05,131.08) .. controls (216.47,132.16) and (212.8,130.4) .. (212.05,125.8) .. controls (212.8,130.4) and (209.89,133.24) .. (203.31,134.32)(206.27,133.84) -- (187.27,136.95) .. controls (182.67,137.71) and (180.75,140.39) .. (181.5,145) ;

\draw (169,144.4) node [anchor=north west][inner sep=0.75pt]    {$\overline{o}$};
\draw (252,135.4) node [anchor=north west][inner sep=0.75pt]    {$\overline{v}$};
\draw (383,208.4) node [anchor=north west][inner sep=0.75pt]    {$o$};
\draw (433,90.4) node [anchor=north west][inner sep=0.75pt]    {$V$};
\draw (508,96.4) node [anchor=north west][inner sep=0.75pt]    {$v\in V_{L}$};
\draw (408,151.4) node [anchor=north west][inner sep=0.75pt]    {$L$};
\draw (485.5,140.9) node [anchor=north west][inner sep=0.75pt]    {$ >L$};
\draw (82,43.4) node [anchor=north west][inner sep=0.75pt]    {$\p_K(\f):\bar S_{n}=\{\bar v: d(\bar o, \bar v)=n\}$};
\draw (371,43.4) node [anchor=north west][inner sep=0.75pt]    {$\U:T'( L)=\cup_{\bar v\in \bar S_n} V_L$};
\draw (207,108.4) node [anchor=north west][inner sep=0.75pt]    {$n$};

\end{tikzpicture}
    \caption{Visual spheres in projection complex (left) and there lifts with $L$-truncation in the original space (right)}
    \label{fig:visualsphere}
\end{figure} 

We are actually going to focus on a much smaller subset.  Let us fix $\Delta := \min\{d(o, go): g\in E( f), f\in F\}$, so  for any $L>0$, the annulus set $$
\begin{array}{rl}
V(L,\Delta)&:=\;\{v\in V: |d_V(o,v) - L|\le \Delta\}
\end{array}
$$ is always non-empty, and has at most $|A(o, L, \Delta)|$ elements. We refer to Fig. \ref{fig:visualsphere} for illustrating these sets.

In contrast with the subset $T_n'(L)$, the union of $V(L,\Delta)$ over $\bar v\in\bar S_n$ is definitely not a net of $T_n$, because $V(L,\Delta)$ truncates the part outside the ball $B(o,L+\Delta)$. Thanks to the linear growth of $V$,     $V(L,\Delta)$ still contributes a major proportion of the Poincar\'e series associated to $V$. This fact stated below shall be useful in the proof of Lemma \ref{BddTnSeriesLem}.

\begin{lem}\label{CosetShadowLem}
For any $\omega>0$ and $L>0$, there exists $\theta=\theta(L,\omega)>0$ such that for any $V\in \f$, we have 
$$
\sum_{v\in V} e^{-\omega d(o,v)} \le \theta \sum_{v\in V(L,\Delta)} e^{-\omega d(o, v)}.$$
 
\end{lem}
\begin{proof}
Write $V=g E(f) \cdot o$ for some $g\in G$.  Let $u\in V$ such that $d(o, u)=d(o, V)$, so $|d(u,v)- L|\le \Delta$ for $v\in V(L, \Delta)$. The  $C$-contracting property further implies for any $v\in V$: $$|d(o, v)-d(o, u)-d(u,v)|\le 4C.$$ Thus, by the triangle inequality,
$$
\begin{array}{rl}
\displaystyle\sum_{v\in V} e^{-\omega d(o,v)} & \prec_{C,\Delta}  e^{-\omega d(o,u)} \displaystyle\sum_{n\ge 1} |A(o, n, \Delta)\cap \ax(f)| e^{-\omega n}.
\end{array}
$$
As   the maximal elementary group $E( f)$  is virtually $\mathbb Z$ and $\ax( f)=E(f)o$ is a quasi-geodesic, we see  $|A(o, n, \Delta)\cap \ax(f)| $ is a  linear function of $n$. For any given $L>0$, we ignore the contribution from $V_{L+\Delta}$ in the series on the right-hand side, so that 
$$
\begin{array}{rl}
\displaystyle\sum_{v\in V} e^{-\omega d(o,v)}\le \theta\cdot\left( \sum\limits_{v\in V(L,\Delta)}   e^{-\omega d(o, v)}\right)
\end{array}
$$
where the implied constant $\theta$ depends only on $L, C,\Delta$ and $E(f)$.  
\end{proof}

The following lemma is analogous to Lemma \ref{OverlapAnnulus}, but the proof is much more involved. Recall  the definition of $T_n'(L)$ is given in  (\ref{TnLEQ}). 

\begin{lem}\label{bddoverlapLem}
There exists $L\gg 0$ such that  for any $\xi \in\mathcal C$ and $n>0$, the set of orbital points $$T_n(\xi)\quad :=\quad \{v\in    T_n'(L): \; \xi \in \Pi_o(v, r)\}$$ is contained in at most one  axis $V$ in $\f$. That is, $T_n(\xi)= V_L$.
\end{lem}
\begin{proof}
Let   $go, ho\in T_n'(L)$  so that $\xi \in \Pi_o(go,r)\cap \Pi_o(ho,r)$. By definition of $T_n'(L)$, there are $U, V\in \f$ so that $go\in U$ and $ho\in V$ and the associated vertices $\bar u, \bar v $ in the projection complex  lie on the $n$-sphere $ \bar S_n$. Moreover, we have 
\begin{equation}\label{dUogoEq}
d_U(o, go)\ge L,\quad d_V(o, ho)\ge L.
\end{equation} 
Arguing by way of contradiction,  assume  that  $U\ne V$, which thus a bounded intersection. We may assume   further  that  $d(o, U)\le d(o, V)$ for concreteness. 

According to definition of  $\xi \in \Pi_o(go,r)\cap \Pi_o(ho,r)$, there exist  two sequences of points in the corresponding cones $$\{x_m\}\subseteq \Omega_o(go, r),\quad \{y_m\}\subseteq \Omega_o(ho, r)$$ which    tend to $\xi$ as $m\to\infty$. 
We refer to Fig. \ref{fig:bddmultiplicity2} for the configuration of these points.

By definition, $d_U(o,go)\ge L$ for $go\in U_L$, which yields $\|N_r(U)\cap [o,go]\|\ge L-4r$ by Lemma \ref{Transform}.  The contracting property of $U$ implies that the entry point  of $[o, x_m]$ at $N_r(U)$ is $2C$-close to $\pi_U(o)$ by Lemma \ref{BigFive}(3).  As $go$ lies within a distance $r$ to $[o,x_m]$, we obtain $$\|N_r(U)\cap [o,x_m]\|\ge \|N_r(U)\cap [o,go]\| -2C-r\ge L-5r-2C.$$    Thus,  setting $M=9r+2C$, Lemma \ref{Transform} shows
$$d_U(o, x_m)\ge \|N_r(U)\cap [o,x_m]\| - 4r \ge  L-M.$$ 
Similarly, the same reasoning shows $d_V(o, y_m)\ge  L-M.$

As $U\ne V$ is assumed, we have $[\Lambda U]\cap [\Lambda V]=\emptyset$. Indeed, recall that $\Lambda U$ means the limit set of the axis $U\in \f$, which is here the fixed points of a contracting element. By Lemma \ref{DisjFixedSet}, distinct non-pinched contracting elements have disjoint fixed points, so $[\Lambda U]\cap [\Lambda V]=\emptyset$ follows. 

For concreteness, assume that $\xi\notin [\Lambda U]$; the case  $\xi\notin [\Lambda V]$ is similar. 
\begin{figure}
    \centering

\tikzset{every picture/.style={line width=0.75pt}} 

\begin{tikzpicture}[x=0.75pt,y=0.75pt,yscale=-1,xscale=1]

\draw    (45.5,173) -- (95.5,132) ;
\draw [shift={(95.5,132)}, rotate = 320.65] [color={rgb, 255:red, 0; green, 0; blue, 0 }  ][fill={rgb, 255:red, 0; green, 0; blue, 0 }  ][line width=0.75]      (0, 0) circle [x radius= 3.35, y radius= 3.35]   ;
\draw [shift={(45.5,173)}, rotate = 320.65] [color={rgb, 255:red, 0; green, 0; blue, 0 }  ][fill={rgb, 255:red, 0; green, 0; blue, 0 }  ][line width=0.75]      (0, 0) circle [x radius= 3.35, y radius= 3.35]   ;
\draw    (186.5,28) .. controls (148.5,56) and (78.5,113) .. (98.5,137) .. controls (118.5,161) and (211.5,79) .. (240.5,53) ;
\draw    (95.5,132) -- (155.5,83) ;
\draw [shift={(155.5,83)}, rotate = 320.76] [color={rgb, 255:red, 0; green, 0; blue, 0 }  ][fill={rgb, 255:red, 0; green, 0; blue, 0 }  ][line width=0.75]      (0, 0) circle [x radius= 3.35, y radius= 3.35]   ;
\draw    (45.5,173) -- (228.5,173) ;
\draw [shift={(228.5,173)}, rotate = 0] [color={rgb, 255:red, 0; green, 0; blue, 0 }  ][fill={rgb, 255:red, 0; green, 0; blue, 0 }  ][line width=0.75]      (0, 0) circle [x radius= 3.35, y radius= 3.35]   ;
\draw    (45.5,173) .. controls (167.27,76.97) and (294.92,71.11) .. (351.8,75.85) ;
\draw [shift={(353.5,76)}, rotate = 185.1] [color={rgb, 255:red, 0; green, 0; blue, 0 }  ][line width=0.75]    (10.93,-3.29) .. controls (6.95,-1.4) and (3.31,-0.3) .. (0,0) .. controls (3.31,0.3) and (6.95,1.4) .. (10.93,3.29)   ;
\draw    (403.5,146) .. controls (346.5,121) and (226.5,140) .. (228.5,173) .. controls (230.5,206) and (316.5,199) .. (380.5,198) ;
\draw    (228.5,173) -- (328.5,173) ;
\draw [shift={(328.5,173)}, rotate = 0] [color={rgb, 255:red, 0; green, 0; blue, 0 }  ][fill={rgb, 255:red, 0; green, 0; blue, 0 }  ][line width=0.75]      (0, 0) circle [x radius= 3.35, y radius= 3.35]   ;
\draw    (45.5,173) .. controls (200.72,167.03) and (316.34,158.09) .. (369.7,112.69) ;
\draw [shift={(370.5,112)}, rotate = 139.04] [color={rgb, 255:red, 0; green, 0; blue, 0 }  ][line width=0.75]    (10.93,-3.29) .. controls (6.95,-1.4) and (3.31,-0.3) .. (0,0) .. controls (3.31,0.3) and (6.95,1.4) .. (10.93,3.29)   ;
\draw  [line width=5.25] [line join = round][line cap = round] (391.3,90.22) .. controls (391.3,90.22) and (391.3,90.22) .. (391.3,90.22) ;
\draw  [line width=5.25] [line join = round][line cap = round] (371.3,112.22) .. controls (371.3,112.22) and (371.3,112.22) .. (371.3,112.22) ;
\draw  [line width=5.25] [line join = round][line cap = round] (353.3,76.22) .. controls (353.3,76.22) and (353.3,76.22) .. (353.3,76.22) ;

\draw (35,178.4) node [anchor=north west][inner sep=0.75pt]    {$o$};
\draw (86,58.4) node [anchor=north west][inner sep=0.75pt]    {$U$};
\draw (160,57.4) node [anchor=north west][inner sep=0.75pt]    {$go\in U_{L}$};
\draw (340,159.4) node [anchor=north west][inner sep=0.75pt]    {$ho\in V_{L}$};
\draw (259,206.4) node [anchor=north west][inner sep=0.75pt]    {$V$};
\draw (340,48.4) node [anchor=north west][inner sep=0.75pt]    {$x_{m} \in \Omega _{o}^{F}( go,r)$};
\draw (401,79.4) node [anchor=north west][inner sep=0.75pt]    {$\xi\in \mathcal C$};
\draw (367.5,112.4) node [anchor=north west][inner sep=0.75pt]    {$y_{m} \in \Omega _{o}^{F}( go,r)$};

\end{tikzpicture}
    \caption{If $d(o,U)\gg d(o,V)$, we shall see that $U$ appears before $V$ in $\p_K(\f)$: $\bar u$ lies on $[\bar o, \bar v]$. This contradicts the assumption of $\bar u,\bar v\in \bar S_n$.}
    \label{fig:bddmultiplicity2}
\end{figure}
By  \ref{AssumpC}, we have $d(o, [x_m, y_m])\to\infty$  for $\xi\in\mathcal C$. By Lemma \ref{BdryProjCor}, as $\xi\notin \Lambda U$, we have $d_U(x_m, y_m)\le C$ for $m\gg 0$. The triangle property for $d_U$ then gives  $$d_U(o, y_m)\ge d_U(o, x_m)-d_U(x_m, y_m)\ge L-C-M$$

If we choose $L\gg M+5C$, then $[o, y_m]\cap N_C(U)\ne \emptyset$ by Lemma \ref{BigFive}. Let $z$ be the exit point $[o, y_m]$ at $N_C(U)$, so $d(z, \pi_U(y_m))\le 2C$ by Lemma \ref{BigThree}. 

As   $U\ne V$ has bounded intersection and   $d(o, U)\le d(o,V)$,  $[o, y_m]$ exits $N_C(U)$ before entering $N_C(V)$ with a   possible overlap bounded by a constant $B$.
As  $d(ho, [o, y_m])\le r$ and $ho\in V$, there exists a  constant still denoted by $M=M(r)>0$ such that $$d_U(o, ho)\ge d_U(o, y_m) - M.$$

Choose  $L>\hat K+2M+C$,  where $\hat K$ is given by Lemma \ref{ForcingLem}. Thus, $d_U(o, ho)\ge \hat K$, so that $\bar u$ appears in any geodesic $[\bar o, \bar v]$. This is a contradiction, as $d(\bar o, \bar u)=d(\bar o, \bar v)=n$. 

In conclusion, the set $T_n(\xi)$ is contained in one axis $U=V$. In particular, $T_n(\xi)\subseteq V_L$ follows. The result is proved. 
\end{proof}

Thanks to Lemma \ref{bddoverlapLem}, any boundary point is contained in a uniform number of shadows over $T_n'(L)$, so we deduce the following main corollary from it.  
\begin{lem}\label{BddTnSeriesLem}
There exists a uniform constant $c>0$ independent of $n\ge 1$ such that  
\begin{equation}\label{TnSeriesEQ}
\forall n\ge 1:\quad \sum\limits_{go \in T_n} e^{-\omega d(o, go)} \le  c.
\end{equation}
 
\end{lem}
\begin{proof}
Fix a constant  $L>0$  satisfying   Lemma \ref{bddoverlapLem}, so if $T_n(\xi)$ for   $\xi \in \mathcal C$ is non-empty, then it is contained in one $V\in \f$.  Precisely, $T_n(\xi)\subseteq V_L$. 
By Lemma \ref{IntersectShadow}, there exists  $N=N(L)$ such that any point  $\xi \in \mathcal C$ are shadowed by at most $N$     elements in $V(L,\Delta)$: 
$$ \sum\limits_{\bar v\in \bar S_n} \left( \sum\limits_{u \in V(L,\Delta)} \mu_1(\Pi_o^F(u,r)) \right)\le N \cdot \mu_o( \mathcal C)\le N$$
By Lemma \ref{CosetShadowLem}, the majority of the Poincar\'e series is concentrated over $V(L,\Delta)$. We thus obtain $$\sum\limits_{v \in T_n} e^{-\omega d(o, v)} \le \theta  \sum\limits_{\bar v\in \bar S_n} \left(\sum\limits_{u \in V(L,\Delta)}e^{-\omega d(o, u)}\right) $$ 
where $V$ is the axes associated to $\bar v$.

The conclusion then follows from the shadow lemma \ref{ShadowLem} that $\mu_1(\Pi_o^F(u,r))\asymp \mathrm{e}^{-\omega d(o,u)}$.
\end{proof}

Recall the definition of partial Poincar\'e series over the $n$-th visual sphere $T_n$:
$$
\p_{ T_{n}}(s,o,o) = \sum_{v\in  T_{n}} e^{-\omega d(o, v)}
$$ 
The following crucial lemma shall be used to write the Poincar\'e series as geometric series in Step 1 below in the next subsection.
\begin{lem}\label{KeyInequalitiesLem}
For any integer $m, n\ge 1$ and for any $s>\e G$, we have
\begin{align}
\label{TnSeriesCstEQ}
\p_{  T_{n}}(s,o,o)\asymp \p_{T_{n+1}}(s,o,o)\\
\label{TnmSeriesEQ}
\p_{T_{n+m}}(s,o,o) \prec \p_{ T_{n}}(s,o,o) \cdot \p_{  T_{m}}(s,o,o)
\end{align} 
where the implied constants do not depend on $m, n$ and $s$.
\end{lem}
\begin{proof}
\textbf{(1).}  We start by proving the ``$\prec $" direction of (\ref{TnSeriesCstEQ}).

Let $\hat K=\hat K(K)\ge B$ be given by Lemma \ref{ForcingLem}, where $B$ is the bounded projection constant of $\f$. Consider the $(2\hat K)$-net $T_n':=T_n'(2\hat K)$  in $T_n$, so we have  $$\p_{T_n}(s,o,o) \asymp_{\hat K} \p_{T_n'}(s,o,o).$$
The goal of the proof is  to prove for any $go\in T_n'$, there exists $f'\in F$ such that $gf'o\in T_{n+1}$.  Recall that if $V$ is an axis given by $g\ax(f)$, we say that $f$ is the {\it type} of the axis $V$ or the vertex $\bar v$.   

Indeed, according to definition of $go\in T_n'(2\hat K)$,   there exist an associated vertex $\bar u\in \bar S_n$ of type $f\in F$ and  the vertex space $U=g\ax(f)$ such that $d_U(o, go)>2\hat K.$ By Lemma \ref{extend3} there exists $f'\ne f\in F$ such that $gf'$ labels a quasi-geodesic. The axis $V:=g\ax(f')$ is adjacent to $U$ in $\PC$, so $d(\bar u, \bar v)=1$ and thus $|d(\bar v, \bar o)-n|\le 1$. We claim that $\bar v \in   \bar S_{n+1}$.  The point is  to exclude the possibility of $d(\bar v, \bar o)\le n$.

Since $U=g\ax(f)$ and $V=g\ax(f')$ have $B$-bounded intersection for some $B>0$, we obtain that   $[go, gf'o]$ has $B$-bounded projection to $U$.  This implies  $$d_U(O, V)\ge d(o,go)-B\ge 2\hat K-B\ge \hat K$$ so  by Lemma \ref{ForcingLem},    $\bar u$ is contained in any geodesic from $\bar o$ to $\bar v$ in $\PC$. 

If  $d(\bar o, \bar v)\le n$ was true, we would obtain a contradiction: $d(\bar u, \bar o)< n$. Hence,  $d(\bar v, \bar o)\le n$ is impossible: that is,  $d(\bar o,\bar v)=n+1$ and $\bar v \in   \bar S_{n+1}$. 

In summary, we proved that for any $go\in T_n'$ there exists $f'\in F$ such that $gf'o\in T_{n+1}$ and  $d(o,gf'o,o)\le d(o,go)+\|Fo\|$. Moreover, as $Go$ is a discrete set, only a uniform finite number of $go\in T_n'$ shares the common $gf'o\in T_{n+1}$ in a ball of radius $\|Fo\|$. Unveiling the definition of $\p_{ T_{n+1}}(s,o,o)$ with the triangle inequality, we see that    $$\p_{ T_{n+1}}(s,o,o) \geq  c\cdot \p_{T_n'}(s,o,o)$$ some constant $c$ depending on $\|Fo\|$. 
Recall that $\p_{T_n'}(s,o,o) \asymp_{\hat K}   \p_{ T_{n}}(s,o,o)$. The direction ``$\prec $"  of the inequlity (\ref{TnSeriesCstEQ}) follows as desired.

The other direction would follow from the inequality (\ref{TnmSeriesEQ}), which is  the next goal.

\textbf{(2).}
For any $go\in T_{n+m}$, let $\bar u\in\bar S_{n+m}$  be an vertex for which the associated axis $U$ contains $go$. The vertex $\bar u$ is non-unique, as any point of $Go$ is contained in three axis from $\f$. 

Consider the standard path $\alpha$ from $\bar o$ to $\bar u$. By Lemma \ref{LiftPathLem}, there exists an $(L,B)$-admissible path $\tilde \alpha$ between $o$ and $go$ so that the saturation $\f(\tilde \alpha)$ consists of the  axes associated to the vertices on  $\alpha$. In particular, we have $[o,go]$ intersects the $r$-neighborhood of every $V\in \f(\tilde \alpha)$. 

Let  $\bar v \in \bar S_n\cap \alpha$ so that $d(\bar v, \bar o)=n$.    To be explicit, write the corresponding axis $V=h\ax(f')$ for some $h\in G$ and $f'\in F$.
By definition, \emph{any} element $ho\in V$ in the axis $V$ is contained in  the $n$-th visual sphere $T_n$. Let us choose the one $ho$ so that $d(ho, [o,go])\le r$ (as  $[o,go]$ intersects the $r$-neighborhood of $V$). Thus, denoting $a=h^{-1}g$, $$
d(o,go)\ge d(o, ho)+d(o, ao) -2r
$$
As the standard path $\alpha$ is a $(2,1)$-quasi-geodesic (\cite[Corollary 3.7]{BBFS}), it follows that $|d(\bar v,\bar u)-m|\le N_0$ for a universal integer $N_0$. Indeed, according to the proof of Theorem \ref{projcplxThm}, we can choose $N_0=2$, since  it is proved that any geodesic passes with a distance at most 2 to any vertex on the standard path.

It is possible that $f'\ne f$; that is, $\bar v$ might not be in the orbit $G\bar o$. Let $\bar w$ be the vertex adjacent to $\bar v$ so that the associated axis $W$ is of form $h\ax(f)$. As $|d(\bar w, \bar u)-m|\le N_0+1$ and $h^{-1}\bar w=\bar o$, we thus have    $h^{-1}go$ lies in the following $(N_0+1)$-thicken annulus:  $$A_m:=\bigcup_{|i-m|\le N_0+1} T_i$$ That is,  we can write $g=h \cdot a$ where $ho\in T_n$ and $ao\in A_m$. 
We now see   that
$$\begin{aligned}
\p_{ T_{n+m}}(s,o,o) &=\sum_{go\in  T_{n+m}} e^{-\omega d(o, go)}\\ &\prec  \p_{ T_{n}}(s,o,o) \cdot \p_{A_m}(s,o,o),    
\end{aligned}
$$
By Lemma \ref{BddTnSeriesLem}, $\p_{A_1}(s,o,o)\le c_1:=2(N_0+1)c$  for $m=1$, so we get $$\p_{ T_{n+1}}(s,o,o)\le c_1\p_{ T_{n}}(s,o,o)$$ This completes the proof of the other direction in (\ref{TnSeriesCstEQ}), as promised above. This further implies  $\p_{A_m}(s,o,o)\le c_2 \p_{ T_{m}}(s,o,o)$ for some $c_2$ depending on $c_1$ and $N_0$. Hence, the general case  in  (\ref{TnmSeriesEQ}) follows. 
\end{proof}

\subsection{Proof of Hopf--Tsuji--Sullivan Dichotomy} 
\label{SecHTSTProof}

Recall that the union $\bigcup_{i\ge 1}  T_i$ covers the $G$-orbit $Go$ with multiplicity at most  3.  
The shadow lemma \ref{ShadowLem} then implies  
\begin{equation}\label{TnPSeriesEQ}
\p_{G}(\omega, o,o)\quad  \asymp\quad \sum\limits_{i\ge 1} \Bigg(\sum\limits_{g \in  T_i} e^{-\omega d(o, go)}\Bigg)\quad \asymp_r\quad \sum\limits_{i\ge 1} \Bigg(\sum\limits_{g \in  T_i} \mu_1\big(\Pi_o^F(go, r)\big)\Bigg)
\end{equation}

We subdivide the proof of Theorem  \ref{HTSThm} in the following   steps. 

Before the proof, we remark that different definitions  \ref{ConicalDef1prime}, \ref{ConicalDef2} and \ref{ConicalDef3}  of   conical points are quantitatively equivalent in Lemmas \ref{QuantativeEquiv12} and \ref{QuantativeEquiv23}: that is, each set is contained into the other with appropriate parameters. Thus, we shall  insist on the definition \ref{ConicalDef2} of conical points via shadows, as the proof uses extensively the shadow lemma. This does not effect the generality of Theorem \ref{HTSThm}, as the positive / null measure charged  on conical points shall be established to be equivalent with divergence / convergence of Poincar\'e series.   

First of all, the direction ``(\ref{itm:2}) $\Longrightarrow$ (\ref{itm:4})" follows as $\Lambda_r^F(Go)\subseteq \mG$, and  ``(\ref{itm:3}) $\Longrightarrow$ (\ref{itm:2})" is trivial. We now give the proof of the other directions. 

\subsubsection*{(\ref{itm:1})  \texorpdfstring{$\Longrightarrow$}{} (\ref{itm:4})} If the set  $\Lambda_{r}^F(Go)$  is $\mu_1$-null, then  we prove that $G$ is of convergence type. Denote $$\Lambda_n:= \bigcup\big\{     \Pi_o^{F}(go, r): g \in \cup_{k\ge n}T_k \big\}.$$ 
We have $\displaystyle  \lim_{n\to \infty} \Lambda_n \subseteq  \Lambda_r^F(Go)$ by (\ref{rFConicalEQ}). By hypothesis,  $\displaystyle\lim_{n\to \infty} \mu_1(\Lambda_n)=0$. Using     Lemma \ref{ShadowLem}, there exist $n_0>0$ and $0<q<1$ such that $\theta q\le 1$ and $$\sum\limits_{g \in  T_{n_0}} e^{-\omega d(o, go)} < q$$ 
where $\theta>0$ is the implied constant given by Lemma \ref{KeyInequalitiesLem}. For each $k\ge 0$ and $1\le i\le n_0$, we obtain from (\ref{TnSeriesCstEQ}, \ref{TnmSeriesEQ})  that    
$$
\begin{array}{ccc}
\sum\limits_{k\ge 1}\left(\sum\limits_{v \in   T_{kn_0}} e^{-\omega d(o, v)}\right) 
&<   \sum\limits_{k\ge 1} \left(\theta \sum\limits_{v \in   T_{n_0}} e^{-\omega d(o, v)}  \right)^k  &\le  \sum\limits_{k\ge 1} (\theta q)^k <\infty
\end{array}
$$
and
$$
\sum\limits_{v \in T_{kn_0+1}} e^{-\omega d(o, v)}< C \left(\sum\limits_{v \in T_{kn_0}} e^{-\omega  d(o, v)}\right) 
$$
By (\ref{TnPSeriesEQ}), we obtain that $\p_{G}(\omega,o,o)$ is convergent. 

Recall  the usual conical points $\Lambda_c(G)$  are defined in (\ref{ConicalEQ}) using the usual shadow $\Pi_x(y, r)$ without index $F$: $$
\Lambda_{c}(Go) :=\bigcup_{r\ge 0}\bigcup_{x\in Go} \left(\limsup_{y\in Go} \Pi_x(y, r)\right).$$
It is clear that $\Lambda_{r}^F(Go)$ is a subset of $\Lambda_{c}(Go)$, so $\mu_1(\Lambda_c(G))=0$   implies  $\p_{G}(\omega,o,o)<\infty$.

\subsubsection*{(\ref{itm:4})  \texorpdfstring{$\Longrightarrow$}{} (\ref{itm:1})} If   $G$ is of convergent type, then the set $\Lambda_{r}^F(Go)$ is $\mu_1$-null. 

Indeed,  assume that $\mu_1(\Lambda_r^F(Go))>0$ and we are going to prove that $\p_{G}(\omega, o,o)$ is convergent. By Corollary \ref{FixLightSource}, we could fix the light source at $o$ and increase $r$, so that the following family of sets $$\Lambda_n:=  \bigcup\limits_{k\ge n}\Bigg(\bigcup\limits_{v \in   A(o, k,\Delta) }  \Pi_{o}^F(v,r)\Bigg)$$ tends to $\Lambda_{r}^F(Go)$, where $A(o,n,\Delta)$ is defined in (\ref{AnnulusEQ}). Thus, $\mu_1(\Lambda_n)\ge \frac{1}{2}\mu_1(\Lambda_r^F(Go))$ for $n\gg 0$.
By  Lemma    \ref{ShadowLem}, we obtain a uniform lower bound of the partial sum  of Poincar\'e series $\p_{G}(\omega,o,o)$ for any $n\gg 0$
$$
\bigcup\limits_{k\ge n}\Bigg(\bigcup\limits_{v \in   A(o, k,\Delta) }    e^{-\omega n}\Bigg) \succ \mu_1(\Lambda_n)
$$   contradicting to the convergence of $\p_{G}(\omega,o,o)$. 

To prove that $\Lambda_c(G)$ is $\mu_1$-null, we re-define $\Lambda_n$ as in the exact form as above where $\Pi_{o}^F(v,r)$ is replaced by $\Pi_{o}(v,r)$. As the shadow Lemma    \ref{ShadowLem} works exactly for $\Pi_{o}(v,r)$, the remaining argument proves in the same way that $\Lambda_{c}(Go)$ is $\mu_1$-null.  

\subsubsection*{(\ref{itm:4})  \texorpdfstring{$\Longrightarrow$}{} (\ref{itm:5})} If   the set $A:=\Lambda_r^F(Go)$ or $A:=\Lambda_c(G)$   is $\mu_1$-positive then  it is $\mu_1$-full in $\mathcal C$: $\mu_1(A)=\mu_1(\mathcal C)$. 

Suppose to the contrary  that  the complement  $\mathcal C\setminus A$ is $\mu_1$-positive. The restriction of $\{\mu_x\}_{x\in \U}$ on $\mathcal C\setminus A$   gives a conformal density $\{\nu_x\}_{x\in \U}$ of the same    dimension $\omega$, so that $\nu_1(\Lambda_r^F(Go))=0$. By   ``\textbf{(\ref{itm:1}) $\Longrightarrow$ (\ref{itm:4})}", we have that $G$ is of convergent type: this in turn contradicts the ``\textbf{(\ref{itm:4}) $\Longrightarrow$ (\ref{itm:1})}". 

\subsubsection*{(\ref{itm:5})  \texorpdfstring{$\Longrightarrow$}{} (\ref{itm:3})} Applying ``(\ref{itm:1}) $\Longrightarrow$  (\ref{itm:4})",  we have that for every $\f$ in (\ref{SystemFDefIntro}),     the set $\Lambda_r^F(Go)$ is   $\mu_1$-full. By Corollary \ref{MyrbergConical}, their countable intersection gives the Myrberg  set.

The proof of Theorem \ref{HTSThm} is complete.

\section{Conformal and harmonic  measures on the reduced Myrberg  set}\label{SecUniqueness}
This section investigates  the uniqueness of conformal density with   dimension at the critical exponent for divergence group actions. Our goal here is to  prove Theorem \ref{UniqueConfThm} from Introduction, stated below as  Lemma \ref{Unique}, that conformal measures shall be unique and ergodic, when pushed forward to the reduced horofunction boundary.     

Suppose that $G\act \U$ as in (\ref{GActUAssump}) is of divergent type, where $\U$ is compactified by a convergence boundary $\pU$.  
Let $\{\mu_x\}_{x\in \U}$ be a $\omega$-dimensional $G$-quasi-equivariant quasi-conformal density on $\pU$ with $\lambda\ge 1$ given in Definition \ref{ConformalDensityDefn}.  

According to the Hopf--Tsuji--Sullivan dichotomy (Theorem \ref{HTSThm}), as $G\act \U$   is of divergent type,  the conformal measures $\mu_x$ are fully charged on the set of conical points or even the smaller subset of Myrberg points.

By Lemma \ref{MyrbergGood} and \ref{GoodPartition}, the finite difference relation restricted on $\mG$ induces a closed surjective map from $\mG$ to $[\mG]$ with compact fibers. This implies that $[\mG]$ is metrizable and second countable. 

\subsection{Push forward conformal density to the quotient}
According to Lemma \ref{EmbedConical}, there exists a topological embedding
$$
[\mG] \longrightarrow \partial \QT
$$
so we could endow $[\mG]$ with the visual metric on the Gromov boundary $\partial \QT$. 
Push forward $\{\mu_x\}_{x\in \U}$ to $[\mG]$ via the quotient map 
$$
\mG \longrightarrow [\mG]
$$ We obtain a $\e G$-dimensional $G$-quasi-equivariant quasi-conformal density $\{[\mu_x]\}_{x\in \U}$ on $[\mG]$, where $[\mu_x]$ denotes the image measure of $\mu_x$. The constant $\lambda'$ now depends on $\lambda$ and $C$, as any two horofunctions in a $[\cdot]$-class of Myrberg points has $(20C)$-difference by Lemma \ref{UnifDiffLem}.  Here the contracting constant $C$ could taken to be the infimum over all $\f$ in (\ref{SystemFDefIntro}).  

The proof consists in  releasing conformal measures as a Hausdorff-Caratheodory measure on $[\mG]$. To that end, we first state a variant of Vitali covering lemma in \cite[Appendix A]{YANG7}.  


Let $B \in \mathbb B$ be a family  of subsets in a topological space $X$ with  extensions $B\subseteq \tilde B$ of each $B$. Then $\mathbb B$  is called \textit{hierarchically locally finite}  if there exist a positive integer $N>0$ and a disjoint decomposition $\mathbb B=\bigsqcup \mathbb B_n$   for $n\ge 0$ where  $\mathbb B_n$ is possibly empty for certain $n$ such that  the following two properties are true:
\begin{enumerate}
\item
Given $B\in \mathbb B_n$ for $n\ge 0$,   the set $\mathcal N(B)$ of $A\in\mathbb B_{n-1} \cup \mathbb B_n\cup \mathbb B_{n+1}$ with  $A\cap B\ne \emptyset$ has cardinality at most  $N,$ where $\mathbb B_{-1}:=\emptyset$.
\item
If $A\in \mathbb B_n$ intersects $B\in \mathbb B_m$ for $n\ge m-2$,  then $A$  is contained in the extension $\tilde B$ of $B$.  
\end{enumerate}


By definition, any subfamily of $\mathbb B$ is   hierarchically locally finite.

\begin{lem}\cite[Lemma A.1]{YANG7}\label{covering} 
Let $\mathbb B$ be a hierarchically locally finite family of subsets
in a space $X$. Then there exists a sub-family
$\mathbb B' \subseteq \mathbb B$ of pairwise disjoint subsets such
that  
\begin{equation}\label{disjointcover}
\bigcup\limits_{B \in \mathbb B} B \subseteq \bigcup_{A \in
\mathbb B'}  \left(\bigcup_{B\in \mathcal N(A)} \tilde B \right).
\end{equation}
\end{lem}



\subsection{Hierarchically locally finite covers}

In order to use Lemma \ref{covering}, we define the  family    $\mathbb B$ of sets  $B(v)$ and their \textit{extension} $\tilde B(v)$ as follows
$$B(v):=[\Pi_o^F(v,  r)], \quad \tilde B(v):=[\Pi_o^F(v,  \hat r)]$$ for every $v\in Go$. Here $r, \tilde r$ are given by Convention \ref{ConvF}.

 By abuse of language, we define the \textit{diameter} of it: $\|B(v)\|:=e^{-d(o, v)}$, even though there may   exist no metric on $\pU$ inducing it.    
 A  subfamily of $\mathbb B'\subseteq \mathbb B$  is called an \textit{$\epsilon$-covering}, if $\|B\|\le \epsilon$ for every $B\in \mathbb B'$.  
\begin{lem}
For any $v\in Go$, we have
\begin{equation}\label{shadowdiskEQ}
\mu_1(B(v)) \asymp  \| B(v)\|^\omega.  
\end{equation}    
\end{lem}
\begin{proof}
The lower bound on $\mu_1(B(v))$ follows directly from the shadow lemma \ref{HoroShadowLem}. For the upper bound, $B(v)$ contains the locus of generic points in $\Pi_o^F(v,  r)$. Since $G$ is of divergent type, the measure $\mu_1$ is  supported on $\Lambda_r^F(Go)$, so by Lemma \ref{ClosedFinDiff}, any two horofunctions in the same locus of points in $\Pi_o^F(v,  r)$ have uniform difference. Thus, the proof of  Lemma \ref{ShadowLem} for the upper bound works to give an uniform coefficient before $e^{-d(o,v)}.$
\end{proof}

We define a Hausdorff-Carath\'eodory outer measure $\mathcal H^\omega$ on $[\mG]$ relative to $\mathbb B$ and the gauge function $r^\omega$. Namely, for any $\epsilon>0$ and $A\subset [\mG]$,
\begin{equation}\label{HausdorffMeasure}
\mathcal H_\epsilon^\omega(A)=\inf\Bigg\{\sum_{B\in\mathbb B'} \|B\|^\omega: A\subseteq \bigcup_{B\in \mathbb B'} B,\;  \|B\|\le \epsilon \Bigg\}    
\end{equation}
where the infimum is taken over $\epsilon$-cover $\mathbb B'$ over $A$. Such $\epsilon$-covers exist for any $\epsilon$, since   Myrberg points  and   conical points are  contained in infinitely many shadows.  
Then $$\mathcal H^\omega(A) :=\sup_{\epsilon>0} \mathcal H_\epsilon^\omega(A).$$

We now prove  that $\mathcal H^\omega$ is a \textit{metric outer measure}. Thus, Borel subsets are $\mathcal H^\omega$-measurable.  
\begin{lem} 

If $A, B$ are disjoint closed subsets in $[\mG]$, then $$\mathcal H^\omega(A\cup B)=\mathcal H^\omega(A )+\mathcal H^\omega( B)$$\end{lem}
\begin{proof}
By Lemma \ref{EmbedConical}, $[\mG]$ is topologically embedded into the Gromov boundary $\partial\QT$. Let $A, B$ be disjoint closed subsets in $\mG$. If $\partial\QT$ is equipped with visual metric, it is readily checked that the corresponding  $\epsilon$-covers for $A$ and $B$ are disjoint for small enough $\epsilon$. The details is left to the interested reader.  Consequently, the conclusion follows from the definition of Hausdorff measures.
\end{proof}



Fix a large  $L>0$, and subdivide $\mathbb B$ into a sequence of (possibly empty) sub-families $\mathbb B_n$ for $n\ge 0$:  $$\mathbb B_n:=\{B(v)\in \mathbb B:  nL\le d(o, v)< (n+1)L\}.$$

\begin{lem}\label{coveringlemma}
There exists a uniform constant $L>0$ such that  $\mathbb B$  defined as above is a hierarchically locally finite family over $[\mG]$.
\end{lem}
\begin{proof}
We verify the two conditions in order. Let $B(v)=[\Pi_o^F(v,   r)]\in \mathbb B_n$. Applying  Lemma \ref{OverlapAnnulus} to the annulus $A(o, n, 2L)$, there are at most $N=N(L)$   
sets $$B(w)   \in \mathbb B_{n-1}\cup \mathbb B_n\cup \mathbb B_{n+1}$$ which intersect $B(v)$. 
 
Let $B(w) \in \mathbb B_m$ for $m<n$ intersecting  $B(v)$, so $d(v, w)\ge d(o, v)- d(o, w)\ge L$. 
 We have   $B(v)\subseteq \tilde B(w)=[\Pi_o^F(w,   \hat r)]$ by Lemma \ref{IntersectShadow}. Both conditions are thus verified and the lemma is proved. 
\end{proof}

\subsection{Uniqueness of conformal measures} 

Assuming the action is of divergent type, we are now ready to prove the main result in this Section Theorem \ref{UniqueConfThm} restated below, saying that the push-forward of PS measures becomes unique on the reduced Myrberg  set.

\begin{lem}\label{Unique}
Let $\{[\mu_x]\}_{x\in \U}$ be a $\e G$-dimensional $G$-quasi-equivariant quasi-conformal density    on the quotient  $[\mG]$. Then we have
$$
\mathcal H^{\omega}(A) \asymp [\mu_1](A)
$$
for any Borel subset $A \subseteq [\mG]$, where the implied constants depend on the constant $\lambda$ in Definition \ref{ConformalDensityDefn}. 

In particular, $[\mu_x]$ is unique in
the following sense: if $[\mu], [\mu']$ are two such quasi-conformal densities, then the
Radon-Nikodym derivative $d[\mu_o]/d[\mu_o']$ is bounded from above and below by a constant depending only on    $\lambda$.  
\end{lem}
 
\begin{proof}
Take an $\epsilon$-covering $\mathbb B' \subseteq \mathbb B$  of $A$, so that  $[\mu_1(A)] \le \sum_{B \in
\mathbb B'} [\mu_1](B)$. Letting $\epsilon \to 0$, we obtain from (\ref{shadowdiskEQ}) that $[\mu_1](A) \prec \mathcal H^{\sigma}(A)$.

To establish the other inequality,  for any $0 <
\epsilon < \epsilon_0$, let $\mathbb B_1 \subseteq \mathbb B$ be an
$\epsilon$-covering of $A$. By Lemma \ref{covering},   there exist a disjoint sub-family
$\mathbb B_2$ of $\mathbb B_1$ and  an integer $N>1$  such
that  
$$
\mathcal H^\omega(A)  \prec
\sum_{B \in \mathbb B_2} N \cdot \|{B}\|^\omega \prec  [\mu_1](U),
$$
  Letting
$\epsilon\to 0$ yields $\mathcal H^\omega(A) \prec [\mu_1](A)$.
\end{proof}

As in the usual Hausdorff measure, we define the \textit{Hausdorff dimension} $\HD(A)$ of a subset $A\subseteq [\mG]$: 
\begin{equation}\label{HDimDef}
\HD(A) := \sup\{\omega\ge 0: \mathcal H^\omega(A)=\infty \}=\inf\{\omega\ge 0: \mathcal H^\omega(A)=0 \}    
\end{equation}
where $\sup\emptyset :=\infty$ and $\inf\emptyset :=0$ by convention. The following corollary is immediate.
\begin{cor}\label{HDimCor}
The Hausdorff dimension of $[\mG]$ equals $\e G$.
\end{cor}

We also obtain the ergodicity from the uniqueness.
\begin{lem}\label{Ergodic}
Under the assumption of Lemma \ref{Unique}, let $\{[\mu_x]\}_{x\in \U}$ be a $\e G$-dimensional $G$-quasi-equivariant quasi-conformal density  
  on $[\mG]$. Then $\{[\mu_x]\}_{x\in \U}$ is ergodic with respect to the action of $G$ on $[\mG]$.
\end{lem}
\begin{proof}
Let $A$ be a $G$-invariant Borel subset in $[\mG]$ such that $[\mu_x](A)
>0$. Restricting $[\mu_x]$ on $A$ gives rise to a $\e G$-dimensional conformal density $\nu_x:=[\mu_x]|_{A}$  on $[\mG]$. By Lemma \ref{Unique}, for any subset
$X \subseteq [\mG]$, we have $\nu_x(X) \prec \mathcal H^\omega(A \cap X)$.  Thus, $[\mu_x]([\mG] \setminus A)=0$.
\end{proof}

\subsection{Myrberg  set as Poisson boundary }\label{ProofPoissonBdry}
In this subsection, we explain that the reduced Myrberg set could serve as the Poisson boundary for random walks on groups with contracting elements. This is the content of Theorem \ref{PoissonBdry}.

Let $G\act \U$ be a proper action with contracting elements as in (\ref{GActUAssump}) with a convergence boundary $\pU$. Let $\mu$ be a probability measure on $G$ so that its support generates $G$ as a semigroup. This defines a $\mu$-random walk on the countable state set $G$ as follows. 

Let $\{s_n \in G:n\ne 1\}$ be the sequence of  identical independent distributed (i.i.d.) random variables with law $\mu$ defined on $(G^{\mathbb N}, \mu^{\mathbb N})$. Formally, we define $s_n$ to the $n$-th projection of $G^{\mathbb N}$ to $G$.  The $\mu$-random walk $\omega_n = \prod_{i=1}^n s_i$ (starting at identity) is then product of i.i.d. random variables, recording the random position of the walker with i.i.d. increments $s_i$. The law of $\omega_n$ is given by the $n$-th convolution $\mu^{\star n}$. 

We look at the trajectory of $\omega_n$ on the space $\U$, the orbital points $\{\omega_n\cdot o\}$,  and  assume that the  $\mu$-random walk has   finite   logarithmic moment with respect to the metric on $\U$:
$$\sum_{g\in G}\mu(g)\cdot |\log d(o, go)|<\infty.$$

Let $F$ be a set of three independent non-pinched contracting elements in $G$. Let $\f$ as in (\ref{SystemFDefIntro}) be the associated $G$-translated axes of elements in $F$. By \cite[Theorem 5.10]{BBFS}, the action of $G$ on the projection complex $\PC$ is acylindrical  (see \cite[Thm 6.3]{GYANG} for a detailed proof).  

By the construction of projection complex, the orbital map $$go\in \U\longmapsto g\bar o\in \PC$$ is clearly a Lipschitz map. In particular, if a $\mu$-random walk on $\U$ has   finite   logarithmic moment, then   the induced random walk on $\PC$ has the finite  logarithmic moment  as well.  \cite[Theorem 1.5]{MT} then says that almost every trajectory on $\PC$ tends to a point in the Gromov boundary of $\PC$. 

If   $H(\mu)<\infty$, the Gromov boundary of $\PC$    endowed with hitting measure $\mu$ is the Poisson boundary for the $\mu$-random walk on $G$. Via Lemma \ref{EmbedGromov}, we obtain that almost every trajectory on $\U$ tends to the locus of a conical point in $\hU$.

We vary   $F$ over all such sets as above  in $G$. 
By   Corollary \ref{MyrbergConical},    the countable intersection of the embeddings of   Gromov boundaries $\partial\PC$ gives the Myrberg  set $\mG$.   Hence, we conclude  that almost every trajectory on the original space $\U$ converges to the locus of a Myrberg point.

By Lemma \ref{VisualMetriconMyrberg}, $[\mG]$ is a   second countable metrizable space. As the hitting measure $\nu$ is supported on $\mG$,  the     partition $[\cdot]$ thus forms  a \textit{measurable partition} in sense of Rokhlin: there are a countably many of measurable sets $B_n$ in $\pU$, which are union of $[\cdot]$-classes,   separating almost every two $[\cdot]$-classes (\cite[Def. 15.1]{C16Book}). 

By Rokhlin's theory, the measurable quotient  space $([\mG], [\nu])$  by the above discussion is isomorphic to the Poisson boundary with hitting measure. The proof of Theorem \ref{PoissonBdry} is therefore complete.

\begin{rem}\label{NoMomentCondtion}
Using the recent work \cite{CFFT} \footnote{appearing after the first version of this paper}, the finite  logarithmic moment condition could be removed, still under the finite entropy. Indeed, there the authors proved that, without moment condition, the Gromov boundary is identified with Poisson boundary for the random walk for the acylindrical action on a hyperbolic space. The above argument works without changes modulo this ingredient.     
\end{rem}

\section{Cogrowth of normal subgroups in  groups of divergent type}\label{SecCogrowth}
From this section, we begin to  present the applications of results established in previous sections. The first application is Theorem \ref{CogrowthThm} from Introduction, which is a general result concerning the co-growth tightness   of normal subgroups.
\begin{thm}\label{SecCogrowthThm}
Assume that $G\act \U$ is a   non-elementary group of divergent type with a contracting element. Then for any non-elementary normal subgroup $H$ of $G$, we have $$\e H>\frac{\e G}{2}.$$
Moreover, if $H$   is of divergent type, then $\e H=\e G$.
\end{thm}
The strict inequality ``$\e H>\frac{\e G}{2}$" is sometimes referred to as cogrowth tightness in literature.  

The remainder of this section is devoted to the proof, which are subdivided into the following two cases.  
\subsection{Normal   subgroups are of divergent type} Assume that $H\act \U$ is of divergent type.  We are going to prove that $\e H=\e G$. The proof   presented here is inspired by the one in \cite{MYJ20}.

Let $ \{\mu_x\}_{x\in\U} $ be a PS measure on the horofunction boundary $\hU$, which is an accumulation point of $\{\mu_{x,s}^o\}$ in (\ref{PattersonEQ}) supported on $Ho$ for a sequence $s\searrow \e H$. This is a $\e H$-dimensional $H$-equivariant conformal density by Lemma \ref{confdensity}. That is,  $\lambda=1$ and $\epsilon=0$   in Definition \ref{ConformalDensityDefn} and \ref{AssumpE}. 

Keep in mind that all the Patterson construction below uses the action $H\act \U$ instead of $G\act \U$, so gives (quasi)conformal density for $H$ (rather than for $G$).  

The goal of proof is providing       a constant $M>0$ such that    
\begin{equation}\label{BoundedMassEQ}
\forall x\in Go: \quad \|\mu_x\|\le M    
\end{equation}
Once this is proved, the shadow principal     \ref{ShadowPrinciple} implies that 
$$
\mu_1([\Pi_o^F(go, r)]) \asymp e^{-\e H d(o, go)}$$
for any $go\in Go$. The same argument as in Proposition \ref{ballgrowth} proves $\e H\ge \e G$. Thus the remainder of the proof  is to prove (\ref{BoundedMassEQ}). For this purpose, for given $go$, we consider another PS-measure $ \{\mu_x^{go}\}_{x\in\U} $ on $\hU$   as an accumulation point of $\{\mu_{x,s}^{go}\}$ in (\ref{PattersonEQ}) supported on $Hgo$ for a sequence $s\searrow \e H$.
 
As $H\act \U$ is of divergent type, the Patterson's construction  gives 
$$
\forall x\in \U:\quad \|\mu_x^{go}\| = \lim_{s\to \omega+} \frac{\p_H(s, x, go)}{\p_H(s, o, go)},\quad \|\mu_x^{o}\| = \lim_{s\to \omega+} \frac{\p_H(s, x, o)}{\p_H(s, o, o)},
$$
From the definition of Poincar\'e series (\ref{PoincareEQ}), we have 
$$
\forall s>0:\quad \frac{\p_H(s, go, go)}{\p_H(s, o, go)} =  \Bigg[ \frac{\p_H(s, go, o)}{\p_H(s, o, o)}\Bigg]^{-1}
$$
For given $go\in Go$, let us take the same sequence of $s$ (depending on $go$) so that   $\{\mu_{x, s}^{o}\}_{x\in\U}$  and $\{\mu_{x, s}^{go}\}_{x\in\U}$  converge to $ \{\mu_x^o\}_{x\in\U} $ and $ \{\mu_x^{go}\}_{x\in\U} $  respectively, so 
\begin{equation}\label{MassEqualityEQ}\|\mu_{go}^{go}\|= \|\mu_{go}^o\|^{-1}
\end{equation}

Push forward the limiting measures $ \{\mu_x^o\}_{x\in\U} $ and $ \{\mu_x^{go}\}_{x\in\U} $  to  the  quasi-conformal densities on the quotient $[\partial_m H]$, which we keep by the same notation.  The constant $\lambda$ in Definition \ref{ConformalDensityDefn} is universal for any $ \{\mu_x^{go}\}_{x\in\U} $ with $go\in Go$, as the difference of two horofunctions in the same locus of a Myrberg point is universal by Lemma \ref{UnifDiffLem}. Thus, Lemma \ref{Unique}  gives a constant $M=M(\lambda)$ independent of $go$ such that 
for any $go\in Go$: $$M^{-1} \|\mu_{go}^{go}\| \le \|\mu_{go}^o\|\le M \|\mu_{go}^{go}\|$$
so   the relation (\ref{MassEqualityEQ}) implies $\|\mu_{go}^o\|^2\le M$. 

It still remains to note that $\{\mu_x\}$ and $\{\mu_x^o\}$ may be different limit points of $\{\mu_{x, s}^{o}\}_{x\in\U}$, where $\{\mu_x\}$ is the PS measure fixed at the beginning of the proof. However, applying Lemma \ref{Unique} again gives  $\|\mu_{x}\|\asymp_\lambda \|\mu_{x}^o\|$ for any $x\in \U$ and thus for $x=go$.  As desired in (\ref{BoundedMassEQ}), we obtain a upper bound on $\|\mu_{go}\|$ depending only on $\lambda, M$. The proof of the Case 1 is completed.

\subsection{Normal   subgroups are of convergent type} Assume that $H\act \U$ is  of convergent type. It suffices to prove $$\omega(H)> \frac{\omega(G)}{2}.$$

Indeed, if denote $A_n=A(o,n,\Delta)$ defined in (\ref{AnnulusEQ}), we have 
$$
\p_G(\omega, o,o) \asymp_\Delta  \sum_{n\ge 0} |A_{n}|e^{-n\omega }.
$$
A simple  exercise shows that for any $\Gamma\subseteq G$, we have
\begin{equation}\label{AnnulusGrowthEQ}  
\limsup_{n\to \infty}\frac{\log \sharp \left(A(o, n, \Delta)\cap \Gamma\right)}{n} = \e \Gamma
\end{equation}

Let $r, F>0$ given by  Lemma \ref{extend3} and    $R=\|Fo\|+4r+4\Delta$.
Let  $B_n$ be a maximal $R$-separated subset of $A_n$:
\begin{enumerate}
\item
For any $v\ne w\in B_n$ we have $d(v,w)>R$
\item
For any $w\in A_n$ there exists $w\in B_n$ such that $d(v,w)\le R.$
\end{enumerate}

The idea of the proof is to produce at least $|B_n|$ elements in $H\cap A(o, {2n}, R)$:
\begin{equation}\label{HAnnulusEQ}  
  |B_n|\le |H\cap A(o, {2n}, R)|
\end{equation}  

Indeed, this follows from \cite[Lemma 2.19]{YANG10}. As it    is short,   let us give the proof at the convenience of reader.   

Apply Lemma \ref{extend3} to the pair $(g, g^{-1})$ for each $g o\in B_n$. There exists $f\in F$ such that the admissible path labeled by $gfg^{-1}$ $r$-fellow travels a geodesic $\gamma=[o, gfg^{-1}o]$  so    $d(go, \gamma)\le r$ and
$$
|d(o, gfg^{-1}o)-2d(o, go)|\le \|Fo\|+4r. 
$$
If $gfg^{-1} o=g'f'g'^{-1}o$ for two $go\ne g'o\in B_n$,  then $d(go, \gamma), d(g'o, \gamma)\le r$. As $go, g'o\in A_n$ lies in the same annulus of width $\Delta$, we see that $d(go, g'o)\le 2r+2\Delta$. This contradicts to the item (1). Hence,  (\ref{HAnnulusEQ}) is proved and by (\ref{AnnulusGrowthEQ}), $$\e H\ge  \frac{\e G}{2}.$$
 By the item (2), $|A_n|\asymp_R |B_n|$. If    $\e H= \frac{\e G}{2}$, then  
$$
\p_G(\e G, o,o) \asymp_R  \sum_{n\ge 1} |B_{n}|e^{-n\e G }\le \sum\limits_{n\ge 1} |A(o, 2n,R)\cap H|e^{-\e H n} <\infty
$$ 
where the last inequality follows from the convergence of $\p_H(\e H, o,o)$ at $\e H$.
This contradicts  the divergent action of $G\act \U$. Hence, the proof of the second case is completed.

\section{Relatively hyperbolic groups: Case Study (I)}\label{SecRHG}
This section is to illustrate some aspects of the convergence boundary through  Floyd and Bowditch boundaries of relatively hyperbolic groups. The latter two boundaries were serving as a main motivation to the introduction of \ref{AssumpA}, \ref{AssumpB} and \ref{AssumpC} of a convergence boundary. 

Most materials presented here are not really new to the experts in the field, though   the discussion concerning  the   Myrberg sets in this setup seems not be found in literature.  In this sense, a new result that we shall establish, Theorem \ref{HomeoMybergRHG}, is  that the Myrberg set in Floyd boundary is homeomorphic to the one in Bowditch boundary. Consequently, the Myrberg set is  invariant among the peripheral structures endowed on a fixed relatively hyperbolic group. The same conclusion holds for Cannon--Thurston map (Theorem \ref{HomeoMybergCT}).

\subsection{Convergence group actions}\label{SSecConvAction}

We start by  introducing the notion of convergence group actions, which sets the stage for further discussion. We refer the reader to \cite{Bow2, Tukia2} for more details and relevant discussion.

A  group action of $G$    by
homeomorphisms on  a compact metrizable space $M$ is called  \textit{convergence group action} if the diagonal action of $G$ on the triple space $\Theta^3(M)=\{(x,y,z)\in M^3:x\ne y\ne z\ne x \}$ of $M$ is properly discontinuous: for any compact set $K$ in $\Theta^3(M)$, the following set 
$$
\sharp \{g\in G: gK\cap K\ne\emptyset\}<\infty
$$ 
The prototype example of a convergence action is the boundary action on the Gromov boundary induced from a proper action on the hyperbolic space, and the action on the Floyd boundary introduced in the next subsection.   

By the general theory of convergence groups,   the elements in $G$ are subdivided into the classes of \textit{elliptic}, \textit{parabolic} and \textit{hyperbolic} elements. The latter two  are  infinite order  elements having exactly one and two fixed points in $M$ accordingly. We denote by $h^-, h^+$ the \textit{repelling} and \textit{attracting} fixed points of a hyperbolic element $h$. 
\begin{lem}\label{HypinConv}
Let $h\in G$ be a hyperbolic element with respect to $G\act M$. Then
\begin{enumerate}
    \item 
    $\langle h\rangle $ acts properly and compactly outside $M\setminus h^\pm$. In particular, 
    \begin{itemize}
        \item 
        it admits North--South dynamics as in Lemma \ref{SouthNorthLem} (with $[\cdot]$ being singleton).
        \item 
        the set stabilizer of $h^\pm$ is the maximal elementary subgroup containing $h$. 
    \end{itemize}
    \item 
    The fixed points of two hyperbolic elements are either disjoint or equal.
    \item 
    If $h, k$ are two hyperbolic elements with disjoint fixed points, then $g_n=h^nk^n$ for $n\gg 0$ is hyperbolic, so that the  fixed points $g_n^+$ and $g_n^-$ tend  to $h^+$ and $k^-$ as $n\to \infty$.
\end{enumerate}   
\end{lem}
\begin{proof}
These are  proved by Tukia \cite{Tukia2}: Theorem 2G, Theorem 2I and Lemma 2J.    
\end{proof}

The \textit{limit set} $\Lambda H$ of a subgroup $H<G$ is the set of accumulation points of all $H$-orbits in $M$. If the limit set $\Lambda H$ contains more 3 points, then it is the unique minimal $H$-invariant closed subset. The limit set contains the essential dynamical information, on which $H$  restricts as a convergence group action. 

We hereafter assume that $\sharp \Lambda G\ge 3$ and $G$ admits a minimal convergence action on $M$ by restricting on its limit set. That is, $M=\Lambda G$ consists of limit points.


We introduce  several classes of limit points, which shall be used later on.
\begin{defn}\label{LimitPointsDefn}
Let $\dG $ denote the set of unordered pairs of distinct points in $\Lambda G$. 
\begin{enumerate}
\item
A point $\xi \in \Lambda G$ is called \textit{conical}  if there exists a
sequence of elements $g_n \in G$ ($n\ge 1$)  such that the closure of $\{g_n(\xi, \eta): n\ge 1\}$ in $\dG $ is disjoint from the diagonal $\Delta(\Lambda G)=\{(x, x): x\in \Lambda G\}$
for any $\eta \in \Lambda G \setminus \xi$.

\item
A point $\xi \in \Lambda G$ is called \textit{Myrberg point}  if for any $\eta \in \Lambda G \setminus \xi$, the set of $G$-orbits of $(\eta, \xi)$ is dense in the set of distinct pairs  $\dG$.

\item
A point $\xi \in \Lambda G$ is called \textit{bounded parabolic} if the
stabilizer $G_\xi$ of $\xi$ in $G$ is infinite, and acts properly
and co-compactly on $\Lambda G\setminus \xi$. The subgroup $G_\xi$ is called  {\it maximal parabolic subgroup}.
\item
A convergence group action of $G$ on $M$ is called \textit{geometrically
finite} if every limit point $\xi \in \Lambda G$ is either a conical point or a
bounded parabolic point.
\end{enumerate}

\end{defn}
\begin{rem}
\begin{enumerate}
    \item 
    We demand the dense property in the space of distinct pairs only for symmetric reasons: if it is  dense in $\Lambda G\Join \Lambda G$, it is so in the product $\Lambda G\times \Lambda G$. 
    \item 
    According to definition, a Myrberg point must be a conical point, but the converse is  false in general. By the South-North dynamics, the fixed points $h^\pm$ of a hyperbolic element $h\in G$ are conical. However, $h^\pm$ are not Myrberg if  the  two fixed points $\{h^-,h^+\}$ of a hyperbolic element $h\in G$ forms a \textit{null sequence} under the action of $G$: only finitely many pairs have diameter bigger than any fixed  constant.   For example, this property holds for the convergence action on Gromov boundary induced from a proper action on a hyperbolic space. See   Lemma \ref{ContrMyrberg} for the similar results concerning about fixed points of contracting elements.
\end{enumerate}    
\end{rem}
In practice, it suffices  to approach the fixed points of all hyperbolic elements in the above definition of  a Myrberg point.
\begin{lem}\label{MyrbergCharinConv}
A point $\xi\in M$ is a Myrberg point if and only if for any hyperbolic element $h\in G$  and for any $\eta\ne \xi$, there exists $g_n\in G$ so that $g_n(\eta,\xi)\to (h^-,h^+)$.    
\end{lem}
\begin{proof}
As the limit set is $G$-minimal, the set of fixed points of hyperbolic element is dense in $\Lambda G$.  Thus by Lemma \ref{HypinConv}.(3), the set of pairs of fixed points of all hyperbolic elements is dense in $\Lambda G\times \Lambda G$: if $h, k$ are hyperbolic element, then $g_n=h^nk^n$ for $n\gg 0$ is a hyperbolic element so that  $g_n^+$ and $g_n^-$ tends to $h^+$ and $k^-$ respectively.     
\end{proof}

In the next lemma, we could restrict to a smaller class of hyperbolic elements.
\begin{lem}\label{FindDenseHypElemts}
Suppose that $\Phi: G\act M \to G\act N$ is a $G$-equivariant map between two minimal convergence group actions. Then the following hold,
\begin{enumerate}
    \item The preimage of a conical point $\xi\in N$  consists of one conical point. 
    \item 
    Let $A$  be the set of all hyperbolic elements with respect to the action  $G\act N$. Then    the pairs of fixed points of $h\in A$  are  dense in $M\Join M$.
\end{enumerate}   
\end{lem}
\begin{proof}
We only prove the assertion (2) here; the assertion (1) is due to \cite[Prop. 7.5.2]{Ge2} and \cite[Lemma 2.3]{MOY13}. By the classification of elements,  a hyperbolic element in $G\act N$ is also hyperbolic  with respect to  $G\act M$, for an infinite order element cannot have three fixed points. Let $\Lambda\subseteq M$ be the   union of  fixed points of hyperbolic elements in $A$ with respect to the convergence action $G\act M$. 

Let $U,V$ be two open neighborhoods of $\xi$ and $\eta$ respectively  in $M$. Note that $M$ is the unique minimal $G$-invariant set, so  $\Lambda$ is dense in $M$.  Let us choose $h,k\in A$ so that $h^+\in U$ and $k^-\in V$ are fixed points in $M$. By Lemma \ref{HypinConv}.(3), the element  $g_n:=h^nk^n$ for $n\gg 0$ is hyperbolic with respect to $G\act N$ (not being considered w.r.t. $G\act M$), so $g_n$ is contained in $A$. At the same time,   their two fixed points of $g_n$ in $M$ tend to $\xi$ and $\eta$ as $n\to\infty$. This establishes the assertion (2) and the proof is complete.    
\end{proof}

\subsection{Floyd boundary as convergence boundary}\label{SSecFloyd}
We introduce a compactification of any locally finite graph due to W. Floyd \cite{Floyd}. The Cayley graph of a finitely generated group shall be our main focus. We follows closely the  exposition in      \cite{Ge2}, \cite{GePo2} and \cite{Ka}.

Let $G$ be a group with a finite generating set $S$. Assume that
$1\notin S$ and $S=S^{-1}$.  Let  $\Gx$ denote  the \textit{Cayley graph} of $G$ with respect to $S$, equipped with   the word  metric  $d$. We define a Floyd metric  on $\Gx$ by rescaling the word metric as follows.  

Fix a Floyd scaling function $f: \mathbb N_{\ge 0}\to \mathbb R_{>0}$, satisfying for some $0 < \lambda <1$, so that
\begin{itemize}
    \item $\sum_{n\ge 0} f(n)<\infty$
    \item $\lambda f(n)\le f(n+1)\le  f(n)$
\end{itemize}
The \textit{Floyd length} $\ell_{o,f}(e)$ of an edge $e$ in $\Gx $ is $\lambda^n$, where $n =
d(o, e)$. The Floyd length $\ell_{o,f}(\gamma)$ of a path $\gamma$ is the sum of Floyd lengths of its edges. This induces a length metric $\rho_{o,f}$ on $\Gx$, which is the infimum of Floyd lengths of all possible paths between two points.

Let $\Gf$ be the Cauchy completion of $G$ with respect to $\rho_{o,f}$.
The complement $\pGf$ of $\Gx$ in $\Gf$ is called \textit{Floyd
boundary} of $G$. The $\pGf$ is called \textit{non-trivial} if
$\sharp \pGf>2$. The non-triviality of Floyd boundary does not depend on the choice of generating sets \cite[Lemma~7.1]{YANG7}. Most of groups do have trivial Floyd boundary: we refer to \cite{KN04, Lev20} for relevant discussion. Currently, the  most general class of groups known  with non-trivial Floyd boundary are relatively hyperbolic groups, to be introduced in the next subsection.

In what follows, we shall omit the Floyd function $f$ in the notations $\rho_{o,f}, \ell_{o,f}$ and write $\rho_o, \ell_o$ for simplicity.

By construction, we have the following equivariant property 
\begin{align}
\label{equiv}
\rho_o (x, y) =
\rho_{go}(gx, gy)\\
\label{lambdabilip}
\lambda^{d(o, o')} \le \frac{\rho_{o}(x, y)}{\rho_{o'}(x, y)} \le
\lambda^{-d(o, o')}    
\end{align}
for any two points $o, o' \in G$. So for different basepoints, the corresponding Floyd compactifications are bi-Lipschitz. As a consequence of (\ref{equiv} and \ref{lambdabilip}), one derives that the left-multiplication on $G$ extends to the boundary as bi-Lipschitz homeomorphism.  Note  that the topology  may  depend on the choice of the rescaling function and the generating set.

As fore-mentioned, the action on the Floyd boundary provides an important source of convergence group actions. If $\sharp \pGf\ge 3$, Karlsson proved in \cite{Ka}  that $\Gamma$ acts  by homeomorphism on $\pGf$ as a convergence group action. 
Moreover, the cardinality of $\pGf$ is either 0, 1, 2 or uncountably infinite. By   \cite[Proposition~7]{Ka}, $\pGf= 2$ exactly when the group $\Gamma$ is virtually infinite cyclic.
These are based on the following fact  in   \cite{Ka}, which is the main tool  in understanding the Floyd geometry.
\begin{lem}[Visibility lemma] \label{karlssonlem} 
There is a function $\varphi: \mathbb R_{\ge 0} \to \mathbb R_{\ge 0}$ such that for
any $v \in G$ and any geodesic $\gamma$ in $\Gx$, we have if
$\ell_v(\gamma) \ge \kappa,$ then $d(v,
\gamma) \le \varphi(\kappa)$.
\end{lem}
From this lemma, we could deduce that Floyd boundary is \textit{visual}: any quasi-geodesic ray converges to a boundary point, and any two points $x, y\in G\cup\pGf$ can be connected by a two-sided or one-sided infinite geodesic. See \cite[Prop. 2.4]{GePo2} for a proof.

Let $f\in G$ be a hyperbolic element with two distinct fixed points $f^-,f^+\in \pGf$, so that the action on $\pGf\setminus \{f^\pm\}$ admits South-North dynamics (cf. Lemma \ref{HypinConv}). If $E(f)$ denotes the set stabilizer of the two fixed points $f^\pm$, then it is the unique maximal elementary subgroup containing $f$. Moreover, by \cite[Lemma 7.2]{YANG6} and \cite[Prop 8.2.4]{GePo4}, the quasi-axis $\ax(f)=E(f)$ is contracting in the sense of Definition \ref{ContrDefn}. Compare these results with  the ones in \textsection\ref{SecBoundary} using the sole contracting property. 

The following easy consequence shall be used later on.
\begin{lem}\label{CloseToHypFixedPts}
If $\gamma_n=[x_n,y_n]$ is a sequence of geodesics so that $\|\ax(f)\cap N_r(\gamma_n)\|$ tends to $\infty$ for a large fixed $r>0$, then $x_n\to f^-$ and $y_n\to f^+$ (up to swapping $f^-$ and $f^+$). 
\end{lem} 
\begin{proof}
Let  $u_n\in \ax(f)$ be a shortest projection point of $x_n$ to $\ax(f)$. The $C$-contracting  property of $\ax(f)$ implies that any geodesic $[x_n, z]$ for given $z\in \ax(f)$ intersects the $C$-neighborhood of $u_n$. Let $r>C$. By assumption, $\|\ax(f)\cap N_r(\gamma_n)\|\to \infty$, so we see that $d(1,x_n)\to\infty$. Letting $z\to f^-$, we obtain that $d(1,[z,x_n])\to \infty$, so $x_n \to f^-$ follows from Lemma \ref{karlssonlem}. The proof for $y_n\to f^+$ is completely analogous and left to the interested reader.
\end{proof}

Here is another easy corollary relating the Floyd geometry to the existence of a barrier along a geodesic in Definition \ref{barriers}.
\begin{lem}\label{LargeFloydDistBarrier}
There exists $\kappa>0$ depending on $\ax(f)$ with the following property for any $L\gg 0$.
Assume that a geodesic $\gamma $ contains a   segment $\alpha$   in $N_r(\ax(f))$ with length at least $L$. Let $x$ be the middle point of  $\alpha$. Then $\rho_x(\gamma^-,\gamma^+)>\kappa$.      
\end{lem}

To conclude this subsection, we prove that Floyd boundary is a convergence boundary with maximal partition in  \textsection \ref{SecBoundary}.
We first start with the following fact  due to A. Karlsson.
\begin{lem}\label{HoroFloydMap}
There exists a continuous $G$-equivariant map from $\partial_{\mathrm H} G$ to $\partial_{\mathrm F} G$. 
\end{lem}
\begin{proof}
We show that if $x_n\to \xi$ in the horofunction compactification, then $x_n$ converges to a point $\eta$ in the Floyd compactification. Arguing by contradiction, let us assume (after dropping and re-indexing) the sequences $x_{2n}\to\eta$ and $x_{2n+1}\to \zeta$ for $\eta\ne \zeta$ in the horofunction compactification. By \cite[Prop. 2.4]{GePo2}, there exists a geodesic between any two distinct $\xi$ and $\eta$. This contradicts to the \ref{AssumpC} of the horofunction boundary. Therefore, we obtain a well-defined map from $\partial_{\mathrm H} G$ to $\partial_{\mathrm F} G$. 
\end{proof}

We now verify the following. 
\begin{lem}\label{FloydConvBdry}
The nontrivial Floyd boundary of the Cayley graph of a finitely generated group is a convergence boundary with maximal partition, where conical points are non-pinched.    
\end{lem}
\begin{proof}
Any geodesic ray converges to a limit point in $\pGf$ by \cite[Prop. 2.4]{GePo2}. By Lemma \ref{HoroFloydMap}, the   \ref{AssumpB} for horofunction boundary implies \ref{AssumpB} for Floyd boundary. These verify the \ref{AssumpA} and \ref{AssumpB}. It remains to show that  a conical point is non-pinched. This is a consequence of the fact that there exists no horocycle at a conical point by \cite[Lemma 3.6]{GePo2}. Recall that a horocycle at a boundary point $\xi$ means a bi-infinite geodesic with two half-rays tending to $\xi$.

Indeed, let $x_n, y_n\in G$ tend to a conical point $\xi$. If  $[x_n,y_n]$ for all $n\ge 1$ intersects a fixed ball around the identity $1$, then we could extract a subsequence of $[x_n,y_n]$ still denoted by $[x_n,y_n]$ so that $[x_n,y_n]$ converges locally uniformly to a bi-infinite geodesic $\gamma=[\gamma^-,\gamma_+]$. We claim that $\gamma$ would be a horocycle at $\xi$; that is, $\gamma^-=\gamma_+=\xi$. This gives the desired contradiction to \cite[Lemma 3.6]{GePo2}.  To this end, consider the common overlap $[u_n,v_n]=[x_n,y_n]\cap \gamma$. The length of $[u_n,v_n]$ tends to $\infty$, so $\gamma\setminus [u_n,v_n]$ leaves every compact set, so $\rho_1(u_n,\gamma^-)\to 0$ and  $\rho(v_n,\gamma^+)\to 0$ by Lemma \ref{karlssonlem}. Hence, it suffices to show that $\rho_1(u_n,x_n)\to 0$ and $\rho_1(v_n,y_n)\to 0$. This again follows from Lemma \ref{karlssonlem}, as $[u_n,x_n]$ and $[v_n,y_n]$ are escaping. Therefore, we conclude that $\gamma^-=\gamma_+=\xi$   and the proof of the lemma is  complete. 
\end{proof}
Recall that  Floyd boundary surjects onto the end boundary for any finitely generated group, where the map is injective on the set of thin ends (cf. \cite[Prop. 7.18]{GGPY}). We refer to \cite{GGPY} for  the relevant discussion. We record the following corollary.
\begin{cor}\label{EndConvBdry}
The end boundary of (the Cayley graph of) a finitely generated group is a convergence boundary, where the thin ends are non-pinched.    
\end{cor}

\subsection{Homeomorphic Myrberg sets for relatively hyperbolic groups}\label{SSecRHG}
There are various equivalent ways in literature (\cite{Farb, Bow1,Osin,DruSapir,Ge1} etc) to introduce the relative hyperbolicity for a finitely generated group $G$ (even for countable groups which however is beyond the scope of this paper). We adopt here the point of view from the boundary theory. Recall Definition \ref{LimitPointsDefn} for introduction of several classes of limit points.

\begin{defn}
Let $\mathbb P$ be a finite collection of subgroups in $G$. A pair   $\GP$ is  \textit{relatively hyperbolic} if $G$ admits a geometrically finite group action on a compact metrizable space $M$ such that $\{gPg^{-1}: P\in \mathbb P, g\in G\}$ are exactly  the collection of all maximal parabolic subgroups. 
\end{defn}

Using the relative Cayley graph, one can construct the limit set $\Lambda G$ of the action $G\act M$ with the Gromov boundary of this graph \cite{Bow1}. We will often  call \textit{Bowditch boundary}, denoted by $\partial_{\mathrm B}G$, the limit set $\Lambda G$ of a geometrically finite action.  Bowditch proved that if $G$ is finitely generated then $\partial_{\mathrm B}G$ up to an equivariant homeomorphism depends only on the pair $\GP$ \cite{Bow2}. 
\begin{rem}
To be distinguished with the limit set on Bowditch boundary, we shall use $\partial H$ to denote the limit set of $H$ in the Floyd boundary $\pGf$, which are the accumulation points of $H$ in $\pGf$. Equivalently, it is the topological closure of the subset $H$ at the boundary in the compactification $\Gf$.    
\end{rem}
  
Assume that $G$ is non-elementary: each maximal parabolic subgroup is of infinite index; equivalently, the action on the Bowditch boundary has at least three limit points. 

The following    result was first proved by Floyd \cite{Floyd} for geometrically finite Kleinian group $\Gamma$, and then generalized to any relatively hyperbolic group by Gerasimov \cite{Ge2}. 
\begin{thm}\label{mapGerasimov}\cite[Map Theorem]{Ge2}
There exists a Floyd scaling function $f$  so that the identity on $G$ extends to a continuous and equivariant surjection $\phi$ from the Floyd compactification to the Bowditch compactification of $G$.
\end{thm}
In \cite{Ge2}, Gerasimov proved the theorem for the scaling function $f(n)=\lambda^n$ for $\lambda\in (0,1)$, but any scaling function $f(n)>\lambda^n$ such as $f(n)=n^p$ for $p>1$ would also work. Indeed,  the Floyd boundary with respect to $f(n)=n^p$  covers  the one constructed using $f(n)=\lambda^n$.   

\subsection{Conical and Myrberg points}
Let $\GP$ be a relatively hyperbolic pair. We say that   an element $g\in G$ is  \textit{hyperbolic} if it is infinite order and is not conjugated into any $P$. Equivalently, it is a hyperbolic  element on the Bowditch boundary  in the sense of convergence boundary. Two hyperbolic elements are called \textit{independent} if no power of them  are conjugate. 

The notion of a hyperbolic element plays a key role here in that it is contracting with respect to any word metric on $G$. The  results obtained from the sole contracting property in previous subsections hold of course in this specific setup. In particular, a hyperbolic element $f$ is always contained in the unique maximal elementary subgroup $E(f)$ (in Lemma \ref{elementarygroup}), which is equivalently the stabilizer in $G$ of the set of two fixed points $\{f^{-},f^+\}$ of $f$.    In what follows, we however shall take the point of view from Floyd metric, which would be   more powerful and convenient, to study the conical and Myrberg points.

Let $F$ be a finite  set of independent hyperbolic elements. From a different perspective, Osin  \cite{Osin2} proved   that $G$ is hyperbolic relative to the union  $\mathbb P\cup\mathbb E$,  where $\mathbb E=\{E(f):   f\in F\}$. Let $\mathcal P=\{gP: g\in G, P\in \mathbb P\}$. Note that the previously-defined  collection $\f$ of $G$-translated axis of $f\in F$ are exactly the one of   all left cosets of $E(f)$ for $f\in F$. It is well-known that $\mathcal P\cup \f$ forms a contracting system  with bounded intersection (cf. \textsection\ref{SContractingSubset}). 

Let us first recall the following characterization of conical points.
\begin{lem}\cite[Lemma 2.14]{YANG10}\label{CharConicBowd}
A point $\xi$ in $\partial_{\mathrm B}G$ is conical if and only if for any geodesic ray $\gamma$ ending at $\xi$, there exists infinitely many elements $g_n\in G$ so that $d(g_n, \gamma)\le r$, where $r$ is a uniform constant independent of  $\xi$ and $\gamma$. 

Moreover, the direction ``$\Longrightarrow$" holds for a conical point $\xi$ in $\pGf$ with the constant $r$ depending on $\xi$.
\end{lem}

According to the definition \ref{ConicalDef2},  a point $\xi$ is $(F,r)$-conical for some large $r>0$, if it is contained in the  partial shadows $\Pi_1^F(g_n,r)$ for infinitely many $g_n\in G$. That is, there exists a geodesic ray $\gamma=[1,\xi]$ so that $d(g_n,\gamma), d(g_nf,\gamma)\le r$. 

\begin{lem}\label{rFconicalisconical}
Let $\xi$ be an  $(F,r)$-conical point in $\pGf$. Then for $\|F\|\gg 0$,    $\xi$  is a conical point in the sense of convergence boundary (cf. Definition \ref{ConicalDef2}).    
\end{lem}
\begin{proof}
Let $\gamma$ be a geodesic ray ending at an $(F,r)$-conical point $\xi\in \pGf$.  We first show that the endpoint $\eta$ of $\gamma$ in $\partial_{\mathrm B}G$ is a conical point by using Lemma \ref{CharConicBowd}.  

Indeed, according to definition, $\gamma$ contains infinitely many $(r,f)$-barriers $g_n$ for $f\in F$. That is, $\gamma\cap g_n N_r(\ax(f))$ has diameter at least $d(o,fo)$. By Lemma \ref{LargeFloydDistBarrier},  if $\|F\|=\min\{d(o,fo):f\in F\}\gg 0$, $\gamma$ contains infinitely many  points $x_n$ so that $\rho_{x_n}(\gamma^-,\gamma^+)\ge \kappa$, where $\kappa$ is a positive constant independent of $x_n$.  According to Lemma \ref{karlssonlem}, it is easy to show that any other geodesic $\beta$ ending at $\eta$ has to pass through a fixed neighborhood of all but finitely many $x_n$. Hence, $\eta$ is a conical point by Lemma \ref{CharConicBowd}, so the preimage $\xi$ is a conical point in $\pGf$ by Lemma \ref{FindDenseHypElemts}. The proof is complete.  
\end{proof}

According to Definition \ref{LimitPointsDefn}, let us  recast the notion of Myrberg points in the Floyd boundary. Let $\pGf\Join \pGf$ denote the set of unordered pairs of distinct points in $\pGf$.

\begin{defn}
A point $\xi\in \pGf$ is called a \textit{Myrberg   point} if for any $x\in G$, the $G$-orbit of $(x,\xi)$ is dense in $\pGf\Join \pGf$. To be precise, given  $\eta\ne \zeta\in \pGf$, there exists a sequence of $g_n\in G$ so that $g_nx\to\eta$ and $g_n\xi\to\zeta$.     
\end{defn}

\begin{rem} According to definition of Floyd length, if $\sup_{n\ge 1}d(x_n,y_n)<\infty$ and $d(1,x_n)\to\infty$,  we see that $\rho_1(x_n,y_n)\to 0$.   Hence, if $g_nx\to \xi$, then $g_ny\to \xi$ for any $y\in G$. Thus, we could assume $x=1$ in the definition for convenience.  This also implies that the point $x$ can be assumed to any boundary point different from $\xi$.  
\end{rem}

Using the geometry of the Cayley graphs, we  describe the Myrberg point in the Floyd boundary. Denote by $\partial_m G$ the set of Myrberg points in $\pGf$.

\begin{lem}\label{MyrbergCharinFloyd}
Let $A$ be a  set of hyperbolic elements in $G$, which are closed under conjugation, so that their fixed point pairs are dense in $\pGf\Join \pGf$. Then $\xi\in \pGf$ is a Myrberg point if and only if the following property holds:
\begin{itemize}
    \item 
    for any $f\in A$ and for any integer $n>0$, there exist infinitely many distinct $g\in G$ so that $$d(g,[x,\xi]), d(gf^n,[x,\xi])\le r$$ where $r$ depends on the contracting constant of $\ax(f)$.
\end{itemize}  
The same characterization holds for Myrberg points in Bowditch boundary, with $\pGf$ replaced by $\partial_{\mathrm B}G$ in the above statement. 
\end{lem}
\begin{proof}
First of all, by Corollary \ref{MyrbergConical}, it suffices to show that $\xi\in \pGf$ satisfying the following property is a Myrberg  point:  
\begin{itemize}
    \item Given any $f\in A$ and any integer $n>0$,  a geodesic ray   $\gamma=[x,\xi]$ starting at $x\in G$ contains infinitely many $(r,f^n)$-barrier, where   $r>C$ is  a large constant depending on the contraction constant of $\ax(f)$. 
\end{itemize}
To this end, consider a pair of distinct points $(\zeta, \eta)$ in $\pGf$. We shall find a sequence of $g_n\in G$ so that $[x,\xi]\to (\zeta, \eta)$. By assumption, let     $h_n$ be a sequence of hyperbolic elements in $A$ so that $(h_n^-,h_n^+)$ tends to $(\zeta, \eta)$. According to the above property,  for each fixed $f:=h_n$, we have that $\gamma$ contains infinitely many $(r,f^n)$-barriers, where $r$ depends on the contracting constant of the axis of $f$. Let us choose one $(r,f^n)$-barrier $g_n$ among them, so that the diamter of $g_n\gamma\cap N_{r}(\ax(f))$ tends to $\infty$. By Lemma \ref{CloseToHypFixedPts}, we obtain   $$\rho_o(g_nx, f^-), \; \rho_o(g_n\gamma^+,f^+)\le 1/n.$$ Note that this holds for each $f:=h_n$, hence the assumption that $(h_n^-,h_n^+)$ tends to $(\zeta, \eta)$ shows that   $g_n(x,\xi)$ tends   to $(\zeta, \xi)$.  This concludes the proof of the lemma.
\end{proof}

The same proof also gives the following description of Myrberg points in the limit set of discrete groups on  hyperbolic spaces.
\begin{lem}\label{MyrbergCharinHyp}
Assume that $G$ act properly on a hyperbolic space $\U$.
Let $A$ be a  set of hyperbolic elements in $G$ closed under conjugation so that their fixed point pairs are dense in $\Lambda G\Join \Lambda G$. Then $\xi\in \Lambda G$ is a Myrberg point if and only if the following property holds:
\begin{itemize}
    \item 
    for any $f\in A$ and for any integer $n>0$, there exist infinitely many distinct $g\in G$ so that $$d(g,[x,\xi]), d(gf^n,[x,\xi])\le r$$ where $r$ depends on the contracting constant of $\ax(f)$.
\end{itemize}  
Here $[x,\xi]$ denotes a geodesic in the hyperbolic space $\U$.  
\end{lem}

\subsection{Homeomorphisms between Myrberg sets}
Let $\Phi: \pGf\to\Lambda G$ be the   map from the Floyd boundary to Bowditch boundary given by Theorem \ref{mapGerasimov}. By the continuity, a Myrberg point in $\pGf$ is of course sent to  a Myrberg point in $\Lambda G$.  Let $\Phi_m: \partial_m G\to\Lambda G$ denotes  the map of $\Phi$ restricted  onto the Myrberg  set $\partial_m G$.
\begin{thm}\label{HomeoMybergRHG}
The map $\Phi_m:\partial_m G \to \Lambda_m G $ is a homeomorphism  onto the set of Myrberg points in $\Lambda G$.   
\end{thm}
\begin{proof}
Let   $\xi\in \Lambda G$ be a Myrberg point. It is a conical point, so the preimage under $\Phi$ is a singleton consisting of a conical point in $\pGf$ by Lemma \ref{FindDenseHypElemts}. Hence, it follows that $\Phi_m$ must be  injective. We are now going  to  prove the surjectivity of the map.

Let $A$ be the set of hyperbolic elements relative to the convergence action of $G$ on the Bowditch boundary.  The pairs of fixed points in $\Lambda G$ of the elements in $A$ is dense in $\Lambda G\Join \Lambda G$, and moreover, by Lemma \ref{FindDenseHypElemts}, their preimages under $\Phi$ (i.e. fixed points in $\pGf$) are dense in $\pGf\Join \pGf$.  

As $\xi\in\Lambda_mG$ is Myrberg  point, for any fixed $n$ and hyperbolic element $f$ in $A$, the geodesic ray $[1,\xi)$ contains infinitely many segments labeled by $f^n$ in its $r$-neighborhood for some $r$ depending on $\ax(f)$. This property involves the geometry of the Cayley graph, without using  the boundary, so by Lemma \ref{MyrbergCharinFloyd}, the preimage $\Phi^{-1}(\xi)$ is also a Myrberg point in the Floyd boundary. Therefore,   we proved that  $\Phi_m$ is a continuous and bijective map to $\Lambda_m G$.

At last,  we  prove that the inverse  of $\Phi$ is continuous on $\Lambda_m G$. To this end, let $\xi_n\to \xi\in \Lambda_m G$.     Let $\alpha=[1,\xi]$ and $\alpha_n=[1,\xi_n]$. If $\xi_n\to\xi$, then by Lemma \ref{MyrbergCharinHyp}, $\alpha_n$ eventually contains infinitely copies of $g_nf^m$ in its $r$-neighborhood. This proves $\xi_n\to \xi\in \partial_m G$. The proof is complete. 
\end{proof}

We now record the main conclusion from the above result, saying that Myrberg  set is invariant independent of the peripheral structure of $G$. 
\begin{cor}
Let $(G,\mathbb P)$ and $(G,\mathbb P')$ be two relative hyperbolic pairs for the same group $G$. Then there exists a homeomorphism between their   Myrberg  sets in $\partial_{\mathrm B}(G,\mathbb P)$ and $\partial_{\mathrm B}(G,\mathbb P')$, which are homeomorphic to the Myrberg  set in the Floyd boundary $\partial_{\mathrm F}(G)$.
\end{cor}

To conclude this subsection, let us prove that the same conclusion holds for the Cannon--Thurston map. 

Assume that $G$ is a hyperbolic group, which acts properly on a hyperbolic space $\U$. Assume that the orbit map 
$$
g\in G \longmapsto go\in \U
$$
extends continuously (and uniquely) to  a map from the Gromov boundary $\partial G$ of $G$ to the limit set of $G$. Such a boundary map $\Phi$ is referred to as \textit{Cannon--Thurston map} in literature (\cite{CT07, Mj14}).
\begin{thm}\label{HomeoMybergCT}
The map $\Phi: \partial G\to \Lambda G$ restricts to  a homeomorphism from $\partial_m G$ to the Myrberg set in $\Lambda G$.    
\end{thm}
\begin{proof}
The map $\Phi$ is a $G$-equivariant map from $\partial G\to \Lambda G$. By Lemma \ref{FindDenseHypElemts}, the preimage of any conical point in $\Lambda G$ is just singleton of one (conical) point in $\partial G$.

Let $A$ be the set of hyperbolic elements in   $G$ acting on the hyperbolic space $\U$.   By  Lemma \ref{FindDenseHypElemts},  $A$ satisfies the assumption of Lemma \ref{MyrbergCharinHyp}. The proof follows closely the one for Lemma \ref{HomeoMybergRHG}. We leave it to the interested reader. 
\end{proof}

\section{CAT(0) groups with rank-1 elements: Case Study (II)}\label{SecCAT0}
The groups acting properly on CAT(0) spaces with rank-1 elements are another important class of groups motivating the study of the paper.  Applications are presented with an emphasis  to the Myrberg  set and the contracting boundary for CAT(0) spaces. 

We assume the reader is familiar with the basics on CAT(0) geometry and refer the reader to \cite{BriHae} for the reference. 

\subsection{Myrberg set in the visual boundary}

We first prove that the visual boundary of a proper CAT(0) space is a convergence boundary. The proper assumption is only used to guarantee the compactification.
\begin{lem}\label{ContractiveVisualBdry}
Let $\U$ be a proper  CAT(0) space. Then the visual boundary $\bV$ is a convergence boundary with respect to the maximal partition, so that all boundary points are non-pinched. 
\end{lem} 
\begin{proof}
It is well-known that the visual boundary is homeomorphic to horofunction boundary (\cite{BriHae}). By Theorem \ref{ContractiveThm1}, it suffices to prove that this finite difference partition is trivial: every point is minimal.   

Indeed, arguing by contradiction, let $\xi\ne \eta\in \bV$ be       represented respectively by two geodesic rays $\alpha, \beta$ from $o\in \U$, so that $\|b_\xi-b_\eta\|_\infty\le K$ for some $K>0$. On one hand, for any $x\in \alpha$, we have $b_\xi(x)=d(o,x)$ so $|b_\eta(x)-d(o,x)|\le K$. Taking a sequence $y_n\in\beta$ so that $b_{y_n}(x)=d(x, y_n)-d(o, y_n)$ tends to $b_\xi(x)$, we have $$|d(o,x)+d(o,y_n)-d(x,y_n)|\le K+1$$ for all $n\gg 0$. On the other hand,  the Alexandrov angle $\theta:=\angle (\alpha,\beta)$ is positive for $\alpha\ne \beta$ (cf. \cite{BriHae}), so  the Euclidean comparison triangle for  $\Delta(o, x, y_n)$ has the Euclidean angle at $o$ bigger than $\theta>0$. Thus the  Euclidean geometry tells us that $|d(o,x)+d(o,y_n)-d(x,y_n)|$ must tend to $\infty$ as $d(o,y_n)\to \infty$. This is a contradiction.  The lemma is thus proved.
\end{proof}

The following result is crucial in the further discussion, roughly saying that the set of geodesics with positive flat strip is an open set.
\begin{lem}\label{Rank1Open}\cite[Lemma III.3.1]{Ballmann}
Let $\gamma:\mathbb R\to\U$ be a geodesic in a proper CAT(0) space $\U$ without bounding a flat strip of width $R>0$. Then there exist neighborhoods $U $ of $\gamma^+$ and $V$ of $\gamma^-$ in $\bU$ such that any $\xi\in U$ and any $\eta\in V$, there exists a geodesic from $\xi$ to $\eta$ within $R$-distance to $\gamma(0)$ and any such geodesic does not bound a $(2R)$-flat strip. 
\end{lem}

We now give the main result of this subsection, recovering \cite[Lemma 7.1]{L18}  via completely different methods (also compare with \cite[Theorem 2]{R17}). Such results are crucial in applying Hopf's argument in proving ergodicity of geodesic flows (see \cite[Section 7]{L18} for details).
\begin{lem}\label{ZeroAxisLem}
Assume that  $G\act \U$  in (\ref{GActUAssump})  contains a rank-1 element with zero width axis. Then
\begin{enumerate}
\item
any bi-infinite geodesic $\gamma$ with one endpoint at a Myrberg   point has zero width.
\item
any two geodesic rays $\beta, \gamma$ ending at a Myrberg   point are asymptotic: there exists $a\in \mathbb R$ such that 
$$
d(\beta(t),\gamma(t+a))\to 0, \quad t\to \infty.
$$ 
\end{enumerate}
\end{lem}

\begin{proof}
Let $h$ be a rank-1 element with a geodesic axis $\alpha$   with zero width: for any $R>0$ there is no flat strip of width $R$ with  $\alpha$ as one boundary component. Fix $o\in \alpha$.

\textbf{(1).} We are going to prove that   $\gamma$  has zero width. Suppose to the contrary that   $\gamma$ bounds an flat strip of width $R$ for some $R>0$.  By Corollary \ref{MyrbergConical},   $\gamma$ contains infinitely many $(r, f_n)$-barriers $g_no$, where  $f_n\in \langle h\rangle$ is any  sequence of  elements with unbounded length. Thus,
\begin{equation}\label{LargeIntersectionEQ}
 \|N_r(g_n\alpha)\cap \gamma\|\to \infty 
\end{equation}
Thus, there exists   a segment  $p_n$ of $g_n\alpha$ such that $p_n\subseteq N_r(\gamma)$ and   $\len(p_n)\to +\infty$. 

For this $R/2$, let $U, V$ be their open neighborhoods of $\alpha^-,\alpha^+$ satisfying      Lemma \ref{Rank1Open}. As $\len(p_n)\to +\infty$, the cone topology on the visual boundary of $\U$ implies  that $g_n^{-1}\gamma^-\in U, g_n^{-1}\gamma^+\in V$ for $n\gg 0$. By Lemma \ref{Rank1Open}, the geodesic $g_n^{-1}\gamma$ does not bound an  flat strip of width $R$. This contradicts the hypothesis, so we proved that $\gamma$ has zero width. 

\textbf{(2).}   As the CAT(0) distance has  convexity property, it suffices to find two unbounded sequences of points $x_n\in \beta, y_n\in \gamma$ such that $d(x_n, y_n)\to 0$ as $n\to \infty$. By the same argument as above,  we have (\ref{LargeIntersectionEQ}) holds for $\gamma$ and $\beta$, so  for any $R>0$,   Lemma \ref{Rank1Open} implies that $\gamma, \beta$ are $R/2$-close to $g_n\alpha(0)$ for any $n\gg 0$. Thus we obtain $x_n\in \beta, y_n\in \gamma$ so that $d(x_n, y_n)\le R$. Letting $R\to 0$ completes the proof.  
\end{proof}

\subsection{Roller boundary of CAT(0) cube complexes}
Let $\U$ be a proper CAT(0) cube complex, that is, a proper CAT(0) space built out of unit Euclidean $n$-cubes $[0,1]^n$ for $n\ge 1$ by side isometric gluing. The maximal dimension of $n$-cubes is called the dimension of $\U$. In the sequel, we only consider finite dimensional CAT(0) cube complexes.

The collection $\mathcal H$ of \textit{hyperplanes} or called \textit{walls} as   connected components of mid-cubes (i.e. copies of $[0,1]^{n-1}\times \{1/2\}$ in $n$-cubes) endows $\U$ with a wall structure: every wall separates the space into two  connected components called \textit{half-spaces}. We say that an edge is \textit{duel} to a wall if the edge is parallel to the mid-cube in the wall.  An \textit{orientation} of walls picks up exactly one half-space for each wall, and is called \textit{consistent} if any two such chosen half-spaces intersect non-trivially. 

For every vertex $x$ of $\U^0$, the set $U_x$ of half-spaces containing $x$ forms a consistent orientation denoted by $\mathbf{x}\in 2^{\mathcal H}$.   We equip the space $2^{\mathcal H}$ of all orientations    with the compact product topology, where all consistent  orientations form a closed subset. The \textit{Roller compactification} $\bU$ of $\U$ is thus the closure  of  the vertex set   $\U^0$ (identified with just described consistent orientations) in the space of all consistent  orientations on walls, and    \textit{Roller boundary} is   $\bR:=\bU\setminus \U^0$. 

The CAT(0) metric induces an $\ell^2$-metric on the  1-skeleton $\U^1$. In this subsection, we are mainly interested in the $\ell^1$-metric  on   $\U^1$: $$\forall x, y\in \U^0: \; d(x,y):=\frac{1}{2}|U_x\Delta U_y|$$ where $\Delta$ denotes the  symmetric difference of sets.  If $\U$ is finite dimensional, the $\ell^1$-metric and $\ell^2$-metric are quasi-isometric. The $\ell^1$-metric is exactly the combinatorial metric on $\U^1$ and usually non-unique geodesic  metric.  However,  the set of walls duel to the edges on any geodesic is more canonical and does not depend on the geodesic, which is exactly the set of walls separating $x$ and $y$. By abuse of language, we also say that the wall separating $x$ and $y$ is duel to a geodesic $[x,y]$.

Recall that an action of $G$ on the CAT(0) cubical complex $\U$ is called \textit{essential} if no $G$-orbit is contained in a finite neighborhood of some half-space. If $G$ has no fixed points in the visual boundary of $\U$ or has only finitely many orbits of walls, then the action could be made essential by passing to its (convex) essential core (see \cite[Prop 3.5]{CapSag}).

Assume that $\U$ is \textit{irreducible}: it cannot be written as a nontrivial metric product of two CAT(0) cube complexes. If $G$ acts essentially on $\U$ without fixed points in the visual boundary, this is  equivalent to one of the following 
\begin{enumerate}
    \item 
      existence of 
\textit{strongly separated} of two half spaces:   no wall transverses both of them (\cite[Prop. 5.1]{CapSag}). 
    \item 
      existence of a contracting isometry on $\U$ with respect to CAT(0) metric  (\cite[Theorem 6.3]{CapSag}).
\end{enumerate}

\begin{lem}\cite[Lemma 8.3]{GYANG}\label{rank1l1}
Let $G \act \U$ be an essential  proper action on an irreducible CAT(0) cube complex. Then $G$ contains a {contracting} element with respect to the action on 1-skeleton of $\U$ with the $\ell^1$-metric. Moreover, \begin{itemize}
    \item 
    A contracting   element in CAT(0) metric is   contracting in the $\ell^1$-metric, and vice versa. 
    \item 
    A contracting   element in $\ell^1$-metric preserves an $\ell^1$-geodesic by translation.
\end{itemize}
\end{lem}

Following \cite{F18}, a boundary point (i.e. a consistent orientation of walls) is called \textit{regular} if the orientation of walls  contains an infinite sequence of a descending chain of pairwise {strongly separated} half-spaces.  We  equip $\bR$ with the \textit{finite symmetric difference partition} $[\cdot]$: two (consistent) orientations are equivalent if their symmetric difference is finite. It is proved in \cite[Lemma 5.2]{FLM} that the regular points are $[\cdot]$-minimal: the $[\cdot]$-class is a singleton. The fixed points of contracting isometry are regular and thus minimal.

An unpublished result of Bader--Guralnik says that the Roller boundary is homeomorphic to  the horofunction boundary of $(\U^0, d)$ (see a proof \cite[Prop. 6.20]{FLM}). Here is the main result relating the previous study of convergence boundary to Roller boundary.

\begin{lem}\label{Roller}
The Roller boundary $\bR$ is a convergence boundary, with finite symmetric difference partition  $[\cdot]$ so that $[\cdot]$  coincides with the partition of finite difference of horofunctions. 

Moreover, assume that a non-elementary group $G$ acts essentially  on $\U$. 
Then the following hold. 
\begin{enumerate}

    \item
    Every contracting element satisfies the  North--South dynamics with  minimal fixed points.
   
    \item
    $G$ admits a unique invariant closed subset, where the    fixed point pairs of contracting elements are dense in the  distinct pairs.
     \item
    The Myrberg points of $G$ are regular, and in particular, are minimal.
\end{enumerate}
       
\end{lem}

\begin{proof}
By \cite[Prop A.2]{Genevois}, every point in $\bR$ is represented by a geodesic ray (called Buseman point),  via the identification with the half-spaces into which the ray eventually enters.   We first prove that if  two horofunctions  $b_\xi, b_\eta\in \bR$ have finite difference at most $K$, then  $\xi$ and $\eta$ are in the same symmetric difference $[\cdot]$-class. 

Fix a basepoint $o\in \U$. Let $\alpha, \beta$ be two (combinatorial) geodesic rays issuing from $o$ and ending at $\xi, \eta$. Recall that $\bR$ is the horofunction boundary with respect to  the combinatorial metric $d$. We obtain similarly the same estimates on Gromov product  as  in the proof of Lemma \ref{ContractiveVisualBdry}: $$2\langle o, y_n \rangle_x=d(o,x)+d(o,y_n)-d(x,y_n)\le K+1$$ for any $x\in\alpha, y_n\in \beta$ with $n\gg 0$. As half-spaces are combinatorially convex, the geodesic rays $\alpha$ and $\beta$   share the same set of half-spaces containing the common initial point $o$. Thus, it suffices to consider the difference on the half-spaces whose bounding walls are duel to edges of $\alpha$ or $\beta$.  

\begin{figure}
    \centering

\tikzset{every picture/.style={line width=0.75pt}} 

\begin{tikzpicture}[x=0.75pt,y=0.75pt,yscale=-1,xscale=1]

\draw    (17.5,74) -- (299.51,97.83) ;
\draw [shift={(301.5,98)}, rotate = 184.83] [color={rgb, 255:red, 0; green, 0; blue, 0 }  ][line width=0.75]    (10.93,-3.29) .. controls (6.95,-1.4) and (3.31,-0.3) .. (0,0) .. controls (3.31,0.3) and (6.95,1.4) .. (10.93,3.29)   ;
\draw [shift={(17.5,74)}, rotate = 4.83] [color={rgb, 255:red, 0; green, 0; blue, 0 }  ][fill={rgb, 255:red, 0; green, 0; blue, 0 }  ][line width=0.75]      (0, 0) circle [x radius= 3.35, y radius= 3.35]   ;
\draw    (109.5,82) -- (149.5,29) ;
\draw [shift={(149.5,29)}, rotate = 307.04] [color={rgb, 255:red, 0; green, 0; blue, 0 }  ][fill={rgb, 255:red, 0; green, 0; blue, 0 }  ][line width=0.75]      (0, 0) circle [x radius= 3.35, y radius= 3.35]   ;
\draw [shift={(109.5,82)}, rotate = 307.04] [color={rgb, 255:red, 0; green, 0; blue, 0 }  ][fill={rgb, 255:red, 0; green, 0; blue, 0 }  ][line width=0.75]      (0, 0) circle [x radius= 3.35, y radius= 3.35]   ;
\draw    (149.5,29) -- (307.51,40.85) ;
\draw [shift={(309.5,41)}, rotate = 184.29] [color={rgb, 255:red, 0; green, 0; blue, 0 }  ][line width=0.75]    (10.93,-3.29) .. controls (6.95,-1.4) and (3.31,-0.3) .. (0,0) .. controls (3.31,0.3) and (6.95,1.4) .. (10.93,3.29)   ;
\draw  [line width=6] [line join = round][line cap = round] (212.8,91.22) .. controls (212.8,91.22) and (212.8,91.22) .. (212.8,91.22) ;
\draw  [line width=0.75]  (145.5,29) .. controls (141.72,26.26) and (138.46,26.78) .. (135.72,30.56) -- (126.59,43.15) .. controls (122.68,48.54) and (118.83,49.87) .. (115.05,47.13) .. controls (118.83,49.87) and (118.76,53.94) .. (114.84,59.33)(116.61,56.91) -- (106.95,70.22) .. controls (104.2,74) and (104.72,77.26) .. (108.5,80) ;

\draw (2,70.4) node [anchor=north west][inner sep=0.75pt]    {$o$};
\draw (150,6.4) node [anchor=north west][inner sep=0.75pt]    {$x$};
\draw (206,95.4) node [anchor=north west][inner sep=0.75pt]    {$y_{n}$};
\draw (224,11.4) node [anchor=north west][inner sep=0.75pt]    {$\alpha $};
\draw (218,66.4) node [anchor=north west][inner sep=0.75pt]    {$\beta $};
\draw (100,88.4) node [anchor=north west][inner sep=0.75pt]    {$m$};
\draw (305.5,44.4) node [anchor=north west][inner sep=0.75pt]    {$\xi $};
\draw (293,99.4) node [anchor=north west][inner sep=0.75pt]    {$\eta $};
\draw (34,36.4) node [anchor=north west][inner sep=0.75pt]    {$( K+1) /2\geq $};

\end{tikzpicture}
    \caption{The walls due to $[m,x]$}
    \label{fig:medianCCC}
\end{figure}
Recall that    the \textit{median} $m$ of $(o, x, y_n)$ is the unique common point  on all three sides of some geodesic triangle with these vertices. Alternatively, $m$ is the intersection of all geodesic segments between pairs of points in $(o, x, y_n)$. From the above inequality, we obtain $$\langle o, y_n \rangle_x=d(x, m)\le (K+1)/2.$$ Replacing the initial subpaths of $\alpha$ and $\beta$ with appropriate geodesics with the same endpoints, we may assume that the median point $m$  lies  both on $\alpha$ and $\beta$ so that $[o, m]_\alpha=[o, m]_\beta$. Thus,   the   walls intersecting $[o, x]_\alpha $ but not $ [o, m]_\alpha=[o, m]_\beta$ are due to the geodesic $[x,m]$, which are at most $d(x, m)\le (K+1)/2$ walls.  Note that $m$ depends on $x$. See Fig. \ref{fig:medianCCC}.

Letting $x\to \xi$ and choosing $m$ accordingly,   a limiting argument shows  that the symmetric difference of the walls intersecting $\alpha$ and the ones intersecting $\beta$ has at most  $(K+1)/2$ half-spaces. This shows that $\xi, \eta$ have finite difference on walls. 
 
 Conversely, if $\xi, \eta$ has finite symmetric difference on $K$ walls, the similar reasoning as above produces two unbounded sequence of  points $x_n\in \alpha, m_n\in \beta$ with $d(x_n, m_n)\le K/2$. Thus,  $\|b_\xi-b_\eta\|_\infty\le K/2$.

To complete the proof, it remains to   show that a Myrberg point is \textit{regular}: it contains an infinite sequence of a descending chain of pairwise {strongly separated} half-spaces. Once this is proved, the statements (1) and (2) are immediate consequences of Lemmas \ref{SouthNorthLem} and  \ref{UniqueLimitSet}. 

Let $\gamma$ be  a  geodesic ray  ending at a  Myrberg point $\xi$. By   \cite[Theorem 3.9]{Genevois}, for any contracting isometry $h$, the combinatorial geodesic axis    is duel to a pair of strongly separated half-spaces with arbitrarily large distance. As each wall separates $\U$ into two convex components, for any two pairs $(x,y)$ and $(z,w)$ with $d(x,z), d(y,w)\le r$, the walls due to $[x,y]$ differ by at most $2r$ from the ones to $[z,w]$.  By Corollary \ref{MyrbergConical}, for any   large $L>0$,  $\gamma$ contains a sequence of  disjoint   sub-segments of length $L$ within  the $r$-neighborhood of        a translated  contracting axis of $h$. Hence, $\gamma$ has to be duel to  infinitely many pairs of strongly separated walls. By strong separateness,  the   half-spaces bounded by those walls and into which  $\gamma$ eventually enters form a descending chain, so  $\xi$ is a regular point by definition. Thus the lemma is proved. 
\end{proof}

\subsection{Contracting boundary in conformal measure}
We start by proving a general result for  any convergence boundary $\pU$. Recall that a point $\xi$ in $\pU$ is called visual if  there exists a geodesic ray $\gamma$ from any point in $ \U$ ending at $[\xi]$. If in addition, $\gamma$ is  contracting,  then we call  $\xi\in \pU$ a \textit{contracting point}.

\begin{thm}\label{nullity}
Suppose that $G\act \U$ as in (\ref{GActUAssump}) is a co-compact action.  Let $\{\mu_x\}_{x\in \U}$ be a $\e G$-dimensional quasi-conformal density on a convergence boundary $\pU$. Then the set of contracting points in   $\pU$   is $\mu_1$-null if and only if $G$ is not a hyperbolic group.  
\end{thm}

Fix a basepoint $o\in \U$. As the action $G\act \U$ is co-compact, let $M>0$ such that $N_M(Go)=\U$. In the proof, we need the following result which is of independent interest.

Let $\Lambda\subseteq \pU$ be a set of visual points. Denote by $\textrm{Hull}_R(\Lambda)$  the set of orbital points $v \in Go$ such that $v$ is $R$-closed to a  geodesic from     $o$ to $\xi\in \Lambda$ for $R\ge 0$.
\begin{lem}\label{growthhull}
   
Suppose $\Lambda\subseteq \pU$ has positive $\mu_1$-measure.  Then for any $R\gg M$, there exists $c=c(\Lambda, R)>0$ such that
$$| \textrm{Hull}_R(\Lambda)\cap N(o,n)| \ge c\cdot e^{\e G n},$$
for any $n>0$.
\end{lem}
\begin{proof}
Denote $A_n= \textrm{Hull}_R(\Lambda) \cap A(o, n, \Delta)$.     Set $R>\max\{r_0, M\}$, where $r_0>0$ is given by Lemma \ref{ShadowLem}. By definition, we have    $$\Lambda \subseteq \bigcup_{v \in A_n} [\Pi_o^F(v, R)]$$ 
which yields  
$$
\mu_1(\Lambda) \le \sum_{h \in A_n} \mu_1([\Pi_o(v, R)]).
$$ A co-compact action must be of divergent type, so  by Theorem \ref{HTSThm}, $\mu_1$ is supported on Myrberg  set. By Lemma \ref{UnifDiffLem}, any two horofunctions in the locus of a Myrberg point have uniformly bounded difference. The proof for the upper bound in     Lemma \ref{ShadowLem} gives
$$
\mu_1([\Pi_o(v, R)]) \prec_R e^{\e G d(o, v)}
$$ 
Hence, there exists a constant $c>0$ such that
$$
|A_n|\ge c \cdot e^{\e G n}
$$
completing the proof.
\end{proof}

We are now ready to give the proof.

\begin{proof}[Proof of Theorem \ref{nullity}]
If $G$ is a hyperbolic group, then the co-compact action implies that $\U$ is hyperbolic and the action   $G\act \U$ is of divergent type. By Lemma \ref{CharConicalLem}, any $(r, F)$-conical point is visual and thus contracting, because every geodesic ray in a hyperbolic space is contracting. By Theorem \ref{HTSThm}, the set of contracting points is  $\mu_1$-full, hence the ``$\Longrightarrow$" direction follows. 

Let us prove  the other direction by assuming that $G$ is not hyperbolic. Let $\Lambda$ denote the set of contracting points, and assume   $\mu(\Lambda)>0$ for a contradiction. 

Let   $\Lambda_C(x) $ be the set of  $\xi\in \Lambda$ so that there exists a $C$-contracting geodesic   ray at $\xi$ starting at $x\in \U$.  Then $\Lambda$ admits the following countable union:       $$ \Lambda =\bigcup_{C\in\mathbb N} \Bigg(\bigcup_{x\in \U}\Lambda_C(x)\Bigg)$$ where $\Lambda(gx)=g\Lambda(x)$ for any $g\in G, x\in \U$.  Since the conformal measures are almost $G$-equivariant and the $M$-neighborhood of a $G$-orbit $Go$ covers $\U$, there exists $C>0$ such that   $\mu_1(\Lambda_C(o))>0$ for some (or any) $o\in \U$.  

Let $U=\textrm{Hull}_R(\Lambda_C(o))$ be the  hull of $\Lambda_C(o)$ for some $R\gg 0$. As $\mu_1(\Lambda_C(o))>0$, the set $U\subseteq Go$ has the growth rate  $\e G$ by Lemma \ref{growthhull}.

By definition, any segment $[o,v]$ for $v\in U$ has two endpoints $R$-close to a $C$-contracting geodesic ray. Recall that by \cite[Prop. 2.2]{YANG11},   there exists  a constant $D=D(C, R)$ such that any geodesic segment with two endpoints in an  $2R$-neighborhood of a $C$-contracting geodesic is $D$-contracting. Thus, $[o,v]$ for $v\in U$ is $D$-contracting. 

It is well-known that if  every geodesic segment in $\U$ is uniformly contracting then $\U$ is Gromov-hyperbolic. Thus  there exists  a segment  $\alpha$,  so that $\alpha$ is not $D$-contracting. As $G$ acts    co-compactly on $\U$, we may assume without loss of generality that $\alpha=[o, ho]$ for a basepoint $o\in \U$ and $h\in G$ is not $D$-contracting. 

Observe that any $[o,v]$ for $v\in U$ cannot contain $\alpha$ in its $R$-neighborhood. Indeed, if not,  $\alpha$ lies in $2R$-neighborhood of a $C$-contracting geodesic ray, so must be  $D$-contracting.  That is to say, any $[o,v]$ for $v\in U$ does not contain any $(R, h)$-barrier in a sense of \cite{YANG10}. By \cite[Theorem C]{YANG10}, the set $U$ is \textit{growth tight}: $\e U<\e G$. This   contradicts the equality obtained as above. Hence, it is proved that $\mu_1(\Lambda)=0$.
\end{proof}

\section{Conformal density for mapping class groups: Case Study (III)}\label{SecMCG}

Let $\Sigma_g$ be an orientable closed surface of genus $g\ge 2$. Consider  the mapping class group    $$ \mcg=\pi_0(\textrm{Homeo}^+(\Sigma_g))$$i.e.:   isotopy classes of orientation preserving homeomorphisms of $\Sigma_g$.  

Uniformization theorem for surfaces sets up one-to-one correspondence between complex structures and hyperbolic metrics on $\Sigma_g$.  The Teichm\"uller space, denoted by $\T$,  of $\Sigma_g$ is the set of    isotopy classes of hyperbolic metrics or of complex structures       on  $\Sigma_g$. We equip $\T$  with   Teichm\"uller metric    $d_T$, for which the isometry group is exactly  the mapping class groups $\mcg$. It is well-known that pseudo-Anosov mapping classes are contracting in Teichm\"uller metric \cite{Minsky}, so the results   in the previous subsections apply.  In the sequel, we refer to \cite{FMbook, FLP} for more details on the undefined notions.

\subsection{Preliminaries on Thurston and Gardiner--Masur boundaries}\label{SecThurstonBdry}
Let $\s$ be the set of (isotopy classes of) essential simple closed curves on $\Sigma_g$. Given $x\in \T$, each element $c\in \s$ can be assigned either   the   length $\ell_x(c)$ of a closed geodesic representative on $\Sigma_g$ equipped with a constant  curvature $-1$ of Riemannian metric  or the extremal length $\ext_x(c)$ on $\Sigma_g$ with a Riemann surface structure. 

Denote by $\mathbb{R}_{>0}^\s$ the set of positive functions on $\s$.  We then obtain the embedding of $\U$ into $\mathbb{R}_{>0}^\s$ in two ways 
\begin{equation}\label{CurveEmbed}
x\in \T\mapsto \Big(\ell_x(c)\Big)_{c\in \s}\; \textrm{ and } \quad x\in \T\mapsto \Big(\sqrt{\ext_x(c)}\Big)_{c\in \s}    
\end{equation}
whose topological closures in the projectived space $\mathbb P\mathbb{R}_{>0}^\s$ give the corresponding Thurston and Gardiner--Masur compactifications of $\T$ denoted by $\T\cup \bTh$ and $\T\cup \bGM$ (\cite{FLP, GM91}). We first describe the Thurston boundary $\bTh$ in more details.

\subsubsection*{Thurston boundary} Let $\mf$ be the set of measured foliations  on $\Sigma_g$ with a natural topology, so that $\mf$   admits a topological embedding into $\mathbb{R}_{>0}^\s$, by taking the intersection number $I(\xi, c)$ of a measured foliation $\xi$ with a curve $c$ in $\s$.  As the   weighted multi-curves   form a dense subset in $\mf$, their geometric intersection number extends continuously to  a bi-linear intersection form  $$I(\cdot, \cdot): \quad \mf\times \mf\to \mathbb R_{\ge 0}.$$
This allows to understand   an essential simple closed curves $c\in \mathcal S$ as a measured foliation, where  the transversal measure is     given by  the geometric intersection number.

Let $\mf\to \pmf$  be the projection to the set $\pmf$ of projective measured foliations  by positive reals scaling. Thurston proved that $ \pmf$ is homeomorphic to the sphere of dimension $6g-7$, and compactifies $\T$ in the projective space $\mathbb P\mathbb{R}_{>0}^\s$ as a closed unit ball (\cite{FLP}).  Then $\bTh:=\pmf$ is the so-called Thurston  boundary.

By abuse of language, recalling    $I(\cdot,\cdot )$ is bi-linear, we  shall write     $I(\xi, \eta)=0$ for two \textbf{projective} measured foliations $\xi, \eta\in \pmf$ to mean $I(\check \xi, \check \eta)=0$ for any of their lifts $\check \xi, \check \eta\in \mf$.  

\subsubsection*{Gardiner--Masur boundary} In \cite{GM91}, Gardiner--Masur proved that the Gardiner--Masur boundary  $\bGM$ from the embedding (\ref{CurveEmbed}) contains $\pmf$ as a proper subset. In \cite[Corollary 1]{Miy13}, Miyachi proved that the identification between the Teichm\"uller space $\T$ extends continuously  
\begin{equation}\label{ThGMHomeoEQ}
\T\cup \ue \subseteq \T\cup\bTh\quad \longmapsto\quad  \T\cup  \ue \subseteq \T\cup\bGM   
\end{equation}
to a homeomorphism between the subsets   $\ue$ in both boundaries. Furthermore, the following result will be crucially used later on. 
\begin{lem}\label{GromovProduct}\cite[Corollary 1]{Miy14b}
For any $o\in \T$, there exists a unique continuous extension of  the Gromov product   
$$
(x, y)\in \T\times \T \quad \longmapsto \quad \langle x, y\rangle_o=\frac{d(x,o)+d(y,o)-d(x,y)}{2}\in \mathbb R_{\ge 0}
$$
to the Gardiner--Masur boundary $\pmf\subseteq \bGM$ so that for any $\xi, \eta\in \pmf$ 
\begin{equation}\label{GromovProductExtendsI}
e^{-2\langle \xi, \eta\rangle_o} = \frac{I(\check\xi,\check\eta)}{\ext_o(\check\xi)^{1/2}\ext_o(\check\eta)^{1/2}}
\end{equation}
where $\check{\xi},\check{\eta}$ are any lifts of $\xi, \eta$ in $\mf$.
\end{lem}
\begin{rem}(extension in Thurston boundary)
Thanks to (\ref{ThGMHomeoEQ}), the equality (\ref{GromovProductExtendsI}) holds for uniquely ergodic points $\xi,\eta\in \bTh$. We now explain it also holds for $\xi,\eta \in \mathcal S\subseteq \pmf=\bTh$. By Hubbard-Masur theorem, for any basepoint $o\in \T$ and for any $c\in \mathcal S\subseteq \pmf$,   there exists a quadratic differential $q$ on the Riemann surface $o$ whose   vertical foliation $q^+$ realizes $c$. Recall that   the  Teichm\"uller geodesic ray issuing from $o$ directed by $q^+$ has the limit point  $c$ in either Thurston or Gardiner--Masur boundary (cf. \cite[Theorem]{Ma82b}, \cite[Theorem 2]{Miy13}).  Let $x_n,y_n$ be on the geodesic rays ending at $\xi\in \bTh$ and $\eta\in \bTh$ so that   $x_n\to \xi$ and $y_n\to \eta$ in both compactifications.  Thus, (\ref{GromovProductExtendsI}) holds for  $\xi,\eta \in \mathcal S \subseteq \bTh$ as the right side of (\ref{GromovProductExtendsI}) involves only the Teichm\"uller metric without caring about compactification. H. Miyachi also indicated to the author that this also follows from the discussion \cite[Sec. 5.3]{Miy13}.  
\end{rem}

\subsection{Kaimanovich--Masur partition on Thurston boundary}
Following Kaimanovich--Masur \cite{KaMasur}, we shall define a partition on $\pmf$ as follows.      Let $\minf$ be the subset of projective \textit{minimal} foliations $\xi\in \pmf$ so that $I(\xi, c)>0$ for any $c\in \s$. The $[\cdot]$-class of   $\xi\in \minf$  defined as $$[\xi]\;:=\;\{\eta\in \minf:\; I( \xi,  \eta)=0\}$$  forms a partition of   $\minf$, as  $I(\xi, \eta)=0$ for $\xi\in \minf$ implies $\eta\in \minf$ (\cite{Rees81}).  
The $[\cdot]$-class of minimal foliations $\xi\in \minf$ being singleton forms exactly the set $\ue$ of uniquely ergodic foliations. 
The remaining non-minimal foliations $\mf\setminus \minf$ are then partitioned into    countably many classes $$[\xi]\;:=\;\{\eta: \; \forall c\in \s,\;  I( \eta, c)=0 \Longleftrightarrow  I(\xi,c)=0 \}$$  according to whether they share the same disjoint set of  curves in $\s$. It is proved in \cite[Lemma 1.1.2]{KaMasur} that the partition $[\cdot]$, generated by a countable partition, is measurable in the sense of Rokhlin.  
 
The following useful fact is   proved in \cite[Lemma 1.4.2]{KaMasur} where $\xi$ lies in $\ue$ and the conclusion follows as $y_n\to\xi$.   We provide some explanation on the changes.
\begin{lem}\label{PMFConv}
Assume that  $x_n\in \T$ tends to $\xi\in\minf$.     If   a sequence of points $y_n \in \T$ satisfies $$d(o, y_n)-d(x_n, y_n)\to +\infty$$  then   $y_n\to [\xi]$.

\end{lem}
\begin{proof}
In the proof \cite[Lemma 1.4.1]{KaMasur}, a sequence of $\beta_n\in \s$ is produced such that $\ext_{x_n}(\beta_n)$
is uniformly bounded and 
$$
I(\beta_n, F_n)\to 0
$$
where $F_n$ is the vertical measured foliation of the terminal quadratic differential of the Teichm\"uller map from $x_0$ to $x_n$. Taking convergence in  $\pmf$ and  passing to a subsequence,  assume that $$\beta_n \stackrel{\pmf}{\Longrightarrow}  F, \quad F_n\stackrel{\pmf}{\Longrightarrow}  F_\infty$$ for some $F, F_\infty\in \mf$. Then  $I(F, F_\infty)=0=I(F_\infty, \xi)$ by   \cite[Lemma 43]{Ka2}.  As all the foliations in $[\xi]$ for $\xi\in \minf$ is minimal (\cite{Rees81}),  we obtain $ F\in [\xi]$.

The remaining proof   from page 245 of \cite[Lemma 1.4.1]{KaMasur} shows that $y_n$ tends to $ F\in \pmf$. This is what we wanted.
\end{proof}

The Thurston boundary with the above partition $[\cdot]$ satisfies the following stronger version of \ref{AssumpB}, which also holds for CAT(0) and hyperbolic spaces (see Lemma \ref{UnifDiffLemCAT0}). To be precise, the cones in the following definition are contained in the cones in \ref{AssumpB}, where  $\|[o,y]\cap N_C(X_n)\|\gg 0$ is required.
\begin{lem}\label{StrongerAssumpB}
Let $X_n$ be an escaping  sequence of contracting subsets in $\T$. Denote the cone $Y_n=\Omega_o(X_n)=\{y\in \T: [o,y]\cap X_n\ne\emptyset\}$. Then there exist a subsequence of $Y_n$ and a boundary point $\xi$ so that the following holds. Any convergent sequence $y_n\in Y_n$ tends into $[\xi]$.
\end{lem}
\begin{proof}
Indeed, after passage to subsequence, let $X_n$ be an escaping sequence of $C$-contracting subsets so that $x_n\in \pi_{X_n}(o)$ tends to $\xi\in \pmf$. Take any  sequence of points $y_n\in \Omega_o(X_n)$ converging to some $\eta\in \pmf$ by compactness. Recall that $\Omega_o(X_n)$ is the set of points $y\in \T$ with $[o,y]\cap X_n\ne\emptyset.$ The contracting property in Lemma \ref{BigThree} implies $d(\pi_{X_n}(o), [o, y_n])\le C$, so we have 
\begin{equation}\label{almostgeodEQ}
|d(o, y_n) - d(o,x_n)-d(x_n, y_n|\le 2C    
\end{equation}  If $\xi\in \minf$, then $\eta\in [\xi]$ follows by Lemma \ref{PMFConv} as $d(o,x_n)\to\infty$.     

To finish the proof of  \ref{AssumpB}, it thus remains to consider the case $\xi\in \pmf\setminus \minf$, and prove $[\xi]=[\eta]$: they has the same  set of disjoint curves in $\s$. That is to say,    $I( \xi, c)=0$ for $c\in\s$ is equivalent to  $I( \eta, c)=0$. The two directions are symmetric. We only prove that if $I( \xi, c)=0$ then  $I(\eta, c)=0$.

Let $z_n$ be be a sequence of points on a geodesic ray ending at $c\in \pmf=\bTh$, so $z_n\to c$   also takes place in $\bGM$. By Lemma \ref{GromovProduct}, as $I(\xi, c)=0$,  we have $ \langle x_n, z_n\rangle_o\to\infty$.  By definition of Gromov product, $$
\begin{array}{rlr}
2\langle y_n,z_n\rangle_o &= d(y_n,o)+d(z_n,o)-d(y_n,z_n) &\\
& \ge d(y_n, x_n)+d(o, x_n) +d(z_n, o)-d(y_n, z_n)-2C & (\ref{almostgeodEQ}) \\
& \ge d(o, x_n) + d(z_n, o) -d(x_n,z_n)-2C & (\Delta\textrm{-ineq.})\\
&= 2\langle x_n, z_n\rangle_o -2C & 
\end{array} 
$$ so by Lemma \ref{GromovProduct} with the Remark after it, we have $I(\eta, c)=0$.     
\end{proof}

We are now ready to prove the main result, Theorem \ref{ContractiveThm2}, of this subsection.
\begin{prop}\label{ThurstonBdryContractive}
The Thurston boundary $\bTh$ with the above Kaimanovich--Masur partition $[\cdot]$   is a  convergence boundary of Teichm\"uller space $\T$, so that the uniquely ergodic points $\ue$  are   non-pinched points. 
\end{prop}
\begin{proof}
It is well-known that a contracting geodesic ray $X$ tends to  a unique ergodic point (see \cite{Co17}). Any sequence $y_n$ with an escaping $x_n\in \pi_X(y_n)$ satisfies the condition of  Lemma \ref{PMFConv}, so \ref{AssumpA}  follows. The non-pinched points in \ref{AssumpC} contain $\ue$, which follows    by Lemma \ref{GromovProduct} and (\ref{ThGMHomeoEQ}).

\ref{AssumpB} follows from Lemma \ref{StrongerAssumpB}. The proof is thus complete.
\end{proof}

In \cite{McPapa}, McCarthy--Papadopoulos  defined a notion of the limit set, denoted by $\Lambda_{\mathcal{MP}}(H)$, in Thurston boundary for a non-elementary  subgroup $H<\mcg$  with a pseudo-Anosov element, as the closure of fixed points of pseudo-Anosov elements. They further showed that the set $\Lambda(Ho)$ of accumulation points of a $H$-orbit in $\T$ is contained in the intersection locus of $\Lambda$ (those measured foliations with zero intersection with the ones in $\Lambda$). Their limit set $\Lambda_{\mathcal{MP}}(H)$ is exactly the set $\Lambda$ in Lemma \ref{UniqueLimitSet}, so we have $[\Lambda]=[\Lambda Go]$. Here $[\cdot]$ denotes the KM partition, but not the intersection locus in  \cite{McPapa}.  Therefore, it would be interesting to examine the exact relation between the Kaimanovich--Masur partition $[\cdot]$ and the intersection locus. 

In \cite{LS14}, Liu--Su  proved that Gardiner--Masur boundary is the horofunction boundary of Teichm\"uller space (with Teichm\"uller metric). We list a few consequences of the general theory, in view of the corresponding results on Thurston boundary \cite{McPapa}.

\begin{lem}\label{GMBdryConvergence}
 $\bGM$ is a convergence boundary with respect to finite difference relation $[\cdot]$. Moreover, 
\begin{enumerate}
    \item 
    The finite difference relation $[\cdot]$ restricts on $\ue\subseteq \bGM$ as a maximal partition.
    \item
    All $(r,F)$-conical points in Definition \ref{ConicalDef2}  are uniquely ergodic points.
    \item
    Every pseudo-Anosov elements are non-pinched contracting elements with minimal fixed points and have the   North--South dynamics on $\bGM$.
    \item
    Every non-elementary subgroup $G<\mcg$ has a unique minimal $G$-invariant closed subset $\Lambda$, so that  $[\pG ]=[\Lambda]$ for any $o\in \T$.  
\end{enumerate}
\end{lem}
\begin{proof}
The assertions (3) and (4) are corollaries of Lemmas \ref{SouthNorthLem} and \ref{UniqueLimitSet}, provided that (1) and (2) are proved.

\textbf{(1).} Let $[\xi]=[\eta]$ with $\xi\in \ue$, so  $\|b_\xi-b_\eta\|\le K<\infty$. Take  a geodesic ray  $\gamma=[o,\xi]$ and $x_n\in \gamma$ tends to $\xi$. Let $y_n\to \eta$. As $|b_\xi(x_n)-b_{y_n}(x_n)|=|d(o,x_n)+d(o, y_n)-d(x_n, y_n)|\le K$, then $y_n\to \xi$ by Lemma \ref{PMFConv}, so $\xi=\eta$ is proved. 

\textbf{(2).} Projecting to the quotient $\T/{\mcg}$, any geodesic ray ending at an $(r, F)$-conical point $\xi\in \bGM$ is recurrent into the $r$-neighborhood of closed geodesics corresponding to pseudo-Anosov elements in $F$.  Thus, through (\ref{ThGMHomeoEQ}),  $\xi$ is uniquely ergodic by a criterion of Masur \cite{Ma80}.    
\end{proof}

\subsection{Conformal density on Thurston and Gardiner--Masur boundaries}\label{SecMCGDensity}

We first introduce the (formal) conformal density defined in \cite{ABEM} on $\bTh$, and then clarify the relation  with the conformal density obtained from the Patterson's construction. 

The  train track parametrization  endows the set $\mf$ of  measured foliations with a Piecewise Linear structure. Let $\mu_{\mathrm{Th}}$ be the resulted $\mcg$-invariant \textit{Thurston measure} on $\mf$, in the Lebesgue measure class, pasting from the Lebesgue measure on the train track charts on $\mf$. By Masur's ergodicity theorem,  $\mu_{\mathrm{Th}}$ is a  $\mcg$-invariant ergodic measure, which is thus unique up to scaling. 

According to \cite{ABEM}, a family of conformal measures are constructed via $\mu_{\mathrm{Th}}$ as follows. Let $U$ be a Borel subset in $\pmf$.  Define for any $x\in \T$, $$\lambda_x(U) \;:=\; \mu_{\mathrm{Th}}(\{\xi \in \mf: \check \xi \in U, \; \ext_x(\xi) \le 1\}).$$ 
where $\check \xi$ denotes the image of $\xi$ in $\pmf$. 

By definition, the family of $\mcg$-equivariant measures $\{\lambda_x\}_{x\in \T}$ are mutually absolutely continuous, and   are related as follows
$$\mu_y-\textrm{a.e. }  \xi\in \pmf:\quad \frac{d\lambda_x}{d\lambda_y}(\xi)=e^{-\e G B_{\xi}(x,y)}$$
where $\e G=6g-6$ and   the   cocycle  $B_{\xi}: \T\times \T\to \mathbb R$ is is given by
\begin{equation}\label{BCocycleEQ}
\forall  x,y\in \T,\quad B_{\xi}(x,y) :=\frac{1}{2}\cdot\log  \frac{{\ext_{x}(\hat \xi)}}{{\ext_{y}(\hat \xi)}}.
\end{equation}
where   $\hat\xi\in \mf$ is any lift of  $\xi$.   By \cite[Theorem 1]{Walsh19}, one can verify that for $\xi\in \ue$, $B_\xi(\cdot,\cdot)$ is indeed the Buseman cocycle at $\xi$ in the horofunction boundary $\bGM$ (see \cite{TYANG} for details). By (\ref{ThGMHomeoEQ}), the Buseman cocycles extend continuously to $\ue\subseteq \bTh$, so \ref{AssumpE} is satisfied with $\epsilon=0$. 

It is well-known that $\lambda_x$ is supported on $\ue$ (\cite{Ma82a, V82}).  In terms of Definition \ref{ConformalDensityDefn},  $\{\lambda_x\}_{x\in \T}$ is a $(6g-6)$-dimensional $\mcg$-equivariant conformal density on $\bTh$. 
The shadow lemma \ref{ShadowLem} and   principle \ref{ShadowPrinciple} thus hold in this context, where the former has been already obtained in \cite{TYANG} via a different method. 

We emphasize the above construction is not obtained via the Patterson's construction. However, the action $\mcg \act \T$  has purely exponential growth (\cite[Thm 1.2]{ABEM}), hence it  is of divergent type. By Theorem \ref{UniqueConfThm},  there exists a unique    $(6g-6)$-dimensional conformal density on $\bTh$, which was first proved in \cite{H09b, LM08}. Thus, the conformal density obtained by Patterson's construction coincides with $\lambda_x$.



We record the following result for further counting applications.

\begin{lem}\label{ModNullityLem}
Let $G<\mcg$ be a non-elementary subgroup with pseudo-Anosove elements. Assume that  the limit set $\Lambda G$ of $G$ is a proper subset of $\pmf$. Then   its Poincar\'e series $\p_G(s, o, o)$ is convergent at $6g-6$. 

In particular,  the Thurston measure of the conical limit set $\Lambda_c(G)$ defined in (\ref{ConicalEQ}) is zero.     
\end{lem}
\begin{rem}
Analogous to the  Ahlfors area theorem in Kleinian groups, it would be interesting to study which subgroups of $\mcg$ have limit set of zero Lebesgue measure. 
If $G$ is a convex-cocompact subgroup, then  the conical limit set $\Lambda_c(G)$ is the whole limit set, so we recover   \cite[Corollary 4.8]{KeLe2}. The limit set of geometrically finite Kleinian groups consists of conical points and countable many parabolic points. If the Lebesgue measure has no charge on conical limit points, then the whole limit set is null. We expect the ``in particular" statement helpful to solve the zero measure problem for certain limit sets.   
\end{rem}
\begin{proof}
A non-elementary subgroup contains at least two independent pseudo-Anosov elements, so is sufficiently large in terms of \cite{McPapa}. With respect to Kaimanovich--Masur $[\cdot]$-partition, any pseudo-Anosov element $h$ has minimal fixed points $[h^-]=\{\xi\}$ for some $\xi\in \ue$. By \cite{McPapa}, $G$ possesses a unique $G$-invariant minimal closed subset $\Lambda\subseteq \pmf$ defined as the closure of all fixed points of pseudo-Anosov elements (see also Lemma \ref{UniqueLimitSet}). Let $Z(\Lambda):=\{\eta\in \pmf: I(\xi, \eta)=0 \text{ for some }\xi\in \Lambda\}$.  Then $\pG \subseteq Z(\Lambda)$ is proved by \cite[Prop 8.1]{McPapa}. 

By \cite[Prop 6.1, Thm 6.16]{McPapa}, if  $\Lambda \ne\pmf$, then $Z(\Lambda) \ne\pmf$ is  closed and   $G$ acts properly discontinuously on the complement $\Omega:=\pmf\setminus Z(\Lambda)$.  By Lemma \ref{DoubleDense}, the  fixed point pairs of all pseudo-Anosov elements are dense in distinct pairs of $\pmf$.  Choose  a pseudo-Anosov element $h\in \mcg$ and  a closed neighborhood $U$ of $h^\pm$ such that $U$ is contained in $\Omega$. By the proper action of $G\act \Omega$, 
the set $A:=\{1\ne g\in G: U\cap gU\ne\emptyset\}$ is finite. As $G<\mcg$ is residually finite, let us choose a finite index subgroup  $\dot G<G$ such that $\dot G\cap A=\emptyset.$ Thus, any nontrivial element $g\in \dot G$ sends $U$ into $V:=\pmf\setminus U$.

By the Nouth-South dynamics, there exists  
 a high power of $h$ (with the same notation) such that $h^n(U)\subseteq V$ for any $n\ne 0$. Thus, $\langle h\rangle$ and $ \dot G$ are ping-pong players on $(U, V)$, so   generate a free product $\langle h \rangle\star \dot G$. By \cite[Lemma 2.23]{YANG10}, the Poincar\'e series of $\dot G$ converges at $\e {\mcg}=6g-6$. If $G\act \T$ is of divergent type, we have $\e G=\e {\dot G}<6g-6.$ 
 
By Theorem \ref{HTSModThm}, we have the Thurston measure of the conical limit set of $G$ is null.  
\end{proof}

For further reference, we record the following  corollary of Theorem \ref{UniqueConfThm} for Gardiner--Masur boundary.
\begin{lem}
Up to scaling, the Gardiner--Masur boundary $\bGM$ admits a unique and ergodic $\mcg$-equivariant conformal density of dimension $6g-6$, which is  supported on the Myrberg  set (a subset of uniquely ergodic points).
\end{lem}

\subsection{Bowen--Margulis--Sullivan measure on Teichm\"uller geodesic flow}\label{SecTeichFlow}
Let $G$ be a non-elementary subgroup of $\mathrm{Mod}(\Sigma_g)$ with pseudo-Anosov elements, so that $\M:=\T/G$ is a (branched) cover of the moduli space of $\Sigma_g$. Let $\{\mu_x\}_{x\in \T}$ be a $\e G$-dimensional  conformal density on $\bTh=\pmf$ in Theorem \ref{HTSModThm}. 

By Theorem \ref{HTSThm}, it suffices to prove  the direction ``(3) $\Longleftrightarrow$ (4)" of Theorem \ref{HTSModThm}. This requires to run the well-known Hopf's argument which almost follows verbatim the case for CAT(-1) spaces  \cite{H71, N89, BuMo}. What follows sets up the necessary background, and then points out the required ingredients.

Let   $\pi: \QD\to \T$ be the vector bundle of quadratic differentials over Teichm\"uller space.   We first explain   a  $G$-invariant geodesic flow on  the   sub-bundle $\UQT$ of unit area quadratic differentials, and then construct a flow invariant measure called Bowen--Margulis--Sullivan measure. This construction is due to Sullivan in hyperbolic spaces.

Teichm\"uller geodesics  $\gamma$  are uniquely determined by a pair of projective classes $(\check\xi, \check \eta )$ of {transversal} measured foliations $(\xi, \eta)$, and vice versa. Such pairs $(\xi, \eta)$ form exactly the complement of the following big diagonal
$$\Delta \;:=\; \{(\xi, \eta)\in \mf\times \mf:\quad \exists c\in \s,\; I(c, \xi)+I(c, \eta)=0\}$$
According to Hubbard-Masur theorem \cite{HM79}, the transversal pair $(\xi, \eta)$ is  realized as the corresponding  vertical and  horizontal foliations $(q^+, q^-)$ of a unique quadratic differential $q\in \QD$.  According to  \cite{LM08}, we consider the following $G$-equivariant homeomorphism 
$$
\begin{array}{rl}
    \QD \quad & \longrightarrow\quad \mf\times \mf \setminus \Delta\\
    q  \quad & \longmapsto \quad (q^+, q^-).
\end{array}
$$
\begin{rem}
By abuse of language, we write $\gamma=(\check \xi, \check \eta)=(q^+,q^-)$ and $\gamma^+=\check \xi, \gamma^-=\check \eta$. Indeed, the two half rays of $\gamma$ do not necessarily converge to $\check \xi, \check \eta$ in $\T\cup \bTh$, which however do converge for $\xi, \eta\in \ue$.  
\end{rem}

The geodesic flow $(\mathcal G^t)_{t\in \mathbb R}$ on $\UQT$ is defined as $$\mathcal G^t(q):=(e^tq^+, e^{-t}q^-)$$ for any $q\in \UQT$ and $t\in \mathbb R$. The flow line $t\mapsto q_t:=\mathcal G^t(q)$ gives  the   lift of the Teichm\"uller geodesic $\gamma=(q_0^+,q_0^-)$ in the moduli space $\QD$ of quadratic differentials. Endow $\UQT$  with a $\mcg$-invariant distance $d_Q$ as follows:
\begin{equation}\label{MetricOnFlowlines}
d_Q(q, q'): =\int_{-\infty}^{+\infty} \frac{d_T(\pi(q_t), \pi(q_t'))}{2e^t}dt    
\end{equation}
which descends to  the quotient metric  still denoted by  $d_Q$ on $\UQM:=\UQT/G$. 

\subsubsection*{Hopf parameterization of Teichm\"uller flows}
Set $$\dT:=\Big((\pmf\times \pmf)\setminus \Delta\Big).$$ Fix a basepoint $o\in \T$,  the Hopf parameterization of  $\UQT$ is given by the following $G$-equivariant homeomorphism 
$$
\begin{array}{ccc}
    \UQT & \longrightarrow & \dT \times \mathbb R\\
    q & \longmapsto & (\xi, \eta, B_{\xi}(o, \pi(q)))
\end{array}
$$
where  $B_{\xi}(\cdot, \cdot)$ is the cocycle in (\ref{BCocycleEQ}). The $G$-action on the target is given by 
$$
\forall g\in G:\; g(q^-,q^+, s)=(gq^-,gq^+, s+B_{\xi}(g^{-1}o, o)).
$$
As a consequence, the geodesic flow $(\mathcal G^t)_{t\in \mathbb R}$ on $\UQT$ is conjugate to the additive action on the time factor of $ \dT \times \mathbb R$ as follows:
$$(\xi, \eta, s)\longmapsto (\xi,\eta, t+s).$$

Following \cite[Section 2.3.1]{ABEM}, define 
$$
\beta(x, \xi, \eta) := \big[e^{B_\xi(x, p)+B_\eta(x, p)}\big]^{6g-6}= {\left [\frac{ \sqrt{\ext_x(\hat\xi)\ext_x(\hat\eta)} }{I(\hat\xi, \hat\eta)} \right]}^{6g-6} 
$$
for any $p\in [\xi, \eta]$, where $\hat \xi, \hat \eta\in \mf$ are lifts of $\xi, \eta$, and $\ext_p(\check\xi)\ext_p(\check\eta)=I(\check\xi,\check\eta)^2$ by \cite[Thm 5.1]{GM91}. Define  the \textit{Bowen--Margulis--Sullivan measure} $ m$ on $\dT \times \mathbb R$  for any $x\in \T$,
$$
m:=\beta(x, \xi, \eta)^{-1} \cdot \mu_x \times \mu_x \times \textrm{Leb}
$$
which is independent of  $x$.  As $m$ is $G$-invariant,  it descends to the  flow invariant measure $\mathbf m$ on $\UQM$.

\subsection{Hopf--Tsuji--Sullivan dichotomy for geodesic flows}
We first give a brief introduction to Hopf decomposition of  a flow dynamical system. This is classical fact, for which  we  follow  the exposition from \cite{L18} and refer the reader to the more references therein.

Let $\Omega$ be a locally compact and $\sigma$-compact Hausdorff topological space. A flow $\phi$ on $\Omega$ is a continuous map $\phi(t,\omega): \mathbb R \times \Omega \to \Omega$ such that $\phi(0,\omega)=\omega$ and $\phi(t,\phi(s,\omega))=\phi(t+s,\omega)$. Write $\phi^t(\omega)=\phi(t,\omega)$. 

Assume that $\mathbf m$ is a $\phi$-invariant Borel measure on $\Omega$. By Hopf decomposition theorem, $\Omega$ admits a unique decomposition  ($\mathbf m$-mod 0) of $\phi$-invariant Borel subsets called \textit{conservative} and \textit{dissipative} components, denoted by $\Omega_C$ and $\Omega_D$, as follows:
\begin{itemize}
    \item 
    $\Omega_C$ admits no wandering set: there exists no Borel subset $W$ such that $\phi^n(W)\cap W=\emptyset$ for any $n\in\mathbb Z\setminus  0$.
    \item 
    $\Omega_C$ admits a measurable fundamental set $W$ so that  $\cup_{n\in \mathbb Z} \phi^n(W)=\Omega$. 
\end{itemize}
According to Poincar\'e's recurrence theorem, $\mathbf m$-almost every point $\omega$ of $\Omega_C$ is positively recurrent: there exists $t_n\in \mathbb R\to\infty+$ so that $\phi^{t_n}(\omega)\to\omega$. On the other hand, by Hopf's divergence theorem, $\mathbf m$-almost every point $\omega$  of $\Omega_D$ is positively divergent:  for every compact $K\subset \Omega$, there exists $T$ so that $\phi^{t}(\omega)\notin K$ for all $t>T$.

Let us return to our setup: the geodesic flow $(\UQM, \mathcal G^t, \textbf m)$ is a measure preserving dynamical system, so the above Hopf decomposition holds. The point of Theorem \ref{HTSModThm} says that a dichotomy of Hopf decomposition holds according to whether the conical points have positive measure or not.

\subsubsection*{Conservativity \texorpdfstring{$\Longleftrightarrow $}{} Full measure on conical points}

If $(\UQM, \mathcal G^t, \textbf m)$ is   {conservative}, then there is no  $\textbf m$-positive {wandering set} $W $ so that $\{\mathcal G^n W: n\in \mathbb Z\}$ is pairwise disjoint. By Poincar\'e recurrence theorem, almost every point $q\in \UQM$ is {positively recurrent}: $ \mathcal G^{t_n} q\to q$ for an unbounded sequence $t_n>0$.  

Let $\gamma$ be a Teichm\"uller geodesic ray  ending at (the projective class of) $q^+$. As $\UQM$ is a sphere bundle over $\T/G$,  we see that $\gamma$ returns infinitely often to a compact subset of $\T/G$. In other words, $q^+\in \Lambda_c(G)$ is a conical point as defined in (\ref{ConicalEQ}).  

By the dichotomy of Theorem \ref{HTSThm}, the $\mu_x$-measure of conical points is either null or full, so $\mu_x$ is fully  supported on the set of conical points and thus on the Myrberg set. 

For the converse direction, by Hopf decomposition, $(\UQM,\mathcal G^t,  \mathbf m)$  consists of the   conservative   and {dissipative components}, uniquely up to $\textbf m$-null sets. The flow lines in the dissipative component exit every compact sets, so if the conical points are of full measure, then the geodesic flow on $\UQM$ must be conservative.   

\subsubsection*{Conservativity \texorpdfstring{$\Longleftrightarrow$}{} Ergodicity of flows} 
If $(\mathcal G^t)_{t\in \mathbb R}$ on $\UQT$ is ergodic without atoms, then it is of course   conservative. The converse direction requires to run the Hopf argument. 

To that end, we need  a positive $\mathbf m$-integrable function $\Phi(x)$ on $\UQM$. If $\mathbf m(\UQM)$ is finite, we can choose $\Phi=1$. If $G=\mcg$,   $\mathbf m$ is the Masur-Veech measure which is finite. The  proof of ergodicity via Hopf argument is given in \cite[Theorem 4]{Ma82a}.

In general,  define $\Phi(x):=d_Q(x, \mathbf{o})$ for some $\mathbf o\in \UQM$.  The same argument as in \cite[Lemma 8.3.1]{N89} proves that   $\Phi(x)$ is $\mathbf m$-integrable function  and  for some $c>0$ and all $x, y\in \UQM$ with $d_Q(x,y)<1$
$$
\frac{|\Phi(x)-\Phi(y)|}{\Phi(y)} <c \cdot d_Q (x,y)
$$
where $d_Q$ is a metric on $\UQM$ defined  in (\ref{MetricOnFlowlines}).

Recall that $\mu_x$ is supported on  the Myrberg set $\mG$, which consists of   uniquely ergodic points. By \cite{Ma80}, any two geodesic rays $\alpha, \beta$ in $\T$ ending at a common unique ergodic point are asymptotic: there exists $a\in \mathbb R$ such that $d_T(\alpha(t+a), \beta(a))\to 0$ and thus for two $p, q\in \UQM$ with the same vertical foliation,
$$
d_Q(\mathcal G^{t+a} (p), \mathcal G^t (q))\to 0,\quad t\to\infty.
$$

With these ingredients, the ergodicity of the geodesic flow and thus product measures $\mu_x\times \mu_x$ follows verbatim the Hopf's argument (see \cite[Thm 8.3.2]{N89} for details).  So  Theorem \ref{HTSModThm} is proved.



\bibliographystyle{alpha}
 \bibliography{bibliography}

\end{document}